\documentclass[preprint,3p,sort&compress]{elsarticle}
\journal{{\tt Journal on Scientific Computing}}

\usepackage[dvips]{epsfig}
\usepackage{graphicx} 
\usepackage{pdfpages}
\usepackage{latexsym}
\usepackage{verbatim}
\usepackage{amsmath,amsthm}
\usepackage[utf8]{inputenc}
\usepackage{amssymb}
\usepackage{bbm}
\usepackage{fonts}
\usepackage{enumerate}
\usepackage{placeins}
\usepackage{standalone}
\usepackage{paralist}
\usepackage{subcaption}

\usepackage{diagbox}

\usepackage{algorithm,algorithmicx,algpseudocode}
\usepackage{caption}

\usepackage[]{hyperref} 

\usepackage{pgfplots}
\pgfplotsset{compat=newest}       
\usepackage[]{externaltikz}

 \usepackage{relinput}
\newtheorem{theorem}{Theorem}[section]
\newtheorem{definition}[theorem]{Definition}
\newtheorem{remark}[theorem]{Remark}
\newtheorem{example}[theorem]{Example}
\newtheorem{assumption}[theorem]{Assumption}

\newcounter{tikzsubfigcounter}[figure]
\renewcommand{\thetikzsubfigcounter}{\the\numexpr\value{figure}+1\relax\alph{tikzsubfigcounter}}

\newcounter{tikzsubfigcounterinvisible}[figure]
\renewcommand{\thetikzsubfigcounterinvisible}{\the\numexpr\value{figure}+1\relax\alph{tikzsubfigcounterinvisible}}

\usepackage{bm}

\numberwithin{equation}{section}

\title{Oscillation Mitigation of Hyperbolicity-Preserving Intrusive Uncertainty Quantification Methods for Systems of Conservation Laws}

\author[jk]{Jonas Kusch}
\address[jk]{Computational Science and Mathematical Methods, Karlsruhe Institute of Technology, Englerstraße 2, 76128 Karlsruhe, {\tt jonas.kusch@kit.edu}}

\author[ls]{Louisa Schlachter}
\address[ls]{Fachbereich Mathematik, TU Kaiserslautern, Erwin-Schr\"odinger-Str., 67663 Kaiserslautern, Germany, {\tt schlacht@mathematik.uni-kl.de}}

\date{}

\definecolor{greenyellow}   {cmyk}{0.15, 0   , 0.69, 0   }
\definecolor{yellow}        {cmyk}{0   , 0   , 1   , 0   }
\definecolor{goldenrod}     {cmyk}{0   , 0.10, 0.84, 0   }
\definecolor{dandelion}     {cmyk}{0   , 0.29, 0.84, 0   }
\definecolor{apricot}       {cmyk}{0   , 0.32, 0.52, 0   }
\definecolor{peach}         {cmyk}{0   , 0.50, 0.70, 0   }
\definecolor{melon}         {cmyk}{0   , 0.46, 0.50, 0   }
\definecolor{yelloworange}  {cmyk}{0   , 0.42, 1   , 0   }
\definecolor{orange}        {cmyk}{0   , 0.61, 0.87, 0   }
\definecolor{burntorange}   {cmyk}{0   , 0.51, 1   , 0   }
\definecolor{bittersweet}   {cmyk}{0   , 0.75, 1   , 0.24}
\definecolor{redorange}     {cmyk}{0   , 0.77, 0.87, 0   }
\definecolor{mahogany}      {cmyk}{0   , 0.85, 0.87, 0.35}
\definecolor{maroon}        {cmyk}{0   , 0.87, 0.68, 0.32}
\definecolor{brickred}      {cmyk}{0   , 0.89, 0.94, 0.28}
\definecolor{red}           {cmyk}{0   , 1   , 1   , 0   }
\definecolor{orangered}     {cmyk}{0   , 1   , 0.50, 0   }
\definecolor{rubinered}     {cmyk}{0   , 1   , 0.13, 0   }
\definecolor{wildstrawberry}{cmyk}{0   , 0.96, 0.39, 0   }
\definecolor{salmon}        {cmyk}{0   , 0.53, 0.38, 0   }
\definecolor{carnationpink} {cmyk}{0   , 0.63, 0   , 0   }
\definecolor{magenta}       {cmyk}{0   , 1   , 0   , 0   }
\definecolor{violetred}     {cmyk}{0   , 0.81, 0   , 0   }
\definecolor{rhodamine}     {cmyk}{0   , 0.82, 0   , 0   }
\definecolor{mulberry}      {cmyk}{0.34, 0.90, 0   , 0.02}
\definecolor{redviolet}     {cmyk}{0.07, 0.90, 0   , 0.34}
\definecolor{fuchsia}       {cmyk}{0.47, 0.91, 0   , 0.08}
\definecolor{lavender}      {cmyk}{0   , 0.48, 0   , 0   }
\definecolor{thistle}       {cmyk}{0.12, 0.59, 0   , 0   }
\definecolor{orchid}        {cmyk}{0.32, 0.64, 0   , 0   }
\definecolor{darkorchid}    {cmyk}{0.40, 0.80, 0.20, 0   }
\definecolor{purple}        {cmyk}{0.45, 0.86, 0   , 0   }
\definecolor{plum}          {cmyk}{0.50, 1   , 0   , 0   }
\definecolor{violet}        {cmyk}{0.79, 0.88, 0   , 0   }
\definecolor{royalpurple}   {cmyk}{0.75, 0.90, 0   , 0   }
\definecolor{blueviolet}    {cmyk}{0.86, 0.91, 0   , 0.04}
\definecolor{periwinkle}    {cmyk}{0.57, 0.55, 0   , 0   }
\definecolor{cadetblue}     {cmyk}{0.62, 0.57, 0.23, 0   }
\definecolor{cornflowerblue}{cmyk}{0.65, 0.13, 0   , 0   }
\definecolor{midnightblue}  {cmyk}{0.98, 0.13, 0   , 0.43}
\definecolor{navyblue}      {cmyk}{0.94, 0.54, 0   , 0   }
\definecolor{royalblue}     {cmyk}{1   , 0.50, 0   , 0   }
\definecolor{blue}          {cmyk}{1   , 1   , 0   , 0   }
\definecolor{cerulean}      {cmyk}{0.94, 0.11, 0   , 0   }
\definecolor{cyan}          {cmyk}{1   , 0   , 0   , 0   }
\definecolor{processblue}   {cmyk}{0.96, 0   , 0   , 0   }
\definecolor{skyblue}       {cmyk}{0.62, 0   , 0.12, 0   }
\definecolor{turquoise}     {cmyk}{0.85, 0   , 0.20, 0   }
\definecolor{tealblue}      {cmyk}{0.86, 0   , 0.34, 0.02}
\definecolor{aquamarine}    {cmyk}{0.82, 0   , 0.30, 0   }
\definecolor{bluegreen}     {cmyk}{0.85, 0   , 0.33, 0   }
\definecolor{emerald}       {cmyk}{1   , 0   , 0.50, 0   }
\definecolor{junglegreen}   {cmyk}{0.99, 0   , 0.52, 0   }
\definecolor{seagreen}      {cmyk}{0.69, 0   , 0.50, 0   }
\definecolor{green}         {cmyk}{1   , 0   , 1   , 0   }
\definecolor{forestgreen}   {cmyk}{0.91, 0   , 0.88, 0.12}
\definecolor{pinegreen}     {cmyk}{0.92, 0   , 0.59, 0.25}
\definecolor{limegreen}     {cmyk}{0.50, 0   , 1   , 0   }
\definecolor{yellowgreen}   {cmyk}{0.44, 0   , 0.74, 0   }
\definecolor{springgreen}   {cmyk}{0.26, 0   , 0.76, 0   }
\definecolor{olivegreen}    {cmyk}{0.64, 0   , 0.95, 0.40}
\definecolor{rawsienna}     {cmyk}{0   , 0.72, 1   , 0.45}
\definecolor{sepia}         {cmyk}{0   , 0.83, 1   , 0.70}
\definecolor{brown}         {cmyk}{0   , 0.81, 1   , 0.60}
\definecolor{tan}           {cmyk}{0.14, 0.42, 0.56, 0   }
\definecolor{gray}          {cmyk}{0   , 0   , 0   , 0.50}
\definecolor{black}         {cmyk}{0   , 0   , 0   , 1   }
\definecolor{white}         {cmyk}{0   , 0   , 0   , 0   } 

\usepgfplotslibrary{groupplots}
\usepackage{pgfplotstable}

\pgfplotsset{
	colormap={jet}{
rgb(0.000000 pt)=(0.000000,0.000000,0.504000);
rgb(1.000000 pt)=(0.000000,0.000000,0.508000);
rgb(2.000000 pt)=(0.000000,0.000000,0.512000);
rgb(3.000000 pt)=(0.000000,0.000000,0.516000);
rgb(4.000000 pt)=(0.000000,0.000000,0.520000);
rgb(5.000000 pt)=(0.000000,0.000000,0.524000);
rgb(6.000000 pt)=(0.000000,0.000000,0.528000);
rgb(7.000000 pt)=(0.000000,0.000000,0.532000);
rgb(8.000000 pt)=(0.000000,0.000000,0.536000);
rgb(9.000000 pt)=(0.000000,0.000000,0.540000);
rgb(10.000000 pt)=(0.000000,0.000000,0.544000);
rgb(11.000000 pt)=(0.000000,0.000000,0.548000);
rgb(12.000000 pt)=(0.000000,0.000000,0.552000);
rgb(13.000000 pt)=(0.000000,0.000000,0.556000);
rgb(14.000000 pt)=(0.000000,0.000000,0.560000);
rgb(15.000000 pt)=(0.000000,0.000000,0.564000);
rgb(16.000000 pt)=(0.000000,0.000000,0.568000);
rgb(17.000000 pt)=(0.000000,0.000000,0.572000);
rgb(18.000000 pt)=(0.000000,0.000000,0.576000);
rgb(19.000000 pt)=(0.000000,0.000000,0.580000);
rgb(20.000000 pt)=(0.000000,0.000000,0.584000);
rgb(21.000000 pt)=(0.000000,0.000000,0.588000);
rgb(22.000000 pt)=(0.000000,0.000000,0.592000);
rgb(23.000000 pt)=(0.000000,0.000000,0.596000);
rgb(24.000000 pt)=(0.000000,0.000000,0.600000);
rgb(25.000000 pt)=(0.000000,0.000000,0.604000);
rgb(26.000000 pt)=(0.000000,0.000000,0.608000);
rgb(27.000000 pt)=(0.000000,0.000000,0.612000);
rgb(28.000000 pt)=(0.000000,0.000000,0.616000);
rgb(29.000000 pt)=(0.000000,0.000000,0.620000);
rgb(30.000000 pt)=(0.000000,0.000000,0.624000);
rgb(31.000000 pt)=(0.000000,0.000000,0.628000);
rgb(32.000000 pt)=(0.000000,0.000000,0.632000);
rgb(33.000000 pt)=(0.000000,0.000000,0.636000);
rgb(34.000000 pt)=(0.000000,0.000000,0.640000);
rgb(35.000000 pt)=(0.000000,0.000000,0.644000);
rgb(36.000000 pt)=(0.000000,0.000000,0.648000);
rgb(37.000000 pt)=(0.000000,0.000000,0.652000);
rgb(38.000000 pt)=(0.000000,0.000000,0.656000);
rgb(39.000000 pt)=(0.000000,0.000000,0.660000);
rgb(40.000000 pt)=(0.000000,0.000000,0.664000);
rgb(41.000000 pt)=(0.000000,0.000000,0.668000);
rgb(42.000000 pt)=(0.000000,0.000000,0.672000);
rgb(43.000000 pt)=(0.000000,0.000000,0.676000);
rgb(44.000000 pt)=(0.000000,0.000000,0.680000);
rgb(45.000000 pt)=(0.000000,0.000000,0.684000);
rgb(46.000000 pt)=(0.000000,0.000000,0.688000);
rgb(47.000000 pt)=(0.000000,0.000000,0.692000);
rgb(48.000000 pt)=(0.000000,0.000000,0.696000);
rgb(49.000000 pt)=(0.000000,0.000000,0.700000);
rgb(50.000000 pt)=(0.000000,0.000000,0.704000);
rgb(51.000000 pt)=(0.000000,0.000000,0.708000);
rgb(52.000000 pt)=(0.000000,0.000000,0.712000);
rgb(53.000000 pt)=(0.000000,0.000000,0.716000);
rgb(54.000000 pt)=(0.000000,0.000000,0.720000);
rgb(55.000000 pt)=(0.000000,0.000000,0.724000);
rgb(56.000000 pt)=(0.000000,0.000000,0.728000);
rgb(57.000000 pt)=(0.000000,0.000000,0.732000);
rgb(58.000000 pt)=(0.000000,0.000000,0.736000);
rgb(59.000000 pt)=(0.000000,0.000000,0.740000);
rgb(60.000000 pt)=(0.000000,0.000000,0.744000);
rgb(61.000000 pt)=(0.000000,0.000000,0.748000);
rgb(62.000000 pt)=(0.000000,0.000000,0.752000);
rgb(63.000000 pt)=(0.000000,0.000000,0.756000);
rgb(64.000000 pt)=(0.000000,0.000000,0.760000);
rgb(65.000000 pt)=(0.000000,0.000000,0.764000);
rgb(66.000000 pt)=(0.000000,0.000000,0.768000);
rgb(67.000000 pt)=(0.000000,0.000000,0.772000);
rgb(68.000000 pt)=(0.000000,0.000000,0.776000);
rgb(69.000000 pt)=(0.000000,0.000000,0.780000);
rgb(70.000000 pt)=(0.000000,0.000000,0.784000);
rgb(71.000000 pt)=(0.000000,0.000000,0.788000);
rgb(72.000000 pt)=(0.000000,0.000000,0.792000);
rgb(73.000000 pt)=(0.000000,0.000000,0.796000);
rgb(74.000000 pt)=(0.000000,0.000000,0.800000);
rgb(75.000000 pt)=(0.000000,0.000000,0.804000);
rgb(76.000000 pt)=(0.000000,0.000000,0.808000);
rgb(77.000000 pt)=(0.000000,0.000000,0.812000);
rgb(78.000000 pt)=(0.000000,0.000000,0.816000);
rgb(79.000000 pt)=(0.000000,0.000000,0.820000);
rgb(80.000000 pt)=(0.000000,0.000000,0.824000);
rgb(81.000000 pt)=(0.000000,0.000000,0.828000);
rgb(82.000000 pt)=(0.000000,0.000000,0.832000);
rgb(83.000000 pt)=(0.000000,0.000000,0.836000);
rgb(84.000000 pt)=(0.000000,0.000000,0.840000);
rgb(85.000000 pt)=(0.000000,0.000000,0.844000);
rgb(86.000000 pt)=(0.000000,0.000000,0.848000);
rgb(87.000000 pt)=(0.000000,0.000000,0.852000);
rgb(88.000000 pt)=(0.000000,0.000000,0.856000);
rgb(89.000000 pt)=(0.000000,0.000000,0.860000);
rgb(90.000000 pt)=(0.000000,0.000000,0.864000);
rgb(91.000000 pt)=(0.000000,0.000000,0.868000);
rgb(92.000000 pt)=(0.000000,0.000000,0.872000);
rgb(93.000000 pt)=(0.000000,0.000000,0.876000);
rgb(94.000000 pt)=(0.000000,0.000000,0.880000);
rgb(95.000000 pt)=(0.000000,0.000000,0.884000);
rgb(96.000000 pt)=(0.000000,0.000000,0.888000);
rgb(97.000000 pt)=(0.000000,0.000000,0.892000);
rgb(98.000000 pt)=(0.000000,0.000000,0.896000);
rgb(99.000000 pt)=(0.000000,0.000000,0.900000);
rgb(100.000000 pt)=(0.000000,0.000000,0.904000);
rgb(101.000000 pt)=(0.000000,0.000000,0.908000);
rgb(102.000000 pt)=(0.000000,0.000000,0.912000);
rgb(103.000000 pt)=(0.000000,0.000000,0.916000);
rgb(104.000000 pt)=(0.000000,0.000000,0.920000);
rgb(105.000000 pt)=(0.000000,0.000000,0.924000);
rgb(106.000000 pt)=(0.000000,0.000000,0.928000);
rgb(107.000000 pt)=(0.000000,0.000000,0.932000);
rgb(108.000000 pt)=(0.000000,0.000000,0.936000);
rgb(109.000000 pt)=(0.000000,0.000000,0.940000);
rgb(110.000000 pt)=(0.000000,0.000000,0.944000);
rgb(111.000000 pt)=(0.000000,0.000000,0.948000);
rgb(112.000000 pt)=(0.000000,0.000000,0.952000);
rgb(113.000000 pt)=(0.000000,0.000000,0.956000);
rgb(114.000000 pt)=(0.000000,0.000000,0.960000);
rgb(115.000000 pt)=(0.000000,0.000000,0.964000);
rgb(116.000000 pt)=(0.000000,0.000000,0.968000);
rgb(117.000000 pt)=(0.000000,0.000000,0.972000);
rgb(118.000000 pt)=(0.000000,0.000000,0.976000);
rgb(119.000000 pt)=(0.000000,0.000000,0.980000);
rgb(120.000000 pt)=(0.000000,0.000000,0.984000);
rgb(121.000000 pt)=(0.000000,0.000000,0.988000);
rgb(122.000000 pt)=(0.000000,0.000000,0.992000);
rgb(123.000000 pt)=(0.000000,0.000000,0.996000);
rgb(124.000000 pt)=(0.000000,0.000000,1.000000);
rgb(125.000000 pt)=(0.000000,0.004000,1.000000);
rgb(126.000000 pt)=(0.000000,0.008000,1.000000);
rgb(127.000000 pt)=(0.000000,0.012000,1.000000);
rgb(128.000000 pt)=(0.000000,0.016000,1.000000);
rgb(129.000000 pt)=(0.000000,0.020000,1.000000);
rgb(130.000000 pt)=(0.000000,0.024000,1.000000);
rgb(131.000000 pt)=(0.000000,0.028000,1.000000);
rgb(132.000000 pt)=(0.000000,0.032000,1.000000);
rgb(133.000000 pt)=(0.000000,0.036000,1.000000);
rgb(134.000000 pt)=(0.000000,0.040000,1.000000);
rgb(135.000000 pt)=(0.000000,0.044000,1.000000);
rgb(136.000000 pt)=(0.000000,0.048000,1.000000);
rgb(137.000000 pt)=(0.000000,0.052000,1.000000);
rgb(138.000000 pt)=(0.000000,0.056000,1.000000);
rgb(139.000000 pt)=(0.000000,0.060000,1.000000);
rgb(140.000000 pt)=(0.000000,0.064000,1.000000);
rgb(141.000000 pt)=(0.000000,0.068000,1.000000);
rgb(142.000000 pt)=(0.000000,0.072000,1.000000);
rgb(143.000000 pt)=(0.000000,0.076000,1.000000);
rgb(144.000000 pt)=(0.000000,0.080000,1.000000);
rgb(145.000000 pt)=(0.000000,0.084000,1.000000);
rgb(146.000000 pt)=(0.000000,0.088000,1.000000);
rgb(147.000000 pt)=(0.000000,0.092000,1.000000);
rgb(148.000000 pt)=(0.000000,0.096000,1.000000);
rgb(149.000000 pt)=(0.000000,0.100000,1.000000);
rgb(150.000000 pt)=(0.000000,0.104000,1.000000);
rgb(151.000000 pt)=(0.000000,0.108000,1.000000);
rgb(152.000000 pt)=(0.000000,0.112000,1.000000);
rgb(153.000000 pt)=(0.000000,0.116000,1.000000);
rgb(154.000000 pt)=(0.000000,0.120000,1.000000);
rgb(155.000000 pt)=(0.000000,0.124000,1.000000);
rgb(156.000000 pt)=(0.000000,0.128000,1.000000);
rgb(157.000000 pt)=(0.000000,0.132000,1.000000);
rgb(158.000000 pt)=(0.000000,0.136000,1.000000);
rgb(159.000000 pt)=(0.000000,0.140000,1.000000);
rgb(160.000000 pt)=(0.000000,0.144000,1.000000);
rgb(161.000000 pt)=(0.000000,0.148000,1.000000);
rgb(162.000000 pt)=(0.000000,0.152000,1.000000);
rgb(163.000000 pt)=(0.000000,0.156000,1.000000);
rgb(164.000000 pt)=(0.000000,0.160000,1.000000);
rgb(165.000000 pt)=(0.000000,0.164000,1.000000);
rgb(166.000000 pt)=(0.000000,0.168000,1.000000);
rgb(167.000000 pt)=(0.000000,0.172000,1.000000);
rgb(168.000000 pt)=(0.000000,0.176000,1.000000);
rgb(169.000000 pt)=(0.000000,0.180000,1.000000);
rgb(170.000000 pt)=(0.000000,0.184000,1.000000);
rgb(171.000000 pt)=(0.000000,0.188000,1.000000);
rgb(172.000000 pt)=(0.000000,0.192000,1.000000);
rgb(173.000000 pt)=(0.000000,0.196000,1.000000);
rgb(174.000000 pt)=(0.000000,0.200000,1.000000);
rgb(175.000000 pt)=(0.000000,0.204000,1.000000);
rgb(176.000000 pt)=(0.000000,0.208000,1.000000);
rgb(177.000000 pt)=(0.000000,0.212000,1.000000);
rgb(178.000000 pt)=(0.000000,0.216000,1.000000);
rgb(179.000000 pt)=(0.000000,0.220000,1.000000);
rgb(180.000000 pt)=(0.000000,0.224000,1.000000);
rgb(181.000000 pt)=(0.000000,0.228000,1.000000);
rgb(182.000000 pt)=(0.000000,0.232000,1.000000);
rgb(183.000000 pt)=(0.000000,0.236000,1.000000);
rgb(184.000000 pt)=(0.000000,0.240000,1.000000);
rgb(185.000000 pt)=(0.000000,0.244000,1.000000);
rgb(186.000000 pt)=(0.000000,0.248000,1.000000);
rgb(187.000000 pt)=(0.000000,0.252000,1.000000);
rgb(188.000000 pt)=(0.000000,0.256000,1.000000);
rgb(189.000000 pt)=(0.000000,0.260000,1.000000);
rgb(190.000000 pt)=(0.000000,0.264000,1.000000);
rgb(191.000000 pt)=(0.000000,0.268000,1.000000);
rgb(192.000000 pt)=(0.000000,0.272000,1.000000);
rgb(193.000000 pt)=(0.000000,0.276000,1.000000);
rgb(194.000000 pt)=(0.000000,0.280000,1.000000);
rgb(195.000000 pt)=(0.000000,0.284000,1.000000);
rgb(196.000000 pt)=(0.000000,0.288000,1.000000);
rgb(197.000000 pt)=(0.000000,0.292000,1.000000);
rgb(198.000000 pt)=(0.000000,0.296000,1.000000);
rgb(199.000000 pt)=(0.000000,0.300000,1.000000);
rgb(200.000000 pt)=(0.000000,0.304000,1.000000);
rgb(201.000000 pt)=(0.000000,0.308000,1.000000);
rgb(202.000000 pt)=(0.000000,0.312000,1.000000);
rgb(203.000000 pt)=(0.000000,0.316000,1.000000);
rgb(204.000000 pt)=(0.000000,0.320000,1.000000);
rgb(205.000000 pt)=(0.000000,0.324000,1.000000);
rgb(206.000000 pt)=(0.000000,0.328000,1.000000);
rgb(207.000000 pt)=(0.000000,0.332000,1.000000);
rgb(208.000000 pt)=(0.000000,0.336000,1.000000);
rgb(209.000000 pt)=(0.000000,0.340000,1.000000);
rgb(210.000000 pt)=(0.000000,0.344000,1.000000);
rgb(211.000000 pt)=(0.000000,0.348000,1.000000);
rgb(212.000000 pt)=(0.000000,0.352000,1.000000);
rgb(213.000000 pt)=(0.000000,0.356000,1.000000);
rgb(214.000000 pt)=(0.000000,0.360000,1.000000);
rgb(215.000000 pt)=(0.000000,0.364000,1.000000);
rgb(216.000000 pt)=(0.000000,0.368000,1.000000);
rgb(217.000000 pt)=(0.000000,0.372000,1.000000);
rgb(218.000000 pt)=(0.000000,0.376000,1.000000);
rgb(219.000000 pt)=(0.000000,0.380000,1.000000);
rgb(220.000000 pt)=(0.000000,0.384000,1.000000);
rgb(221.000000 pt)=(0.000000,0.388000,1.000000);
rgb(222.000000 pt)=(0.000000,0.392000,1.000000);
rgb(223.000000 pt)=(0.000000,0.396000,1.000000);
rgb(224.000000 pt)=(0.000000,0.400000,1.000000);
rgb(225.000000 pt)=(0.000000,0.404000,1.000000);
rgb(226.000000 pt)=(0.000000,0.408000,1.000000);
rgb(227.000000 pt)=(0.000000,0.412000,1.000000);
rgb(228.000000 pt)=(0.000000,0.416000,1.000000);
rgb(229.000000 pt)=(0.000000,0.420000,1.000000);
rgb(230.000000 pt)=(0.000000,0.424000,1.000000);
rgb(231.000000 pt)=(0.000000,0.428000,1.000000);
rgb(232.000000 pt)=(0.000000,0.432000,1.000000);
rgb(233.000000 pt)=(0.000000,0.436000,1.000000);
rgb(234.000000 pt)=(0.000000,0.440000,1.000000);
rgb(235.000000 pt)=(0.000000,0.444000,1.000000);
rgb(236.000000 pt)=(0.000000,0.448000,1.000000);
rgb(237.000000 pt)=(0.000000,0.452000,1.000000);
rgb(238.000000 pt)=(0.000000,0.456000,1.000000);
rgb(239.000000 pt)=(0.000000,0.460000,1.000000);
rgb(240.000000 pt)=(0.000000,0.464000,1.000000);
rgb(241.000000 pt)=(0.000000,0.468000,1.000000);
rgb(242.000000 pt)=(0.000000,0.472000,1.000000);
rgb(243.000000 pt)=(0.000000,0.476000,1.000000);
rgb(244.000000 pt)=(0.000000,0.480000,1.000000);
rgb(245.000000 pt)=(0.000000,0.484000,1.000000);
rgb(246.000000 pt)=(0.000000,0.488000,1.000000);
rgb(247.000000 pt)=(0.000000,0.492000,1.000000);
rgb(248.000000 pt)=(0.000000,0.496000,1.000000);
rgb(249.000000 pt)=(0.000000,0.500000,1.000000);
rgb(250.000000 pt)=(0.000000,0.504000,1.000000);
rgb(251.000000 pt)=(0.000000,0.508000,1.000000);
rgb(252.000000 pt)=(0.000000,0.512000,1.000000);
rgb(253.000000 pt)=(0.000000,0.516000,1.000000);
rgb(254.000000 pt)=(0.000000,0.520000,1.000000);
rgb(255.000000 pt)=(0.000000,0.524000,1.000000);
rgb(256.000000 pt)=(0.000000,0.528000,1.000000);
rgb(257.000000 pt)=(0.000000,0.532000,1.000000);
rgb(258.000000 pt)=(0.000000,0.536000,1.000000);
rgb(259.000000 pt)=(0.000000,0.540000,1.000000);
rgb(260.000000 pt)=(0.000000,0.544000,1.000000);
rgb(261.000000 pt)=(0.000000,0.548000,1.000000);
rgb(262.000000 pt)=(0.000000,0.552000,1.000000);
rgb(263.000000 pt)=(0.000000,0.556000,1.000000);
rgb(264.000000 pt)=(0.000000,0.560000,1.000000);
rgb(265.000000 pt)=(0.000000,0.564000,1.000000);
rgb(266.000000 pt)=(0.000000,0.568000,1.000000);
rgb(267.000000 pt)=(0.000000,0.572000,1.000000);
rgb(268.000000 pt)=(0.000000,0.576000,1.000000);
rgb(269.000000 pt)=(0.000000,0.580000,1.000000);
rgb(270.000000 pt)=(0.000000,0.584000,1.000000);
rgb(271.000000 pt)=(0.000000,0.588000,1.000000);
rgb(272.000000 pt)=(0.000000,0.592000,1.000000);
rgb(273.000000 pt)=(0.000000,0.596000,1.000000);
rgb(274.000000 pt)=(0.000000,0.600000,1.000000);
rgb(275.000000 pt)=(0.000000,0.604000,1.000000);
rgb(276.000000 pt)=(0.000000,0.608000,1.000000);
rgb(277.000000 pt)=(0.000000,0.612000,1.000000);
rgb(278.000000 pt)=(0.000000,0.616000,1.000000);
rgb(279.000000 pt)=(0.000000,0.620000,1.000000);
rgb(280.000000 pt)=(0.000000,0.624000,1.000000);
rgb(281.000000 pt)=(0.000000,0.628000,1.000000);
rgb(282.000000 pt)=(0.000000,0.632000,1.000000);
rgb(283.000000 pt)=(0.000000,0.636000,1.000000);
rgb(284.000000 pt)=(0.000000,0.640000,1.000000);
rgb(285.000000 pt)=(0.000000,0.644000,1.000000);
rgb(286.000000 pt)=(0.000000,0.648000,1.000000);
rgb(287.000000 pt)=(0.000000,0.652000,1.000000);
rgb(288.000000 pt)=(0.000000,0.656000,1.000000);
rgb(289.000000 pt)=(0.000000,0.660000,1.000000);
rgb(290.000000 pt)=(0.000000,0.664000,1.000000);
rgb(291.000000 pt)=(0.000000,0.668000,1.000000);
rgb(292.000000 pt)=(0.000000,0.672000,1.000000);
rgb(293.000000 pt)=(0.000000,0.676000,1.000000);
rgb(294.000000 pt)=(0.000000,0.680000,1.000000);
rgb(295.000000 pt)=(0.000000,0.684000,1.000000);
rgb(296.000000 pt)=(0.000000,0.688000,1.000000);
rgb(297.000000 pt)=(0.000000,0.692000,1.000000);
rgb(298.000000 pt)=(0.000000,0.696000,1.000000);
rgb(299.000000 pt)=(0.000000,0.700000,1.000000);
rgb(300.000000 pt)=(0.000000,0.704000,1.000000);
rgb(301.000000 pt)=(0.000000,0.708000,1.000000);
rgb(302.000000 pt)=(0.000000,0.712000,1.000000);
rgb(303.000000 pt)=(0.000000,0.716000,1.000000);
rgb(304.000000 pt)=(0.000000,0.720000,1.000000);
rgb(305.000000 pt)=(0.000000,0.724000,1.000000);
rgb(306.000000 pt)=(0.000000,0.728000,1.000000);
rgb(307.000000 pt)=(0.000000,0.732000,1.000000);
rgb(308.000000 pt)=(0.000000,0.736000,1.000000);
rgb(309.000000 pt)=(0.000000,0.740000,1.000000);
rgb(310.000000 pt)=(0.000000,0.744000,1.000000);
rgb(311.000000 pt)=(0.000000,0.748000,1.000000);
rgb(312.000000 pt)=(0.000000,0.752000,1.000000);
rgb(313.000000 pt)=(0.000000,0.756000,1.000000);
rgb(314.000000 pt)=(0.000000,0.760000,1.000000);
rgb(315.000000 pt)=(0.000000,0.764000,1.000000);
rgb(316.000000 pt)=(0.000000,0.768000,1.000000);
rgb(317.000000 pt)=(0.000000,0.772000,1.000000);
rgb(318.000000 pt)=(0.000000,0.776000,1.000000);
rgb(319.000000 pt)=(0.000000,0.780000,1.000000);
rgb(320.000000 pt)=(0.000000,0.784000,1.000000);
rgb(321.000000 pt)=(0.000000,0.788000,1.000000);
rgb(322.000000 pt)=(0.000000,0.792000,1.000000);
rgb(323.000000 pt)=(0.000000,0.796000,1.000000);
rgb(324.000000 pt)=(0.000000,0.800000,1.000000);
rgb(325.000000 pt)=(0.000000,0.804000,1.000000);
rgb(326.000000 pt)=(0.000000,0.808000,1.000000);
rgb(327.000000 pt)=(0.000000,0.812000,1.000000);
rgb(328.000000 pt)=(0.000000,0.816000,1.000000);
rgb(329.000000 pt)=(0.000000,0.820000,1.000000);
rgb(330.000000 pt)=(0.000000,0.824000,1.000000);
rgb(331.000000 pt)=(0.000000,0.828000,1.000000);
rgb(332.000000 pt)=(0.000000,0.832000,1.000000);
rgb(333.000000 pt)=(0.000000,0.836000,1.000000);
rgb(334.000000 pt)=(0.000000,0.840000,1.000000);
rgb(335.000000 pt)=(0.000000,0.844000,1.000000);
rgb(336.000000 pt)=(0.000000,0.848000,1.000000);
rgb(337.000000 pt)=(0.000000,0.852000,1.000000);
rgb(338.000000 pt)=(0.000000,0.856000,1.000000);
rgb(339.000000 pt)=(0.000000,0.860000,1.000000);
rgb(340.000000 pt)=(0.000000,0.864000,1.000000);
rgb(341.000000 pt)=(0.000000,0.868000,1.000000);
rgb(342.000000 pt)=(0.000000,0.872000,1.000000);
rgb(343.000000 pt)=(0.000000,0.876000,1.000000);
rgb(344.000000 pt)=(0.000000,0.880000,1.000000);
rgb(345.000000 pt)=(0.000000,0.884000,1.000000);
rgb(346.000000 pt)=(0.000000,0.888000,1.000000);
rgb(347.000000 pt)=(0.000000,0.892000,1.000000);
rgb(348.000000 pt)=(0.000000,0.896000,1.000000);
rgb(349.000000 pt)=(0.000000,0.900000,1.000000);
rgb(350.000000 pt)=(0.000000,0.904000,1.000000);
rgb(351.000000 pt)=(0.000000,0.908000,1.000000);
rgb(352.000000 pt)=(0.000000,0.912000,1.000000);
rgb(353.000000 pt)=(0.000000,0.916000,1.000000);
rgb(354.000000 pt)=(0.000000,0.920000,1.000000);
rgb(355.000000 pt)=(0.000000,0.924000,1.000000);
rgb(356.000000 pt)=(0.000000,0.928000,1.000000);
rgb(357.000000 pt)=(0.000000,0.932000,1.000000);
rgb(358.000000 pt)=(0.000000,0.936000,1.000000);
rgb(359.000000 pt)=(0.000000,0.940000,1.000000);
rgb(360.000000 pt)=(0.000000,0.944000,1.000000);
rgb(361.000000 pt)=(0.000000,0.948000,1.000000);
rgb(362.000000 pt)=(0.000000,0.952000,1.000000);
rgb(363.000000 pt)=(0.000000,0.956000,1.000000);
rgb(364.000000 pt)=(0.000000,0.960000,1.000000);
rgb(365.000000 pt)=(0.000000,0.964000,1.000000);
rgb(366.000000 pt)=(0.000000,0.968000,1.000000);
rgb(367.000000 pt)=(0.000000,0.972000,1.000000);
rgb(368.000000 pt)=(0.000000,0.976000,1.000000);
rgb(369.000000 pt)=(0.000000,0.980000,1.000000);
rgb(370.000000 pt)=(0.000000,0.984000,1.000000);
rgb(371.000000 pt)=(0.000000,0.988000,1.000000);
rgb(372.000000 pt)=(0.000000,0.992000,1.000000);
rgb(373.000000 pt)=(0.000000,0.996000,1.000000);
rgb(374.000000 pt)=(0.000000,1.000000,1.000000);
rgb(375.000000 pt)=(0.004000,1.000000,0.996000);
rgb(376.000000 pt)=(0.008000,1.000000,0.992000);
rgb(377.000000 pt)=(0.012000,1.000000,0.988000);
rgb(378.000000 pt)=(0.016000,1.000000,0.984000);
rgb(379.000000 pt)=(0.020000,1.000000,0.980000);
rgb(380.000000 pt)=(0.024000,1.000000,0.976000);
rgb(381.000000 pt)=(0.028000,1.000000,0.972000);
rgb(382.000000 pt)=(0.032000,1.000000,0.968000);
rgb(383.000000 pt)=(0.036000,1.000000,0.964000);
rgb(384.000000 pt)=(0.040000,1.000000,0.960000);
rgb(385.000000 pt)=(0.044000,1.000000,0.956000);
rgb(386.000000 pt)=(0.048000,1.000000,0.952000);
rgb(387.000000 pt)=(0.052000,1.000000,0.948000);
rgb(388.000000 pt)=(0.056000,1.000000,0.944000);
rgb(389.000000 pt)=(0.060000,1.000000,0.940000);
rgb(390.000000 pt)=(0.064000,1.000000,0.936000);
rgb(391.000000 pt)=(0.068000,1.000000,0.932000);
rgb(392.000000 pt)=(0.072000,1.000000,0.928000);
rgb(393.000000 pt)=(0.076000,1.000000,0.924000);
rgb(394.000000 pt)=(0.080000,1.000000,0.920000);
rgb(395.000000 pt)=(0.084000,1.000000,0.916000);
rgb(396.000000 pt)=(0.088000,1.000000,0.912000);
rgb(397.000000 pt)=(0.092000,1.000000,0.908000);
rgb(398.000000 pt)=(0.096000,1.000000,0.904000);
rgb(399.000000 pt)=(0.100000,1.000000,0.900000);
rgb(400.000000 pt)=(0.104000,1.000000,0.896000);
rgb(401.000000 pt)=(0.108000,1.000000,0.892000);
rgb(402.000000 pt)=(0.112000,1.000000,0.888000);
rgb(403.000000 pt)=(0.116000,1.000000,0.884000);
rgb(404.000000 pt)=(0.120000,1.000000,0.880000);
rgb(405.000000 pt)=(0.124000,1.000000,0.876000);
rgb(406.000000 pt)=(0.128000,1.000000,0.872000);
rgb(407.000000 pt)=(0.132000,1.000000,0.868000);
rgb(408.000000 pt)=(0.136000,1.000000,0.864000);
rgb(409.000000 pt)=(0.140000,1.000000,0.860000);
rgb(410.000000 pt)=(0.144000,1.000000,0.856000);
rgb(411.000000 pt)=(0.148000,1.000000,0.852000);
rgb(412.000000 pt)=(0.152000,1.000000,0.848000);
rgb(413.000000 pt)=(0.156000,1.000000,0.844000);
rgb(414.000000 pt)=(0.160000,1.000000,0.840000);
rgb(415.000000 pt)=(0.164000,1.000000,0.836000);
rgb(416.000000 pt)=(0.168000,1.000000,0.832000);
rgb(417.000000 pt)=(0.172000,1.000000,0.828000);
rgb(418.000000 pt)=(0.176000,1.000000,0.824000);
rgb(419.000000 pt)=(0.180000,1.000000,0.820000);
rgb(420.000000 pt)=(0.184000,1.000000,0.816000);
rgb(421.000000 pt)=(0.188000,1.000000,0.812000);
rgb(422.000000 pt)=(0.192000,1.000000,0.808000);
rgb(423.000000 pt)=(0.196000,1.000000,0.804000);
rgb(424.000000 pt)=(0.200000,1.000000,0.800000);
rgb(425.000000 pt)=(0.204000,1.000000,0.796000);
rgb(426.000000 pt)=(0.208000,1.000000,0.792000);
rgb(427.000000 pt)=(0.212000,1.000000,0.788000);
rgb(428.000000 pt)=(0.216000,1.000000,0.784000);
rgb(429.000000 pt)=(0.220000,1.000000,0.780000);
rgb(430.000000 pt)=(0.224000,1.000000,0.776000);
rgb(431.000000 pt)=(0.228000,1.000000,0.772000);
rgb(432.000000 pt)=(0.232000,1.000000,0.768000);
rgb(433.000000 pt)=(0.236000,1.000000,0.764000);
rgb(434.000000 pt)=(0.240000,1.000000,0.760000);
rgb(435.000000 pt)=(0.244000,1.000000,0.756000);
rgb(436.000000 pt)=(0.248000,1.000000,0.752000);
rgb(437.000000 pt)=(0.252000,1.000000,0.748000);
rgb(438.000000 pt)=(0.256000,1.000000,0.744000);
rgb(439.000000 pt)=(0.260000,1.000000,0.740000);
rgb(440.000000 pt)=(0.264000,1.000000,0.736000);
rgb(441.000000 pt)=(0.268000,1.000000,0.732000);
rgb(442.000000 pt)=(0.272000,1.000000,0.728000);
rgb(443.000000 pt)=(0.276000,1.000000,0.724000);
rgb(444.000000 pt)=(0.280000,1.000000,0.720000);
rgb(445.000000 pt)=(0.284000,1.000000,0.716000);
rgb(446.000000 pt)=(0.288000,1.000000,0.712000);
rgb(447.000000 pt)=(0.292000,1.000000,0.708000);
rgb(448.000000 pt)=(0.296000,1.000000,0.704000);
rgb(449.000000 pt)=(0.300000,1.000000,0.700000);
rgb(450.000000 pt)=(0.304000,1.000000,0.696000);
rgb(451.000000 pt)=(0.308000,1.000000,0.692000);
rgb(452.000000 pt)=(0.312000,1.000000,0.688000);
rgb(453.000000 pt)=(0.316000,1.000000,0.684000);
rgb(454.000000 pt)=(0.320000,1.000000,0.680000);
rgb(455.000000 pt)=(0.324000,1.000000,0.676000);
rgb(456.000000 pt)=(0.328000,1.000000,0.672000);
rgb(457.000000 pt)=(0.332000,1.000000,0.668000);
rgb(458.000000 pt)=(0.336000,1.000000,0.664000);
rgb(459.000000 pt)=(0.340000,1.000000,0.660000);
rgb(460.000000 pt)=(0.344000,1.000000,0.656000);
rgb(461.000000 pt)=(0.348000,1.000000,0.652000);
rgb(462.000000 pt)=(0.352000,1.000000,0.648000);
rgb(463.000000 pt)=(0.356000,1.000000,0.644000);
rgb(464.000000 pt)=(0.360000,1.000000,0.640000);
rgb(465.000000 pt)=(0.364000,1.000000,0.636000);
rgb(466.000000 pt)=(0.368000,1.000000,0.632000);
rgb(467.000000 pt)=(0.372000,1.000000,0.628000);
rgb(468.000000 pt)=(0.376000,1.000000,0.624000);
rgb(469.000000 pt)=(0.380000,1.000000,0.620000);
rgb(470.000000 pt)=(0.384000,1.000000,0.616000);
rgb(471.000000 pt)=(0.388000,1.000000,0.612000);
rgb(472.000000 pt)=(0.392000,1.000000,0.608000);
rgb(473.000000 pt)=(0.396000,1.000000,0.604000);
rgb(474.000000 pt)=(0.400000,1.000000,0.600000);
rgb(475.000000 pt)=(0.404000,1.000000,0.596000);
rgb(476.000000 pt)=(0.408000,1.000000,0.592000);
rgb(477.000000 pt)=(0.412000,1.000000,0.588000);
rgb(478.000000 pt)=(0.416000,1.000000,0.584000);
rgb(479.000000 pt)=(0.420000,1.000000,0.580000);
rgb(480.000000 pt)=(0.424000,1.000000,0.576000);
rgb(481.000000 pt)=(0.428000,1.000000,0.572000);
rgb(482.000000 pt)=(0.432000,1.000000,0.568000);
rgb(483.000000 pt)=(0.436000,1.000000,0.564000);
rgb(484.000000 pt)=(0.440000,1.000000,0.560000);
rgb(485.000000 pt)=(0.444000,1.000000,0.556000);
rgb(486.000000 pt)=(0.448000,1.000000,0.552000);
rgb(487.000000 pt)=(0.452000,1.000000,0.548000);
rgb(488.000000 pt)=(0.456000,1.000000,0.544000);
rgb(489.000000 pt)=(0.460000,1.000000,0.540000);
rgb(490.000000 pt)=(0.464000,1.000000,0.536000);
rgb(491.000000 pt)=(0.468000,1.000000,0.532000);
rgb(492.000000 pt)=(0.472000,1.000000,0.528000);
rgb(493.000000 pt)=(0.476000,1.000000,0.524000);
rgb(494.000000 pt)=(0.480000,1.000000,0.520000);
rgb(495.000000 pt)=(0.484000,1.000000,0.516000);
rgb(496.000000 pt)=(0.488000,1.000000,0.512000);
rgb(497.000000 pt)=(0.492000,1.000000,0.508000);
rgb(498.000000 pt)=(0.496000,1.000000,0.504000);
rgb(499.000000 pt)=(0.500000,1.000000,0.500000);
rgb(500.000000 pt)=(0.504000,1.000000,0.496000);
rgb(501.000000 pt)=(0.508000,1.000000,0.492000);
rgb(502.000000 pt)=(0.512000,1.000000,0.488000);
rgb(503.000000 pt)=(0.516000,1.000000,0.484000);
rgb(504.000000 pt)=(0.520000,1.000000,0.480000);
rgb(505.000000 pt)=(0.524000,1.000000,0.476000);
rgb(506.000000 pt)=(0.528000,1.000000,0.472000);
rgb(507.000000 pt)=(0.532000,1.000000,0.468000);
rgb(508.000000 pt)=(0.536000,1.000000,0.464000);
rgb(509.000000 pt)=(0.540000,1.000000,0.460000);
rgb(510.000000 pt)=(0.544000,1.000000,0.456000);
rgb(511.000000 pt)=(0.548000,1.000000,0.452000);
rgb(512.000000 pt)=(0.552000,1.000000,0.448000);
rgb(513.000000 pt)=(0.556000,1.000000,0.444000);
rgb(514.000000 pt)=(0.560000,1.000000,0.440000);
rgb(515.000000 pt)=(0.564000,1.000000,0.436000);
rgb(516.000000 pt)=(0.568000,1.000000,0.432000);
rgb(517.000000 pt)=(0.572000,1.000000,0.428000);
rgb(518.000000 pt)=(0.576000,1.000000,0.424000);
rgb(519.000000 pt)=(0.580000,1.000000,0.420000);
rgb(520.000000 pt)=(0.584000,1.000000,0.416000);
rgb(521.000000 pt)=(0.588000,1.000000,0.412000);
rgb(522.000000 pt)=(0.592000,1.000000,0.408000);
rgb(523.000000 pt)=(0.596000,1.000000,0.404000);
rgb(524.000000 pt)=(0.600000,1.000000,0.400000);
rgb(525.000000 pt)=(0.604000,1.000000,0.396000);
rgb(526.000000 pt)=(0.608000,1.000000,0.392000);
rgb(527.000000 pt)=(0.612000,1.000000,0.388000);
rgb(528.000000 pt)=(0.616000,1.000000,0.384000);
rgb(529.000000 pt)=(0.620000,1.000000,0.380000);
rgb(530.000000 pt)=(0.624000,1.000000,0.376000);
rgb(531.000000 pt)=(0.628000,1.000000,0.372000);
rgb(532.000000 pt)=(0.632000,1.000000,0.368000);
rgb(533.000000 pt)=(0.636000,1.000000,0.364000);
rgb(534.000000 pt)=(0.640000,1.000000,0.360000);
rgb(535.000000 pt)=(0.644000,1.000000,0.356000);
rgb(536.000000 pt)=(0.648000,1.000000,0.352000);
rgb(537.000000 pt)=(0.652000,1.000000,0.348000);
rgb(538.000000 pt)=(0.656000,1.000000,0.344000);
rgb(539.000000 pt)=(0.660000,1.000000,0.340000);
rgb(540.000000 pt)=(0.664000,1.000000,0.336000);
rgb(541.000000 pt)=(0.668000,1.000000,0.332000);
rgb(542.000000 pt)=(0.672000,1.000000,0.328000);
rgb(543.000000 pt)=(0.676000,1.000000,0.324000);
rgb(544.000000 pt)=(0.680000,1.000000,0.320000);
rgb(545.000000 pt)=(0.684000,1.000000,0.316000);
rgb(546.000000 pt)=(0.688000,1.000000,0.312000);
rgb(547.000000 pt)=(0.692000,1.000000,0.308000);
rgb(548.000000 pt)=(0.696000,1.000000,0.304000);
rgb(549.000000 pt)=(0.700000,1.000000,0.300000);
rgb(550.000000 pt)=(0.704000,1.000000,0.296000);
rgb(551.000000 pt)=(0.708000,1.000000,0.292000);
rgb(552.000000 pt)=(0.712000,1.000000,0.288000);
rgb(553.000000 pt)=(0.716000,1.000000,0.284000);
rgb(554.000000 pt)=(0.720000,1.000000,0.280000);
rgb(555.000000 pt)=(0.724000,1.000000,0.276000);
rgb(556.000000 pt)=(0.728000,1.000000,0.272000);
rgb(557.000000 pt)=(0.732000,1.000000,0.268000);
rgb(558.000000 pt)=(0.736000,1.000000,0.264000);
rgb(559.000000 pt)=(0.740000,1.000000,0.260000);
rgb(560.000000 pt)=(0.744000,1.000000,0.256000);
rgb(561.000000 pt)=(0.748000,1.000000,0.252000);
rgb(562.000000 pt)=(0.752000,1.000000,0.248000);
rgb(563.000000 pt)=(0.756000,1.000000,0.244000);
rgb(564.000000 pt)=(0.760000,1.000000,0.240000);
rgb(565.000000 pt)=(0.764000,1.000000,0.236000);
rgb(566.000000 pt)=(0.768000,1.000000,0.232000);
rgb(567.000000 pt)=(0.772000,1.000000,0.228000);
rgb(568.000000 pt)=(0.776000,1.000000,0.224000);
rgb(569.000000 pt)=(0.780000,1.000000,0.220000);
rgb(570.000000 pt)=(0.784000,1.000000,0.216000);
rgb(571.000000 pt)=(0.788000,1.000000,0.212000);
rgb(572.000000 pt)=(0.792000,1.000000,0.208000);
rgb(573.000000 pt)=(0.796000,1.000000,0.204000);
rgb(574.000000 pt)=(0.800000,1.000000,0.200000);
rgb(575.000000 pt)=(0.804000,1.000000,0.196000);
rgb(576.000000 pt)=(0.808000,1.000000,0.192000);
rgb(577.000000 pt)=(0.812000,1.000000,0.188000);
rgb(578.000000 pt)=(0.816000,1.000000,0.184000);
rgb(579.000000 pt)=(0.820000,1.000000,0.180000);
rgb(580.000000 pt)=(0.824000,1.000000,0.176000);
rgb(581.000000 pt)=(0.828000,1.000000,0.172000);
rgb(582.000000 pt)=(0.832000,1.000000,0.168000);
rgb(583.000000 pt)=(0.836000,1.000000,0.164000);
rgb(584.000000 pt)=(0.840000,1.000000,0.160000);
rgb(585.000000 pt)=(0.844000,1.000000,0.156000);
rgb(586.000000 pt)=(0.848000,1.000000,0.152000);
rgb(587.000000 pt)=(0.852000,1.000000,0.148000);
rgb(588.000000 pt)=(0.856000,1.000000,0.144000);
rgb(589.000000 pt)=(0.860000,1.000000,0.140000);
rgb(590.000000 pt)=(0.864000,1.000000,0.136000);
rgb(591.000000 pt)=(0.868000,1.000000,0.132000);
rgb(592.000000 pt)=(0.872000,1.000000,0.128000);
rgb(593.000000 pt)=(0.876000,1.000000,0.124000);
rgb(594.000000 pt)=(0.880000,1.000000,0.120000);
rgb(595.000000 pt)=(0.884000,1.000000,0.116000);
rgb(596.000000 pt)=(0.888000,1.000000,0.112000);
rgb(597.000000 pt)=(0.892000,1.000000,0.108000);
rgb(598.000000 pt)=(0.896000,1.000000,0.104000);
rgb(599.000000 pt)=(0.900000,1.000000,0.100000);
rgb(600.000000 pt)=(0.904000,1.000000,0.096000);
rgb(601.000000 pt)=(0.908000,1.000000,0.092000);
rgb(602.000000 pt)=(0.912000,1.000000,0.088000);
rgb(603.000000 pt)=(0.916000,1.000000,0.084000);
rgb(604.000000 pt)=(0.920000,1.000000,0.080000);
rgb(605.000000 pt)=(0.924000,1.000000,0.076000);
rgb(606.000000 pt)=(0.928000,1.000000,0.072000);
rgb(607.000000 pt)=(0.932000,1.000000,0.068000);
rgb(608.000000 pt)=(0.936000,1.000000,0.064000);
rgb(609.000000 pt)=(0.940000,1.000000,0.060000);
rgb(610.000000 pt)=(0.944000,1.000000,0.056000);
rgb(611.000000 pt)=(0.948000,1.000000,0.052000);
rgb(612.000000 pt)=(0.952000,1.000000,0.048000);
rgb(613.000000 pt)=(0.956000,1.000000,0.044000);
rgb(614.000000 pt)=(0.960000,1.000000,0.040000);
rgb(615.000000 pt)=(0.964000,1.000000,0.036000);
rgb(616.000000 pt)=(0.968000,1.000000,0.032000);
rgb(617.000000 pt)=(0.972000,1.000000,0.028000);
rgb(618.000000 pt)=(0.976000,1.000000,0.024000);
rgb(619.000000 pt)=(0.980000,1.000000,0.020000);
rgb(620.000000 pt)=(0.984000,1.000000,0.016000);
rgb(621.000000 pt)=(0.988000,1.000000,0.012000);
rgb(622.000000 pt)=(0.992000,1.000000,0.008000);
rgb(623.000000 pt)=(0.996000,1.000000,0.004000);
rgb(624.000000 pt)=(1.000000,1.000000,0.000000);
rgb(625.000000 pt)=(1.000000,0.996000,0.000000);
rgb(626.000000 pt)=(1.000000,0.992000,0.000000);
rgb(627.000000 pt)=(1.000000,0.988000,0.000000);
rgb(628.000000 pt)=(1.000000,0.984000,0.000000);
rgb(629.000000 pt)=(1.000000,0.980000,0.000000);
rgb(630.000000 pt)=(1.000000,0.976000,0.000000);
rgb(631.000000 pt)=(1.000000,0.972000,0.000000);
rgb(632.000000 pt)=(1.000000,0.968000,0.000000);
rgb(633.000000 pt)=(1.000000,0.964000,0.000000);
rgb(634.000000 pt)=(1.000000,0.960000,0.000000);
rgb(635.000000 pt)=(1.000000,0.956000,0.000000);
rgb(636.000000 pt)=(1.000000,0.952000,0.000000);
rgb(637.000000 pt)=(1.000000,0.948000,0.000000);
rgb(638.000000 pt)=(1.000000,0.944000,0.000000);
rgb(639.000000 pt)=(1.000000,0.940000,0.000000);
rgb(640.000000 pt)=(1.000000,0.936000,0.000000);
rgb(641.000000 pt)=(1.000000,0.932000,0.000000);
rgb(642.000000 pt)=(1.000000,0.928000,0.000000);
rgb(643.000000 pt)=(1.000000,0.924000,0.000000);
rgb(644.000000 pt)=(1.000000,0.920000,0.000000);
rgb(645.000000 pt)=(1.000000,0.916000,0.000000);
rgb(646.000000 pt)=(1.000000,0.912000,0.000000);
rgb(647.000000 pt)=(1.000000,0.908000,0.000000);
rgb(648.000000 pt)=(1.000000,0.904000,0.000000);
rgb(649.000000 pt)=(1.000000,0.900000,0.000000);
rgb(650.000000 pt)=(1.000000,0.896000,0.000000);
rgb(651.000000 pt)=(1.000000,0.892000,0.000000);
rgb(652.000000 pt)=(1.000000,0.888000,0.000000);
rgb(653.000000 pt)=(1.000000,0.884000,0.000000);
rgb(654.000000 pt)=(1.000000,0.880000,0.000000);
rgb(655.000000 pt)=(1.000000,0.876000,0.000000);
rgb(656.000000 pt)=(1.000000,0.872000,0.000000);
rgb(657.000000 pt)=(1.000000,0.868000,0.000000);
rgb(658.000000 pt)=(1.000000,0.864000,0.000000);
rgb(659.000000 pt)=(1.000000,0.860000,0.000000);
rgb(660.000000 pt)=(1.000000,0.856000,0.000000);
rgb(661.000000 pt)=(1.000000,0.852000,0.000000);
rgb(662.000000 pt)=(1.000000,0.848000,0.000000);
rgb(663.000000 pt)=(1.000000,0.844000,0.000000);
rgb(664.000000 pt)=(1.000000,0.840000,0.000000);
rgb(665.000000 pt)=(1.000000,0.836000,0.000000);
rgb(666.000000 pt)=(1.000000,0.832000,0.000000);
rgb(667.000000 pt)=(1.000000,0.828000,0.000000);
rgb(668.000000 pt)=(1.000000,0.824000,0.000000);
rgb(669.000000 pt)=(1.000000,0.820000,0.000000);
rgb(670.000000 pt)=(1.000000,0.816000,0.000000);
rgb(671.000000 pt)=(1.000000,0.812000,0.000000);
rgb(672.000000 pt)=(1.000000,0.808000,0.000000);
rgb(673.000000 pt)=(1.000000,0.804000,0.000000);
rgb(674.000000 pt)=(1.000000,0.800000,0.000000);
rgb(675.000000 pt)=(1.000000,0.796000,0.000000);
rgb(676.000000 pt)=(1.000000,0.792000,0.000000);
rgb(677.000000 pt)=(1.000000,0.788000,0.000000);
rgb(678.000000 pt)=(1.000000,0.784000,0.000000);
rgb(679.000000 pt)=(1.000000,0.780000,0.000000);
rgb(680.000000 pt)=(1.000000,0.776000,0.000000);
rgb(681.000000 pt)=(1.000000,0.772000,0.000000);
rgb(682.000000 pt)=(1.000000,0.768000,0.000000);
rgb(683.000000 pt)=(1.000000,0.764000,0.000000);
rgb(684.000000 pt)=(1.000000,0.760000,0.000000);
rgb(685.000000 pt)=(1.000000,0.756000,0.000000);
rgb(686.000000 pt)=(1.000000,0.752000,0.000000);
rgb(687.000000 pt)=(1.000000,0.748000,0.000000);
rgb(688.000000 pt)=(1.000000,0.744000,0.000000);
rgb(689.000000 pt)=(1.000000,0.740000,0.000000);
rgb(690.000000 pt)=(1.000000,0.736000,0.000000);
rgb(691.000000 pt)=(1.000000,0.732000,0.000000);
rgb(692.000000 pt)=(1.000000,0.728000,0.000000);
rgb(693.000000 pt)=(1.000000,0.724000,0.000000);
rgb(694.000000 pt)=(1.000000,0.720000,0.000000);
rgb(695.000000 pt)=(1.000000,0.716000,0.000000);
rgb(696.000000 pt)=(1.000000,0.712000,0.000000);
rgb(697.000000 pt)=(1.000000,0.708000,0.000000);
rgb(698.000000 pt)=(1.000000,0.704000,0.000000);
rgb(699.000000 pt)=(1.000000,0.700000,0.000000);
rgb(700.000000 pt)=(1.000000,0.696000,0.000000);
rgb(701.000000 pt)=(1.000000,0.692000,0.000000);
rgb(702.000000 pt)=(1.000000,0.688000,0.000000);
rgb(703.000000 pt)=(1.000000,0.684000,0.000000);
rgb(704.000000 pt)=(1.000000,0.680000,0.000000);
rgb(705.000000 pt)=(1.000000,0.676000,0.000000);
rgb(706.000000 pt)=(1.000000,0.672000,0.000000);
rgb(707.000000 pt)=(1.000000,0.668000,0.000000);
rgb(708.000000 pt)=(1.000000,0.664000,0.000000);
rgb(709.000000 pt)=(1.000000,0.660000,0.000000);
rgb(710.000000 pt)=(1.000000,0.656000,0.000000);
rgb(711.000000 pt)=(1.000000,0.652000,0.000000);
rgb(712.000000 pt)=(1.000000,0.648000,0.000000);
rgb(713.000000 pt)=(1.000000,0.644000,0.000000);
rgb(714.000000 pt)=(1.000000,0.640000,0.000000);
rgb(715.000000 pt)=(1.000000,0.636000,0.000000);
rgb(716.000000 pt)=(1.000000,0.632000,0.000000);
rgb(717.000000 pt)=(1.000000,0.628000,0.000000);
rgb(718.000000 pt)=(1.000000,0.624000,0.000000);
rgb(719.000000 pt)=(1.000000,0.620000,0.000000);
rgb(720.000000 pt)=(1.000000,0.616000,0.000000);
rgb(721.000000 pt)=(1.000000,0.612000,0.000000);
rgb(722.000000 pt)=(1.000000,0.608000,0.000000);
rgb(723.000000 pt)=(1.000000,0.604000,0.000000);
rgb(724.000000 pt)=(1.000000,0.600000,0.000000);
rgb(725.000000 pt)=(1.000000,0.596000,0.000000);
rgb(726.000000 pt)=(1.000000,0.592000,0.000000);
rgb(727.000000 pt)=(1.000000,0.588000,0.000000);
rgb(728.000000 pt)=(1.000000,0.584000,0.000000);
rgb(729.000000 pt)=(1.000000,0.580000,0.000000);
rgb(730.000000 pt)=(1.000000,0.576000,0.000000);
rgb(731.000000 pt)=(1.000000,0.572000,0.000000);
rgb(732.000000 pt)=(1.000000,0.568000,0.000000);
rgb(733.000000 pt)=(1.000000,0.564000,0.000000);
rgb(734.000000 pt)=(1.000000,0.560000,0.000000);
rgb(735.000000 pt)=(1.000000,0.556000,0.000000);
rgb(736.000000 pt)=(1.000000,0.552000,0.000000);
rgb(737.000000 pt)=(1.000000,0.548000,0.000000);
rgb(738.000000 pt)=(1.000000,0.544000,0.000000);
rgb(739.000000 pt)=(1.000000,0.540000,0.000000);
rgb(740.000000 pt)=(1.000000,0.536000,0.000000);
rgb(741.000000 pt)=(1.000000,0.532000,0.000000);
rgb(742.000000 pt)=(1.000000,0.528000,0.000000);
rgb(743.000000 pt)=(1.000000,0.524000,0.000000);
rgb(744.000000 pt)=(1.000000,0.520000,0.000000);
rgb(745.000000 pt)=(1.000000,0.516000,0.000000);
rgb(746.000000 pt)=(1.000000,0.512000,0.000000);
rgb(747.000000 pt)=(1.000000,0.508000,0.000000);
rgb(748.000000 pt)=(1.000000,0.504000,0.000000);
rgb(749.000000 pt)=(1.000000,0.500000,0.000000);
rgb(750.000000 pt)=(1.000000,0.496000,0.000000);
rgb(751.000000 pt)=(1.000000,0.492000,0.000000);
rgb(752.000000 pt)=(1.000000,0.488000,0.000000);
rgb(753.000000 pt)=(1.000000,0.484000,0.000000);
rgb(754.000000 pt)=(1.000000,0.480000,0.000000);
rgb(755.000000 pt)=(1.000000,0.476000,0.000000);
rgb(756.000000 pt)=(1.000000,0.472000,0.000000);
rgb(757.000000 pt)=(1.000000,0.468000,0.000000);
rgb(758.000000 pt)=(1.000000,0.464000,0.000000);
rgb(759.000000 pt)=(1.000000,0.460000,0.000000);
rgb(760.000000 pt)=(1.000000,0.456000,0.000000);
rgb(761.000000 pt)=(1.000000,0.452000,0.000000);
rgb(762.000000 pt)=(1.000000,0.448000,0.000000);
rgb(763.000000 pt)=(1.000000,0.444000,0.000000);
rgb(764.000000 pt)=(1.000000,0.440000,0.000000);
rgb(765.000000 pt)=(1.000000,0.436000,0.000000);
rgb(766.000000 pt)=(1.000000,0.432000,0.000000);
rgb(767.000000 pt)=(1.000000,0.428000,0.000000);
rgb(768.000000 pt)=(1.000000,0.424000,0.000000);
rgb(769.000000 pt)=(1.000000,0.420000,0.000000);
rgb(770.000000 pt)=(1.000000,0.416000,0.000000);
rgb(771.000000 pt)=(1.000000,0.412000,0.000000);
rgb(772.000000 pt)=(1.000000,0.408000,0.000000);
rgb(773.000000 pt)=(1.000000,0.404000,0.000000);
rgb(774.000000 pt)=(1.000000,0.400000,0.000000);
rgb(775.000000 pt)=(1.000000,0.396000,0.000000);
rgb(776.000000 pt)=(1.000000,0.392000,0.000000);
rgb(777.000000 pt)=(1.000000,0.388000,0.000000);
rgb(778.000000 pt)=(1.000000,0.384000,0.000000);
rgb(779.000000 pt)=(1.000000,0.380000,0.000000);
rgb(780.000000 pt)=(1.000000,0.376000,0.000000);
rgb(781.000000 pt)=(1.000000,0.372000,0.000000);
rgb(782.000000 pt)=(1.000000,0.368000,0.000000);
rgb(783.000000 pt)=(1.000000,0.364000,0.000000);
rgb(784.000000 pt)=(1.000000,0.360000,0.000000);
rgb(785.000000 pt)=(1.000000,0.356000,0.000000);
rgb(786.000000 pt)=(1.000000,0.352000,0.000000);
rgb(787.000000 pt)=(1.000000,0.348000,0.000000);
rgb(788.000000 pt)=(1.000000,0.344000,0.000000);
rgb(789.000000 pt)=(1.000000,0.340000,0.000000);
rgb(790.000000 pt)=(1.000000,0.336000,0.000000);
rgb(791.000000 pt)=(1.000000,0.332000,0.000000);
rgb(792.000000 pt)=(1.000000,0.328000,0.000000);
rgb(793.000000 pt)=(1.000000,0.324000,0.000000);
rgb(794.000000 pt)=(1.000000,0.320000,0.000000);
rgb(795.000000 pt)=(1.000000,0.316000,0.000000);
rgb(796.000000 pt)=(1.000000,0.312000,0.000000);
rgb(797.000000 pt)=(1.000000,0.308000,0.000000);
rgb(798.000000 pt)=(1.000000,0.304000,0.000000);
rgb(799.000000 pt)=(1.000000,0.300000,0.000000);
rgb(800.000000 pt)=(1.000000,0.296000,0.000000);
rgb(801.000000 pt)=(1.000000,0.292000,0.000000);
rgb(802.000000 pt)=(1.000000,0.288000,0.000000);
rgb(803.000000 pt)=(1.000000,0.284000,0.000000);
rgb(804.000000 pt)=(1.000000,0.280000,0.000000);
rgb(805.000000 pt)=(1.000000,0.276000,0.000000);
rgb(806.000000 pt)=(1.000000,0.272000,0.000000);
rgb(807.000000 pt)=(1.000000,0.268000,0.000000);
rgb(808.000000 pt)=(1.000000,0.264000,0.000000);
rgb(809.000000 pt)=(1.000000,0.260000,0.000000);
rgb(810.000000 pt)=(1.000000,0.256000,0.000000);
rgb(811.000000 pt)=(1.000000,0.252000,0.000000);
rgb(812.000000 pt)=(1.000000,0.248000,0.000000);
rgb(813.000000 pt)=(1.000000,0.244000,0.000000);
rgb(814.000000 pt)=(1.000000,0.240000,0.000000);
rgb(815.000000 pt)=(1.000000,0.236000,0.000000);
rgb(816.000000 pt)=(1.000000,0.232000,0.000000);
rgb(817.000000 pt)=(1.000000,0.228000,0.000000);
rgb(818.000000 pt)=(1.000000,0.224000,0.000000);
rgb(819.000000 pt)=(1.000000,0.220000,0.000000);
rgb(820.000000 pt)=(1.000000,0.216000,0.000000);
rgb(821.000000 pt)=(1.000000,0.212000,0.000000);
rgb(822.000000 pt)=(1.000000,0.208000,0.000000);
rgb(823.000000 pt)=(1.000000,0.204000,0.000000);
rgb(824.000000 pt)=(1.000000,0.200000,0.000000);
rgb(825.000000 pt)=(1.000000,0.196000,0.000000);
rgb(826.000000 pt)=(1.000000,0.192000,0.000000);
rgb(827.000000 pt)=(1.000000,0.188000,0.000000);
rgb(828.000000 pt)=(1.000000,0.184000,0.000000);
rgb(829.000000 pt)=(1.000000,0.180000,0.000000);
rgb(830.000000 pt)=(1.000000,0.176000,0.000000);
rgb(831.000000 pt)=(1.000000,0.172000,0.000000);
rgb(832.000000 pt)=(1.000000,0.168000,0.000000);
rgb(833.000000 pt)=(1.000000,0.164000,0.000000);
rgb(834.000000 pt)=(1.000000,0.160000,0.000000);
rgb(835.000000 pt)=(1.000000,0.156000,0.000000);
rgb(836.000000 pt)=(1.000000,0.152000,0.000000);
rgb(837.000000 pt)=(1.000000,0.148000,0.000000);
rgb(838.000000 pt)=(1.000000,0.144000,0.000000);
rgb(839.000000 pt)=(1.000000,0.140000,0.000000);
rgb(840.000000 pt)=(1.000000,0.136000,0.000000);
rgb(841.000000 pt)=(1.000000,0.132000,0.000000);
rgb(842.000000 pt)=(1.000000,0.128000,0.000000);
rgb(843.000000 pt)=(1.000000,0.124000,0.000000);
rgb(844.000000 pt)=(1.000000,0.120000,0.000000);
rgb(845.000000 pt)=(1.000000,0.116000,0.000000);
rgb(846.000000 pt)=(1.000000,0.112000,0.000000);
rgb(847.000000 pt)=(1.000000,0.108000,0.000000);
rgb(848.000000 pt)=(1.000000,0.104000,0.000000);
rgb(849.000000 pt)=(1.000000,0.100000,0.000000);
rgb(850.000000 pt)=(1.000000,0.096000,0.000000);
rgb(851.000000 pt)=(1.000000,0.092000,0.000000);
rgb(852.000000 pt)=(1.000000,0.088000,0.000000);
rgb(853.000000 pt)=(1.000000,0.084000,0.000000);
rgb(854.000000 pt)=(1.000000,0.080000,0.000000);
rgb(855.000000 pt)=(1.000000,0.076000,0.000000);
rgb(856.000000 pt)=(1.000000,0.072000,0.000000);
rgb(857.000000 pt)=(1.000000,0.068000,0.000000);
rgb(858.000000 pt)=(1.000000,0.064000,0.000000);
rgb(859.000000 pt)=(1.000000,0.060000,0.000000);
rgb(860.000000 pt)=(1.000000,0.056000,0.000000);
rgb(861.000000 pt)=(1.000000,0.052000,0.000000);
rgb(862.000000 pt)=(1.000000,0.048000,0.000000);
rgb(863.000000 pt)=(1.000000,0.044000,0.000000);
rgb(864.000000 pt)=(1.000000,0.040000,0.000000);
rgb(865.000000 pt)=(1.000000,0.036000,0.000000);
rgb(866.000000 pt)=(1.000000,0.032000,0.000000);
rgb(867.000000 pt)=(1.000000,0.028000,0.000000);
rgb(868.000000 pt)=(1.000000,0.024000,0.000000);
rgb(869.000000 pt)=(1.000000,0.020000,0.000000);
rgb(870.000000 pt)=(1.000000,0.016000,0.000000);
rgb(871.000000 pt)=(1.000000,0.012000,0.000000);
rgb(872.000000 pt)=(1.000000,0.008000,0.000000);
rgb(873.000000 pt)=(1.000000,0.004000,0.000000);
rgb(874.000000 pt)=(1.000000,0.000000,0.000000);
rgb(875.000000 pt)=(0.996000,0.000000,0.000000);
rgb(876.000000 pt)=(0.992000,0.000000,0.000000);
rgb(877.000000 pt)=(0.988000,0.000000,0.000000);
rgb(878.000000 pt)=(0.984000,0.000000,0.000000);
rgb(879.000000 pt)=(0.980000,0.000000,0.000000);
rgb(880.000000 pt)=(0.976000,0.000000,0.000000);
rgb(881.000000 pt)=(0.972000,0.000000,0.000000);
rgb(882.000000 pt)=(0.968000,0.000000,0.000000);
rgb(883.000000 pt)=(0.964000,0.000000,0.000000);
rgb(884.000000 pt)=(0.960000,0.000000,0.000000);
rgb(885.000000 pt)=(0.956000,0.000000,0.000000);
rgb(886.000000 pt)=(0.952000,0.000000,0.000000);
rgb(887.000000 pt)=(0.948000,0.000000,0.000000);
rgb(888.000000 pt)=(0.944000,0.000000,0.000000);
rgb(889.000000 pt)=(0.940000,0.000000,0.000000);
rgb(890.000000 pt)=(0.936000,0.000000,0.000000);
rgb(891.000000 pt)=(0.932000,0.000000,0.000000);
rgb(892.000000 pt)=(0.928000,0.000000,0.000000);
rgb(893.000000 pt)=(0.924000,0.000000,0.000000);
rgb(894.000000 pt)=(0.920000,0.000000,0.000000);
rgb(895.000000 pt)=(0.916000,0.000000,0.000000);
rgb(896.000000 pt)=(0.912000,0.000000,0.000000);
rgb(897.000000 pt)=(0.908000,0.000000,0.000000);
rgb(898.000000 pt)=(0.904000,0.000000,0.000000);
rgb(899.000000 pt)=(0.900000,0.000000,0.000000);
rgb(900.000000 pt)=(0.896000,0.000000,0.000000);
rgb(901.000000 pt)=(0.892000,0.000000,0.000000);
rgb(902.000000 pt)=(0.888000,0.000000,0.000000);
rgb(903.000000 pt)=(0.884000,0.000000,0.000000);
rgb(904.000000 pt)=(0.880000,0.000000,0.000000);
rgb(905.000000 pt)=(0.876000,0.000000,0.000000);
rgb(906.000000 pt)=(0.872000,0.000000,0.000000);
rgb(907.000000 pt)=(0.868000,0.000000,0.000000);
rgb(908.000000 pt)=(0.864000,0.000000,0.000000);
rgb(909.000000 pt)=(0.860000,0.000000,0.000000);
rgb(910.000000 pt)=(0.856000,0.000000,0.000000);
rgb(911.000000 pt)=(0.852000,0.000000,0.000000);
rgb(912.000000 pt)=(0.848000,0.000000,0.000000);
rgb(913.000000 pt)=(0.844000,0.000000,0.000000);
rgb(914.000000 pt)=(0.840000,0.000000,0.000000);
rgb(915.000000 pt)=(0.836000,0.000000,0.000000);
rgb(916.000000 pt)=(0.832000,0.000000,0.000000);
rgb(917.000000 pt)=(0.828000,0.000000,0.000000);
rgb(918.000000 pt)=(0.824000,0.000000,0.000000);
rgb(919.000000 pt)=(0.820000,0.000000,0.000000);
rgb(920.000000 pt)=(0.816000,0.000000,0.000000);
rgb(921.000000 pt)=(0.812000,0.000000,0.000000);
rgb(922.000000 pt)=(0.808000,0.000000,0.000000);
rgb(923.000000 pt)=(0.804000,0.000000,0.000000);
rgb(924.000000 pt)=(0.800000,0.000000,0.000000);
rgb(925.000000 pt)=(0.796000,0.000000,0.000000);
rgb(926.000000 pt)=(0.792000,0.000000,0.000000);
rgb(927.000000 pt)=(0.788000,0.000000,0.000000);
rgb(928.000000 pt)=(0.784000,0.000000,0.000000);
rgb(929.000000 pt)=(0.780000,0.000000,0.000000);
rgb(930.000000 pt)=(0.776000,0.000000,0.000000);
rgb(931.000000 pt)=(0.772000,0.000000,0.000000);
rgb(932.000000 pt)=(0.768000,0.000000,0.000000);
rgb(933.000000 pt)=(0.764000,0.000000,0.000000);
rgb(934.000000 pt)=(0.760000,0.000000,0.000000);
rgb(935.000000 pt)=(0.756000,0.000000,0.000000);
rgb(936.000000 pt)=(0.752000,0.000000,0.000000);
rgb(937.000000 pt)=(0.748000,0.000000,0.000000);
rgb(938.000000 pt)=(0.744000,0.000000,0.000000);
rgb(939.000000 pt)=(0.740000,0.000000,0.000000);
rgb(940.000000 pt)=(0.736000,0.000000,0.000000);
rgb(941.000000 pt)=(0.732000,0.000000,0.000000);
rgb(942.000000 pt)=(0.728000,0.000000,0.000000);
rgb(943.000000 pt)=(0.724000,0.000000,0.000000);
rgb(944.000000 pt)=(0.720000,0.000000,0.000000);
rgb(945.000000 pt)=(0.716000,0.000000,0.000000);
rgb(946.000000 pt)=(0.712000,0.000000,0.000000);
rgb(947.000000 pt)=(0.708000,0.000000,0.000000);
rgb(948.000000 pt)=(0.704000,0.000000,0.000000);
rgb(949.000000 pt)=(0.700000,0.000000,0.000000);
rgb(950.000000 pt)=(0.696000,0.000000,0.000000);
rgb(951.000000 pt)=(0.692000,0.000000,0.000000);
rgb(952.000000 pt)=(0.688000,0.000000,0.000000);
rgb(953.000000 pt)=(0.684000,0.000000,0.000000);
rgb(954.000000 pt)=(0.680000,0.000000,0.000000);
rgb(955.000000 pt)=(0.676000,0.000000,0.000000);
rgb(956.000000 pt)=(0.672000,0.000000,0.000000);
rgb(957.000000 pt)=(0.668000,0.000000,0.000000);
rgb(958.000000 pt)=(0.664000,0.000000,0.000000);
rgb(959.000000 pt)=(0.660000,0.000000,0.000000);
rgb(960.000000 pt)=(0.656000,0.000000,0.000000);
rgb(961.000000 pt)=(0.652000,0.000000,0.000000);
rgb(962.000000 pt)=(0.648000,0.000000,0.000000);
rgb(963.000000 pt)=(0.644000,0.000000,0.000000);
rgb(964.000000 pt)=(0.640000,0.000000,0.000000);
rgb(965.000000 pt)=(0.636000,0.000000,0.000000);
rgb(966.000000 pt)=(0.632000,0.000000,0.000000);
rgb(967.000000 pt)=(0.628000,0.000000,0.000000);
rgb(968.000000 pt)=(0.624000,0.000000,0.000000);
rgb(969.000000 pt)=(0.620000,0.000000,0.000000);
rgb(970.000000 pt)=(0.616000,0.000000,0.000000);
rgb(971.000000 pt)=(0.612000,0.000000,0.000000);
rgb(972.000000 pt)=(0.608000,0.000000,0.000000);
rgb(973.000000 pt)=(0.604000,0.000000,0.000000);
rgb(974.000000 pt)=(0.600000,0.000000,0.000000);
rgb(975.000000 pt)=(0.596000,0.000000,0.000000);
rgb(976.000000 pt)=(0.592000,0.000000,0.000000);
rgb(977.000000 pt)=(0.588000,0.000000,0.000000);
rgb(978.000000 pt)=(0.584000,0.000000,0.000000);
rgb(979.000000 pt)=(0.580000,0.000000,0.000000);
rgb(980.000000 pt)=(0.576000,0.000000,0.000000);
rgb(981.000000 pt)=(0.572000,0.000000,0.000000);
rgb(982.000000 pt)=(0.568000,0.000000,0.000000);
rgb(983.000000 pt)=(0.564000,0.000000,0.000000);
rgb(984.000000 pt)=(0.560000,0.000000,0.000000);
rgb(985.000000 pt)=(0.556000,0.000000,0.000000);
rgb(986.000000 pt)=(0.552000,0.000000,0.000000);
rgb(987.000000 pt)=(0.548000,0.000000,0.000000);
rgb(988.000000 pt)=(0.544000,0.000000,0.000000);
rgb(989.000000 pt)=(0.540000,0.000000,0.000000);
rgb(990.000000 pt)=(0.536000,0.000000,0.000000);
rgb(991.000000 pt)=(0.532000,0.000000,0.000000);
rgb(992.000000 pt)=(0.528000,0.000000,0.000000);
rgb(993.000000 pt)=(0.524000,0.000000,0.000000);
rgb(994.000000 pt)=(0.520000,0.000000,0.000000);
rgb(995.000000 pt)=(0.516000,0.000000,0.000000);
rgb(996.000000 pt)=(0.512000,0.000000,0.000000);
rgb(997.000000 pt)=(0.508000,0.000000,0.000000);
rgb(998.000000 pt)=(0.504000,0.000000,0.000000);
rgb(999.000000 pt)=(0.500000,0.000000,0.000000);
}}

\pgfplotsset{
	colormap={myhot}{
		rgb(0pt)=(1.000000,1.000000,1.000000);
		rgb(1pt)=(1.000000,1.000000,0.996000);
		rgb(2pt)=(1.000000,1.000000,0.992000);
		rgb(3pt)=(1.000000,1.000000,0.988000);
		rgb(4pt)=(1.000000,1.000000,0.984000);
		rgb(5pt)=(1.000000,1.000000,0.980000);
		rgb(6pt)=(1.000000,1.000000,0.976000);
		rgb(7pt)=(1.000000,1.000000,0.972000);
		rgb(8pt)=(1.000000,1.000000,0.968000);
		rgb(9pt)=(1.000000,1.000000,0.964000);
		rgb(10pt)=(1.000000,1.000000,0.960000);
		rgb(11pt)=(1.000000,1.000000,0.956000);
		rgb(12pt)=(1.000000,1.000000,0.952000);
		rgb(13pt)=(1.000000,1.000000,0.948000);
		rgb(14pt)=(1.000000,1.000000,0.944000);
		rgb(15pt)=(1.000000,1.000000,0.940000);
		rgb(16pt)=(1.000000,1.000000,0.936000);
		rgb(17pt)=(1.000000,1.000000,0.932000);
		rgb(18pt)=(1.000000,1.000000,0.928000);
		rgb(19pt)=(1.000000,1.000000,0.924000);
		rgb(20pt)=(1.000000,1.000000,0.920000);
		rgb(21pt)=(1.000000,1.000000,0.916000);
		rgb(22pt)=(1.000000,1.000000,0.912000);
		rgb(23pt)=(1.000000,1.000000,0.908000);
		rgb(24pt)=(1.000000,1.000000,0.904000);
		rgb(25pt)=(1.000000,1.000000,0.900000);
		rgb(26pt)=(1.000000,1.000000,0.896000);
		rgb(27pt)=(1.000000,1.000000,0.892000);
		rgb(28pt)=(1.000000,1.000000,0.888000);
		rgb(29pt)=(1.000000,1.000000,0.884000);
		rgb(30pt)=(1.000000,1.000000,0.880000);
		rgb(31pt)=(1.000000,1.000000,0.876000);
		rgb(32pt)=(1.000000,1.000000,0.872000);
		rgb(33pt)=(1.000000,1.000000,0.868000);
		rgb(34pt)=(1.000000,1.000000,0.864000);
		rgb(35pt)=(1.000000,1.000000,0.860000);
		rgb(36pt)=(1.000000,1.000000,0.856000);
		rgb(37pt)=(1.000000,1.000000,0.852000);
		rgb(38pt)=(1.000000,1.000000,0.848000);
		rgb(39pt)=(1.000000,1.000000,0.844000);
		rgb(40pt)=(1.000000,1.000000,0.840000);
		rgb(41pt)=(1.000000,1.000000,0.836000);
		rgb(42pt)=(1.000000,1.000000,0.832000);
		rgb(43pt)=(1.000000,1.000000,0.828000);
		rgb(44pt)=(1.000000,1.000000,0.824000);
		rgb(45pt)=(1.000000,1.000000,0.820000);
		rgb(46pt)=(1.000000,1.000000,0.816000);
		rgb(47pt)=(1.000000,1.000000,0.812000);
		rgb(48pt)=(1.000000,1.000000,0.808000);
		rgb(49pt)=(1.000000,1.000000,0.804000);
		rgb(50pt)=(1.000000,1.000000,0.800000);
		rgb(51pt)=(1.000000,1.000000,0.796000);
		rgb(52pt)=(1.000000,1.000000,0.792000);
		rgb(53pt)=(1.000000,1.000000,0.788000);
		rgb(54pt)=(1.000000,1.000000,0.784000);
		rgb(55pt)=(1.000000,1.000000,0.780000);
		rgb(56pt)=(1.000000,1.000000,0.776000);
		rgb(57pt)=(1.000000,1.000000,0.772000);
		rgb(58pt)=(1.000000,1.000000,0.768000);
		rgb(59pt)=(1.000000,1.000000,0.764000);
		rgb(60pt)=(1.000000,1.000000,0.760000);
		rgb(61pt)=(1.000000,1.000000,0.756000);
		rgb(62pt)=(1.000000,1.000000,0.752000);
		rgb(63pt)=(1.000000,1.000000,0.748000);
		rgb(64pt)=(1.000000,1.000000,0.744000);
		rgb(65pt)=(1.000000,1.000000,0.740000);
		rgb(66pt)=(1.000000,1.000000,0.736000);
		rgb(67pt)=(1.000000,1.000000,0.732000);
		rgb(68pt)=(1.000000,1.000000,0.728000);
		rgb(69pt)=(1.000000,1.000000,0.724000);
		rgb(70pt)=(1.000000,1.000000,0.720000);
		rgb(71pt)=(1.000000,1.000000,0.716000);
		rgb(72pt)=(1.000000,1.000000,0.712000);
		rgb(73pt)=(1.000000,1.000000,0.708000);
		rgb(74pt)=(1.000000,1.000000,0.704000);
		rgb(75pt)=(1.000000,1.000000,0.700000);
		rgb(76pt)=(1.000000,1.000000,0.696000);
		rgb(77pt)=(1.000000,1.000000,0.692000);
		rgb(78pt)=(1.000000,1.000000,0.688000);
		rgb(79pt)=(1.000000,1.000000,0.684000);
		rgb(80pt)=(1.000000,1.000000,0.680000);
		rgb(81pt)=(1.000000,1.000000,0.676000);
		rgb(82pt)=(1.000000,1.000000,0.672000);
		rgb(83pt)=(1.000000,1.000000,0.668000);
		rgb(84pt)=(1.000000,1.000000,0.664000);
		rgb(85pt)=(1.000000,1.000000,0.660000);
		rgb(86pt)=(1.000000,1.000000,0.656000);
		rgb(87pt)=(1.000000,1.000000,0.652000);
		rgb(88pt)=(1.000000,1.000000,0.648000);
		rgb(89pt)=(1.000000,1.000000,0.644000);
		rgb(90pt)=(1.000000,1.000000,0.640000);
		rgb(91pt)=(1.000000,1.000000,0.636000);
		rgb(92pt)=(1.000000,1.000000,0.632000);
		rgb(93pt)=(1.000000,1.000000,0.628000);
		rgb(94pt)=(1.000000,1.000000,0.624000);
		rgb(95pt)=(1.000000,1.000000,0.620000);
		rgb(96pt)=(1.000000,1.000000,0.616000);
		rgb(97pt)=(1.000000,1.000000,0.612000);
		rgb(98pt)=(1.000000,1.000000,0.608000);
		rgb(99pt)=(1.000000,1.000000,0.604000);
		rgb(100pt)=(1.000000,1.000000,0.600000);
		rgb(101pt)=(1.000000,1.000000,0.596000);
		rgb(102pt)=(1.000000,1.000000,0.592000);
		rgb(103pt)=(1.000000,1.000000,0.588000);
		rgb(104pt)=(1.000000,1.000000,0.584000);
		rgb(105pt)=(1.000000,1.000000,0.580000);
		rgb(106pt)=(1.000000,1.000000,0.576000);
		rgb(107pt)=(1.000000,1.000000,0.572000);
		rgb(108pt)=(1.000000,1.000000,0.568000);
		rgb(109pt)=(1.000000,1.000000,0.564000);
		rgb(110pt)=(1.000000,1.000000,0.560000);
		rgb(111pt)=(1.000000,1.000000,0.556000);
		rgb(112pt)=(1.000000,1.000000,0.552000);
		rgb(113pt)=(1.000000,1.000000,0.548000);
		rgb(114pt)=(1.000000,1.000000,0.544000);
		rgb(115pt)=(1.000000,1.000000,0.540000);
		rgb(116pt)=(1.000000,1.000000,0.536000);
		rgb(117pt)=(1.000000,1.000000,0.532000);
		rgb(118pt)=(1.000000,1.000000,0.528000);
		rgb(119pt)=(1.000000,1.000000,0.524000);
		rgb(120pt)=(1.000000,1.000000,0.520000);
		rgb(121pt)=(1.000000,1.000000,0.516000);
		rgb(122pt)=(1.000000,1.000000,0.512000);
		rgb(123pt)=(1.000000,1.000000,0.508000);
		rgb(124pt)=(1.000000,1.000000,0.504000);
		rgb(125pt)=(1.000000,1.000000,0.500000);
		rgb(126pt)=(1.000000,1.000000,0.496000);
		rgb(127pt)=(1.000000,1.000000,0.492000);
		rgb(128pt)=(1.000000,1.000000,0.488000);
		rgb(129pt)=(1.000000,1.000000,0.484000);
		rgb(130pt)=(1.000000,1.000000,0.480000);
		rgb(131pt)=(1.000000,1.000000,0.476000);
		rgb(132pt)=(1.000000,1.000000,0.472000);
		rgb(133pt)=(1.000000,1.000000,0.468000);
		rgb(134pt)=(1.000000,1.000000,0.464000);
		rgb(135pt)=(1.000000,1.000000,0.460000);
		rgb(136pt)=(1.000000,1.000000,0.456000);
		rgb(137pt)=(1.000000,1.000000,0.452000);
		rgb(138pt)=(1.000000,1.000000,0.448000);
		rgb(139pt)=(1.000000,1.000000,0.444000);
		rgb(140pt)=(1.000000,1.000000,0.440000);
		rgb(141pt)=(1.000000,1.000000,0.436000);
		rgb(142pt)=(1.000000,1.000000,0.432000);
		rgb(143pt)=(1.000000,1.000000,0.428000);
		rgb(144pt)=(1.000000,1.000000,0.424000);
		rgb(145pt)=(1.000000,1.000000,0.420000);
		rgb(146pt)=(1.000000,1.000000,0.416000);
		rgb(147pt)=(1.000000,1.000000,0.412000);
		rgb(148pt)=(1.000000,1.000000,0.408000);
		rgb(149pt)=(1.000000,1.000000,0.404000);
		rgb(150pt)=(1.000000,1.000000,0.400000);
		rgb(151pt)=(1.000000,1.000000,0.396000);
		rgb(152pt)=(1.000000,1.000000,0.392000);
		rgb(153pt)=(1.000000,1.000000,0.388000);
		rgb(154pt)=(1.000000,1.000000,0.384000);
		rgb(155pt)=(1.000000,1.000000,0.380000);
		rgb(156pt)=(1.000000,1.000000,0.376000);
		rgb(157pt)=(1.000000,1.000000,0.372000);
		rgb(158pt)=(1.000000,1.000000,0.368000);
		rgb(159pt)=(1.000000,1.000000,0.364000);
		rgb(160pt)=(1.000000,1.000000,0.360000);
		rgb(161pt)=(1.000000,1.000000,0.356000);
		rgb(162pt)=(1.000000,1.000000,0.352000);
		rgb(163pt)=(1.000000,1.000000,0.348000);
		rgb(164pt)=(1.000000,1.000000,0.344000);
		rgb(165pt)=(1.000000,1.000000,0.340000);
		rgb(166pt)=(1.000000,1.000000,0.336000);
		rgb(167pt)=(1.000000,1.000000,0.332000);
		rgb(168pt)=(1.000000,1.000000,0.328000);
		rgb(169pt)=(1.000000,1.000000,0.324000);
		rgb(170pt)=(1.000000,1.000000,0.320000);
		rgb(171pt)=(1.000000,1.000000,0.316000);
		rgb(172pt)=(1.000000,1.000000,0.312000);
		rgb(173pt)=(1.000000,1.000000,0.308000);
		rgb(174pt)=(1.000000,1.000000,0.304000);
		rgb(175pt)=(1.000000,1.000000,0.300000);
		rgb(176pt)=(1.000000,1.000000,0.296000);
		rgb(177pt)=(1.000000,1.000000,0.292000);
		rgb(178pt)=(1.000000,1.000000,0.288000);
		rgb(179pt)=(1.000000,1.000000,0.284000);
		rgb(180pt)=(1.000000,1.000000,0.280000);
		rgb(181pt)=(1.000000,1.000000,0.276000);
		rgb(182pt)=(1.000000,1.000000,0.272000);
		rgb(183pt)=(1.000000,1.000000,0.268000);
		rgb(184pt)=(1.000000,1.000000,0.264000);
		rgb(185pt)=(1.000000,1.000000,0.260000);
		rgb(186pt)=(1.000000,1.000000,0.256000);
		rgb(187pt)=(1.000000,1.000000,0.252000);
		rgb(188pt)=(1.000000,1.000000,0.248000);
		rgb(189pt)=(1.000000,1.000000,0.244000);
		rgb(190pt)=(1.000000,1.000000,0.240000);
		rgb(191pt)=(1.000000,1.000000,0.236000);
		rgb(192pt)=(1.000000,1.000000,0.232000);
		rgb(193pt)=(1.000000,1.000000,0.228000);
		rgb(194pt)=(1.000000,1.000000,0.224000);
		rgb(195pt)=(1.000000,1.000000,0.220000);
		rgb(196pt)=(1.000000,1.000000,0.216000);
		rgb(197pt)=(1.000000,1.000000,0.212000);
		rgb(198pt)=(1.000000,1.000000,0.208000);
		rgb(199pt)=(1.000000,1.000000,0.204000);
		rgb(200pt)=(1.000000,1.000000,0.200000);
		rgb(201pt)=(1.000000,1.000000,0.196000);
		rgb(202pt)=(1.000000,1.000000,0.192000);
		rgb(203pt)=(1.000000,1.000000,0.188000);
		rgb(204pt)=(1.000000,1.000000,0.184000);
		rgb(205pt)=(1.000000,1.000000,0.180000);
		rgb(206pt)=(1.000000,1.000000,0.176000);
		rgb(207pt)=(1.000000,1.000000,0.172000);
		rgb(208pt)=(1.000000,1.000000,0.168000);
		rgb(209pt)=(1.000000,1.000000,0.164000);
		rgb(210pt)=(1.000000,1.000000,0.160000);
		rgb(211pt)=(1.000000,1.000000,0.156000);
		rgb(212pt)=(1.000000,1.000000,0.152000);
		rgb(213pt)=(1.000000,1.000000,0.148000);
		rgb(214pt)=(1.000000,1.000000,0.144000);
		rgb(215pt)=(1.000000,1.000000,0.140000);
		rgb(216pt)=(1.000000,1.000000,0.136000);
		rgb(217pt)=(1.000000,1.000000,0.132000);
		rgb(218pt)=(1.000000,1.000000,0.128000);
		rgb(219pt)=(1.000000,1.000000,0.124000);
		rgb(220pt)=(1.000000,1.000000,0.120000);
		rgb(221pt)=(1.000000,1.000000,0.116000);
		rgb(222pt)=(1.000000,1.000000,0.112000);
		rgb(223pt)=(1.000000,1.000000,0.108000);
		rgb(224pt)=(1.000000,1.000000,0.104000);
		rgb(225pt)=(1.000000,1.000000,0.100000);
		rgb(226pt)=(1.000000,1.000000,0.096000);
		rgb(227pt)=(1.000000,1.000000,0.092000);
		rgb(228pt)=(1.000000,1.000000,0.088000);
		rgb(229pt)=(1.000000,1.000000,0.084000);
		rgb(230pt)=(1.000000,1.000000,0.080000);
		rgb(231pt)=(1.000000,1.000000,0.076000);
		rgb(232pt)=(1.000000,1.000000,0.072000);
		rgb(233pt)=(1.000000,1.000000,0.068000);
		rgb(234pt)=(1.000000,1.000000,0.064000);
		rgb(235pt)=(1.000000,1.000000,0.060000);
		rgb(236pt)=(1.000000,1.000000,0.056000);
		rgb(237pt)=(1.000000,1.000000,0.052000);
		rgb(238pt)=(1.000000,1.000000,0.048000);
		rgb(239pt)=(1.000000,1.000000,0.044000);
		rgb(240pt)=(1.000000,1.000000,0.040000);
		rgb(241pt)=(1.000000,1.000000,0.036000);
		rgb(242pt)=(1.000000,1.000000,0.032000);
		rgb(243pt)=(1.000000,1.000000,0.028000);
		rgb(244pt)=(1.000000,1.000000,0.024000);
		rgb(245pt)=(1.000000,1.000000,0.020000);
		rgb(246pt)=(1.000000,1.000000,0.016000);
		rgb(247pt)=(1.000000,1.000000,0.012000);
		rgb(248pt)=(1.000000,1.000000,0.008000);
		rgb(249pt)=(1.000000,1.000000,0.004000);
		rgb(250pt)=(1.000000,1.000000,0.000000);
		rgb(251pt)=(1.000000,0.997333,0.000000);
		rgb(252pt)=(1.000000,0.994667,0.000000);
		rgb(253pt)=(1.000000,0.992000,0.000000);
		rgb(254pt)=(1.000000,0.989333,0.000000);
		rgb(255pt)=(1.000000,0.986667,0.000000);
		rgb(256pt)=(1.000000,0.984000,0.000000);
		rgb(257pt)=(1.000000,0.981333,0.000000);
		rgb(258pt)=(1.000000,0.978667,0.000000);
		rgb(259pt)=(1.000000,0.976000,0.000000);
		rgb(260pt)=(1.000000,0.973333,0.000000);
		rgb(261pt)=(1.000000,0.970667,0.000000);
		rgb(262pt)=(1.000000,0.968000,0.000000);
		rgb(263pt)=(1.000000,0.965333,0.000000);
		rgb(264pt)=(1.000000,0.962667,0.000000);
		rgb(265pt)=(1.000000,0.960000,0.000000);
		rgb(266pt)=(1.000000,0.957333,0.000000);
		rgb(267pt)=(1.000000,0.954667,0.000000);
		rgb(268pt)=(1.000000,0.952000,0.000000);
		rgb(269pt)=(1.000000,0.949333,0.000000);
		rgb(270pt)=(1.000000,0.946667,0.000000);
		rgb(271pt)=(1.000000,0.944000,0.000000);
		rgb(272pt)=(1.000000,0.941333,0.000000);
		rgb(273pt)=(1.000000,0.938667,0.000000);
		rgb(274pt)=(1.000000,0.936000,0.000000);
		rgb(275pt)=(1.000000,0.933333,0.000000);
		rgb(276pt)=(1.000000,0.930667,0.000000);
		rgb(277pt)=(1.000000,0.928000,0.000000);
		rgb(278pt)=(1.000000,0.925333,0.000000);
		rgb(279pt)=(1.000000,0.922667,0.000000);
		rgb(280pt)=(1.000000,0.920000,0.000000);
		rgb(281pt)=(1.000000,0.917333,0.000000);
		rgb(282pt)=(1.000000,0.914667,0.000000);
		rgb(283pt)=(1.000000,0.912000,0.000000);
		rgb(284pt)=(1.000000,0.909333,0.000000);
		rgb(285pt)=(1.000000,0.906667,0.000000);
		rgb(286pt)=(1.000000,0.904000,0.000000);
		rgb(287pt)=(1.000000,0.901333,0.000000);
		rgb(288pt)=(1.000000,0.898667,0.000000);
		rgb(289pt)=(1.000000,0.896000,0.000000);
		rgb(290pt)=(1.000000,0.893333,0.000000);
		rgb(291pt)=(1.000000,0.890667,0.000000);
		rgb(292pt)=(1.000000,0.888000,0.000000);
		rgb(293pt)=(1.000000,0.885333,0.000000);
		rgb(294pt)=(1.000000,0.882667,0.000000);
		rgb(295pt)=(1.000000,0.880000,0.000000);
		rgb(296pt)=(1.000000,0.877333,0.000000);
		rgb(297pt)=(1.000000,0.874667,0.000000);
		rgb(298pt)=(1.000000,0.872000,0.000000);
		rgb(299pt)=(1.000000,0.869333,0.000000);
		rgb(300pt)=(1.000000,0.866667,0.000000);
		rgb(301pt)=(1.000000,0.864000,0.000000);
		rgb(302pt)=(1.000000,0.861333,0.000000);
		rgb(303pt)=(1.000000,0.858667,0.000000);
		rgb(304pt)=(1.000000,0.856000,0.000000);
		rgb(305pt)=(1.000000,0.853333,0.000000);
		rgb(306pt)=(1.000000,0.850667,0.000000);
		rgb(307pt)=(1.000000,0.848000,0.000000);
		rgb(308pt)=(1.000000,0.845333,0.000000);
		rgb(309pt)=(1.000000,0.842667,0.000000);
		rgb(310pt)=(1.000000,0.840000,0.000000);
		rgb(311pt)=(1.000000,0.837333,0.000000);
		rgb(312pt)=(1.000000,0.834667,0.000000);
		rgb(313pt)=(1.000000,0.832000,0.000000);
		rgb(314pt)=(1.000000,0.829333,0.000000);
		rgb(315pt)=(1.000000,0.826667,0.000000);
		rgb(316pt)=(1.000000,0.824000,0.000000);
		rgb(317pt)=(1.000000,0.821333,0.000000);
		rgb(318pt)=(1.000000,0.818667,0.000000);
		rgb(319pt)=(1.000000,0.816000,0.000000);
		rgb(320pt)=(1.000000,0.813333,0.000000);
		rgb(321pt)=(1.000000,0.810667,0.000000);
		rgb(322pt)=(1.000000,0.808000,0.000000);
		rgb(323pt)=(1.000000,0.805333,0.000000);
		rgb(324pt)=(1.000000,0.802667,0.000000);
		rgb(325pt)=(1.000000,0.800000,0.000000);
		rgb(326pt)=(1.000000,0.797333,0.000000);
		rgb(327pt)=(1.000000,0.794667,0.000000);
		rgb(328pt)=(1.000000,0.792000,0.000000);
		rgb(329pt)=(1.000000,0.789333,0.000000);
		rgb(330pt)=(1.000000,0.786667,0.000000);
		rgb(331pt)=(1.000000,0.784000,0.000000);
		rgb(332pt)=(1.000000,0.781333,0.000000);
		rgb(333pt)=(1.000000,0.778667,0.000000);
		rgb(334pt)=(1.000000,0.776000,0.000000);
		rgb(335pt)=(1.000000,0.773333,0.000000);
		rgb(336pt)=(1.000000,0.770667,0.000000);
		rgb(337pt)=(1.000000,0.768000,0.000000);
		rgb(338pt)=(1.000000,0.765333,0.000000);
		rgb(339pt)=(1.000000,0.762667,0.000000);
		rgb(340pt)=(1.000000,0.760000,0.000000);
		rgb(341pt)=(1.000000,0.757333,0.000000);
		rgb(342pt)=(1.000000,0.754667,0.000000);
		rgb(343pt)=(1.000000,0.752000,0.000000);
		rgb(344pt)=(1.000000,0.749333,0.000000);
		rgb(345pt)=(1.000000,0.746667,0.000000);
		rgb(346pt)=(1.000000,0.744000,0.000000);
		rgb(347pt)=(1.000000,0.741333,0.000000);
		rgb(348pt)=(1.000000,0.738667,0.000000);
		rgb(349pt)=(1.000000,0.736000,0.000000);
		rgb(350pt)=(1.000000,0.733333,0.000000);
		rgb(351pt)=(1.000000,0.730667,0.000000);
		rgb(352pt)=(1.000000,0.728000,0.000000);
		rgb(353pt)=(1.000000,0.725333,0.000000);
		rgb(354pt)=(1.000000,0.722667,0.000000);
		rgb(355pt)=(1.000000,0.720000,0.000000);
		rgb(356pt)=(1.000000,0.717333,0.000000);
		rgb(357pt)=(1.000000,0.714667,0.000000);
		rgb(358pt)=(1.000000,0.712000,0.000000);
		rgb(359pt)=(1.000000,0.709333,0.000000);
		rgb(360pt)=(1.000000,0.706667,0.000000);
		rgb(361pt)=(1.000000,0.704000,0.000000);
		rgb(362pt)=(1.000000,0.701333,0.000000);
		rgb(363pt)=(1.000000,0.698667,0.000000);
		rgb(364pt)=(1.000000,0.696000,0.000000);
		rgb(365pt)=(1.000000,0.693333,0.000000);
		rgb(366pt)=(1.000000,0.690667,0.000000);
		rgb(367pt)=(1.000000,0.688000,0.000000);
		rgb(368pt)=(1.000000,0.685333,0.000000);
		rgb(369pt)=(1.000000,0.682667,0.000000);
		rgb(370pt)=(1.000000,0.680000,0.000000);
		rgb(371pt)=(1.000000,0.677333,0.000000);
		rgb(372pt)=(1.000000,0.674667,0.000000);
		rgb(373pt)=(1.000000,0.672000,0.000000);
		rgb(374pt)=(1.000000,0.669333,0.000000);
		rgb(375pt)=(1.000000,0.666667,0.000000);
		rgb(376pt)=(1.000000,0.664000,0.000000);
		rgb(377pt)=(1.000000,0.661333,0.000000);
		rgb(378pt)=(1.000000,0.658667,0.000000);
		rgb(379pt)=(1.000000,0.656000,0.000000);
		rgb(380pt)=(1.000000,0.653333,0.000000);
		rgb(381pt)=(1.000000,0.650667,0.000000);
		rgb(382pt)=(1.000000,0.648000,0.000000);
		rgb(383pt)=(1.000000,0.645333,0.000000);
		rgb(384pt)=(1.000000,0.642667,0.000000);
		rgb(385pt)=(1.000000,0.640000,0.000000);
		rgb(386pt)=(1.000000,0.637333,0.000000);
		rgb(387pt)=(1.000000,0.634667,0.000000);
		rgb(388pt)=(1.000000,0.632000,0.000000);
		rgb(389pt)=(1.000000,0.629333,0.000000);
		rgb(390pt)=(1.000000,0.626667,0.000000);
		rgb(391pt)=(1.000000,0.624000,0.000000);
		rgb(392pt)=(1.000000,0.621333,0.000000);
		rgb(393pt)=(1.000000,0.618667,0.000000);
		rgb(394pt)=(1.000000,0.616000,0.000000);
		rgb(395pt)=(1.000000,0.613333,0.000000);
		rgb(396pt)=(1.000000,0.610667,0.000000);
		rgb(397pt)=(1.000000,0.608000,0.000000);
		rgb(398pt)=(1.000000,0.605333,0.000000);
		rgb(399pt)=(1.000000,0.602667,0.000000);
		rgb(400pt)=(1.000000,0.600000,0.000000);
		rgb(401pt)=(1.000000,0.597333,0.000000);
		rgb(402pt)=(1.000000,0.594667,0.000000);
		rgb(403pt)=(1.000000,0.592000,0.000000);
		rgb(404pt)=(1.000000,0.589333,0.000000);
		rgb(405pt)=(1.000000,0.586667,0.000000);
		rgb(406pt)=(1.000000,0.584000,0.000000);
		rgb(407pt)=(1.000000,0.581333,0.000000);
		rgb(408pt)=(1.000000,0.578667,0.000000);
		rgb(409pt)=(1.000000,0.576000,0.000000);
		rgb(410pt)=(1.000000,0.573333,0.000000);
		rgb(411pt)=(1.000000,0.570667,0.000000);
		rgb(412pt)=(1.000000,0.568000,0.000000);
		rgb(413pt)=(1.000000,0.565333,0.000000);
		rgb(414pt)=(1.000000,0.562667,0.000000);
		rgb(415pt)=(1.000000,0.560000,0.000000);
		rgb(416pt)=(1.000000,0.557333,0.000000);
		rgb(417pt)=(1.000000,0.554667,0.000000);
		rgb(418pt)=(1.000000,0.552000,0.000000);
		rgb(419pt)=(1.000000,0.549333,0.000000);
		rgb(420pt)=(1.000000,0.546667,0.000000);
		rgb(421pt)=(1.000000,0.544000,0.000000);
		rgb(422pt)=(1.000000,0.541333,0.000000);
		rgb(423pt)=(1.000000,0.538667,0.000000);
		rgb(424pt)=(1.000000,0.536000,0.000000);
		rgb(425pt)=(1.000000,0.533333,0.000000);
		rgb(426pt)=(1.000000,0.530667,0.000000);
		rgb(427pt)=(1.000000,0.528000,0.000000);
		rgb(428pt)=(1.000000,0.525333,0.000000);
		rgb(429pt)=(1.000000,0.522667,0.000000);
		rgb(430pt)=(1.000000,0.520000,0.000000);
		rgb(431pt)=(1.000000,0.517333,0.000000);
		rgb(432pt)=(1.000000,0.514667,0.000000);
		rgb(433pt)=(1.000000,0.512000,0.000000);
		rgb(434pt)=(1.000000,0.509333,0.000000);
		rgb(435pt)=(1.000000,0.506667,0.000000);
		rgb(436pt)=(1.000000,0.504000,0.000000);
		rgb(437pt)=(1.000000,0.501333,0.000000);
		rgb(438pt)=(1.000000,0.498667,0.000000);
		rgb(439pt)=(1.000000,0.496000,0.000000);
		rgb(440pt)=(1.000000,0.493333,0.000000);
		rgb(441pt)=(1.000000,0.490667,0.000000);
		rgb(442pt)=(1.000000,0.488000,0.000000);
		rgb(443pt)=(1.000000,0.485333,0.000000);
		rgb(444pt)=(1.000000,0.482667,0.000000);
		rgb(445pt)=(1.000000,0.480000,0.000000);
		rgb(446pt)=(1.000000,0.477333,0.000000);
		rgb(447pt)=(1.000000,0.474667,0.000000);
		rgb(448pt)=(1.000000,0.472000,0.000000);
		rgb(449pt)=(1.000000,0.469333,0.000000);
		rgb(450pt)=(1.000000,0.466667,0.000000);
		rgb(451pt)=(1.000000,0.464000,0.000000);
		rgb(452pt)=(1.000000,0.461333,0.000000);
		rgb(453pt)=(1.000000,0.458667,0.000000);
		rgb(454pt)=(1.000000,0.456000,0.000000);
		rgb(455pt)=(1.000000,0.453333,0.000000);
		rgb(456pt)=(1.000000,0.450667,0.000000);
		rgb(457pt)=(1.000000,0.448000,0.000000);
		rgb(458pt)=(1.000000,0.445333,0.000000);
		rgb(459pt)=(1.000000,0.442667,0.000000);
		rgb(460pt)=(1.000000,0.440000,0.000000);
		rgb(461pt)=(1.000000,0.437333,0.000000);
		rgb(462pt)=(1.000000,0.434667,0.000000);
		rgb(463pt)=(1.000000,0.432000,0.000000);
		rgb(464pt)=(1.000000,0.429333,0.000000);
		rgb(465pt)=(1.000000,0.426667,0.000000);
		rgb(466pt)=(1.000000,0.424000,0.000000);
		rgb(467pt)=(1.000000,0.421333,0.000000);
		rgb(468pt)=(1.000000,0.418667,0.000000);
		rgb(469pt)=(1.000000,0.416000,0.000000);
		rgb(470pt)=(1.000000,0.413333,0.000000);
		rgb(471pt)=(1.000000,0.410667,0.000000);
		rgb(472pt)=(1.000000,0.408000,0.000000);
		rgb(473pt)=(1.000000,0.405333,0.000000);
		rgb(474pt)=(1.000000,0.402667,0.000000);
		rgb(475pt)=(1.000000,0.400000,0.000000);
		rgb(476pt)=(1.000000,0.397333,0.000000);
		rgb(477pt)=(1.000000,0.394667,0.000000);
		rgb(478pt)=(1.000000,0.392000,0.000000);
		rgb(479pt)=(1.000000,0.389333,0.000000);
		rgb(480pt)=(1.000000,0.386667,0.000000);
		rgb(481pt)=(1.000000,0.384000,0.000000);
		rgb(482pt)=(1.000000,0.381333,0.000000);
		rgb(483pt)=(1.000000,0.378667,0.000000);
		rgb(484pt)=(1.000000,0.376000,0.000000);
		rgb(485pt)=(1.000000,0.373333,0.000000);
		rgb(486pt)=(1.000000,0.370667,0.000000);
		rgb(487pt)=(1.000000,0.368000,0.000000);
		rgb(488pt)=(1.000000,0.365333,0.000000);
		rgb(489pt)=(1.000000,0.362667,0.000000);
		rgb(490pt)=(1.000000,0.360000,0.000000);
		rgb(491pt)=(1.000000,0.357333,0.000000);
		rgb(492pt)=(1.000000,0.354667,0.000000);
		rgb(493pt)=(1.000000,0.352000,0.000000);
		rgb(494pt)=(1.000000,0.349333,0.000000);
		rgb(495pt)=(1.000000,0.346667,0.000000);
		rgb(496pt)=(1.000000,0.344000,0.000000);
		rgb(497pt)=(1.000000,0.341333,0.000000);
		rgb(498pt)=(1.000000,0.338667,0.000000);
		rgb(499pt)=(1.000000,0.336000,0.000000);
		rgb(500pt)=(1.000000,0.333333,0.000000);
		rgb(501pt)=(1.000000,0.330667,0.000000);
		rgb(502pt)=(1.000000,0.328000,0.000000);
		rgb(503pt)=(1.000000,0.325333,0.000000);
		rgb(504pt)=(1.000000,0.322667,0.000000);
		rgb(505pt)=(1.000000,0.320000,0.000000);
		rgb(506pt)=(1.000000,0.317333,0.000000);
		rgb(507pt)=(1.000000,0.314667,0.000000);
		rgb(508pt)=(1.000000,0.312000,0.000000);
		rgb(509pt)=(1.000000,0.309333,0.000000);
		rgb(510pt)=(1.000000,0.306667,0.000000);
		rgb(511pt)=(1.000000,0.304000,0.000000);
		rgb(512pt)=(1.000000,0.301333,0.000000);
		rgb(513pt)=(1.000000,0.298667,0.000000);
		rgb(514pt)=(1.000000,0.296000,0.000000);
		rgb(515pt)=(1.000000,0.293333,0.000000);
		rgb(516pt)=(1.000000,0.290667,0.000000);
		rgb(517pt)=(1.000000,0.288000,0.000000);
		rgb(518pt)=(1.000000,0.285333,0.000000);
		rgb(519pt)=(1.000000,0.282667,0.000000);
		rgb(520pt)=(1.000000,0.280000,0.000000);
		rgb(521pt)=(1.000000,0.277333,0.000000);
		rgb(522pt)=(1.000000,0.274667,0.000000);
		rgb(523pt)=(1.000000,0.272000,0.000000);
		rgb(524pt)=(1.000000,0.269333,0.000000);
		rgb(525pt)=(1.000000,0.266667,0.000000);
		rgb(526pt)=(1.000000,0.264000,0.000000);
		rgb(527pt)=(1.000000,0.261333,0.000000);
		rgb(528pt)=(1.000000,0.258667,0.000000);
		rgb(529pt)=(1.000000,0.256000,0.000000);
		rgb(530pt)=(1.000000,0.253333,0.000000);
		rgb(531pt)=(1.000000,0.250667,0.000000);
		rgb(532pt)=(1.000000,0.248000,0.000000);
		rgb(533pt)=(1.000000,0.245333,0.000000);
		rgb(534pt)=(1.000000,0.242667,0.000000);
		rgb(535pt)=(1.000000,0.240000,0.000000);
		rgb(536pt)=(1.000000,0.237333,0.000000);
		rgb(537pt)=(1.000000,0.234667,0.000000);
		rgb(538pt)=(1.000000,0.232000,0.000000);
		rgb(539pt)=(1.000000,0.229333,0.000000);
		rgb(540pt)=(1.000000,0.226667,0.000000);
		rgb(541pt)=(1.000000,0.224000,0.000000);
		rgb(542pt)=(1.000000,0.221333,0.000000);
		rgb(543pt)=(1.000000,0.218667,0.000000);
		rgb(544pt)=(1.000000,0.216000,0.000000);
		rgb(545pt)=(1.000000,0.213333,0.000000);
		rgb(546pt)=(1.000000,0.210667,0.000000);
		rgb(547pt)=(1.000000,0.208000,0.000000);
		rgb(548pt)=(1.000000,0.205333,0.000000);
		rgb(549pt)=(1.000000,0.202667,0.000000);
		rgb(550pt)=(1.000000,0.200000,0.000000);
		rgb(551pt)=(1.000000,0.197333,0.000000);
		rgb(552pt)=(1.000000,0.194667,0.000000);
		rgb(553pt)=(1.000000,0.192000,0.000000);
		rgb(554pt)=(1.000000,0.189333,0.000000);
		rgb(555pt)=(1.000000,0.186667,0.000000);
		rgb(556pt)=(1.000000,0.184000,0.000000);
		rgb(557pt)=(1.000000,0.181333,0.000000);
		rgb(558pt)=(1.000000,0.178667,0.000000);
		rgb(559pt)=(1.000000,0.176000,0.000000);
		rgb(560pt)=(1.000000,0.173333,0.000000);
		rgb(561pt)=(1.000000,0.170667,0.000000);
		rgb(562pt)=(1.000000,0.168000,0.000000);
		rgb(563pt)=(1.000000,0.165333,0.000000);
		rgb(564pt)=(1.000000,0.162667,0.000000);
		rgb(565pt)=(1.000000,0.160000,0.000000);
		rgb(566pt)=(1.000000,0.157333,0.000000);
		rgb(567pt)=(1.000000,0.154667,0.000000);
		rgb(568pt)=(1.000000,0.152000,0.000000);
		rgb(569pt)=(1.000000,0.149333,0.000000);
		rgb(570pt)=(1.000000,0.146667,0.000000);
		rgb(571pt)=(1.000000,0.144000,0.000000);
		rgb(572pt)=(1.000000,0.141333,0.000000);
		rgb(573pt)=(1.000000,0.138667,0.000000);
		rgb(574pt)=(1.000000,0.136000,0.000000);
		rgb(575pt)=(1.000000,0.133333,0.000000);
		rgb(576pt)=(1.000000,0.130667,0.000000);
		rgb(577pt)=(1.000000,0.128000,0.000000);
		rgb(578pt)=(1.000000,0.125333,0.000000);
		rgb(579pt)=(1.000000,0.122667,0.000000);
		rgb(580pt)=(1.000000,0.120000,0.000000);
		rgb(581pt)=(1.000000,0.117333,0.000000);
		rgb(582pt)=(1.000000,0.114667,0.000000);
		rgb(583pt)=(1.000000,0.112000,0.000000);
		rgb(584pt)=(1.000000,0.109333,0.000000);
		rgb(585pt)=(1.000000,0.106667,0.000000);
		rgb(586pt)=(1.000000,0.104000,0.000000);
		rgb(587pt)=(1.000000,0.101333,0.000000);
		rgb(588pt)=(1.000000,0.098667,0.000000);
		rgb(589pt)=(1.000000,0.096000,0.000000);
		rgb(590pt)=(1.000000,0.093333,0.000000);
		rgb(591pt)=(1.000000,0.090667,0.000000);
		rgb(592pt)=(1.000000,0.088000,0.000000);
		rgb(593pt)=(1.000000,0.085333,0.000000);
		rgb(594pt)=(1.000000,0.082667,0.000000);
		rgb(595pt)=(1.000000,0.080000,0.000000);
		rgb(596pt)=(1.000000,0.077333,0.000000);
		rgb(597pt)=(1.000000,0.074667,0.000000);
		rgb(598pt)=(1.000000,0.072000,0.000000);
		rgb(599pt)=(1.000000,0.069333,0.000000);
		rgb(600pt)=(1.000000,0.066667,0.000000);
		rgb(601pt)=(1.000000,0.064000,0.000000);
		rgb(602pt)=(1.000000,0.061333,0.000000);
		rgb(603pt)=(1.000000,0.058667,0.000000);
		rgb(604pt)=(1.000000,0.056000,0.000000);
		rgb(605pt)=(1.000000,0.053333,0.000000);
		rgb(606pt)=(1.000000,0.050667,0.000000);
		rgb(607pt)=(1.000000,0.048000,0.000000);
		rgb(608pt)=(1.000000,0.045333,0.000000);
		rgb(609pt)=(1.000000,0.042667,0.000000);
		rgb(610pt)=(1.000000,0.040000,0.000000);
		rgb(611pt)=(1.000000,0.037333,0.000000);
		rgb(612pt)=(1.000000,0.034667,0.000000);
		rgb(613pt)=(1.000000,0.032000,0.000000);
		rgb(614pt)=(1.000000,0.029333,0.000000);
		rgb(615pt)=(1.000000,0.026667,0.000000);
		rgb(616pt)=(1.000000,0.024000,0.000000);
		rgb(617pt)=(1.000000,0.021333,0.000000);
		rgb(618pt)=(1.000000,0.018667,0.000000);
		rgb(619pt)=(1.000000,0.016000,0.000000);
		rgb(620pt)=(1.000000,0.013333,0.000000);
		rgb(621pt)=(1.000000,0.010667,0.000000);
		rgb(622pt)=(1.000000,0.008000,0.000000);
		rgb(623pt)=(1.000000,0.005333,0.000000);
		rgb(624pt)=(1.000000,0.002667,0.000000);
		rgb(625pt)=(1.000000,0.000000,0.000000);
		rgb(626pt)=(0.997333,0.000000,0.000000);
		rgb(627pt)=(0.994667,0.000000,0.000000);
		rgb(628pt)=(0.992000,0.000000,0.000000);
		rgb(629pt)=(0.989333,0.000000,0.000000);
		rgb(630pt)=(0.986667,0.000000,0.000000);
		rgb(631pt)=(0.984000,0.000000,0.000000);
		rgb(632pt)=(0.981333,0.000000,0.000000);
		rgb(633pt)=(0.978667,0.000000,0.000000);
		rgb(634pt)=(0.976000,0.000000,0.000000);
		rgb(635pt)=(0.973333,0.000000,0.000000);
		rgb(636pt)=(0.970667,0.000000,0.000000);
		rgb(637pt)=(0.968000,0.000000,0.000000);
		rgb(638pt)=(0.965333,0.000000,0.000000);
		rgb(639pt)=(0.962667,0.000000,0.000000);
		rgb(640pt)=(0.960000,0.000000,0.000000);
		rgb(641pt)=(0.957333,0.000000,0.000000);
		rgb(642pt)=(0.954667,0.000000,0.000000);
		rgb(643pt)=(0.952000,0.000000,0.000000);
		rgb(644pt)=(0.949333,0.000000,0.000000);
		rgb(645pt)=(0.946667,0.000000,0.000000);
		rgb(646pt)=(0.944000,0.000000,0.000000);
		rgb(647pt)=(0.941333,0.000000,0.000000);
		rgb(648pt)=(0.938667,0.000000,0.000000);
		rgb(649pt)=(0.936000,0.000000,0.000000);
		rgb(650pt)=(0.933333,0.000000,0.000000);
		rgb(651pt)=(0.930667,0.000000,0.000000);
		rgb(652pt)=(0.928000,0.000000,0.000000);
		rgb(653pt)=(0.925333,0.000000,0.000000);
		rgb(654pt)=(0.922667,0.000000,0.000000);
		rgb(655pt)=(0.920000,0.000000,0.000000);
		rgb(656pt)=(0.917333,0.000000,0.000000);
		rgb(657pt)=(0.914667,0.000000,0.000000);
		rgb(658pt)=(0.912000,0.000000,0.000000);
		rgb(659pt)=(0.909333,0.000000,0.000000);
		rgb(660pt)=(0.906667,0.000000,0.000000);
		rgb(661pt)=(0.904000,0.000000,0.000000);
		rgb(662pt)=(0.901333,0.000000,0.000000);
		rgb(663pt)=(0.898667,0.000000,0.000000);
		rgb(664pt)=(0.896000,0.000000,0.000000);
		rgb(665pt)=(0.893333,0.000000,0.000000);
		rgb(666pt)=(0.890667,0.000000,0.000000);
		rgb(667pt)=(0.888000,0.000000,0.000000);
		rgb(668pt)=(0.885333,0.000000,0.000000);
		rgb(669pt)=(0.882667,0.000000,0.000000);
		rgb(670pt)=(0.880000,0.000000,0.000000);
		rgb(671pt)=(0.877333,0.000000,0.000000);
		rgb(672pt)=(0.874667,0.000000,0.000000);
		rgb(673pt)=(0.872000,0.000000,0.000000);
		rgb(674pt)=(0.869333,0.000000,0.000000);
		rgb(675pt)=(0.866667,0.000000,0.000000);
		rgb(676pt)=(0.864000,0.000000,0.000000);
		rgb(677pt)=(0.861333,0.000000,0.000000);
		rgb(678pt)=(0.858667,0.000000,0.000000);
		rgb(679pt)=(0.856000,0.000000,0.000000);
		rgb(680pt)=(0.853333,0.000000,0.000000);
		rgb(681pt)=(0.850667,0.000000,0.000000);
		rgb(682pt)=(0.848000,0.000000,0.000000);
		rgb(683pt)=(0.845333,0.000000,0.000000);
		rgb(684pt)=(0.842667,0.000000,0.000000);
		rgb(685pt)=(0.840000,0.000000,0.000000);
		rgb(686pt)=(0.837333,0.000000,0.000000);
		rgb(687pt)=(0.834667,0.000000,0.000000);
		rgb(688pt)=(0.832000,0.000000,0.000000);
		rgb(689pt)=(0.829333,0.000000,0.000000);
		rgb(690pt)=(0.826667,0.000000,0.000000);
		rgb(691pt)=(0.824000,0.000000,0.000000);
		rgb(692pt)=(0.821333,0.000000,0.000000);
		rgb(693pt)=(0.818667,0.000000,0.000000);
		rgb(694pt)=(0.816000,0.000000,0.000000);
		rgb(695pt)=(0.813333,0.000000,0.000000);
		rgb(696pt)=(0.810667,0.000000,0.000000);
		rgb(697pt)=(0.808000,0.000000,0.000000);
		rgb(698pt)=(0.805333,0.000000,0.000000);
		rgb(699pt)=(0.802667,0.000000,0.000000);
		rgb(700pt)=(0.800000,0.000000,0.000000);
		rgb(701pt)=(0.797333,0.000000,0.000000);
		rgb(702pt)=(0.794667,0.000000,0.000000);
		rgb(703pt)=(0.792000,0.000000,0.000000);
		rgb(704pt)=(0.789333,0.000000,0.000000);
		rgb(705pt)=(0.786667,0.000000,0.000000);
		rgb(706pt)=(0.784000,0.000000,0.000000);
		rgb(707pt)=(0.781333,0.000000,0.000000);
		rgb(708pt)=(0.778667,0.000000,0.000000);
		rgb(709pt)=(0.776000,0.000000,0.000000);
		rgb(710pt)=(0.773333,0.000000,0.000000);
		rgb(711pt)=(0.770667,0.000000,0.000000);
		rgb(712pt)=(0.768000,0.000000,0.000000);
		rgb(713pt)=(0.765333,0.000000,0.000000);
		rgb(714pt)=(0.762667,0.000000,0.000000);
		rgb(715pt)=(0.760000,0.000000,0.000000);
		rgb(716pt)=(0.757333,0.000000,0.000000);
		rgb(717pt)=(0.754667,0.000000,0.000000);
		rgb(718pt)=(0.752000,0.000000,0.000000);
		rgb(719pt)=(0.749333,0.000000,0.000000);
		rgb(720pt)=(0.746667,0.000000,0.000000);
		rgb(721pt)=(0.744000,0.000000,0.000000);
		rgb(722pt)=(0.741333,0.000000,0.000000);
		rgb(723pt)=(0.738667,0.000000,0.000000);
		rgb(724pt)=(0.736000,0.000000,0.000000);
		rgb(725pt)=(0.733333,0.000000,0.000000);
		rgb(726pt)=(0.730667,0.000000,0.000000);
		rgb(727pt)=(0.728000,0.000000,0.000000);
		rgb(728pt)=(0.725333,0.000000,0.000000);
		rgb(729pt)=(0.722667,0.000000,0.000000);
		rgb(730pt)=(0.720000,0.000000,0.000000);
		rgb(731pt)=(0.717333,0.000000,0.000000);
		rgb(732pt)=(0.714667,0.000000,0.000000);
		rgb(733pt)=(0.712000,0.000000,0.000000);
		rgb(734pt)=(0.709333,0.000000,0.000000);
		rgb(735pt)=(0.706667,0.000000,0.000000);
		rgb(736pt)=(0.704000,0.000000,0.000000);
		rgb(737pt)=(0.701333,0.000000,0.000000);
		rgb(738pt)=(0.698667,0.000000,0.000000);
		rgb(739pt)=(0.696000,0.000000,0.000000);
		rgb(740pt)=(0.693333,0.000000,0.000000);
		rgb(741pt)=(0.690667,0.000000,0.000000);
		rgb(742pt)=(0.688000,0.000000,0.000000);
		rgb(743pt)=(0.685333,0.000000,0.000000);
		rgb(744pt)=(0.682667,0.000000,0.000000);
		rgb(745pt)=(0.680000,0.000000,0.000000);
		rgb(746pt)=(0.677333,0.000000,0.000000);
		rgb(747pt)=(0.674667,0.000000,0.000000);
		rgb(748pt)=(0.672000,0.000000,0.000000);
		rgb(749pt)=(0.669333,0.000000,0.000000);
		rgb(750pt)=(0.666667,0.000000,0.000000);
		rgb(751pt)=(0.664000,0.000000,0.000000);
		rgb(752pt)=(0.661333,0.000000,0.000000);
		rgb(753pt)=(0.658667,0.000000,0.000000);
		rgb(754pt)=(0.656000,0.000000,0.000000);
		rgb(755pt)=(0.653333,0.000000,0.000000);
		rgb(756pt)=(0.650667,0.000000,0.000000);
		rgb(757pt)=(0.648000,0.000000,0.000000);
		rgb(758pt)=(0.645333,0.000000,0.000000);
		rgb(759pt)=(0.642667,0.000000,0.000000);
		rgb(760pt)=(0.640000,0.000000,0.000000);
		rgb(761pt)=(0.637333,0.000000,0.000000);
		rgb(762pt)=(0.634667,0.000000,0.000000);
		rgb(763pt)=(0.632000,0.000000,0.000000);
		rgb(764pt)=(0.629333,0.000000,0.000000);
		rgb(765pt)=(0.626667,0.000000,0.000000);
		rgb(766pt)=(0.624000,0.000000,0.000000);
		rgb(767pt)=(0.621333,0.000000,0.000000);
		rgb(768pt)=(0.618667,0.000000,0.000000);
		rgb(769pt)=(0.616000,0.000000,0.000000);
		rgb(770pt)=(0.613333,0.000000,0.000000);
		rgb(771pt)=(0.610667,0.000000,0.000000);
		rgb(772pt)=(0.608000,0.000000,0.000000);
		rgb(773pt)=(0.605333,0.000000,0.000000);
		rgb(774pt)=(0.602667,0.000000,0.000000);
		rgb(775pt)=(0.600000,0.000000,0.000000);
		rgb(776pt)=(0.597333,0.000000,0.000000);
		rgb(777pt)=(0.594667,0.000000,0.000000);
		rgb(778pt)=(0.592000,0.000000,0.000000);
		rgb(779pt)=(0.589333,0.000000,0.000000);
		rgb(780pt)=(0.586667,0.000000,0.000000);
		rgb(781pt)=(0.584000,0.000000,0.000000);
		rgb(782pt)=(0.581333,0.000000,0.000000);
		rgb(783pt)=(0.578667,0.000000,0.000000);
		rgb(784pt)=(0.576000,0.000000,0.000000);
		rgb(785pt)=(0.573333,0.000000,0.000000);
		rgb(786pt)=(0.570667,0.000000,0.000000);
		rgb(787pt)=(0.568000,0.000000,0.000000);
		rgb(788pt)=(0.565333,0.000000,0.000000);
		rgb(789pt)=(0.562667,0.000000,0.000000);
		rgb(790pt)=(0.560000,0.000000,0.000000);
		rgb(791pt)=(0.557333,0.000000,0.000000);
		rgb(792pt)=(0.554667,0.000000,0.000000);
		rgb(793pt)=(0.552000,0.000000,0.000000);
		rgb(794pt)=(0.549333,0.000000,0.000000);
		rgb(795pt)=(0.546667,0.000000,0.000000);
		rgb(796pt)=(0.544000,0.000000,0.000000);
		rgb(797pt)=(0.541333,0.000000,0.000000);
		rgb(798pt)=(0.538667,0.000000,0.000000);
		rgb(799pt)=(0.536000,0.000000,0.000000);
		rgb(800pt)=(0.533333,0.000000,0.000000);
		rgb(801pt)=(0.530667,0.000000,0.000000);
		rgb(802pt)=(0.528000,0.000000,0.000000);
		rgb(803pt)=(0.525333,0.000000,0.000000);
		rgb(804pt)=(0.522667,0.000000,0.000000);
		rgb(805pt)=(0.520000,0.000000,0.000000);
		rgb(806pt)=(0.517333,0.000000,0.000000);
		rgb(807pt)=(0.514667,0.000000,0.000000);
		rgb(808pt)=(0.512000,0.000000,0.000000);
		rgb(809pt)=(0.509333,0.000000,0.000000);
		rgb(810pt)=(0.506667,0.000000,0.000000);
		rgb(811pt)=(0.504000,0.000000,0.000000);
		rgb(812pt)=(0.501333,0.000000,0.000000);
		rgb(813pt)=(0.498667,0.000000,0.000000);
		rgb(814pt)=(0.496000,0.000000,0.000000);
		rgb(815pt)=(0.493333,0.000000,0.000000);
		rgb(816pt)=(0.490667,0.000000,0.000000);
		rgb(817pt)=(0.488000,0.000000,0.000000);
		rgb(818pt)=(0.485333,0.000000,0.000000);
		rgb(819pt)=(0.482667,0.000000,0.000000);
		rgb(820pt)=(0.480000,0.000000,0.000000);
		rgb(821pt)=(0.477333,0.000000,0.000000);
		rgb(822pt)=(0.474667,0.000000,0.000000);
		rgb(823pt)=(0.472000,0.000000,0.000000);
		rgb(824pt)=(0.469333,0.000000,0.000000);
		rgb(825pt)=(0.466667,0.000000,0.000000);
		rgb(826pt)=(0.464000,0.000000,0.000000);
		rgb(827pt)=(0.461333,0.000000,0.000000);
		rgb(828pt)=(0.458667,0.000000,0.000000);
		rgb(829pt)=(0.456000,0.000000,0.000000);
		rgb(830pt)=(0.453333,0.000000,0.000000);
		rgb(831pt)=(0.450667,0.000000,0.000000);
		rgb(832pt)=(0.448000,0.000000,0.000000);
		rgb(833pt)=(0.445333,0.000000,0.000000);
		rgb(834pt)=(0.442667,0.000000,0.000000);
		rgb(835pt)=(0.440000,0.000000,0.000000);
		rgb(836pt)=(0.437333,0.000000,0.000000);
		rgb(837pt)=(0.434667,0.000000,0.000000);
		rgb(838pt)=(0.432000,0.000000,0.000000);
		rgb(839pt)=(0.429333,0.000000,0.000000);
		rgb(840pt)=(0.426667,0.000000,0.000000);
		rgb(841pt)=(0.424000,0.000000,0.000000);
		rgb(842pt)=(0.421333,0.000000,0.000000);
		rgb(843pt)=(0.418667,0.000000,0.000000);
		rgb(844pt)=(0.416000,0.000000,0.000000);
		rgb(845pt)=(0.413333,0.000000,0.000000);
		rgb(846pt)=(0.410667,0.000000,0.000000);
		rgb(847pt)=(0.408000,0.000000,0.000000);
		rgb(848pt)=(0.405333,0.000000,0.000000);
		rgb(849pt)=(0.402667,0.000000,0.000000);
		rgb(850pt)=(0.400000,0.000000,0.000000);
		rgb(851pt)=(0.397333,0.000000,0.000000);
		rgb(852pt)=(0.394667,0.000000,0.000000);
		rgb(853pt)=(0.392000,0.000000,0.000000);
		rgb(854pt)=(0.389333,0.000000,0.000000);
		rgb(855pt)=(0.386667,0.000000,0.000000);
		rgb(856pt)=(0.384000,0.000000,0.000000);
		rgb(857pt)=(0.381333,0.000000,0.000000);
		rgb(858pt)=(0.378667,0.000000,0.000000);
		rgb(859pt)=(0.376000,0.000000,0.000000);
		rgb(860pt)=(0.373333,0.000000,0.000000);
		rgb(861pt)=(0.370667,0.000000,0.000000);
		rgb(862pt)=(0.368000,0.000000,0.000000);
		rgb(863pt)=(0.365333,0.000000,0.000000);
		rgb(864pt)=(0.362667,0.000000,0.000000);
		rgb(865pt)=(0.360000,0.000000,0.000000);
		rgb(866pt)=(0.357333,0.000000,0.000000);
		rgb(867pt)=(0.354667,0.000000,0.000000);
		rgb(868pt)=(0.352000,0.000000,0.000000);
		rgb(869pt)=(0.349333,0.000000,0.000000);
		rgb(870pt)=(0.346667,0.000000,0.000000);
		rgb(871pt)=(0.344000,0.000000,0.000000);
		rgb(872pt)=(0.341333,0.000000,0.000000);
		rgb(873pt)=(0.338667,0.000000,0.000000);
		rgb(874pt)=(0.336000,0.000000,0.000000);
		rgb(875pt)=(0.333333,0.000000,0.000000);
		rgb(876pt)=(0.330667,0.000000,0.000000);
		rgb(877pt)=(0.328000,0.000000,0.000000);
		rgb(878pt)=(0.325333,0.000000,0.000000);
		rgb(879pt)=(0.322667,0.000000,0.000000);
		rgb(880pt)=(0.320000,0.000000,0.000000);
		rgb(881pt)=(0.317333,0.000000,0.000000);
		rgb(882pt)=(0.314667,0.000000,0.000000);
		rgb(883pt)=(0.312000,0.000000,0.000000);
		rgb(884pt)=(0.309333,0.000000,0.000000);
		rgb(885pt)=(0.306667,0.000000,0.000000);
		rgb(886pt)=(0.304000,0.000000,0.000000);
		rgb(887pt)=(0.301333,0.000000,0.000000);
		rgb(888pt)=(0.298667,0.000000,0.000000);
		rgb(889pt)=(0.296000,0.000000,0.000000);
		rgb(890pt)=(0.293333,0.000000,0.000000);
		rgb(891pt)=(0.290667,0.000000,0.000000);
		rgb(892pt)=(0.288000,0.000000,0.000000);
		rgb(893pt)=(0.285333,0.000000,0.000000);
		rgb(894pt)=(0.282667,0.000000,0.000000);
		rgb(895pt)=(0.280000,0.000000,0.000000);
		rgb(896pt)=(0.277333,0.000000,0.000000);
		rgb(897pt)=(0.274667,0.000000,0.000000);
		rgb(898pt)=(0.272000,0.000000,0.000000);
		rgb(899pt)=(0.269333,0.000000,0.000000);
		rgb(900pt)=(0.266667,0.000000,0.000000);
		rgb(901pt)=(0.264000,0.000000,0.000000);
		rgb(902pt)=(0.261333,0.000000,0.000000);
		rgb(903pt)=(0.258667,0.000000,0.000000);
		rgb(904pt)=(0.256000,0.000000,0.000000);
		rgb(905pt)=(0.253333,0.000000,0.000000);
		rgb(906pt)=(0.250667,0.000000,0.000000);
		rgb(907pt)=(0.248000,0.000000,0.000000);
		rgb(908pt)=(0.245333,0.000000,0.000000);
		rgb(909pt)=(0.242667,0.000000,0.000000);
		rgb(910pt)=(0.240000,0.000000,0.000000);
		rgb(911pt)=(0.237333,0.000000,0.000000);
		rgb(912pt)=(0.234667,0.000000,0.000000);
		rgb(913pt)=(0.232000,0.000000,0.000000);
		rgb(914pt)=(0.229333,0.000000,0.000000);
		rgb(915pt)=(0.226667,0.000000,0.000000);
		rgb(916pt)=(0.224000,0.000000,0.000000);
		rgb(917pt)=(0.221333,0.000000,0.000000);
		rgb(918pt)=(0.218667,0.000000,0.000000);
		rgb(919pt)=(0.216000,0.000000,0.000000);
		rgb(920pt)=(0.213333,0.000000,0.000000);
		rgb(921pt)=(0.210667,0.000000,0.000000);
		rgb(922pt)=(0.208000,0.000000,0.000000);
		rgb(923pt)=(0.205333,0.000000,0.000000);
		rgb(924pt)=(0.202667,0.000000,0.000000);
		rgb(925pt)=(0.200000,0.000000,0.000000);
		rgb(926pt)=(0.197333,0.000000,0.000000);
		rgb(927pt)=(0.194667,0.000000,0.000000);
		rgb(928pt)=(0.192000,0.000000,0.000000);
		rgb(929pt)=(0.189333,0.000000,0.000000);
		rgb(930pt)=(0.186667,0.000000,0.000000);
		rgb(931pt)=(0.184000,0.000000,0.000000);
		rgb(932pt)=(0.181333,0.000000,0.000000);
		rgb(933pt)=(0.178667,0.000000,0.000000);
		rgb(934pt)=(0.176000,0.000000,0.000000);
		rgb(935pt)=(0.173333,0.000000,0.000000);
		rgb(936pt)=(0.170667,0.000000,0.000000);
		rgb(937pt)=(0.168000,0.000000,0.000000);
		rgb(938pt)=(0.165333,0.000000,0.000000);
		rgb(939pt)=(0.162667,0.000000,0.000000);
		rgb(940pt)=(0.160000,0.000000,0.000000);
		rgb(941pt)=(0.157333,0.000000,0.000000);
		rgb(942pt)=(0.154667,0.000000,0.000000);
		rgb(943pt)=(0.152000,0.000000,0.000000);
		rgb(944pt)=(0.149333,0.000000,0.000000);
		rgb(945pt)=(0.146667,0.000000,0.000000);
		rgb(946pt)=(0.144000,0.000000,0.000000);
		rgb(947pt)=(0.141333,0.000000,0.000000);
		rgb(948pt)=(0.138667,0.000000,0.000000);
		rgb(949pt)=(0.136000,0.000000,0.000000);
		rgb(950pt)=(0.133333,0.000000,0.000000);
		rgb(951pt)=(0.130667,0.000000,0.000000);
		rgb(952pt)=(0.128000,0.000000,0.000000);
		rgb(953pt)=(0.125333,0.000000,0.000000);
		rgb(954pt)=(0.122667,0.000000,0.000000);
		rgb(955pt)=(0.120000,0.000000,0.000000);
		rgb(956pt)=(0.117333,0.000000,0.000000);
		rgb(957pt)=(0.114667,0.000000,0.000000);
		rgb(958pt)=(0.112000,0.000000,0.000000);
		rgb(959pt)=(0.109333,0.000000,0.000000);
		rgb(960pt)=(0.106667,0.000000,0.000000);
		rgb(961pt)=(0.104000,0.000000,0.000000);
		rgb(962pt)=(0.101333,0.000000,0.000000);
		rgb(963pt)=(0.098667,0.000000,0.000000);
		rgb(964pt)=(0.096000,0.000000,0.000000);
		rgb(965pt)=(0.093333,0.000000,0.000000);
		rgb(966pt)=(0.090667,0.000000,0.000000);
		rgb(967pt)=(0.088000,0.000000,0.000000);
		rgb(968pt)=(0.085333,0.000000,0.000000);
		rgb(969pt)=(0.082667,0.000000,0.000000);
		rgb(970pt)=(0.080000,0.000000,0.000000);
		rgb(971pt)=(0.077333,0.000000,0.000000);
		rgb(972pt)=(0.074667,0.000000,0.000000);
		rgb(973pt)=(0.072000,0.000000,0.000000);
		rgb(974pt)=(0.069333,0.000000,0.000000);
		rgb(975pt)=(0.066667,0.000000,0.000000);
		rgb(976pt)=(0.064000,0.000000,0.000000);
		rgb(977pt)=(0.061333,0.000000,0.000000);
		rgb(978pt)=(0.058667,0.000000,0.000000);
		rgb(979pt)=(0.056000,0.000000,0.000000);
		rgb(980pt)=(0.053333,0.000000,0.000000);
		rgb(981pt)=(0.050667,0.000000,0.000000);
		rgb(982pt)=(0.048000,0.000000,0.000000);
		rgb(983pt)=(0.045333,0.000000,0.000000);
		rgb(984pt)=(0.042667,0.000000,0.000000);
		rgb(985pt)=(0.040000,0.000000,0.000000);
		rgb(986pt)=(0.037333,0.000000,0.000000);
		rgb(987pt)=(0.034667,0.000000,0.000000);
		rgb(988pt)=(0.032000,0.000000,0.000000);
		rgb(989pt)=(0.029333,0.000000,0.000000);
		rgb(990pt)=(0.026667,0.000000,0.000000);
		rgb(991pt)=(0.024000,0.000000,0.000000);
		rgb(992pt)=(0.021333,0.000000,0.000000);
		rgb(993pt)=(0.018667,0.000000,0.000000);
		rgb(994pt)=(0.016000,0.000000,0.000000);
		rgb(995pt)=(0.013333,0.000000,0.000000);
		rgb(996pt)=(0.010667,0.000000,0.000000);
		rgb(997pt)=(0.008000,0.000000,0.000000);
		rgb(998pt)=(0.005333,0.000000,0.000000);
		rgb(999pt)=(0.002667,0.000000,0.000000);
}}

\newlength{\figureheight}
\newlength{\figurewidth}
\pgfplotsset{
	colormap={parula}{
		rgb(0pt)=(0.2081,0.1663,0.5292);
		rgb(1pt)=(0.208355,0.16778,0.532238);
		rgb(2pt)=(0.208611,0.169261,0.535275);
		rgb(3pt)=(0.208866,0.170741,0.538313);
		rgb(4pt)=(0.209121,0.172222,0.54135);
		rgb(5pt)=(0.209376,0.173702,0.544388);
		rgb(6pt)=(0.209632,0.175183,0.547425);
		rgb(7pt)=(0.209887,0.176663,0.550463);
		rgb(8pt)=(0.210134,0.178144,0.553505);
		rgb(9pt)=(0.210338,0.179624,0.556568);
		rgb(10pt)=(0.210542,0.181105,0.559631);
		rgb(11pt)=(0.210746,0.182585,0.562694);
		rgb(12pt)=(0.210944,0.184066,0.565763);
		rgb(13pt)=(0.211123,0.185546,0.568852);
		rgb(14pt)=(0.211302,0.187027,0.57194);
		rgb(15pt)=(0.21148,0.188507,0.575029);
		rgb(16pt)=(0.211642,0.189996,0.578117);
		rgb(17pt)=(0.21177,0.191502,0.581206);
		rgb(18pt)=(0.211897,0.193008,0.584295);
		rgb(19pt)=(0.212025,0.194514,0.587383);
		rgb(20pt)=(0.212132,0.19602,0.590472);
		rgb(21pt)=(0.212208,0.197526,0.59356);
		rgb(22pt)=(0.212285,0.199032,0.596649);
		rgb(23pt)=(0.212361,0.200538,0.599738);
		rgb(24pt)=(0.212413,0.202044,0.602839);
		rgb(25pt)=(0.212438,0.20355,0.605953);
		rgb(26pt)=(0.212464,0.205056,0.609067);
		rgb(27pt)=(0.212489,0.206562,0.612181);
		rgb(28pt)=(0.212471,0.208083,0.61531);
		rgb(29pt)=(0.21242,0.209614,0.61845);
		rgb(30pt)=(0.212368,0.211146,0.621589);
		rgb(31pt)=(0.212317,0.212677,0.624729);
		rgb(32pt)=(0.212216,0.214209,0.627868);
		rgb(33pt)=(0.212088,0.215741,0.631008);
		rgb(34pt)=(0.211961,0.217272,0.634148);
		rgb(35pt)=(0.211833,0.218804,0.637287);
		rgb(36pt)=(0.211668,0.220354,0.640446);
		rgb(37pt)=(0.211489,0.221911,0.643611);
		rgb(38pt)=(0.21131,0.223468,0.646776);
		rgb(39pt)=(0.211132,0.225025,0.649941);
		rgb(40pt)=(0.210848,0.226603,0.653107);
		rgb(41pt)=(0.210541,0.228186,0.656272);
		rgb(42pt)=(0.210235,0.229768,0.659437);
		rgb(43pt)=(0.209929,0.231351,0.662602);
		rgb(44pt)=(0.209553,0.232934,0.665767);
		rgb(45pt)=(0.20917,0.234516,0.668932);
		rgb(46pt)=(0.208787,0.236099,0.672098);
		rgb(47pt)=(0.208405,0.237681,0.675263);
		rgb(48pt)=(0.20787,0.239289,0.678453);
		rgb(49pt)=(0.207334,0.240897,0.681644);
		rgb(50pt)=(0.206798,0.242505,0.684835);
		rgb(51pt)=(0.206255,0.244114,0.688025);
		rgb(52pt)=(0.205617,0.245722,0.691216);
		rgb(53pt)=(0.204979,0.24733,0.694407);
		rgb(54pt)=(0.204341,0.248938,0.697597);
		rgb(55pt)=(0.203675,0.250554,0.700792);
		rgb(56pt)=(0.202858,0.252213,0.704008);
		rgb(57pt)=(0.202041,0.253872,0.707224);
		rgb(58pt)=(0.201225,0.255531,0.710441);
		rgb(59pt)=(0.200372,0.257184,0.713657);
		rgb(60pt)=(0.199402,0.258818,0.716873);
		rgb(61pt)=(0.198432,0.260452,0.720089);
		rgb(62pt)=(0.197462,0.262085,0.723305);
		rgb(63pt)=(0.196419,0.263735,0.726522);
		rgb(64pt)=(0.195219,0.26542,0.729738);
		rgb(65pt)=(0.19402,0.267105,0.732954);
		rgb(66pt)=(0.19282,0.268789,0.73617);
		rgb(67pt)=(0.191549,0.270474,0.739386);
		rgb(68pt)=(0.19017,0.272159,0.742603);
		rgb(69pt)=(0.188792,0.273843,0.745819);
		rgb(70pt)=(0.187414,0.275528,0.749035);
		rgb(71pt)=(0.1859,0.277237,0.752264);
		rgb(72pt)=(0.184241,0.278973,0.755505);
		rgb(73pt)=(0.182581,0.280709,0.758747);
		rgb(74pt)=(0.180922,0.282444,0.761989);
		rgb(75pt)=(0.179133,0.284209,0.765245);
		rgb(76pt)=(0.177244,0.285996,0.768512);
		rgb(77pt)=(0.175356,0.287783,0.77178);
		rgb(78pt)=(0.173467,0.289569,0.775047);
		rgb(79pt)=(0.171363,0.291406,0.778314);
		rgb(80pt)=(0.169142,0.293269,0.781581);
		rgb(81pt)=(0.166922,0.295132,0.784849);
		rgb(82pt)=(0.164701,0.296996,0.788116);
		rgb(83pt)=(0.162238,0.298934,0.791365);
		rgb(84pt)=(0.159686,0.300899,0.794606);
		rgb(85pt)=(0.157133,0.302865,0.797848);
		rgb(86pt)=(0.15458,0.30483,0.80109);
		rgb(87pt)=(0.151738,0.306858,0.804352);
		rgb(88pt)=(0.148828,0.3089,0.80762);
		rgb(89pt)=(0.145918,0.310942,0.810887);
		rgb(90pt)=(0.143008,0.312984,0.814154);
		rgb(91pt)=(0.139687,0.31514,0.81733);
		rgb(92pt)=(0.136318,0.31731,0.820495);
		rgb(93pt)=(0.132949,0.319479,0.82366);
		rgb(94pt)=(0.129579,0.321649,0.826826);
		rgb(95pt)=(0.125811,0.323918,0.829841);
		rgb(96pt)=(0.122033,0.32619,0.832853);
		rgb(97pt)=(0.118256,0.328462,0.835865);
		rgb(98pt)=(0.114458,0.330737,0.838862);
		rgb(99pt)=(0.110349,0.333059,0.841619);
		rgb(100pt)=(0.106239,0.335382,0.844376);
		rgb(101pt)=(0.102129,0.337705,0.847132);
		rgb(102pt)=(0.0979874,0.340021,0.849835);
		rgb(103pt)=(0.093648,0.342292,0.852209);
		rgb(104pt)=(0.0893087,0.344564,0.854583);
		rgb(105pt)=(0.0849694,0.346836,0.856957);
		rgb(106pt)=(0.08063,0.349091,0.859234);
		rgb(107pt)=(0.0762907,0.351286,0.861174);
		rgb(108pt)=(0.0719514,0.353481,0.863114);
		rgb(109pt)=(0.067612,0.355676,0.865053);
		rgb(110pt)=(0.0633195,0.357817,0.866853);
		rgb(111pt)=(0.0591333,0.359833,0.868333);
		rgb(112pt)=(0.0549471,0.36185,0.869814);
		rgb(113pt)=(0.050761,0.363866,0.871294);
		rgb(114pt)=(0.0466838,0.365823,0.872626);
		rgb(115pt)=(0.0427784,0.367687,0.873724);
		rgb(116pt)=(0.038873,0.36955,0.874821);
		rgb(117pt)=(0.0349676,0.371414,0.875919);
		rgb(118pt)=(0.0315066,0.373217,0.876872);
		rgb(119pt)=(0.0285456,0.374953,0.877664);
		rgb(120pt)=(0.0255847,0.376688,0.878455);
		rgb(121pt)=(0.0226237,0.378424,0.879246);
		rgb(122pt)=(0.0202132,0.380061,0.879868);
		rgb(123pt)=(0.0182477,0.381618,0.880353);
		rgb(124pt)=(0.0162823,0.383175,0.880838);
		rgb(125pt)=(0.0143168,0.384732,0.881323);
		rgb(126pt)=(0.0127892,0.386241,0.881695);
		rgb(127pt)=(0.0115129,0.387721,0.882001);
		rgb(128pt)=(0.0102366,0.389202,0.882307);
		rgb(129pt)=(0.00896036,0.390682,0.882614);
		rgb(130pt)=(0.00812372,0.392089,0.88281);
		rgb(131pt)=(0.00746006,0.393468,0.882963);
		rgb(132pt)=(0.0067964,0.394846,0.883116);
		rgb(133pt)=(0.00613273,0.396224,0.883269);
		rgb(134pt)=(0.00581622,0.397562,0.88332);
		rgb(135pt)=(0.00558649,0.398889,0.883346);
		rgb(136pt)=(0.00535676,0.400217,0.883371);
		rgb(137pt)=(0.00512703,0.401544,0.883397);
		rgb(138pt)=(0.00516757,0.402804,0.883332);
		rgb(139pt)=(0.00524414,0.404054,0.883256);
		rgb(140pt)=(0.00532072,0.405305,0.883179);
		rgb(141pt)=(0.0053973,0.406556,0.883103);
		rgb(142pt)=(0.00572012,0.407757,0.882952);
		rgb(143pt)=(0.00605195,0.408957,0.882799);
		rgb(144pt)=(0.00638378,0.410157,0.882646);
		rgb(145pt)=(0.00672643,0.411355,0.882489);
		rgb(146pt)=(0.00728799,0.412529,0.882259);
		rgb(147pt)=(0.00784955,0.413704,0.88203);
		rgb(148pt)=(0.00841111,0.414878,0.8818);
		rgb(149pt)=(0.00898919,0.416045,0.881564);
		rgb(150pt)=(0.00967838,0.417168,0.881283);
		rgb(151pt)=(0.0103676,0.418292,0.881002);
		rgb(152pt)=(0.0110568,0.419415,0.880721);
		rgb(153pt)=(0.011773,0.420532,0.880435);
		rgb(154pt)=(0.0125898,0.42163,0.880129);
		rgb(155pt)=(0.0134066,0.422728,0.879823);
		rgb(156pt)=(0.0142234,0.423825,0.879516);
		rgb(157pt)=(0.0150703,0.424915,0.879195);
		rgb(158pt)=(0.0159892,0.425987,0.878838);
		rgb(159pt)=(0.0169081,0.427059,0.87848);
		rgb(160pt)=(0.017827,0.428132,0.878123);
		rgb(161pt)=(0.0187748,0.429194,0.877746);
		rgb(162pt)=(0.0197703,0.430241,0.877338);
		rgb(163pt)=(0.0207658,0.431287,0.876929);
		rgb(164pt)=(0.0217613,0.432334,0.876521);
		rgb(165pt)=(0.0227802,0.43338,0.876113);
		rgb(166pt)=(0.0238267,0.434427,0.875704);
		rgb(167pt)=(0.0248733,0.435473,0.875296);
		rgb(168pt)=(0.0259198,0.43652,0.874887);
		rgb(169pt)=(0.0269802,0.437553,0.874451);
		rgb(170pt)=(0.0280523,0.438574,0.873992);
		rgb(171pt)=(0.0291243,0.439595,0.873532);
		rgb(172pt)=(0.0301964,0.440616,0.873073);
		rgb(173pt)=(0.0312844,0.441621,0.872614);
		rgb(174pt)=(0.032382,0.442616,0.872154);
		rgb(175pt)=(0.0334796,0.443612,0.871695);
		rgb(176pt)=(0.0345772,0.444607,0.871235);
		rgb(177pt)=(0.0357108,0.445603,0.870758);
		rgb(178pt)=(0.0368595,0.446598,0.870273);
		rgb(179pt)=(0.0380081,0.447594,0.869788);
		rgb(180pt)=(0.0391568,0.448589,0.869303);
		rgb(181pt)=(0.0402652,0.449565,0.868798);
		rgb(182pt)=(0.0413628,0.450535,0.868287);
		rgb(183pt)=(0.0424604,0.451505,0.867777);
		rgb(184pt)=(0.043558,0.452474,0.867266);
		rgb(185pt)=(0.0445889,0.453444,0.866756);
		rgb(186pt)=(0.0456099,0.454414,0.866245);
		rgb(187pt)=(0.0466309,0.455384,0.865735);
		rgb(188pt)=(0.047652,0.456354,0.865224);
		rgb(189pt)=(0.0486,0.457324,0.864714);
		rgb(190pt)=(0.0495444,0.458294,0.864203);
		rgb(191pt)=(0.0504889,0.459264,0.863692);
		rgb(192pt)=(0.0514315,0.460234,0.863181);
		rgb(193pt)=(0.0523249,0.461204,0.862645);
		rgb(194pt)=(0.0532183,0.462174,0.862109);
		rgb(195pt)=(0.0541117,0.463144,0.861573);
		rgb(196pt)=(0.0549991,0.464111,0.861034);
		rgb(197pt)=(0.0558414,0.465056,0.860472);
		rgb(198pt)=(0.0566838,0.466,0.859911);
		rgb(199pt)=(0.0575261,0.466944,0.859349);
		rgb(200pt)=(0.0583532,0.467889,0.858793);
		rgb(201pt)=(0.0591189,0.468833,0.858257);
		rgb(202pt)=(0.0598847,0.469778,0.857721);
		rgb(203pt)=(0.0606505,0.470722,0.857185);
		rgb(204pt)=(0.0614018,0.471667,0.856641);
		rgb(205pt)=(0.0621165,0.472611,0.85608);
		rgb(206pt)=(0.0628312,0.473556,0.855518);
		rgb(207pt)=(0.0635459,0.4745,0.854957);
		rgb(208pt)=(0.064242,0.475444,0.854405);
		rgb(209pt)=(0.0649057,0.476389,0.853868);
		rgb(210pt)=(0.0655694,0.477333,0.853332);
		rgb(211pt)=(0.066233,0.478278,0.852796);
		rgb(212pt)=(0.0668625,0.479222,0.852249);
		rgb(213pt)=(0.0674495,0.480167,0.851687);
		rgb(214pt)=(0.0680366,0.481111,0.851126);
		rgb(215pt)=(0.0686237,0.482056,0.850564);
		rgb(216pt)=(0.0691838,0.483,0.850003);
		rgb(217pt)=(0.0697198,0.483944,0.849441);
		rgb(218pt)=(0.0702559,0.484889,0.84888);
		rgb(219pt)=(0.0707919,0.485833,0.848318);
		rgb(220pt)=(0.0712967,0.486778,0.847772);
		rgb(221pt)=(0.0717817,0.487722,0.847236);
		rgb(222pt)=(0.0722667,0.488667,0.8467);
		rgb(223pt)=(0.0727517,0.489611,0.846164);
		rgb(224pt)=(0.0732012,0.490573,0.845628);
		rgb(225pt)=(0.0736351,0.491543,0.845092);
		rgb(226pt)=(0.0740691,0.492513,0.844556);
		rgb(227pt)=(0.074503,0.493483,0.84402);
		rgb(228pt)=(0.0748973,0.494433,0.843484);
		rgb(229pt)=(0.0752802,0.495378,0.842948);
		rgb(230pt)=(0.0756631,0.496322,0.842412);
		rgb(231pt)=(0.0760459,0.497267,0.841876);
		rgb(232pt)=(0.0763631,0.498233,0.841362);
		rgb(233pt)=(0.0766694,0.499203,0.840851);
		rgb(234pt)=(0.0769757,0.500173,0.840341);
		rgb(235pt)=(0.077282,0.501143,0.83983);
		rgb(236pt)=(0.0775162,0.502137,0.83932);
		rgb(237pt)=(0.0777459,0.503132,0.838809);
		rgb(238pt)=(0.0779757,0.504128,0.838298);
		rgb(239pt)=(0.0782042,0.505123,0.837789);
		rgb(240pt)=(0.0783829,0.506093,0.837304);
		rgb(241pt)=(0.0785616,0.507063,0.836819);
		rgb(242pt)=(0.0787402,0.508033,0.836334);
		rgb(243pt)=(0.0789135,0.509008,0.835851);
		rgb(244pt)=(0.0790411,0.510029,0.835392);
		rgb(245pt)=(0.0791688,0.51105,0.834932);
		rgb(246pt)=(0.0792964,0.512071,0.834473);
		rgb(247pt)=(0.0794048,0.513092,0.834018);
		rgb(248pt)=(0.0794303,0.514113,0.833584);
		rgb(249pt)=(0.0794559,0.515134,0.83315);
		rgb(250pt)=(0.0794814,0.516155,0.832717);
		rgb(251pt)=(0.0794862,0.517183,0.832289);
		rgb(252pt)=(0.0794351,0.51823,0.831881);
		rgb(253pt)=(0.0793841,0.519276,0.831473);
		rgb(254pt)=(0.079333,0.520323,0.831064);
		rgb(255pt)=(0.079255,0.521369,0.830665);
		rgb(256pt)=(0.0791273,0.522416,0.830282);
		rgb(257pt)=(0.0789997,0.523462,0.829899);
		rgb(258pt)=(0.0788721,0.524509,0.829516);
		rgb(259pt)=(0.0786889,0.525589,0.829156);
		rgb(260pt)=(0.0784336,0.526712,0.828824);
		rgb(261pt)=(0.0781784,0.527835,0.828492);
		rgb(262pt)=(0.0779231,0.528958,0.82816);
		rgb(263pt)=(0.077615,0.530081,0.827868);
		rgb(264pt)=(0.0772577,0.531205,0.827613);
		rgb(265pt)=(0.0769003,0.532328,0.827357);
		rgb(266pt)=(0.0765429,0.533451,0.827102);
		rgb(267pt)=(0.0761243,0.534589,0.826862);
		rgb(268pt)=(0.0756649,0.535738,0.826632);
		rgb(269pt)=(0.0752054,0.536886,0.826403);
		rgb(270pt)=(0.0747459,0.538035,0.826173);
		rgb(271pt)=(0.0742168,0.539219,0.825961);
		rgb(272pt)=(0.0736553,0.540418,0.825756);
		rgb(273pt)=(0.0730937,0.541618,0.825552);
		rgb(274pt)=(0.0725321,0.542818,0.825348);
		rgb(275pt)=(0.0718925,0.544037,0.825183);
		rgb(276pt)=(0.0712288,0.545262,0.82503);
		rgb(277pt)=(0.0705652,0.546487,0.824877);
		rgb(278pt)=(0.0699015,0.547713,0.824723);
		rgb(279pt)=(0.0691514,0.548938,0.824614);
		rgb(280pt)=(0.0683856,0.550163,0.824511);
		rgb(281pt)=(0.0676198,0.551388,0.824409);
		rgb(282pt)=(0.0668541,0.552614,0.824307);
		rgb(283pt)=(0.0660408,0.553886,0.824205);
		rgb(284pt)=(0.065224,0.555162,0.824103);
		rgb(285pt)=(0.0644072,0.556439,0.824001);
		rgb(286pt)=(0.0635892,0.557715,0.823899);
		rgb(287pt)=(0.0626703,0.558991,0.823848);
		rgb(288pt)=(0.0617514,0.560268,0.823797);
		rgb(289pt)=(0.0608324,0.561544,0.823746);
		rgb(290pt)=(0.0599087,0.56282,0.823693);
		rgb(291pt)=(0.0589387,0.564096,0.823616);
		rgb(292pt)=(0.0579688,0.565373,0.82354);
		rgb(293pt)=(0.0569988,0.566649,0.823463);
		rgb(294pt)=(0.0560243,0.567925,0.823386);
		rgb(295pt)=(0.0550288,0.569202,0.82331);
		rgb(296pt)=(0.0540333,0.570478,0.823233);
		rgb(297pt)=(0.0530378,0.571754,0.823157);
		rgb(298pt)=(0.0520423,0.57303,0.82308);
		rgb(299pt)=(0.0510468,0.574307,0.823004);
		rgb(300pt)=(0.0500514,0.575583,0.822927);
		rgb(301pt)=(0.0490559,0.576859,0.82285);
		rgb(302pt)=(0.0480604,0.578127,0.822756);
		rgb(303pt)=(0.0470649,0.579377,0.822629);
		rgb(304pt)=(0.0460694,0.580628,0.822501);
		rgb(305pt)=(0.0450739,0.581879,0.822374);
		rgb(306pt)=(0.0441,0.583119,0.822235);
		rgb(307pt)=(0.0431556,0.584344,0.822082);
		rgb(308pt)=(0.0422111,0.585569,0.821929);
		rgb(309pt)=(0.0412667,0.586795,0.821776);
		rgb(310pt)=(0.0403351,0.58802,0.821597);
		rgb(311pt)=(0.0394162,0.589245,0.821392);
		rgb(312pt)=(0.0384973,0.59047,0.821188);
		rgb(313pt)=(0.0375784,0.591695,0.820984);
		rgb(314pt)=(0.0367495,0.592891,0.820735);
		rgb(315pt)=(0.0359838,0.594065,0.820454);
		rgb(316pt)=(0.035218,0.595239,0.820173);
		rgb(317pt)=(0.0344523,0.596413,0.819892);
		rgb(318pt)=(0.0337721,0.597553,0.819595);
		rgb(319pt)=(0.0331339,0.598676,0.819288);
		rgb(320pt)=(0.0324958,0.599799,0.818982);
		rgb(321pt)=(0.0318577,0.600923,0.818676);
		rgb(322pt)=(0.0312964,0.602026,0.818312);
		rgb(323pt)=(0.0307604,0.603124,0.817929);
		rgb(324pt)=(0.0302243,0.604222,0.817546);
		rgb(325pt)=(0.0296883,0.605319,0.817163);
		rgb(326pt)=(0.0292375,0.606395,0.816738);
		rgb(327pt)=(0.0288036,0.607468,0.816304);
		rgb(328pt)=(0.0283697,0.60854,0.81587);
		rgb(329pt)=(0.0279357,0.609612,0.815436);
		rgb(330pt)=(0.0275721,0.610637,0.814955);
		rgb(331pt)=(0.0272147,0.611658,0.81447);
		rgb(332pt)=(0.0268574,0.612679,0.813985);
		rgb(333pt)=(0.0265,0.6137,0.8135);
		rgb(334pt)=(0.0262447,0.614695,0.812964);
		rgb(335pt)=(0.0259895,0.615691,0.812428);
		rgb(336pt)=(0.0257342,0.616686,0.811892);
		rgb(337pt)=(0.0254853,0.61768,0.811352);
		rgb(338pt)=(0.0253066,0.61865,0.810765);
		rgb(339pt)=(0.0251279,0.61962,0.810177);
		rgb(340pt)=(0.0249492,0.62059,0.80959);
		rgb(341pt)=(0.024779,0.621551,0.808995);
		rgb(342pt)=(0.0246514,0.62247,0.808357);
		rgb(343pt)=(0.0245237,0.623389,0.807719);
		rgb(344pt)=(0.0243961,0.624308,0.80708);
		rgb(345pt)=(0.0242748,0.625221,0.80643);
		rgb(346pt)=(0.0241727,0.626114,0.805741);
		rgb(347pt)=(0.0240706,0.627008,0.805051);
		rgb(348pt)=(0.0239685,0.627901,0.804362);
		rgb(349pt)=(0.0238832,0.628786,0.803656);
		rgb(350pt)=(0.0238321,0.629654,0.802916);
		rgb(351pt)=(0.0237811,0.630522,0.802176);
		rgb(352pt)=(0.02373,0.631389,0.801435);
		rgb(353pt)=(0.023679,0.632247,0.800685);
		rgb(354pt)=(0.0236279,0.633089,0.799919);
		rgb(355pt)=(0.0235769,0.633932,0.799153);
		rgb(356pt)=(0.0235258,0.634774,0.798387);
		rgb(357pt)=(0.0234748,0.635604,0.797596);
		rgb(358pt)=(0.0234237,0.63642,0.79678);
		rgb(359pt)=(0.0233727,0.637237,0.795963);
		rgb(360pt)=(0.0233216,0.638054,0.795146);
		rgb(361pt)=(0.0232706,0.638856,0.794329);
		rgb(362pt)=(0.0232195,0.639647,0.793512);
		rgb(363pt)=(0.0231685,0.640439,0.792695);
		rgb(364pt)=(0.0231174,0.64123,0.791879);
		rgb(365pt)=(0.0230832,0.642005,0.791011);
		rgb(366pt)=(0.0230577,0.64277,0.790118);
		rgb(367pt)=(0.0230321,0.643536,0.789225);
		rgb(368pt)=(0.0230066,0.644302,0.788331);
		rgb(369pt)=(0.0229811,0.645049,0.787438);
		rgb(370pt)=(0.0229556,0.645789,0.786544);
		rgb(371pt)=(0.02293,0.646529,0.785651);
		rgb(372pt)=(0.0229045,0.647269,0.784758);
		rgb(373pt)=(0.022858,0.64801,0.783843);
		rgb(374pt)=(0.0228069,0.64875,0.782924);
		rgb(375pt)=(0.0227559,0.64949,0.782005);
		rgb(376pt)=(0.0227048,0.65023,0.781086);
		rgb(377pt)=(0.0227,0.650947,0.780144);
		rgb(378pt)=(0.0227,0.651662,0.7792);
		rgb(379pt)=(0.0227,0.652377,0.778256);
		rgb(380pt)=(0.0227,0.653092,0.777311);
		rgb(381pt)=(0.0228261,0.653781,0.776341);
		rgb(382pt)=(0.0229538,0.65447,0.775371);
		rgb(383pt)=(0.0230814,0.655159,0.774402);
		rgb(384pt)=(0.0232108,0.655849,0.77343);
		rgb(385pt)=(0.023364,0.656538,0.772434);
		rgb(386pt)=(0.0235171,0.657227,0.771439);
		rgb(387pt)=(0.0236703,0.657916,0.770443);
		rgb(388pt)=(0.0238312,0.658602,0.769444);
		rgb(389pt)=(0.0240354,0.659265,0.768423);
		rgb(390pt)=(0.0242396,0.659929,0.767402);
		rgb(391pt)=(0.0244438,0.660592,0.766381);
		rgb(392pt)=(0.0247021,0.661256,0.765354);
		rgb(393pt)=(0.025136,0.66192,0.764307);
		rgb(394pt)=(0.02557,0.662583,0.763261);
		rgb(395pt)=(0.0260039,0.663247,0.762214);
		rgb(396pt)=(0.0264541,0.663911,0.761168);
		rgb(397pt)=(0.026939,0.664574,0.760121);
		rgb(398pt)=(0.027424,0.665238,0.759074);
		rgb(399pt)=(0.027909,0.665902,0.758028);
		rgb(400pt)=(0.028445,0.666555,0.756971);
		rgb(401pt)=(0.0290577,0.667193,0.755899);
		rgb(402pt)=(0.0296703,0.667832,0.754827);
		rgb(403pt)=(0.0302829,0.66847,0.753755);
		rgb(404pt)=(0.030994,0.669095,0.752683);
		rgb(405pt)=(0.0318108,0.669708,0.751611);
		rgb(406pt)=(0.0326276,0.670321,0.750539);
		rgb(407pt)=(0.0334444,0.670933,0.749467);
		rgb(408pt)=(0.0343045,0.67156,0.748366);
		rgb(409pt)=(0.0351979,0.672198,0.747243);
		rgb(410pt)=(0.0360913,0.672837,0.74612);
		rgb(411pt)=(0.0369847,0.673475,0.744996);
		rgb(412pt)=(0.0380432,0.674096,0.743873);
		rgb(413pt)=(0.0391919,0.674709,0.74275);
		rgb(414pt)=(0.0403405,0.675322,0.741627);
		rgb(415pt)=(0.0414892,0.675934,0.740504);
		rgb(416pt)=(0.0427123,0.676528,0.739381);
		rgb(417pt)=(0.0439631,0.677115,0.738258);
		rgb(418pt)=(0.0452138,0.677702,0.737135);
		rgb(419pt)=(0.0464646,0.678289,0.736011);
		rgb(420pt)=(0.0477153,0.678897,0.734868);
		rgb(421pt)=(0.0489661,0.67951,0.733719);
		rgb(422pt)=(0.0502168,0.680123,0.73257);
		rgb(423pt)=(0.0514676,0.680735,0.731422);
		rgb(424pt)=(0.0529237,0.681325,0.73025);
		rgb(425pt)=(0.0544042,0.681912,0.729076);
		rgb(426pt)=(0.0558847,0.682499,0.727902);
		rgb(427pt)=(0.0573652,0.683086,0.726728);
		rgb(428pt)=(0.0587709,0.683673,0.725553);
		rgb(429pt)=(0.0601748,0.68426,0.724379);
		rgb(430pt)=(0.0615787,0.684847,0.723205);
		rgb(431pt)=(0.0629946,0.685435,0.722028);
		rgb(432pt)=(0.0646027,0.686022,0.720803);
		rgb(433pt)=(0.0662108,0.686609,0.719577);
		rgb(434pt)=(0.0678189,0.687196,0.718352);
		rgb(435pt)=(0.069427,0.687779,0.717131);
		rgb(436pt)=(0.0710351,0.688341,0.715931);
		rgb(437pt)=(0.0726432,0.688902,0.714731);
		rgb(438pt)=(0.0742514,0.689464,0.713532);
		rgb(439pt)=(0.0758709,0.690026,0.712326);
		rgb(440pt)=(0.07753,0.690587,0.711101);
		rgb(441pt)=(0.0791892,0.691149,0.709876);
		rgb(442pt)=(0.0808483,0.69171,0.70865);
		rgb(443pt)=(0.0825387,0.692272,0.707417);
		rgb(444pt)=(0.0843,0.692833,0.706167);
		rgb(445pt)=(0.0860613,0.693395,0.704916);
		rgb(446pt)=(0.0878225,0.693956,0.703665);
		rgb(447pt)=(0.089564,0.694518,0.702405);
		rgb(448pt)=(0.0912742,0.69508,0.701128);
		rgb(449pt)=(0.0929844,0.695641,0.699852);
		rgb(450pt)=(0.0946946,0.696203,0.698576);
		rgb(451pt)=(0.0965009,0.696752,0.697299);
		rgb(452pt)=(0.0984153,0.697288,0.696023);
		rgb(453pt)=(0.10033,0.697824,0.694747);
		rgb(454pt)=(0.102244,0.69836,0.693471);
		rgb(455pt)=(0.10413,0.698896,0.69218);
		rgb(456pt)=(0.105994,0.699432,0.690878);
		rgb(457pt)=(0.107857,0.699968,0.689577);
		rgb(458pt)=(0.10972,0.700505,0.688275);
		rgb(459pt)=(0.111632,0.701041,0.686973);
		rgb(460pt)=(0.113572,0.701577,0.685671);
		rgb(461pt)=(0.115512,0.702113,0.684369);
		rgb(462pt)=(0.117452,0.702649,0.683068);
		rgb(463pt)=(0.119429,0.703185,0.681747);
		rgb(464pt)=(0.12142,0.703721,0.68042);
		rgb(465pt)=(0.123411,0.704257,0.679093);
		rgb(466pt)=(0.125402,0.704793,0.677765);
		rgb(467pt)=(0.127372,0.705308,0.676438);
		rgb(468pt)=(0.129338,0.705819,0.675111);
		rgb(469pt)=(0.131303,0.706329,0.673783);
		rgb(470pt)=(0.133269,0.70684,0.672456);
		rgb(471pt)=(0.135369,0.70735,0.671084);
		rgb(472pt)=(0.137488,0.707861,0.669705);
		rgb(473pt)=(0.139607,0.708371,0.668327);
		rgb(474pt)=(0.141725,0.708882,0.666949);
		rgb(475pt)=(0.143795,0.709392,0.665595);
		rgb(476pt)=(0.145862,0.709903,0.664242);
		rgb(477pt)=(0.14793,0.710414,0.662889);
		rgb(478pt)=(0.150003,0.710924,0.661534);
		rgb(479pt)=(0.152198,0.711435,0.66013);
		rgb(480pt)=(0.154394,0.711945,0.658726);
		rgb(481pt)=(0.156589,0.712456,0.657322);
		rgb(482pt)=(0.158784,0.712963,0.655922);
		rgb(483pt)=(0.160979,0.713448,0.654543);
		rgb(484pt)=(0.163174,0.713933,0.653165);
		rgb(485pt)=(0.16537,0.714418,0.651786);
		rgb(486pt)=(0.16757,0.714908,0.650397);
		rgb(487pt)=(0.169791,0.715419,0.648968);
		rgb(488pt)=(0.172012,0.715929,0.647538);
		rgb(489pt)=(0.174232,0.71644,0.646109);
		rgb(490pt)=(0.176483,0.716935,0.64468);
		rgb(491pt)=(0.178806,0.717395,0.64325);
		rgb(492pt)=(0.181129,0.717854,0.641821);
		rgb(493pt)=(0.183452,0.718314,0.640391);
		rgb(494pt)=(0.185755,0.718783,0.638952);
		rgb(495pt)=(0.188027,0.719268,0.637497);
		rgb(496pt)=(0.190299,0.719753,0.636042);
		rgb(497pt)=(0.192571,0.720238,0.634587);
		rgb(498pt)=(0.194913,0.720711,0.633132);
		rgb(499pt)=(0.197338,0.72117,0.631677);
		rgb(500pt)=(0.199762,0.72163,0.630223);
		rgb(501pt)=(0.202187,0.722089,0.628768);
		rgb(502pt)=(0.204612,0.722549,0.627299);
		rgb(503pt)=(0.207037,0.723008,0.625818);
		rgb(504pt)=(0.209462,0.723468,0.624338);
		rgb(505pt)=(0.211887,0.723927,0.622857);
		rgb(506pt)=(0.214328,0.724386,0.621377);
		rgb(507pt)=(0.216778,0.724846,0.619896);
		rgb(508pt)=(0.219229,0.725305,0.618416);
		rgb(509pt)=(0.221679,0.725765,0.616935);
		rgb(510pt)=(0.224202,0.726188,0.615455);
		rgb(511pt)=(0.226754,0.726597,0.613974);
		rgb(512pt)=(0.229307,0.727005,0.612494);
		rgb(513pt)=(0.231859,0.727414,0.611014);
		rgb(514pt)=(0.234392,0.727842,0.609513);
		rgb(515pt)=(0.236919,0.728276,0.608007);
		rgb(516pt)=(0.239446,0.72871,0.606501);
		rgb(517pt)=(0.241973,0.729144,0.604995);
		rgb(518pt)=(0.244611,0.729556,0.603467);
		rgb(519pt)=(0.247266,0.729964,0.601935);
		rgb(520pt)=(0.24992,0.730372,0.600404);
		rgb(521pt)=(0.252575,0.730781,0.598872);
		rgb(522pt)=(0.25523,0.731189,0.597365);
		rgb(523pt)=(0.257884,0.731598,0.595859);
		rgb(524pt)=(0.260539,0.732006,0.594353);
		rgb(525pt)=(0.263194,0.732414,0.592846);
		rgb(526pt)=(0.265848,0.732796,0.591314);
		rgb(527pt)=(0.268503,0.733179,0.589783);
		rgb(528pt)=(0.271158,0.733562,0.588251);
		rgb(529pt)=(0.27383,0.733945,0.58672);
		rgb(530pt)=(0.276638,0.734328,0.585188);
		rgb(531pt)=(0.279446,0.734711,0.583657);
		rgb(532pt)=(0.282254,0.735094,0.582125);
		rgb(533pt)=(0.285051,0.735471,0.580594);
		rgb(534pt)=(0.287808,0.735829,0.579062);
		rgb(535pt)=(0.290565,0.736186,0.577531);
		rgb(536pt)=(0.293322,0.736544,0.575999);
		rgb(537pt)=(0.2961,0.736894,0.574468);
		rgb(538pt)=(0.298933,0.737226,0.572936);
		rgb(539pt)=(0.301767,0.737557,0.571405);
		rgb(540pt)=(0.3046,0.737889,0.569873);
		rgb(541pt)=(0.307452,0.738221,0.568351);
		rgb(542pt)=(0.310336,0.738553,0.566845);
		rgb(543pt)=(0.313221,0.738885,0.565339);
		rgb(544pt)=(0.316105,0.739217,0.563833);
		rgb(545pt)=(0.318978,0.739537,0.562315);
		rgb(546pt)=(0.321837,0.739843,0.560784);
		rgb(547pt)=(0.324696,0.74015,0.559252);
		rgb(548pt)=(0.327555,0.740456,0.557721);
		rgb(549pt)=(0.330468,0.740749,0.556216);
		rgb(550pt)=(0.333429,0.741029,0.554736);
		rgb(551pt)=(0.336389,0.74131,0.553255);
		rgb(552pt)=(0.33935,0.741591,0.551775);
		rgb(553pt)=(0.342296,0.741872,0.550279);
		rgb(554pt)=(0.345231,0.742153,0.548773);
		rgb(555pt)=(0.348167,0.742433,0.547267);
		rgb(556pt)=(0.351102,0.742714,0.545761);
		rgb(557pt)=(0.354038,0.742977,0.54429);
		rgb(558pt)=(0.356973,0.743232,0.542835);
		rgb(559pt)=(0.359908,0.743488,0.54138);
		rgb(560pt)=(0.362844,0.743743,0.539925);
		rgb(561pt)=(0.365839,0.743959,0.53847);
		rgb(562pt)=(0.368851,0.744163,0.537015);
		rgb(563pt)=(0.371863,0.744367,0.53556);
		rgb(564pt)=(0.374875,0.744571,0.534105);
		rgb(565pt)=(0.377843,0.744775,0.532672);
		rgb(566pt)=(0.380804,0.74498,0.531243);
		rgb(567pt)=(0.383765,0.745184,0.529814);
		rgb(568pt)=(0.386726,0.745388,0.528384);
		rgb(569pt)=(0.389711,0.745568,0.527003);
		rgb(570pt)=(0.392697,0.745747,0.525624);
		rgb(571pt)=(0.395684,0.745926,0.524246);
		rgb(572pt)=(0.39867,0.746104,0.522868);
		rgb(573pt)=(0.401657,0.746257,0.521489);
		rgb(574pt)=(0.404643,0.74641,0.520111);
		rgb(575pt)=(0.40763,0.746563,0.518732);
		rgb(576pt)=(0.410611,0.746716,0.517359);
		rgb(577pt)=(0.413546,0.746869,0.516032);
		rgb(578pt)=(0.416482,0.747023,0.514705);
		rgb(579pt)=(0.419417,0.747176,0.513377);
		rgb(580pt)=(0.422357,0.747319,0.512055);
		rgb(581pt)=(0.425318,0.747421,0.510753);
		rgb(582pt)=(0.428279,0.747523,0.509451);
		rgb(583pt)=(0.43124,0.747626,0.50815);
		rgb(584pt)=(0.43418,0.747735,0.506848);
		rgb(585pt)=(0.437065,0.747862,0.505546);
		rgb(586pt)=(0.439949,0.74799,0.504244);
		rgb(587pt)=(0.442834,0.748117,0.502942);
		rgb(588pt)=(0.445727,0.748227,0.501659);
		rgb(589pt)=(0.448637,0.748304,0.500408);
		rgb(590pt)=(0.451547,0.74838,0.499157);
		rgb(591pt)=(0.454457,0.748457,0.497906);
		rgb(592pt)=(0.457333,0.748522,0.496667);
		rgb(593pt)=(0.460167,0.748573,0.495441);
		rgb(594pt)=(0.463,0.748624,0.494216);
		rgb(595pt)=(0.465833,0.748675,0.492991);
		rgb(596pt)=(0.468667,0.748726,0.491779);
		rgb(597pt)=(0.4715,0.748777,0.490579);
		rgb(598pt)=(0.474333,0.748829,0.48938);
		rgb(599pt)=(0.477167,0.74888,0.48818);
		rgb(600pt)=(0.479969,0.748931,0.486995);
		rgb(601pt)=(0.482752,0.748982,0.485821);
		rgb(602pt)=(0.485534,0.749033,0.484647);
		rgb(603pt)=(0.488316,0.749084,0.483473);
		rgb(604pt)=(0.491081,0.7491,0.482316);
		rgb(605pt)=(0.493838,0.7491,0.481168);
		rgb(606pt)=(0.496595,0.7491,0.480019);
		rgb(607pt)=(0.499351,0.7491,0.47887);
		rgb(608pt)=(0.502069,0.74912,0.477722);
		rgb(609pt)=(0.504775,0.749145,0.476573);
		rgb(610pt)=(0.50748,0.749171,0.475424);
		rgb(611pt)=(0.510186,0.749196,0.474276);
		rgb(612pt)=(0.512892,0.7492,0.47317);
		rgb(613pt)=(0.515598,0.7492,0.472073);
		rgb(614pt)=(0.518303,0.7492,0.470975);
		rgb(615pt)=(0.521009,0.7492,0.469877);
		rgb(616pt)=(0.523644,0.749176,0.46878);
		rgb(617pt)=(0.526273,0.749151,0.467682);
		rgb(618pt)=(0.528902,0.749125,0.466585);
		rgb(619pt)=(0.531531,0.7491,0.465487);
		rgb(620pt)=(0.53416,0.749074,0.464415);
		rgb(621pt)=(0.536789,0.749049,0.463343);
		rgb(622pt)=(0.539418,0.749023,0.462271);
		rgb(623pt)=(0.542043,0.748998,0.461199);
		rgb(624pt)=(0.544621,0.748972,0.460127);
		rgb(625pt)=(0.547199,0.748947,0.459055);
		rgb(626pt)=(0.549777,0.748921,0.457983);
		rgb(627pt)=(0.55235,0.748891,0.45692);
		rgb(628pt)=(0.554903,0.74884,0.455899);
		rgb(629pt)=(0.557456,0.748789,0.454878);
		rgb(630pt)=(0.560008,0.748738,0.453857);
		rgb(631pt)=(0.562554,0.748687,0.452829);
		rgb(632pt)=(0.565081,0.748636,0.451783);
		rgb(633pt)=(0.567608,0.748585,0.450736);
		rgb(634pt)=(0.570135,0.748534,0.449689);
		rgb(635pt)=(0.572653,0.748474,0.44866);
		rgb(636pt)=(0.575155,0.748397,0.447665);
		rgb(637pt)=(0.577656,0.748321,0.446669);
		rgb(638pt)=(0.580158,0.748244,0.445674);
		rgb(639pt)=(0.582649,0.748168,0.444678);
		rgb(640pt)=(0.585125,0.748091,0.443683);
		rgb(641pt)=(0.587601,0.748014,0.442687);
		rgb(642pt)=(0.590077,0.747938,0.441692);
		rgb(643pt)=(0.59254,0.747861,0.440709);
		rgb(644pt)=(0.59499,0.747785,0.439739);
		rgb(645pt)=(0.597441,0.747708,0.438769);
		rgb(646pt)=(0.599891,0.747632,0.437799);
		rgb(647pt)=(0.602311,0.747555,0.436814);
		rgb(648pt)=(0.604711,0.747478,0.435819);
		rgb(649pt)=(0.60711,0.747402,0.434823);
		rgb(650pt)=(0.60951,0.747325,0.433828);
		rgb(651pt)=(0.611909,0.747232,0.432867);
		rgb(652pt)=(0.614308,0.747129,0.431922);
		rgb(653pt)=(0.616708,0.747027,0.430978);
		rgb(654pt)=(0.619107,0.746925,0.430033);
		rgb(655pt)=(0.621487,0.746823,0.429089);
		rgb(656pt)=(0.623861,0.746721,0.428144);
		rgb(657pt)=(0.626235,0.746619,0.4272);
		rgb(658pt)=(0.628609,0.746517,0.426256);
		rgb(659pt)=(0.630962,0.746393,0.425311);
		rgb(660pt)=(0.63331,0.746266,0.424367);
		rgb(661pt)=(0.635658,0.746138,0.423422);
		rgb(662pt)=(0.638007,0.746011,0.422478);
		rgb(663pt)=(0.640332,0.745906,0.421557);
		rgb(664pt)=(0.642654,0.745804,0.420638);
		rgb(665pt)=(0.644977,0.745702,0.419719);
		rgb(666pt)=(0.6473,0.7456,0.4188);
		rgb(667pt)=(0.649623,0.745472,0.417881);
		rgb(668pt)=(0.651946,0.745345,0.416962);
		rgb(669pt)=(0.654268,0.745217,0.416043);
		rgb(670pt)=(0.656587,0.745089,0.415124);
		rgb(671pt)=(0.658859,0.744962,0.414205);
		rgb(672pt)=(0.661131,0.744834,0.413286);
		rgb(673pt)=(0.663402,0.744707,0.412368);
		rgb(674pt)=(0.665674,0.744579,0.411453);
		rgb(675pt)=(0.667946,0.744451,0.410559);
		rgb(676pt)=(0.670218,0.744324,0.409666);
		rgb(677pt)=(0.672489,0.744196,0.408773);
		rgb(678pt)=(0.674755,0.744062,0.407879);
		rgb(679pt)=(0.677001,0.743909,0.406986);
		rgb(680pt)=(0.679247,0.743756,0.406092);
		rgb(681pt)=(0.681494,0.743603,0.405199);
		rgb(682pt)=(0.68374,0.743458,0.404306);
		rgb(683pt)=(0.685986,0.74333,0.403412);
		rgb(684pt)=(0.688232,0.743203,0.402519);
		rgb(685pt)=(0.690479,0.743075,0.401626);
		rgb(686pt)=(0.692704,0.742937,0.400732);
		rgb(687pt)=(0.694899,0.742784,0.399839);
		rgb(688pt)=(0.697094,0.742631,0.398945);
		rgb(689pt)=(0.699289,0.742477,0.398052);
		rgb(690pt)=(0.701497,0.742324,0.397171);
		rgb(691pt)=(0.703718,0.742171,0.396303);
		rgb(692pt)=(0.705939,0.742018,0.395435);
		rgb(693pt)=(0.708159,0.741865,0.394568);
		rgb(694pt)=(0.710351,0.741712,0.3937);
		rgb(695pt)=(0.71252,0.741559,0.392832);
		rgb(696pt)=(0.71469,0.741405,0.391964);
		rgb(697pt)=(0.71686,0.741252,0.391096);
		rgb(698pt)=(0.719029,0.741082,0.390228);
		rgb(699pt)=(0.721199,0.740904,0.38936);
		rgb(700pt)=(0.723369,0.740725,0.388492);
		rgb(701pt)=(0.725538,0.740546,0.387625);
		rgb(702pt)=(0.727708,0.740386,0.386757);
		rgb(703pt)=(0.729878,0.740233,0.385889);
		rgb(704pt)=(0.732047,0.74008,0.385021);
		rgb(705pt)=(0.734217,0.739927,0.384153);
		rgb(706pt)=(0.736366,0.739753,0.383285);
		rgb(707pt)=(0.73851,0.739574,0.382417);
		rgb(708pt)=(0.740654,0.739395,0.38155);
		rgb(709pt)=(0.742798,0.739217,0.380682);
		rgb(710pt)=(0.744919,0.739038,0.379837);
		rgb(711pt)=(0.747038,0.738859,0.378995);
		rgb(712pt)=(0.749156,0.738681,0.378152);
		rgb(713pt)=(0.751275,0.738502,0.37731);
		rgb(714pt)=(0.753394,0.738323,0.376442);
		rgb(715pt)=(0.755512,0.738145,0.375574);
		rgb(716pt)=(0.757631,0.737966,0.374707);
		rgb(717pt)=(0.75975,0.737789,0.373841);
		rgb(718pt)=(0.761868,0.737636,0.372998);
		rgb(719pt)=(0.763987,0.737483,0.372156);
		rgb(720pt)=(0.766105,0.73733,0.371314);
		rgb(721pt)=(0.76822,0.737169,0.370471);
		rgb(722pt)=(0.770313,0.736965,0.369629);
		rgb(723pt)=(0.772406,0.73676,0.368786);
		rgb(724pt)=(0.774499,0.736556,0.367944);
		rgb(725pt)=(0.776592,0.736358,0.367096);
		rgb(726pt)=(0.778686,0.736179,0.366228);
		rgb(727pt)=(0.780779,0.736001,0.36536);
		rgb(728pt)=(0.782872,0.735822,0.364492);
		rgb(729pt)=(0.784957,0.735643,0.363632);
		rgb(730pt)=(0.787024,0.735465,0.36279);
		rgb(731pt)=(0.789092,0.735286,0.361948);
		rgb(732pt)=(0.791159,0.735107,0.361105);
		rgb(733pt)=(0.793227,0.734929,0.360263);
		rgb(734pt)=(0.795295,0.73475,0.359421);
		rgb(735pt)=(0.797362,0.734571,0.358578);
		rgb(736pt)=(0.79943,0.734392,0.357736);
		rgb(737pt)=(0.801485,0.734214,0.356881);
		rgb(738pt)=(0.803527,0.734035,0.356014);
		rgb(739pt)=(0.805569,0.733856,0.355146);
		rgb(740pt)=(0.807611,0.733678,0.354278);
		rgb(741pt)=(0.809668,0.733499,0.353424);
		rgb(742pt)=(0.811735,0.73332,0.352582);
		rgb(743pt)=(0.813803,0.733142,0.35174);
		rgb(744pt)=(0.81587,0.732963,0.350897);
		rgb(745pt)=(0.817921,0.732784,0.350038);
		rgb(746pt)=(0.819963,0.732606,0.349171);
		rgb(747pt)=(0.822005,0.732427,0.348303);
		rgb(748pt)=(0.824047,0.732248,0.347435);
		rgb(749pt)=(0.826071,0.73207,0.346567);
		rgb(750pt)=(0.828087,0.731891,0.345699);
		rgb(751pt)=(0.830104,0.731712,0.344831);
		rgb(752pt)=(0.83212,0.731534,0.343963);
		rgb(753pt)=(0.834158,0.731355,0.343095);
		rgb(754pt)=(0.8362,0.731176,0.342228);
		rgb(755pt)=(0.838242,0.730998,0.34136);
		rgb(756pt)=(0.840284,0.730819,0.340492);
		rgb(757pt)=(0.842303,0.73064,0.339624);
		rgb(758pt)=(0.84432,0.730462,0.338756);
		rgb(759pt)=(0.846336,0.730283,0.337888);
		rgb(760pt)=(0.848353,0.730104,0.33702);
		rgb(761pt)=(0.850369,0.729926,0.336153);
		rgb(762pt)=(0.852386,0.729747,0.335285);
		rgb(763pt)=(0.854402,0.729568,0.334417);
		rgb(764pt)=(0.856419,0.729391,0.333546);
		rgb(765pt)=(0.858435,0.729238,0.332627);
		rgb(766pt)=(0.860452,0.729085,0.331708);
		rgb(767pt)=(0.862468,0.728932,0.330789);
		rgb(768pt)=(0.864481,0.728778,0.329874);
		rgb(769pt)=(0.866472,0.728625,0.32898);
		rgb(770pt)=(0.868463,0.728472,0.328087);
		rgb(771pt)=(0.870454,0.728319,0.327194);
		rgb(772pt)=(0.872445,0.728166,0.326295);
		rgb(773pt)=(0.874436,0.728013,0.325376);
		rgb(774pt)=(0.876427,0.727859,0.324457);
		rgb(775pt)=(0.878418,0.727706,0.323538);
		rgb(776pt)=(0.880417,0.727561,0.322619);
		rgb(777pt)=(0.882433,0.727433,0.3217);
		rgb(778pt)=(0.88445,0.727306,0.320781);
		rgb(779pt)=(0.886466,0.727178,0.319862);
		rgb(780pt)=(0.888463,0.72705,0.318933);
		rgb(781pt)=(0.890429,0.726923,0.317989);
		rgb(782pt)=(0.892394,0.726795,0.317044);
		rgb(783pt)=(0.894359,0.726668,0.3161);
		rgb(784pt)=(0.896337,0.726552,0.315132);
		rgb(785pt)=(0.898328,0.72645,0.314136);
		rgb(786pt)=(0.900319,0.726348,0.313141);
		rgb(787pt)=(0.90231,0.726246,0.312145);
		rgb(788pt)=(0.904301,0.726158,0.31115);
		rgb(789pt)=(0.906292,0.726081,0.310154);
		rgb(790pt)=(0.908283,0.726005,0.309159);
		rgb(791pt)=(0.910274,0.725928,0.308163);
		rgb(792pt)=(0.912249,0.725851,0.307151);
		rgb(793pt)=(0.914214,0.725775,0.30613);
		rgb(794pt)=(0.91618,0.725698,0.305109);
		rgb(795pt)=(0.918145,0.725622,0.304088);
		rgb(796pt)=(0.920111,0.7256,0.303031);
		rgb(797pt)=(0.922076,0.7256,0.301959);
		rgb(798pt)=(0.924041,0.7256,0.300886);
		rgb(799pt)=(0.926007,0.7256,0.299814);
		rgb(800pt)=(0.927972,0.7256,0.298722);
		rgb(801pt)=(0.929938,0.7256,0.297624);
		rgb(802pt)=(0.931903,0.7256,0.296527);
		rgb(803pt)=(0.933869,0.7256,0.295429);
		rgb(804pt)=(0.935812,0.725668,0.294264);
		rgb(805pt)=(0.937752,0.725744,0.29309);
		rgb(806pt)=(0.939692,0.725821,0.291916);
		rgb(807pt)=(0.941632,0.725897,0.290741);
		rgb(808pt)=(0.943571,0.726023,0.289518);
		rgb(809pt)=(0.945511,0.726151,0.288293);
		rgb(810pt)=(0.947451,0.726278,0.287068);
		rgb(811pt)=(0.949389,0.726411,0.285839);
		rgb(812pt)=(0.951278,0.726641,0.284537);
		rgb(813pt)=(0.953167,0.72687,0.283235);
		rgb(814pt)=(0.955056,0.7271,0.281933);
		rgb(815pt)=(0.956938,0.72734,0.280622);
		rgb(816pt)=(0.958776,0.727646,0.279243);
		rgb(817pt)=(0.960614,0.727952,0.277865);
		rgb(818pt)=(0.962451,0.728259,0.276486);
		rgb(819pt)=(0.964273,0.728597,0.275086);
		rgb(820pt)=(0.966034,0.729057,0.273606);
		rgb(821pt)=(0.967795,0.729516,0.272126);
		rgb(822pt)=(0.969557,0.729976,0.270645);
		rgb(823pt)=(0.971288,0.730473,0.269135);
		rgb(824pt)=(0.972947,0.73106,0.267552);
		rgb(825pt)=(0.974606,0.731647,0.265969);
		rgb(826pt)=(0.976265,0.732234,0.264387);
		rgb(827pt)=(0.977857,0.732879,0.262785);
		rgb(828pt)=(0.979338,0.733619,0.261151);
		rgb(829pt)=(0.980818,0.734359,0.259518);
		rgb(830pt)=(0.982299,0.735099,0.257884);
		rgb(831pt)=(0.983697,0.73591,0.256227);
		rgb(832pt)=(0.984999,0.736803,0.254542);
		rgb(833pt)=(0.986301,0.737697,0.252858);
		rgb(834pt)=(0.987603,0.73859,0.251173);
		rgb(835pt)=(0.988753,0.739566,0.249474);
		rgb(836pt)=(0.989774,0.740613,0.247764);
		rgb(837pt)=(0.990795,0.741659,0.246054);
		rgb(838pt)=(0.991816,0.742706,0.244344);
		rgb(839pt)=(0.992677,0.743816,0.242681);
		rgb(840pt)=(0.993443,0.744965,0.241048);
		rgb(841pt)=(0.994209,0.746114,0.239414);
		rgb(842pt)=(0.994975,0.747262,0.23778);
		rgb(843pt)=(0.995578,0.748465,0.236165);
		rgb(844pt)=(0.996114,0.74969,0.234557);
		rgb(845pt)=(0.99665,0.750915,0.232949);
		rgb(846pt)=(0.997186,0.752141,0.231341);
		rgb(847pt)=(0.997562,0.753386,0.229813);
		rgb(848pt)=(0.997893,0.754637,0.228307);
		rgb(849pt)=(0.998225,0.755887,0.226801);
		rgb(850pt)=(0.998557,0.757138,0.225295);
		rgb(851pt)=(0.998711,0.758433,0.223856);
		rgb(852pt)=(0.998839,0.759735,0.222426);
		rgb(853pt)=(0.998966,0.761037,0.220997);
		rgb(854pt)=(0.999094,0.762339,0.219567);
		rgb(855pt)=(0.999076,0.763641,0.218186);
		rgb(856pt)=(0.99905,0.764942,0.216808);
		rgb(857pt)=(0.999025,0.766244,0.21543);
		rgb(858pt)=(0.998995,0.767546,0.214054);
		rgb(859pt)=(0.998868,0.768848,0.212752);
		rgb(860pt)=(0.99874,0.77015,0.21145);
		rgb(861pt)=(0.998613,0.771451,0.210149);
		rgb(862pt)=(0.998473,0.772756,0.208856);
		rgb(863pt)=(0.998243,0.774083,0.207631);
		rgb(864pt)=(0.998014,0.775411,0.206405);
		rgb(865pt)=(0.997784,0.776738,0.20518);
		rgb(866pt)=(0.997539,0.77806,0.20396);
		rgb(867pt)=(0.997232,0.779362,0.20276);
		rgb(868pt)=(0.996926,0.780664,0.201561);
		rgb(869pt)=(0.99662,0.781966,0.200361);
		rgb(870pt)=(0.996299,0.783268,0.199168);
		rgb(871pt)=(0.995942,0.784569,0.197994);
		rgb(872pt)=(0.995584,0.785871,0.19682);
		rgb(873pt)=(0.995227,0.787173,0.195646);
		rgb(874pt)=(0.994842,0.788475,0.19449);
		rgb(875pt)=(0.994408,0.789777,0.193367);
		rgb(876pt)=(0.993974,0.791078,0.192244);
		rgb(877pt)=(0.99354,0.79238,0.191121);
		rgb(878pt)=(0.993083,0.793671,0.190021);
		rgb(879pt)=(0.992598,0.794947,0.188949);
		rgb(880pt)=(0.992113,0.796223,0.187877);
		rgb(881pt)=(0.991628,0.797499,0.186805);
		rgb(882pt)=(0.99113,0.798789,0.185732);
		rgb(883pt)=(0.990619,0.800091,0.18466);
		rgb(884pt)=(0.990109,0.801393,0.183588);
		rgb(885pt)=(0.989598,0.802695,0.182516);
		rgb(886pt)=(0.989072,0.803996,0.18146);
		rgb(887pt)=(0.988536,0.805298,0.180413);
		rgb(888pt)=(0.988,0.8066,0.179367);
		rgb(889pt)=(0.987464,0.807902,0.17832);
		rgb(890pt)=(0.98691,0.809186,0.177291);
		rgb(891pt)=(0.986349,0.810462,0.17627);
		rgb(892pt)=(0.985787,0.811738,0.175249);
		rgb(893pt)=(0.985226,0.813015,0.174228);
		rgb(894pt)=(0.984644,0.814311,0.173207);
		rgb(895pt)=(0.984057,0.815613,0.172186);
		rgb(896pt)=(0.98347,0.816914,0.171165);
		rgb(897pt)=(0.982883,0.818216,0.170144);
		rgb(898pt)=(0.982296,0.819518,0.169145);
		rgb(899pt)=(0.981709,0.82082,0.16815);
		rgb(900pt)=(0.981122,0.822122,0.167154);
		rgb(901pt)=(0.980535,0.823423,0.166159);
		rgb(902pt)=(0.979947,0.824725,0.165163);
		rgb(903pt)=(0.97936,0.826027,0.164168);
		rgb(904pt)=(0.978773,0.827329,0.163172);
		rgb(905pt)=(0.978186,0.828631,0.162177);
		rgb(906pt)=(0.977599,0.829932,0.161181);
		rgb(907pt)=(0.977012,0.831234,0.160186);
		rgb(908pt)=(0.976425,0.832536,0.15919);
		rgb(909pt)=(0.975838,0.833841,0.158195);
		rgb(910pt)=(0.975251,0.835168,0.157199);
		rgb(911pt)=(0.974664,0.836495,0.156204);
		rgb(912pt)=(0.974077,0.837823,0.155208);
		rgb(913pt)=(0.973489,0.83915,0.154213);
		rgb(914pt)=(0.972902,0.840477,0.153217);
		rgb(915pt)=(0.972315,0.841805,0.152222);
		rgb(916pt)=(0.971728,0.843132,0.151226);
		rgb(917pt)=(0.971155,0.844466,0.150224);
		rgb(918pt)=(0.970619,0.845819,0.149203);
		rgb(919pt)=(0.970083,0.847172,0.148182);
		rgb(920pt)=(0.969547,0.848525,0.147161);
		rgb(921pt)=(0.96902,0.849886,0.14614);
		rgb(922pt)=(0.968509,0.851265,0.145119);
		rgb(923pt)=(0.967999,0.852643,0.144098);
		rgb(924pt)=(0.967488,0.854022,0.143077);
		rgb(925pt)=(0.967,0.855411,0.142056);
		rgb(926pt)=(0.966541,0.856815,0.141035);
		rgb(927pt)=(0.966081,0.858219,0.140014);
		rgb(928pt)=(0.965622,0.859623,0.138992);
		rgb(929pt)=(0.965189,0.86104,0.137945);
		rgb(930pt)=(0.96478,0.862469,0.136873);
		rgb(931pt)=(0.964372,0.863899,0.135801);
		rgb(932pt)=(0.963963,0.865328,0.134729);
		rgb(933pt)=(0.96357,0.866773,0.133657);
		rgb(934pt)=(0.963187,0.868228,0.132585);
		rgb(935pt)=(0.962805,0.869683,0.131513);
		rgb(936pt)=(0.962422,0.871138,0.130441);
		rgb(937pt)=(0.962091,0.87261,0.129351);
		rgb(938pt)=(0.961785,0.874091,0.128253);
		rgb(939pt)=(0.961478,0.875571,0.127156);
		rgb(940pt)=(0.961172,0.877052,0.126058);
		rgb(941pt)=(0.960885,0.878571,0.124961);
		rgb(942pt)=(0.960605,0.880103,0.123863);
		rgb(943pt)=(0.960324,0.881634,0.122765);
		rgb(944pt)=(0.960043,0.883166,0.121668);
		rgb(945pt)=(0.959849,0.884719,0.120549);
		rgb(946pt)=(0.95967,0.886276,0.119426);
		rgb(947pt)=(0.959491,0.887833,0.118302);
		rgb(948pt)=(0.959313,0.88939,0.117179);
		rgb(949pt)=(0.959181,0.890995,0.116032);
		rgb(950pt)=(0.959054,0.892603,0.114884);
		rgb(951pt)=(0.958926,0.894211,0.113735);
		rgb(952pt)=(0.958799,0.895819,0.112587);
		rgb(953pt)=(0.958748,0.897453,0.111464);
		rgb(954pt)=(0.958697,0.899086,0.110341);
		rgb(955pt)=(0.958646,0.90072,0.109217);
		rgb(956pt)=(0.958602,0.902359,0.108089);
		rgb(957pt)=(0.958628,0.904043,0.106915);
		rgb(958pt)=(0.958653,0.905728,0.105741);
		rgb(959pt)=(0.958679,0.907413,0.104567);
		rgb(960pt)=(0.958718,0.909102,0.103393);
		rgb(961pt)=(0.95882,0.910812,0.102219);
		rgb(962pt)=(0.958922,0.912522,0.101044);
		rgb(963pt)=(0.959024,0.914232,0.0998703);
		rgb(964pt)=(0.959153,0.915962,0.0986961);
		rgb(965pt)=(0.959357,0.917749,0.0975219);
		rgb(966pt)=(0.959561,0.919536,0.0963477);
		rgb(967pt)=(0.959765,0.921323,0.0951736);
		rgb(968pt)=(0.959996,0.923118,0.0939907);
		rgb(969pt)=(0.960277,0.924931,0.092791);
		rgb(970pt)=(0.960557,0.926743,0.0915913);
		rgb(971pt)=(0.960838,0.928555,0.0903916);
		rgb(972pt)=(0.961151,0.930378,0.0891919);
		rgb(973pt)=(0.961509,0.932216,0.0879922);
		rgb(974pt)=(0.961866,0.934054,0.0867925);
		rgb(975pt)=(0.962223,0.935892,0.0855928);
		rgb(976pt)=(0.96262,0.937768,0.0843802);
		rgb(977pt)=(0.963053,0.939683,0.083155);
		rgb(978pt)=(0.963487,0.941597,0.0819297);
		rgb(979pt)=(0.963921,0.943512,0.0807045);
		rgb(980pt)=(0.964415,0.945426,0.0794643);
		rgb(981pt)=(0.964951,0.947341,0.0782135);
		rgb(982pt)=(0.965487,0.949255,0.0769628);
		rgb(983pt)=(0.966023,0.951169,0.075712);
		rgb(984pt)=(0.966594,0.953118,0.0744441);
		rgb(985pt)=(0.967181,0.955083,0.0731679);
		rgb(986pt)=(0.967768,0.957049,0.0718916);
		rgb(987pt)=(0.968355,0.959014,0.0706153);
		rgb(988pt)=(0.96898,0.960999,0.0693006);
		rgb(989pt)=(0.969619,0.96299,0.0679733);
		rgb(990pt)=(0.970257,0.964981,0.0666459);
		rgb(991pt)=(0.970895,0.966972,0.0653186);
		rgb(992pt)=(0.971554,0.968984,0.0639486);
		rgb(993pt)=(0.972218,0.971001,0.0625703);
		rgb(994pt)=(0.972882,0.973017,0.0611919);
		rgb(995pt)=(0.973545,0.975034,0.0598135);
		rgb(996pt)=(0.974232,0.97705,0.058318);
		rgb(997pt)=(0.974922,0.979067,0.056812);
		rgb(998pt)=(0.975611,0.981083,0.055306);
		rgb(999pt)=(0.9763,0.9831,0.0538)
	}
}

\usetikzlibrary{arrows,shapes,positioning}
\usetikzlibrary{decorations.markings}
\tikzstyle arrowstyle=[scale=1]
\tikzstyle directed=[postaction={decorate,decoration={markings,
		mark=at position .65 with {\arrow[arrowstyle]{stealth}}}}]
\tikzstyle reverse directed=[postaction={decorate,decoration={markings,
		mark=at position .65 with {\arrowreversed[arrowstyle]{stealth};}}}]

\newcommand{\secref}[1]{Section~\ref{#1}}

\newcommand{\defref}[1]{Definition~\ref{#1}}

\newcommand{\tabref}[1]{Table~\ref{#1}}

\usepackage{mathtools}

\newcommand{\R}{\mathbb{R}}

\newcommand{\Domain}{\ensuremath{\mathcal{X}}} 
\newcommand{\dimension}{\ensuremath{d}} 
\newcommand{\timevar}{\ensuremath{t}} 

\newcommand{\indicator}[1]{\ensuremath{\mathbbm{1}_{#1}}}

\newcommand{\Lp}[1]{\ensuremath{L_{#1}}}

\newcommand{\flux}{\ensuremath{\mathbf{f}}}

\newcommand{\entropy}{\ensuremath{h}} 

\newcommand{\x}{\ensuremath{\bx}}
\newcommand{\y}{\ensuremath{y}}

\newcommand{\dt}{\partial_\timevar}

\newcommand{\dtstepsize}{\ensuremath{\Delta \timevar}}

\newcommand{\ncells}{\ensuremath{N_{\Domain}}}

\newcommand{\timeind}{\ensuremath{n}}

\newcommand{\cellind}{\ensuremath{i}}
\newcommand{\cellindx}{\ensuremath{i}}
\newcommand{\cellindy}{\ensuremath{j}}
\newcommand{\cellindR}{\ensuremath{l}}
\newcommand{\cellindRvar}{\ensuremath{\tilde{l}}}

\newcommand{\cell}[1]{\ensuremath{\Domain_{#1}}}
\newcommand{\cellR}[1]{\ensuremath{\Xi_{#1}}}

\newcommand{\limitervariable}{\ensuremath{\theta}}
\newcommand{\limitertolerance}{\ensuremath{\epsilon}}

\newcommand{\density}{\ensuremath{\rho}}

\newcommand{\energy}{\ensuremath{E}}
\newcommand{\pressure}{\ensuremath{p}}

\newcommand{\momentum}{\ensuremath{\rho v}}

\newcommand{\densityvar}{\ensuremath{\widetilde{\rho}}}
\newcommand{\energyvar}{\ensuremath{\widetilde{E}}}
\newcommand{\momentumvar}{\ensuremath{\widetilde{\rho}\widetilde{v}}}

\newcommand{\uncertainty}{\ensuremath{\xi}}

\newcommand{\xiPDF}{\ensuremath{f_\Xi}}
\newcommand{\xiBasisPoly}[1]{\ensuremath{\phi_{#1}}}
\newcommand{\SGsumIndex}{\ensuremath{k}}
\newcommand{\SGsumIndexvar}{\ensuremath{\tilde{k}}}

\newcommand{\SGeqIndex}{\ensuremath{\tilde{k}}}

\newcommand{\xiPDFdxi}{\ensuremath{\xiPDF \mathrm{d}\uncertainty}}
\newcommand{\intRS}{\ensuremath{\int_{\randomSpace}}}

\newcommand{\MEintcellR}[2]{\ensuremath{\int\limits_{\cellR{#1}} #2\,\locxiPDF{#1} \mathrm{d}\uncertainty}}

\newcommand{\SGtruncorder}{\ensuremath{{K_\randomSpace}}}

\newcommand{\nbxiQuadNodes}{\ensuremath{Q_\randomSpace}}

\newcommand{\realizableSet}{\ensuremath{\mathcal{R}}}

\newcommand{\sampleSpace}{\ensuremath{\Omega}}
\newcommand{\sigmaAlgebra}{\ensuremath{\mathcal{F}}}
\newcommand{\probabilityMeasure}{\ensuremath{\mathbb{P}}}
\newcommand{\randomSpace}{\ensuremath{\Xi}}

\newcommand{\solution}{\ensuremath{\bu}}

\newcommand{\cutfun}[1]{h\left(#1\right)}

\newcommand{\entropicVar}{\ensuremath{\boldsymbol{\lambda}}}
\newcommand{\entropicVarSum}{\ensuremath{\boldsymbol{\Lambda}}}
\newcommand{\InvGradEntropy}{\ensuremath{\left( \nabla_{\!\solution}\entropy\right)^{-1}\!}}
\newcommand{\DiffEntropy}{\ensuremath{ \nabla_{\!\solution}\entropy}}
\newcommand{\DiffSol}{\ensuremath{ \nabla_{\!\!\scalebox{0.6}{\entropicVarSum}\!}\solution}}
\newcommand{\entrcm}[2]{\ensuremath{\bar{\entropicVarSum}^{(#1)}_{#2}}}

\newcommand{\hyperbolic}{admissible}
\newcommand{\hypset}{hyperbolicity set}

\newcommand{\randomElement}[1]{\ensuremath{\Xi_{#1}}}
\newcommand{\indicatorVar}[2]{\ensuremath{\mbox{\large $\chi$}_{#1}(#2)}}

\newcommand{\MEIndex}{\ensuremath{l}}
\newcommand{\MEElements}{\ensuremath{{N_{\randomSpace}}}}
\newcommand{\localxiBasisPoly}[2]{\ensuremath{\phi_{#1,#2} }}

\newcommand{\ME}{multi-element}
\newcommand{\uncertaintyR}{\uncertainty^R}
\newcommand{\uncertaintyL}{\uncertainty^L}

\newcommand{\E}{\ensuremath{\mathbb{E}}}
\DeclareMathOperator{\Var}{\text{Var}}

\pgfplotscreateplotcyclelist{color parula}{%
	royalblue!70,every mark/.append style={solid,line width = 0pt,fill=royalblue!60!black},mark=ball\\%
	goldenrod!50!yelloworange,every mark/.append style={solid,fill=goldenrod!30!black},mark=*\\%
	green,every mark/.append style={black,fill=yellowgreen},mark=diamond*\\%
	bluegreen,every mark/.append style={fill=black},mark=triangle*\\%
	royalblue!70,densely dashed,every mark/.append style={solid,line width = 0pt,fill=royalblue!60!black},mark=ball\\%
	goldenrod!50!yelloworange,densely dashed,every mark/.append style={solid,black,fill=goldenrod},mark=diamond*\\%
	yellowgreen,densely dashed,every mark/.append style={solid,fill=yellowgreen!60!royalblue},mark=square*, mark size= 1pt\\%
	bluegreen,densely dashed,mark size = 1.5pt,every mark/.append style={solid,fill=white},mark=otimes*\\
	royalblue,densely dashed,mark size = 1.5pt,every mark/.append style={solid,fill=royalblue!30!black},mark=pentagon*\\
	goldenrod!40!yelloworange,every mark/.append style={solid,fill=goldenrod!30!black},mark=|\\%
}

\pgfplotscreateplotcyclelist{color louisa}{%
	blue,every mark/.append style={solid,line width = 0pt,fill=royalblue!60!black},mark=ball\\%
	red,every mark/.append style={solid,fill=orangered!30!black},mark=*\\%
	green,every mark/.append style={black,fill=yellowgreen},mark=diamond*\\%
	goldenrod!50!yelloworange,every mark/.append style={fill=black},mark=triangle*\\%
	royalblue!70,densely dashed,every mark/.append style={solid,line width = 0pt,fill=royalblue!60!black},mark=ball\\%
	goldenrod!50!yelloworange,densely dashed,every mark/.append style={solid,black,fill=goldenrod},mark=diamond*\\%
	yellowgreen,densely dashed,every mark/.append style={solid,fill=yellowgreen!60!royalblue},mark=square*, mark size= 1pt\\%
	bluegreen,densely dashed,mark size = 1.5pt,every mark/.append style={solid,fill=white},mark=otimes*\\
	royalblue,densely dashed,mark size = 1.5pt,every mark/.append style={solid,fill=royalblue!30!black},mark=pentagon*\\
	goldenrod!40!yelloworange,every mark/.append style={solid,fill=goldenrod!30!black},mark=|\\%
}

\pgfplotscreateplotcyclelist{color fabian}{%
	black,dashed\\%
	red,every mark/.append style={solid},mark=diamond*\\%
	blue,every mark/.append style={solid},mark=*\\%
	goldenrod!50!yelloworange,every mark/.append style={fill=black},mark=triangle*\\%
	royalblue!70,densely dashed,every mark/.append style={solid,line width = 0pt,fill=royalblue!60!black},mark=ball\\%
	goldenrod!50!yelloworange,densely dashed,every mark/.append style={solid,black,fill=goldenrod},mark=diamond*\\%
	yellowgreen,densely dashed,every mark/.append style={solid,fill=yellowgreen!60!royalblue},mark=square*, mark size= 1pt\\%
	bluegreen,densely dashed,mark size = 1.5pt,every mark/.append style={solid,fill=white},mark=otimes*\\
	royalblue,densely dashed,mark size = 1.5pt,every mark/.append style={solid,fill=royalblue!30!black},mark=pentagon*\\
	goldenrod!40!yelloworange,every mark/.append style={solid,fill=goldenrod!30!black},mark=|\\%
}

\usepackage[colorinlistoftodos,obeyFinal]{todonotes}

\newcommand{\nstochdim}{\ensuremath{N}}
\newcommand{\nstochdimidx}{\ensuremath{n}}
\newcommand{\uncertaintydim}[1]{\xi^{#1}}

\newcommand{\locxiPDF}[1]{\ensuremath{f}_{\randomElement{#1}}}

\newcommand{\MExiPDFdxi}{\ensuremath{\locxiPDF{_\MEIndex} \mathrm{d}\uncertainty_\MEIndex}}

\newcommand{\tend}{\ensuremath{T}}

\newcommand{\xiPDFdxiu}{\ensuremath{f_{\Xi_\cellindR}\!(\uncertainty) \mathrm{d}\uncertainty}}
\newcommand{\MEintcellRu}[2]{\ensuremath{\int\limits_{\cellR{#1}} #2\,\locxiPDF{\MEIndex}\!(\uncertainty) \mathrm{d}\uncertainty}}

\newcommand{\lagrange}{\mathcal{L}}
\newcommand{\fluxvec}{\ensuremath{\bF}}
\newcommand{\nbflux}{M}
\newcommand{\fluxidx}{m}
\newcommand{\solutioncoeff}[3]{{\overline{\solution}}_{#1,#2}^{(#3)}}
\newcommand{\solutioncoeffFilter}[3]{{\widehat{\overline{\solution}}}_{#1,#2}^{(#3)}}
\newcommand{\solutioncm}[2]{{\overline{\solution}}_{#2}^{(#1)}}
\newcommand{\solutioncmFilter}[2]{{\widehat{\overline{\solution}}}_{#2}^{(#1)}}
\newcommand{\solutionbar}{{\overline{\solution}}}

\newcommand{\hyperbolLimitm}[1]{\ensuremath{\Lambda\Pi_\SGtruncorder^{#1}}}
\newcommand{\costfun}{J}
\newcommand{\SGsolution}{\solution^{\SGtruncorder}}
\newcommand{\filterSGsolution}{\widehat{\solution}^{\SGtruncorder}}
\newcommand{\filter}{g}
\newcommand{\filterstrength}{\lambda}
\newcommand{\filteroperator}{L}

\newcommand{\indicatorVarl}[1]{\ensuremath{\mbox{\large $\chi$}_{#1}}}
\renewcommand{\comment}[1]{\textcolor{black}{#1}}

\usepackage{media9}

\begin{document}

\begin{abstract}
In this article we study intrusive uncertainty quantification schemes for systems of conservation laws with uncertainty. While intrusive methods inherit certain advantages such as adaptivity and an improved accuracy, they suffer from two key issues. First, intrusive methods tend to show oscillations, especially at shock structures and second, standard intrusive methods can lose hyperbolicity. The aim of this work is to tackle these challenges with the help of two different strategies. First, we combine filters with the \ME{} approach for the hyperbolicity--preserving stochastic Galerkin (hSG) scheme. While the limiter used in the hSG scheme ensures hyperbolicity, the filter as well as the \ME{} ansatz mitigate oscillations. Second, we derive a \ME{} approach for the intrusive polynomial moment (IPM) method. Even though the IPM method is inherently hyperbolic, it suffers from oscillations while requiring the solution of an optimization problem in every spatial cell and every time step. The proposed \ME{} IPM method leads to a decoupling of the optimization problem in every multi-element. Thus, we are able to significantly decrease computational costs while improving parallelizability. Both proposed strategies are extended to adaptivity, allowing to adapt the number of basis functions in each multi-element to the smoothness of the solution. We finally evaluate and compare both approaches on various numerical examples such as a NACA airfoil and a nozzle test case for the two-dimensional Euler equations. In our numerical experiments, we observe the mitigation of spurious artifacts. Furthermore, using the multi-element ansatz for IPM significantly reduces computational costs.
\end{abstract}


\maketitle

\noindent

\section{Introduction}\label{sec:one}
Uncertainties play a key role in modelling hyperbolic systems of conservation laws if 
crucial data or parameters might not be available exactly due to measurement errors. Thus, to achieve an adequate description of reality, one needs to include non-deterministic effects in the approximation of deterministic systems.
In the recent years many  approaches, e.g., methods based on Bayesian inversion, Monte Carlo algorithms or stochastic Galerkin schemes, were proposed to quantify the uncertainties in order to account for them in predictions and simulations \cite{Schillings2017,Abgrall2013,Abgrall2007,Gottlieb2008,Kolb2018,Jin2017c,Bos2018,Yan2019}.  In the context of these Uncertainty Quantification (UQ) methods, we distinguish between so-called non-intrusive and intrusive approaches.

Non-intrusive UQ methods use deterministic solvers of the equations as a black box code in order to solve the model problem by repeated application of this black box
on fixed realizations of the uncertainty. The most widely known non-intrusive UQ method is (Multi-Level) Monte Carlo \cite{lye2016multilevel,Heinrich2001,Giles2008} which is based on statistical sampling methods and can be easily implemented and adopted to any type of uncertain conservation law, but comes with potentially high cost due to repeated application of e.g. finite volume methods (FVM). These so-called MC-FVM schemes have been studied in \cite{Mishra2016,mishra2013} showing a slow convergence rate that is improved by Multi-Level MC-FVM algorithms for conservation laws \cite{Mishra2014,mishra2012multi}. Another non-intrusive UQ scheme is stochastic collocation \cite{Xiu2005} and further developments of this method are described in \cite{Witteveen2009,Ma2009,Nobile2008}.

In this article, we focus on intrusive UQ methods, which aim to increase the overall efficiency but require a modification of the underlying solver for the deterministic problem. The most popular methods are the stochastic Galerkin (SG) and the intrusive polynomial moment (IPM) method \cite{Poette2009}. They both rely on the generalized Polynomial Chaos (gPC) expansion \cite{Xiu2003,Wan2006a,Abgrall2007,Chen2005} which is theoretically based on the Polynomial Chaos expansion from \cite{Wiener1938}. Stochastic Galerkin expands the solution in the conserved variables such that the gPC system results in a weak formulation of the equations with respect to the stochastic variable. IPM expands the stochastic solution in the so-called entropic variables, which results in a hyperbolic gPC system that yields a good approximation quality, however, it is necessary to know a strictly convex entropy of the original system beforehand. 

The biggest challenge of UQ methods for hyperbolic conservation laws lies in the fact that discontinuities in the physical space propagate into the solution manifold \cite{barth2013non} and cause oscillations. These oscillations for example translate into non-physical step-like approximations of the expected value around shocks.
The naive usage of SG for nonlinear hyperbolic problems even typically fails \cite{abgrall2017uncertainty,Poette2009}, since the polynomial expansion of discontinuous data yields huge Gibbs oscillations that result in the loss of hyperbolicity. In order to resolve this problem, we apply the so-called hyperbolicity-preserving limiter to the classical SG approach. This limiter is introduced in \cite{Schlachter2018} and the theory is extended in \cite{Meyer2019a} to high-order Runge-Kutta discontinuous Galerkin schemes as well as the so-called \ME{} approach developed by \cite{Wan2006,Tryoen2010}, where  the random space is divided into disjoint elements in order to define local gPC approximations. The \ME{} ansatz corresponds to an $h$-refinement in the random space. 
Further developments of this method encompass $h$- and $hp$-adaptive refinements in the stochastic space \cite{Wan2005,Wan2009,Tryoen2010}. 
As presented in \cite{Schlachter2019},  a simple example of the uncertain linear transport equation shows that the \ME{} ansatz and the hyperbolicity limiter is still not enough to sufficiently reduce oscillations. Similar situations have been observed in \cite{Poette2009,Barth2012,Ma2009}. Another approach to damp oscillations induced by the Gibbs phenomenon is given in \cite{Kusch2018}, where filters are applied to the gPC coefficients of the SG approximation. Filters are a common technique from kinetic theory \cite{McClarren2010} to reduce oscillations but do not guarantee the hyperbolicity of the gPC system.

On the other hand, the hyperbolicity of the moment system is ensured through the IPM method that has been introduced in \cite{Poette2009}. Expanding the stochastic solution not in the conserved variables but in entropic variables is well known in the radiative transfer community as minimum entropy models \cite{levermore1996moment,hauck2011high,Kusch2018a}. The resulting gPC system has good approximation properties but requires to solve typically expensive optimization problems for nonlinear systems in every space-time cell in order to calculate the entropic variable. IPM may be used to control oscillations of the solution since they are bounded to a certain range though the entropy of the system. Another development of this method is proposed in \cite{Kusch2017}, where a second-order IPM scheme is described which fulfills the maximum principle. More information on IPM can be found in \cite{poette2011treatment,despres2013robust,kusch2020intrusive,Kusch2019,Kusch2018a}. A method similar to IPM is the modal discretization of Roe variables, see e.g. \cite{pettersson2014stochastic,gerster2020entropies} which ensures hyperbolicity. However, similar to IPM and SG, the method experiences non-physical artifacts such as step-like expected value approximations.

Within the scope of this work, we propose two strategies to ensure both, the mitigation of oscillations and the preservation of hyperbolicity. First, we employ the \ME{} ansatz together with the hyperbolicity limiter and filtering to stochastic Galerkin. Though this extension is straight forward, it allows the use of filters in challenging settings, leading to a robust numerical scheme that is able to deal with Gibbs phenomenon. Second, we derive an extension of the IPM method to \ME{} basis functions. Dividing the stochastic domain into \ME{}s promises to reduce oscillations and to decrease computational times within IPM since it decouples and simplifies the optimization problems. In order to further reduce the numerical costs we additionally apply the one-shot implementation for IPM from \cite{kusch2020intrusive}, similar to the idea of one-shot optimization in shape optimization \cite{Hazra2005}. Moreover, each of our numerical schemes is accelerated employing adaptivity, i.e. the truncation order of the gPC polynomials adapts locally to the smoothness of the solution, cf. \comment{\cite{tryoen2012adaptive,buerger2014hybrid,meyer2020posteriori,kusch2020intrusive}}.

The paper is structured as follows. In \secref{sec:two} we describe our problem setting for systems of hyperbolic conservation laws with uncertainty. We then describe in \secref{sec:three} the filtered hyperbolicity-preserving stochastic Galerkin scheme by introducing the \ME{} approach as well as the hyperbolicity limiter and filters for stochastic Galerkin. In \secref{sec:four} we complete the theoretical framework by deriving the \ME{} intrusive polynomial moment method. With these oscillation mitigating intrusive UQ schemes at hand, we present numerical test cases for the one- and two-dimensional Euler equations within \secref{sec:five}. Here, we study the behavior of the proposed methods in terms of errors and computational costs, that verify the reduction of oscillations compared to the standard intrusive schemes. 
Moreover, we are able to significantly reduce the numerical costs by employing the multi-element ansatz in IPM.
\section{Problem Setting}\label{sec:two}
We consider random systems of hyperbolic conservation laws of the form
\begin{subequations}\label{eq:spde}
\begin{equation}\dt \solution(\timevar,\x,\uncertainty) + \nabla_\x\cdot \fluxvec (\solution(\timevar,\x,\uncertainty),\uncertainty) = \bm{0},\end{equation}
for a flux function $\fluxvec=(\flux_1,\ldots,\flux_\nbflux)$, with $\flux_\fluxidx({\solution}): \R^{\dimension} \rightarrow \R^{\dimension}$, $\fluxidx=1,\ldots,\nbflux$, for $\x\in\Domain\subset\R^\nbflux$, and where the solution
$$\solution = \solution(\timevar,\x,\uncertainty):\R_+ \times \R^\nbflux \times \randomSpace \rightarrow \R^{\dimension}$$
is depending on a one-dimensional random variable $\uncertainty$ with probability space $(\sampleSpace,\sigmaAlgebra,\probabilityMeasure)$. Here,
we denote the random space of this uncertainty by $\randomSpace := \uncertainty(\sampleSpace)$ and its probability density function by $\xiPDF(\uncertainty) : \randomSpace \rightarrow \R_+$.
We further assume that the uncertainty can be introduced via the initial conditions, namely
\begin{equation}\solution(\timevar=0,\x,\uncertainty) = \solution^0(\x,\uncertainty).\end{equation}
\end{subequations}
Depending on $\Domain$, additional boundary conditions have to prescribed.
In this article, we consider intrusive uncertainty quantification methods in order to solve \eqref{eq:spde}. 
Therefore, we describe the most commonly used stochastic Galerkin and intrusive polynomial moment method in the following sections, where we modify them according to the theory explained within the introduction. 
The proposed methods are accelerated using techniques discussed in \cite{kusch2020intrusive}. 

\section{Filtered hyperbolicity-preserving stochastic Galerkin scheme}\label{sec:three}
In this section we explain oscillation mitigating approaches for the stochastic Galerkin scheme. As mentioned in \secref{sec:one}, stochastic Galerkin systems are likely to lose hyperbolicity \cite{Poette2009}, since discontinuities that typically arise in hyperbolic equations propagate into the stochastic domain and cause Gibbs oscillations that might yield the loss of hyperbolicity, i.e. the failure of the standard stochastic Galerkin scheme. We therefore apply the \ME{} approach \cite{Wan2006,Tryoen2010} as well as the hyperbolicity-preserving limiter that was introduced in \cite{Schlachter2018}. However, this limiter still does not sufficiently damp Gibbs oscillations as seen in \cite{Schlachter2019}, where an uncertain linear advection example reveals huge overshoots that are not reduced by the hyperbolicity limiter due to the preset hyperbolicity in this scalar equation. 
In this article, we consider another approach that aims on reducing oscillations, i.e. we combine the filtered SG method \cite{Kusch2018} with the hyperbolicity limiter such that the method can be applied on any system of conservation laws. To this end, we filter the gPC coefficients of the SG approximation, which allows a numerically cheap reduction of oscillations. Different filters for SG schemes are derived and described in \cite{Kusch2018}.

\subsection{Multi-Element Stochastic Galerkin}\label{sec:multielement}
We seek for an approximate solution by a finite-term generalized Polynomial Chaos expansion (see e.g. \cite{Gottlieb2008}) \comment{of degree $\SGtruncorder$ which reads}
\begin{equation}\label{SGapproach}
\solution(\timevar,\x,\uncertainty) \approx  \sum\limits_{\SGsumIndex=0}^{\SGtruncorder} \solution_\SGsumIndex(\timevar,\x) \, \xiBasisPoly{\SGsumIndex}(\uncertainty).
\end{equation}
\comment{The polynomials} $\xiBasisPoly{\SGsumIndex}$ of degree $\SGsumIndex$ are supposed to satisfy the orthogonality relation
\begin{equation}\label{eq:orthogonality}
\int\limits_{\randomSpace} \xiBasisPoly{\SGsumIndex}(\uncertainty) \, \xiBasisPoly{\SGeqIndex}(\uncertainty) \, \xiPDF(\uncertainty) \mathrm{d}{\uncertainty} =
\begin{cases}
1, &\quad\text{if }\SGsumIndex=\SGeqIndex\\
0,&\quad\text{else}
\end{cases} \quad \quad \forall \ \SGsumIndex,\SGeqIndex \in \{0, \dots, \SGtruncorder\}.
\end{equation}
Inserting \eqref{SGapproach} into \eqref{eq:spde} and applying a Galerkin projection in the stochastic space leads to the so-called stochastic Galerkin system
\begin{equation}\label{SGsystem}\partial_\timevar \solution_{\SGeqIndex}(\timevar,\x) +  \nabla_\x\cdot \intRS \fluxvec\left(\sum\limits_{\SGsumIndex=0}^{\SGtruncorder} \solution_\SGsumIndex(\timevar,\x) \, \xiBasisPoly{\SGsumIndex}(\uncertainty),\uncertainty\right) \!\xiBasisPoly{\SGeqIndex}(\uncertainty)\xiPDF(\uncertainty)\mathrm{d}\uncertainty = 0, \qquad \SGeqIndex = \bm{0},\ldots,\SGtruncorder.\end{equation}

For discontinuous solutions, the gPC approach may converge slowly or even fail to converge, cf. \cite{Wan2005, Poette2009}.  As presented in \cite{Wan2006,Meyer2019a}, we therefore apply the \ME{} approach, where $\randomSpace$ is divided into disjoint elements with local gPC approximations of  \eqref{eq:spde}. 

We assume that $\randomSpace=(\uncertaintyL,\uncertaintyR)$ and define a decomposition of $\randomSpace$ into $\MEElements$ \ME{}s $\randomElement{\MEIndex} = (\uncertainty_{\MEIndex-\frac12}, \uncertainty_{\MEIndex+\frac12})$, $\cellindR=1,\ldots,\MEElements$, of width $\Delta\uncertainty=\frac{\uncertaintyL-\uncertaintyR}{\MEElements}$.  Moreover, we  introduce an indicator variable $\chi_{\MEIndex}: \Omega \to  \{0,1\}$ on every random element
\begin{equation}\label{def:indicator}
\indicatorVar{\MEIndex}{\uncertainty} :=
\begin{cases}
1&\text{if } \uncertainty \in \randomElement{\MEIndex}, \\
0 &\text{else, }
\end{cases}
\end{equation}
for $\MEIndex=1,\ldots,\MEElements $. 
If we let $\{\localxiBasisPoly{\SGsumIndex}{\MEIndex} (\uncertainty)\}_{\SGsumIndex=0}^\infty$ be  orthonormal polynomials with respect to a conditional probability density function $\locxiPDF{\cellindR}:\cellR{\cellindR}\rightarrow\mathbb{R}$ on the \ME{} $\randomElement{\MEIndex}$, as in \cite{Meyer2019a}, 
the global approximation \eqref{SGapproach} can be written as
\begin{equation}\label{def:globalSG}
\solution(t,\x,\uncertainty) = \sum_{\MEIndex=1}^{\MEElements}  \solution_\MEIndex(t,\x,\uncertainty) \indicatorVar{\MEIndex}{\uncertainty} \approx \sum_{\MEIndex=1}^{\MEElements}  \sum_{\SGsumIndex=0}^{\SGtruncorder}  \solution_{\SGsumIndex, \MEIndex}(\timevar,\x) \localxiBasisPoly{\SGsumIndex}{\MEIndex}(\uncertainty) \indicatorVar{\MEIndex}{\uncertainty}.
\end{equation}
As $\MEElements,\SGtruncorder\to\infty$, the local approximation converges to the global solution in $\Lp{2}(\sampleSpace)$, cf. \cite{Alpert1993}.

\begin{remark}
	For a uniform distribution of the uncertainty, we have $\locxiPDF{\MEIndex}\!(\uncertainty)=\frac{1}{\xi_{\MEIndex+1}-\xi_\MEIndex}\indicatorVar{\cellR{\cellindR}}{\uncertainty}$.
	Moreover, the local probability density function remains a density function of a uniformly distributed random variable. 
	Other distribution types such as the normal distribution require to numerically compute a set of polynomials which are orthogonal with respect to $
	\locxiPDF{\MEIndex}\!(\uncertainty)$, see \cite{Wan2006}. 
\end{remark}

We come up with the following ME stochastic Galerkin system
\begin{equation}\label{SGsystemME}\partial_t \solution_{\SGeqIndex,\MEIndex}(\timevar,\x) +  \nabla_\x\cdot \int\limits_{\cellR{\cellindR}}\! \fluxvec\Big(\!\sum_{\MEIndex=1}^{\MEElements}  \sum_{\SGsumIndex=0}^{\SGtruncorder}  \solution_{\SGsumIndex, \MEIndex}(\timevar,\x) \localxiBasisPoly{\SGsumIndex}{\MEIndex}(\uncertainty),\uncertainty \Big) \xiBasisPoly{\SGeqIndex,\MEIndex}(\uncertainty)\xiPDFdxiu = \bm{0},\end{equation}
for $\SGeqIndex=0,\ldots,\SGtruncorder$ and $\MEIndex=1,\ldots,\MEElements$.

The calculation of the expected value and variance of $\solution$ can be found in \cite{Meyer2019a}, 
\comment{we obtain
\begin{align}
\E(\solution) &\approx  
\sum_{\MEIndex=1}^{\MEElements}  \probabilityMeasure(\indicatorVarl{\MEIndex}=1) \,\solution_{0,\MEIndex}, \label{SGEV}\\[0.2cm]
\Var(\solution) &\approx
\sum_{\MEIndex=1}^{\MEElements} \Big( \Var(\solution_{\MEIndex}) + \big(\solution_{0,\MEIndex} - \mathbb{E}(\solution)\big)^2\Big)\,\probabilityMeasure(\indicatorVarl{\MEIndex}=1),\label{SGSD}
\end{align}
where the local variance $\Var(\solution_{\MEIndex})$ is given by
$$\Var(\solution_{\MEIndex})\approx \int_{\randomElement{\MEIndex}} \! \bigg( \sum_{\SGsumIndex=0}^{\SGtruncorder}  \solution_{\SGsumIndex,\MEIndex}\localxiBasisPoly{\SGsumIndex}{\MEIndex}\bigg)^2  \MExiPDFdxi - \solution_{0,\MEIndex}^2
=\sum_{\SGsumIndex=1}^{\SGtruncorder}  \solution_{\SGsumIndex,\MEIndex}^2.$$}
\renewcommand{\x}{x}
\subsection{Numerical Implementation}\label{sec:numerics}
We now describe our numerical scheme, 
 whereas we consider two-dimensional physical spaces which we require for our numerical examples of the two-dimensional Euler equations. Thus, we divide the spatial domain $\Domain=[\x_L,\x_R]\times[\y_L,\y_R] \subset\R^2$ into a uniform rectangular mesh with cells $\cell{\cellindx,\cellindy} =  [\x_{\cellindx-\frac{1}{2}},\x_{\cellindx+\frac{1}{2}}]\times [\y_{{\cellindy-\frac{1}{2}}},\y_{\cellindy+\frac{1}{2}}]$ where $\x_{\cellindx\pm\frac{1}{2}} = \x_{\cellindx}\pm\frac{\Delta\x }{2}$, $\y_{\cellindy\pm\frac{1}{2}} = \y_{\cellindy}\pm\frac{\Delta\y }{2}$ and $\Delta\x = \frac{\x_R-\x_L}{N_\x}$, $\Delta\y = \frac{\y_R-\x_L}{N_\y}$,
such that $(\cellindx,\cellindy)\in\ncells:=\{(\cellindx,\cellindy)\in\mathbb{N}^2~|~\cellindx\leq N_\x,\,\cellindy\leq N_\y\}$. 

Additionally, we discretize the random space $\randomSpace$  with the \ME{} ansatz from \secref{sec:multielement}, i.e. we divide $\randomSpace$ into $\cellR{\cellindR}$, for $\cellindR=1,\ldots,\MEElements$.
We test \eqref{eq:spde} in each spatial cell $\cell{\cellindx,\cellindy}$ by a test function $v(\x,\y)$ with supp$(v)\subseteq\cell{\cellindx,\cellindy}$ and obtain after one formal integration by parts the following weak formulation 
\begin{subequations}
\begin{align}\partial_\timevar\int\limits_{\cell{\cellindx,\cellindy}}\!\solution(\timevar,\x,\uncertainty) v(\x,\y) \mathrm{d}(\x,\y) &- \int\limits_{\cell{\cellindx,\cellindy}}\!\Big(\flux_1\big(\solution(\timevar,\x,\uncertainty),\uncertainty\big)\partial_\x v(\x,\y) + \flux_2\big(\solution(\timevar,\x,\uncertainty),\uncertainty\big)\partial_\y v(\x,\y)\Big)\, \mathrm{d}(\x,\y)\\
&+ \int\limits_{\y_{\cellindy-\frac{1}{2}}}^{\y_{\cellindy+\frac{1}{2}}}\Big[\flux_1\big(\solution(\timevar,\x,\uncertainty),\uncertainty\big) v(\x,\y) \, \Big]_{\x_{\cellind-\frac{1}{2}}}^{\x_{\cellind+\frac{1}{2}}}\mathrm{d}\y \label{eq:wf2}\\ &+ \int\limits_{\x_{\cellindy-\frac{1}{2}}}^{\x_{\cellindy+\frac{1}{2}}}\Big[\flux_2\big(\solution(\timevar,\x,\uncertainty),\uncertainty\big) v(\x,\y) \, \Big]_{\y_{\cellind-\frac{1}{2}}}^{\y_{\cellind+\frac{1}{2}}}\mathrm{d}\x  =\bm{0}, \label{eq:wf3}
\end{align}
\end{subequations}
for $(\cellindx,\cellindy)\in\ncells$.

Since the solution is discontinuous, we replace the flux at the interfaces $\x_{\cellind\pm\frac{1}{2}}$ with a consistent numerical flux function $\hat{\flux}_\fluxidx$, such that $\hat{\flux}_\fluxidx(\solution,\solution,\uncertainty)=\flux_\fluxidx(\solution,\uncertainty)$ for $\fluxidx=1,2$.
If we use a first order numerical scheme, we choose $v=\frac{1}{\Delta\x\Delta\y}$ and use the midpoint rule to approximate the integrals in \eqref{eq:wf2} and \eqref{eq:wf3}. We then obtain for one time step of the forward Euler scheme 
\begin{subequations}\label{eq:FVschemex}
\begin{align} 
\solutioncm{n+1}{\cellind,\cellindy} = \solutioncm{n}{\cellind,\cellindy} 
&- \frac{\Delta \timevar}{\Delta \x}\Big(\hat{\flux}_1(\solutioncm{n}{\cellind,\cellindy},\solutioncm{n}{\cellind+1,\cellindy},\uncertainty)-\hat{\flux}_1(\solutioncm{n}{\cellind-1,\cellindy},\solutioncm{n}{\cellind,\cellindy},\uncertainty) \Big)\\
&-\frac{\Delta \timevar}{\Delta\y }\Big(\hat{\flux}_2(\solutioncm{n}{\cellind,\cellindy},\solutioncm{n}{\cellind,\cellindy+1},\uncertainty)-\hat{\flux}_2(\solutioncm{n}{\cellind,\cellindy-1},\solutioncm{n}{\cellind,\cellindy},\uncertainty) \Big),
\end{align}
\end{subequations}
where $\solutioncm{n}{\cellind,\cellindy}$, $(\cellindx,\cellindy)\in\ncells$ denotes the spatial cell mean of $\solution$ in $\cell{\cellind,\cellindy}$ at time $\timevar_\timeind$, namely
\begin{equation*}
\solutioncm{\timeind}{\cellind,\cellindy}=\solutioncm{\timeind}{\cellind,\cellindy}(\uncertainty):= \frac{1}{\Delta\x\Delta\y}\int\limits_{\cell{\cellind,\cellindy}}\!\solution(\timevar_\timeind,\x,\y,\uncertainty)\,\mathrm{d}(\x,\y).
\end{equation*}
In our numerical experiments we use the numerical Lax-Friedrichs flux 
\begin{align*}\hat{\flux}_1(\solutionbar_{\cellind,\cellindy},\solutionbar_{\cellind+1,\cellindy},\uncertainty)&=\frac{1}{2}\big(\flux_1(\solutionbar_{\cellind,\cellindy},\uncertainty) + \flux_1(\solutionbar_{\cellind+1,\cellindy},\uncertainty) -  \lambda_1^{\max}(\solutionbar_{\cellind+1,\cellindy} - \solutionbar_{\cellind,\cellindy})\big), \\
\hat{\flux}_2(\solutionbar_{\cellind,\cellindy},\solutionbar_{\cellind,\cellindy+1},\uncertainty)&=\frac{1}{2}\big(\flux_2(\solutionbar_{\cellind,\cellindy},\uncertainty) + \flux_2(\solutionbar_{\cellind,\cellindy+1},\uncertainty) -  \lambda_2^{\max}(\solutionbar_{\cellind,\cellindy+1} - \solutionbar_{\cellind,\cellindy})\big),\end{align*}
where the numerical viscosity constants $\lambda_1^{\max}$ and $\lambda_2^{\max}$ are taken as the global estimates of the absolute value of the largest eigenvalue of $\frac{\partial \flux_1(\solution)}{\partial \solution}$ and $\frac{\partial \flux_2(\solution)}{\partial \solution}$, respectively.
The CFL condition of \eqref{eq:FVschemex} reads
$$\Delta\timevar\leq\frac{\Delta\x}{\lambda_1^{\max}} + \frac{\Delta\y}{\lambda_2^{\max}}.$$

For the coefficients of the gPC polynomial \eqref{def:globalSG} in the multi-element stochastic Galerkin system \eqref{SGsystemME}, the finite volume scheme \eqref{eq:FVschemex} is given by
\begin{subequations}\label{LaxFriedrichs}
\begin{align}\solutioncoeff{\SGsumIndex}{\cellind,\cellindy,\cellindR}{\timeind+1}
= \solutioncoeff{\SGsumIndex}{\cellind,\cellindy,\cellindR}{\timeind}
&- \frac{\dtstepsize}{\Delta\x} \MEintcellR{\cellindR} {\Big(\hat{\flux}_1(\solutioncm{n}{\cellind,\cellindy},\solutioncm{n}{\cellind+1,\cellindy},\uncertainty)-\hat{\flux}_1(\solutioncm{n}{\cellind-1,\cellindy},\solutioncm{n}{\cellind,\cellindy},\uncertainty)\Big)
\xiBasisPoly{\SGsumIndex,\cellindR}}\\
&-\frac{\dtstepsize}{\Delta\y} \MEintcellR{\cellindR} {\Big(\hat{\flux}_2(\solutioncm{n}{\cellind,\cellindy},\solutioncm{n}{\cellind,\cellindy+1},\uncertainty)-\hat{\flux}_2(\solutioncm{n}{\cellind,\cellindy-1},\solutioncm{n}{\cellind,\cellindy},\uncertainty)\Big)
\xiBasisPoly{\SGsumIndex,\cellindR}},
\end{align}
\end{subequations}
for each equation $\SGsumIndex=0,\ldots,\SGtruncorder$, where $\solutioncoeff{\SGsumIndex}{\cellind,\cellindy,\cellindR}{\timeind}$ describes the $\cell{\cellind,\cellindy}\times\cellR{\cellindR}$ cell mean of the $\SGsumIndex$th coefficient in the gPC expansion at time $\timevar_\timeind$, namely
$$ \solutioncm{n}{\SGsumIndex,\cellind,\cellindy,\cellindR}:= \frac{1}{\Delta\x\Delta\y}\int\limits_{\cell{\cellind,\cellindy}}\!\MEintcellRu{\cellindR}{\solution(\timevar_\timeind,\x,\uncertainty)\,\xiBasisPoly{\SGsumIndex,\cellindR}(\uncertainty)}\,\mathrm{d}(\x,\y),$$
for  $\SGsumIndex=0,\ldots,\SGtruncorder$, $(\cellind,\cellindy)\in\ncells$ and $\MEIndex=1,\ldots,\MEElements$.

\begin{remark}\label{rem:kinFlux}
The integrals in \eqref{LaxFriedrichs} are numerically solved using a $\nbxiQuadNodes$-point appropriate quadrature rule on $\randomSpace$. In our numerical examples we pick Gauß-Legendre or Clenshaw-Curtis quadrature rules\comment{, where Clenshaw-Curtis nodes are chosen whenever numerical computations make use of adaptivity, cf. Remark~\ref{rem:acceleration}}. Note that computing moments of an underlying flux is a common numerical flux choice in kinetic theory \cite{deshpande1986kinetic,harten1983upstream,perthame1990boltzmann,perthame1992second} and is frequently used in the context of uncertainty quantification, see e.g. \cite{Schlachter2018,Kusch2019}. A comparison to numerical fluxes working directly on the moments can be found in \cite{debusschere2004numerical} or \cite[Appendix~A]{kusch2020intrusive}.
\end{remark}

\subsection{Hyperbolicity-Preserving Limiter}\label{sec:hypPresLimiter}
Usually, the solution of our system of equations \eqref{eq:spde} has to fulfill certain physical properties. For example, a density should always be nonnegative. For hierarchical models like the Euler equations, the system normally loses hyperbolicity for nonphysical states, cf. \cite{Poette2009}.

\begin{definition}\label{def:hypset}
	We call the set 
	\begin{align*}
	\realizableSet := \left\{\solution\in\R^\dimension~\big|~ \sum_{\fluxidx=1}^\nbflux \alpha_\fluxidx\frac{\partial\flux_\fluxidx(\solution,\uncertainty)}{\partial\solution} \text{ has $\dimension$ real eigenvalues and is diagonalizable } \forall \alpha_1,\ldots,\alpha_\nbflux\in\R,~\forall\uncertainty\in\randomSpace \right\}
	\end{align*}
	the \textbf{\hypset{}}. We call every solution vector $\solution\in\realizableSet$ \textbf{\hyperbolic{}}.
\end{definition}

\begin{assumption}
	\label{ass:RealizableSet}
	In the following we always assume that the \hypset{} $\realizableSet$ is open and convex.
\end{assumption}

\begin{theorem}
If the SG gPC polynomial \eqref{SGapproach} is admissible, then the SG system \eqref{SGsystemME} is hyperbolic.
\end{theorem}
\begin{proof}
The proof can be found in \cite[Theorem 2.1]{Wu2017}.
\end{proof}
This means that we have to ensure that the gPC polynomial does not not leave the hyperbolicity set in order to preserve hyperbolicity of the SG system.



In order to apply the hyperbolicity-limiter from \cite{Schlachter2018}, we require positivity-preserving numerical fluxes $\hat{\flux}_1$, $\hat{\flux}_2$.

\begin{definition}\label{def:pospreserving}
		A scheme \eqref{eq:FVschemex} and the numerical fluxes $\hat{\flux}_1$, $\hat{\flux}_2$ are called \textbf{positivity-preserving}, if $\solutioncm{\timeind}{\cellind,\cellindy} \in \realizableSet$		for all $(\cellind,\cellindy)\in\ncells$ at time $\timevar_\timeind$ implies	that $\solutioncm{\timeind+1}{\cellind,\cellindy} \in \realizableSet$ for all $(\cellind,\cellindy)\in\ncells$ at time $\timevar_{\timeind+1}$. 
\end{definition}

\begin{assumption} \label{ass:ppFlux}
	The numerical fluxes $\hat{\flux}_1$, $\hat{\flux}_2$, are positivity-preserving under a suitable CFL condition
	$$\lambda_1^{\max} \frac{\Delta \timevar }{\Delta \x}+\lambda_2^{\max} \frac{\Delta \timevar }{\Delta \y} \leq C,$$
	where $C \in (0,1]$ is constant.
\end{assumption}

\begin{remark}
	Throughout this work we use the HLL numerical flux \cite{harten1983upstream}. The positivity-preserving property of HLL has been proven in \cite{munz1994godunov}.
\end{remark}



According to \cite{Schlachter2018}, the numerical scheme \eqref{LaxFriedrichs} preserves hyperbolicity of the zeroth SG moment $\solutioncoeff{0}{\cellind,\cellindy,\cellindR}{\timeind}$ for positivity-preserving numerical fluxes and we can now define the hyperbolicity limiter which moves a possibly inadmissible solution towards this admissible moment and inside the hyperbolicity set.
Following the outline of \cite{Schlachter2018}, at time $\timevar_\timeind$ we define the slope-limited SG polynomial in the cell $\cell{\cellind,\cellindy}\times\cellR{\cellindR}$ as
\begin{align}
\hyperbolLimitm{\limitervariable}(\solutioncm{n}{\cellind,\cellindy})\big|_{\cellR{\cellindR}}(\uncertainty) &:= \limitervariable \,\solutioncoeff{0}{\cellind,\cellindy,\cellindR}{\timeind} + (1-\limitervariable) \sum_{\SGsumIndex=0}^\SGtruncorder \solutioncoeff{\SGsumIndex}{\cellind,\cellindy,\cellindR}{\timeind} \xiBasisPoly{\SGsumIndex,\cellindR}(\uncertainty) \nonumber \\
&= \solutioncoeff{0}{\cellind,\cellindy,\cellindR}{\timeind} + (1-\limitervariable) \sum_{\SGsumIndex=1}^\SGtruncorder \solutioncoeff{\SGsumIndex}{\cellind,\cellindy,\cellindR}{\timeind} \xiBasisPoly{\SGsumIndex,\cellindR}(\uncertainty), \label{eq:hyplimitermoments}
\end{align}
for $(\cellind,\cellindy)\in\ncells$ and $\cellindR=1,\ldots,\MEElements$. For more information see \cite{Schlachter2018}.
\comment{To ensure admissibility of the slope-limited SG polynomial \eqref{eq:hyplimitermoments} on every quadrature point $\widehat{\uncertainty}_{\ell}$ (cf. Remark~\ref{rem:kinFlux}), we} choose
$$\hat{\limitervariable}^{(n)}_{\cellind,\cellindy,\cellindR} := \inf\left\{ \tilde{\limitervariable}\in [0,1] ~\big|~ \hyperbolLimitm{\tilde{\limitervariable}}(\solutioncm{n}{\cellind,\cellindy})\big|_{\cellR{\cellindR}}  \in \realizableSet~~ \text{for all quadrature nodes \comment{$\widehat{\uncertainty}_{\ell}$}}\right\}.$$
Due to the openness of $\realizableSet$, we need to modify $\limitervariable$ slightly in order to avoid placing the solution onto the boundary (if the limiter was active). Therefore we use 
\begin{align*}
\limitervariable = \begin{cases}
\hat{\limitervariable}, & \text{ if } \hat{\limitervariable} = 0,\\
\min(\hat{\limitervariable}+\limitertolerance,1), & \text{ if } \hat{\limitervariable} > 0,
\end{cases}
\end{align*}
where $0<\limitertolerance = 10^{-10}$ should be chosen small enough to ensure that the approximation quality is not influenced significantly.

\begin{remark}
	Calculating the slope-limited SG polynomial \eqref{eq:hyplimitermoments}, we can simply replace the SG moments by
	$$\hyperbolLimitm{\limitervariable}(\solutioncoeff{\SGsumIndex}{\cellind,\cellindy,\cellindR}{\timeind} ) = \begin{cases}
	\solutioncoeff{0}{\cellind,\cellindy,\cellindR}{\timeind} , \qquad & \text{if } \SGsumIndex=0,\\
	(1-\limitervariable) \solutioncoeff{\SGsumIndex}{\cellind,\cellindy,\cellindR}{\timeind} , & \text{if } \SGsumIndex>0,
	\end{cases}\qquad \SGsumIndex=0,\ldots,\SGtruncorder,$$
	where $\limitervariable$ is chosen separately for each $\timevar_\timeind$ and $\cell{\cellind,\cellindy}\times\cellR{\cellindR}$, ensuring that $\hyperbolLimitm{\limitervariable}(\solutioncm{\timeind}{\cellind,\cellindy} )\big|_{\cellR{\cellindR}}\in\realizableSet$ in all cells.
\end{remark}

\begin{remark}
We may calculate the the value of the limiter variable $\limitervariable$ directly by solving the minimization problem 
\begin{subequations}\label{eq:thetamin}
\begin{align}
	\min~ &\limitervariable\\
	\text{s.t.}~ &\limitervariable\in[0,1]\\
	&\limitervariable\widetilde{\solution} + (1-\limitervariable)\solution\in\realizableSet,
\end{align}
\end{subequations}
where $\widetilde{\solution}$ represents the admissible zeroth SG moment. 
An exemplary derivation of this parameter for the two-dimensional compressible Euler equations can be found in \cite{Meyer2019a}.
\end{remark}

\begin{example}
We consider the one-dimensional compressible Euler equations
\begin{align}\label{eq:euler}
\partial_t
\begin{pmatrix}
\rho \\ \rho v \\ \rho e
\end{pmatrix}
+\partial_x
\begin{pmatrix}
\rho v \\ \rho v^2 +p \\ v(\rho e+p)
\end{pmatrix}
=\bm{0}.
\end{align}
These equations describe the evolution of a gas with density $\rho$, velocity $v$ and specific total energy $e$. The pressure of the gas can be determined from
\begin{equation*}
p = (\gamma-1)\rho\left(e-\frac{1}{2}v^2\right)\;
\end{equation*}
and $\gamma$ is the heat capacity ratio.
The hyperbolicity set from \defref{def:hypset} is given by solutions with positive density and pressure, i.e.
$$	\realizableSet = \left\{ \solution = \begin{pmatrix}
\rho \\ \rho v \\ \rho e
\end{pmatrix}\,\Bigg|~\density>0,~\pressure = (\gamma-1)\rho\left(e-\frac{1}{2}v^2\right) >0\right\}.$$
Then, the solution of the minimization problem \eqref{eq:thetamin} reads
\begin{align*}
\limitervariable^\star &= \max\left(\cutfun{\limitervariable_1},\cutfun{\limitervariable_{2+}},\cutfun{\limitervariable_{2-}}\right),\label{eq:thetastartwo}\\[.15cm]
\limitervariable_1 &= \frac{\density}{\density-\densityvar},\nonumber\\[.15cm]	\limitervariable_{2\pm} &=\frac{\energy\, \densityvar - 2\,\energy\, \density + \energyvar\, \density - \momentum\, \momentumvar + (\momentum)^2}	{\big(\momentum-\momentumvar\big)^2 +2\big(\energy\, \densityvar + \energyvar\, \density - \energy\, \density  -  \energyvar\, \densityvar\big)}\nonumber\\	&\pm\frac{\sqrt{\big(\energy \densityvar - \energyvar \density  \big)^2 + 2 \big(\energy\, (\momentumvar)^2\, \density +  \energyvar\, (\momentum)^2\, \densityvar - \energy\, \momentum\, \momentumvar\, \densityvar  - \energyvar\, \momentum\, \momentumvar\, \density\big)}}	{\big(\momentum-\momentumvar\big)^2 +2\big(\energy\, \densityvar + \energyvar\, \density - \energy\, \density  -  \energyvar\, \densityvar\big)},\nonumber\\[.15cm]	\cutfun{x} &= x\indicator{[0,1]}(x).\nonumber	\end{align*}
The calculation of the above quantities follows similarly to \cite{Meyer2019a}.
\end{example}

\subsection{Filtered stochastic Galerkin}\label{sec:fSG}
In addition to the hyperbolicity-preserving limiter from \secref{sec:hypPresLimiter}, we apply a filter to the stochastic Galerkin scheme, as introduced in \cite{Kusch2018}. Filters are a common technique in spectral methods \cite{boyd2001chebyshev,hesthaven2007spectral} and have for example been applied in the context of kinetic theory \cite{McClarren2010} in order to dampen oscillations. We will similarly apply them to the gPC coefficients in \eqref{SGapproach} to reduce oscillations in the uncertainty, while maintaining high-order accuracy in deterministic regions.

In the following we present the so-called $L_2$ filter, adding a penalizing term to the $L_2$ error of the solution in order to reduce oscillations within high-order coefficients of the gPC expansion. We follow the outline from \cite{Kusch2018} and include it into the framework of the \ME{} approach. Note that we perform the derivation for the $L_2$ filter to provide an understanding of how filters can be applied to the \ME{} ansatz. This understanding is then used to state the more sophisticated exponential filter \eqref{eq:expfilter} for \ME{} basis functions, which is the filter used throughout the numerical experiments of this work.
The standard SG method chooses the coefficients in the gPC expansion such that the cost functional
$$\costfun := \frac{1}{2}\sum\limits_{\cellindR=1}^\MEElements\,\MEintcellR{\cellindR}{\big\|\solution_\cellindR-\SGsolution_\cellindR\big\|_2^2}$$ 
is minimal with respect to the \ME{} gPC polynomial $\SGsolution$ of degree $\SGtruncorder$, where $\|\cdot\|_2$ denotes the $L_2$ norm and $\solution=\solution(\timevar,\x,\y,\uncertainty)$ with $\uncertainty\in\cellR{\cellindR}$ is the solution to the uncertain conservation law in the $\cellindR$th \ME{} for $\cellindR=1,\ldots,\MEElements$. The minimizer is given by \eqref{def:globalSG}. 

The filtered stochastic Galerkin polynomial is now set to
\begin{equation}\label{eq:filterSGsolution}\filterSGsolution(\timevar,\x,\y,\uncertainty) := \sum_{\cellindR=1}^\MEElements\sum_{\SGsumIndex=0}^{\SGtruncorder} \filter(\SGsumIndex)\solution_{\SGsumIndex,\cellindR}(\timevar,\x,\y)\xiBasisPoly{\SGsumIndex,\cellindR}(\uncertainty)\indicatorVar{\cellindR}{\uncertainty},\end{equation}
with the filter function $\filter$. For the $L_2$ filter, a filter function and therefore a filtered stochastic Galerkin solution, is obtained by minimizing
\begin{equation}\label{eq:filtercostfun}\costfun_\filterstrength := \frac{1}{2}\sum\limits_{\cellindR=1}^\MEElements\, \MEintcellR{\cellindR}{\big\|\solution_\cellindR - \filterSGsolution_\cellindR\big\|_2^2} + \filterstrength\sum\limits_{\cellindR=1}^\MEElements\,\MEintcellR{\cellindR}{\big\|\filteroperator(\filterSGsolution_\cellindR)\big\|_2^2},\end{equation}
for the filter strength $\filterstrength\geq0$. The operator $\filteroperator$ punishes oscillations and is chosen such that the basis polynomials of the gPC expansion are eigenfunctions of $\filteroperator$. Differentiating \eqref{eq:filtercostfun} with respect to the coefficients $\filterSGsolution_{\SGsumIndex,\cellindR}$ and setting the result equal to zero yields the filter function $\filter$.

\begin{remark}
	For a uniform distribution of $\uncertainty$, a common choice of the operator $\filteroperator$ is 
	$$\filteroperator(\solution(\uncertainty))=((1-\uncertainty^2)\solution'(\uncertainty))',$$
	since the Legendre polynomials are eigenfunctions of $\filteroperator$, i.e. they fulfill
	$$\filteroperator(\xiBasisPoly{\SGsumIndex,\cellindR}(\uncertainty)) = -\SGsumIndex(\SGsumIndex+1)\xiBasisPoly{\SGsumIndex,\cellindR}(\uncertainty),$$
	for $\SGsumIndex=0,\ldots,\SGtruncorder$ and $\cellindR=1,\ldots,\MEElements$.
	Other choices of the filter function are Lanczos and ErfcLog \cite{Radice2013}.
	We differentiate the cost functional \eqref{eq:filtercostfun} with respect to $\filterSGsolution_{\SGsumIndex,\cellindR}$ and deduce 
	$$\mathrm{d}_{\filterSGsolution_{\SGsumIndex,\cellindR}}\costfun_\filterstrength = \MEintcellR{\cellindR}{ \bigg|\sum_{\SGsumIndexvar=0}^\SGtruncorder (\solution_{\SGsumIndexvar,\cellindR} - \filterSGsolution_{\SGsumIndexvar,\cellindR})\xiBasisPoly{\SGsumIndexvar,\cellindR} \bigg| \xiBasisPoly{\SGsumIndex,\cellindR}} - \MEintcellR{\cellindR}{ \bigg| \sum_{\SGsumIndexvar=0}^\SGtruncorder \filterstrength \SGsumIndex^2(\SGsumIndex+1)^2 \filterSGsolution_{\SGsumIndexvar,\cellindR} \xiBasisPoly{\SGsumIndexvar,\cellindR}\bigg| \xiBasisPoly{\SGsumIndex,\cellindR}}\overset{!}{=}\bm{0}.$$
	Hence, the filter function within the filtered SG polynomial \eqref{eq:filterSGsolution} for a uniformly distributed uncertainty reads
	\begin{equation}\label{eq:l2filter}\filter(\SGsumIndex) = \frac{1}{1+\filterstrength \SGsumIndex^2(\SGsumIndex+1)^2}.\end{equation}
	Thus, the zeroth SG moment stays unfiltered due to $\filter(0)=1$ which means that the cell means are preserved. For larger values of $\SGsumIndex$, namely higher SG moments, $\filter(\SGsumIndex)$ gets smaller and the moments are more damped. 
\end{remark}

\begin{remark}
	For multi-dimensional uncertainties $\boldsymbol{\uncertainty}\in\mathbb{R}^\nstochdim$, we use the filter operator 
	$$\filteroperator(\solution(\boldsymbol{\uncertainty})) = (\filteroperator_1\circ\cdots\circ\filteroperator_\nstochdim)(\solution(\boldsymbol{\uncertainty})),$$
	where
	$$\filteroperator_\nstochdimidx(\solution(\boldsymbol{\uncertainty})) = \partial_{\uncertaintydim{\nstochdimidx}} \big((1-(\uncertaintydim{\nstochdimidx})^2\big)\partial_{\uncertaintydim{\nstochdimidx}}\solution(\boldsymbol{\uncertainty}).$$
\end{remark}
Finding a good choice of the filter strength $\filterstrength$, which sufficiently damps oscillations while preserving the solution structure, is challenging since the optimal filter strength is problem dependent and a parameter study has to be performed. 

Within the derivation of the $L_2$ filter, the use of the filter function has been motivated by an optimization problem of the form \eqref{eq:filtercostfun}. Another approach is to replace the solution ansatz \eqref{eq:filterSGsolution} by
\begin{equation}\label{eq:filteredSGb}
\filterSGsolution(\timevar,\x,\y,\uncertainty) := \sum_{\cellindR=1}^\MEElements\sum_{\SGsumIndex=0}^{\SGtruncorder} \filter\left(\frac{\SGsumIndex}{\SGtruncorder}\right)^{\lambda\Delta t}\solution_{\SGsumIndex,\cellindR}(\timevar,\x,\y)\xiBasisPoly{\SGsumIndex,\cellindR}(\uncertainty)\indicatorVar{\cellindR}{\uncertainty},
\end{equation}
Here, the time step size $\Delta t$ is included in the filter formulation to enforce independence of the filtering effect from the chosen time discretization. I.e., when the number of time steps increases (i.e. the filter is applied more frequently), the filtering effect by a single filter step decreases. Furthermore, by adding the time step size to the filter strength, the filtering can be interpreted as a smoothing term with respect to the uncertainty. For a detailed discussion applied to the context of kinetic theory, see \cite{radice2013new}. Now, the function $\filter:\mathbb{R}_+ \rightarrow \mathbb{R}_+$ plays the role of a filter function and its choice crucially affects the resulting approximation quality. Let $\filter$ fulfill the following properties:
\begin{enumerate}
\item $\filter(0)=1$,
\item $\filter^{(\ell)}(0) = 0, \qquad \text{ for } \ell = 1,\ldots,\alpha-1$,
\item $\filter^{(\ell)}(1) = 0, \qquad \text{ for } \ell = 0,\ldots,\alpha-1$.
\end{enumerate}
According to \cite{boyd1996erfc}, when approximating a general function $v(s)$ via 
\begin{align}\label{eq:filterGeneral}
v(s) \approx \widehat{v}^{\SGtruncorder}(s) := \sum_{\SGsumIndex=0}^\SGtruncorder \filter\left(\frac{\SGsumIndex}{\SGtruncorder}\right) a_\SGsumIndex \phi_\SGsumIndex(s),
\end{align}
where $\phi_\SGsumIndex(s)$ are basis functions and $a_\SGsumIndex$ the corresponding coefficients of the expansion with $\SGsumIndex=0,\ldots,\SGtruncorder$, the filtered approximation \eqref{eq:filterGeneral} fulfills
\begin{enumerate}
\item If $v$ has $\alpha$ continuous derivatives, then
\begin{align*}
\big\vert v(s) - \widehat{v}^{\SGtruncorder}(s) \big\vert \leq M \cdot \SGtruncorder^{1/2 - \alpha}.
\end{align*}
\item If $v$ has a jump discontinuity at $c^*\in [-1,1]$, then
\begin{align*}
\big\vert v(s) - \widehat{v}^{\SGtruncorder}(s) \big\vert \leq M \cdot d(s)^{1-\alpha} \SGtruncorder^{1 - \alpha},
\end{align*}
where $d(s)$ is the distance to $c^*$ and $M>0$ an adequate constant.
\end{enumerate}
The first property implies that we can achieve spectral accuracy of the filter under a sufficiently high order $\alpha$.  The second inequality shows that even for discontinuous solutions, spectral accuracy is still possible if $s$ is not too close to the shock position. 
One choice for the filter function in \eqref{eq:filteredSGb} is the exponential filter proposed in \cite{hoskins1980representation} which reads
\begin{align}\label{eq:expfilter}
\filter\left(\frac{\SGsumIndex}{\SGtruncorder}\right) = \exp\left(c\left(\frac{\SGsumIndex}{\SGtruncorder}\right)^{\alpha}\right),
\end{align}
for $\SGsumIndex=0,\ldots,\SGtruncorder$.
Here, we use $c = \log(\varepsilon_M)$ where $\varepsilon_M$ is the machine accuracy. The parameter $\alpha$ is an integer called the filter order. A more detailed discussion of the exponential filter can for example be found in \cite[Chapter~18.20]{boyd2001chebyshev}.

\subsection{Algorithm}
We summarize the whole method for the filtered hyperbolicity-preserving stochastic Galerkin scheme in the following algorithm, whereas we denote the solution operator as the right hand side of \eqref{LaxFriedrichs} by 
\comment{\begin{align}L_h^{(\timeind)}(\solutioncoeffFilter{\SGsumIndex}{\cellind,\cellindy,\cellindR}{\timeind}):= \solutioncoeffFilter{\SGsumIndex}{\cellind,\cellindy,\cellindR}{\timeind}
&- \frac{\dtstepsize}{\Delta\x} \MEintcellR{\cellindR} {\Big(\hat{\flux}_1(\solutioncmFilter{n}{\cellind,\cellindy},\solutioncmFilter{n}{\cellind+1,\cellindy},\uncertainty)-\hat{\flux}_1(\solutioncmFilter{n}{\cellind-1,\cellindy},\solutioncmFilter{n}{\cellind,\cellindy},\uncertainty)\Big)
\xiBasisPoly{\SGsumIndex,\cellindR}} \nonumber\\
&-\frac{\dtstepsize}{\Delta\y} \MEintcellR{\cellindR} {\Big(\hat{\flux}_2(\solutioncmFilter{n}{\cellind,\cellindy},\solutioncmFilter{n}{\cellind,\cellindy+1},\uncertainty)-\hat{\flux}_2(\solutioncmFilter{n}{\cellind,\cellindy-1},\solutioncmFilter{n}{\cellind,\cellindy},\uncertainty)\Big)
\xiBasisPoly{\SGsumIndex,\cellindR}}, \label{eq:Lhfilter}
\end{align}}
for $\SGsumIndex=0,\ldots,\SGtruncorder$, $(\cellindx,\cellindy)\in\ncells$ and $\cellindR=1,\ldots,\MEElements$. \comment{The method in pseudo-code can be found in Algorithm~\ref{algo:fhSG}.}

\begin{algorithm}[H] 
	\caption{Filtered Hyperbolicity-Preserving Stochastic Galerkin Scheme}
	\label{algo:fhSG}
	\begin{algorithmic}[1]
		\State $\overline{\solution}^{(0)}\leftarrow \Big(\int_{\cell{\cellind,\cellindy}}\!\int_{\cellR{\cellindR}}\solution^{(0)}\,\xiBasisPoly{\SGsumIndex,\cellindR}\locxiPDF{\cellindR} \mathrm{d}\uncertainty\mathrm{d}(\x,\y)\Big)_{(\cellind,\cellindy)\in\ncells,\cellindR=1:\MEElements,\SGsumIndex=0:\SGtruncorder}$\hfill \textit{\# initial state}\smallskip
		\For{$\timeind=0,\ldots,N_\tend$} \hfill \textit{\# time loop}\smallskip
		\State Choose the filter strength $\filterstrength$ and filter order $\alpha$\smallskip
		\State $\widehat{\overline{\solution}}^{(\timeind)}\leftarrow  \Big(\filter(\cdot)\,\solutioncoeff{\SGsumIndex}{\cellind,\cellindy,\cellindR}{\timeind}\Big)_{(\cellind,\cellindy)\in\ncells,\cellindR=1:\MEElements,\SGsumIndex=0:\SGtruncorder}$\hfill \textit{\# apply filter function \eqref{eq:l2filter} or \eqref{eq:expfilter}}\smallskip
		\State $\widehat{\overline{\solution}}^{(\timeind)}\leftarrow  \hyperbolLimitm{\limitervariable}\big(\widehat{\overline{\solution}}^{(\timeind)}\big)$ \hfill \textit{\# call of hyperbolicity limiter \eqref{eq:hyplimitermoments}}\medskip
		\State $\comment{\overline{\solution}}^{(\timeind+1)}\!\leftarrow
		L_h^{(\timeind)}(\widehat{\overline{\solution}}^{(\timeind)})$\hfill \textit{\# FVM \eqref{LaxFriedrichs}}\smallskip
		\EndFor
		
	\end{algorithmic}
\end{algorithm}
In our numerical examples we will apply the exponential filter \eqref{eq:expfilter}.


\section{Multi-element Intrusive Polynomial Moment Method}\label{sec:four}
The \ME{} approach from \secref{sec:multielement} is known to be applied to the stochastic Galerkin scheme since a local basis in the random variable yields improved solution approximations, even when using small polynomial degrees in each local element. Therefore, it is straightforward to adapt the strategy to the IPM method.
In addition to that, such a local basis seems to be an ideal choice for IPM, since it decouples the IPM optimization problems, enhancing the use of parallel implementations while heavily reducing computational costs. 

Given a strictly convex entropy $\entropy:\mathbb{R}^\dimension\rightarrow\mathbb{R}$ of the uncertain system of hyperbolic conservation laws \eqref{eq:spde}, the IPM system from \cite{Poette2009} is given by
\begin{equation} \label{IPMMsystem}
\dt
\intRS \solution \,\xiBasisPoly{\SGsumIndexvar} \xiPDFdxi
+ \nabla_\x\cdot
\intRS \fluxvec \bigg(\!\!\InvGradEntropy\Big(\sum_{\SGsumIndex=0}^\SGtruncorder \entropicVar_\SGsumIndex \xiBasisPoly{\SGsumIndex}\Big),\uncertainty\bigg)\xiBasisPoly{\SGsumIndexvar} \xiPDFdxi
= \bm{0},
\end{equation}
for $\SGsumIndexvar=0,\ldots,\SGtruncorder$ and where 
$$\entropicVarSum :=  \sum_{\SGsumIndex=0}^\SGtruncorder \entropicVar_\SGsumIndex \xiBasisPoly{\SGsumIndex}$$
is the entropic variable of the system, defining the solution $\solution$ of \eqref{eq:spde} via the one-to-one map $\solution = \InvGradEntropy\!(\entropicVarSum)$. Performing the gPC expansion on the entropic variable guarantees hyperbolicity of the IPM system \eqref{IPMMsystem}. For more information and the derivation of the system \eqref{IPMMsystem} we refer to \cite{Poette2009}. The hyperbolicity of the IPM system is for example shown in \cite[Chapter~4.2.3]{poette2019contribution}.

We now follow the theory of \cite{Poette2009} in order to obtain the IPM system \eqref{IPMMsystem} for the \ME{} approach described in \secref{sec:multielement}. Thus, we divide the random domain $\randomSpace$ into \ME{}s $\cellR{\cellindR}$ with $\cellindR = 1,\ldots, ,\MEElements$ and define a local basis $\xiBasisPoly{\SGsumIndex,\cellindR}$ according to \eqref{eq:orthogonality}. 

We consider the ME system \eqref{SGsystemME} and derive the corresponding IPM closure, i.e. we minimize the functional
\begin{align*}
\lagrange(\solution,\entropicVar) =  \intRS \entropy(\solution) \xiPDFdxi + \sum_{\cellindR = 1}^{\MEElements}\sum_{\SGsumIndex=0}^\SGtruncorder \entropicVar_{\SGsumIndex,\cellindR} \Big(\solution_{\SGsumIndex,\cellindR} - \MEintcellR{\cellindR}{\solution\, \xiBasisPoly{\SGsumIndex,\cellindR}} \Big),
\end{align*}
where we now have Lagrange multipliers (also called dual variables) $\entropicVar_{0,\cellindR},\ldots,\entropicVar_{\SGtruncorder,\cellindR} \in\mathbb{R}^\dimension$ for every \ME{} $\cellR{\cellindR}$, $\cellindR=1,\ldots,\MEElements$. By $\entropicVar$ we denote these multipliers in vectorized form.
In order to compute the exact minimizer of this functional, we calculate the Gâteaux derivative with respect to $\solution$ in an arbitrary direction $\boldsymbol{v}$
\begin{align*}
\mathrm{d}_\solution \lagrange(\solution,\entropicVar)[\boldsymbol{v}] := \frac{d}{dt} \lagrange(\solution+t\cdot \boldsymbol{v},\entropicVar)\large\vert_{t=0}.
\end{align*}
We obtain
\begin{align*}
\mathrm{d}_\solution \lagrange(\solution,\entropicVar)[\boldsymbol{v}] = \intRS \DiffEntropy(\solution) \cdot \boldsymbol{v}\xiPDFdxi - \sum_{\cellindR = 1}^{\MEElements}\sum_{\SGsumIndex=0}^{\SGtruncorder}\,\MEintcellR{\cellindR}{ \entropicVar_{\SGsumIndex,\cellindR} \xiBasisPoly{\SGsumIndex,\cellindR} \cdot \boldsymbol{v}}, 
\end{align*}
which should be zero for any direction $\boldsymbol{v}$, yielding
$$\DiffEntropy(\solution) = \sum_{\cellindR = 1}^{\MEElements}\sum_{\SGsumIndex=0}^{\SGtruncorder} \entropicVar_{\SGsumIndex,\cellindR} \xiBasisPoly{\SGsumIndex,\cellindR}.$$
If we now define the entropic variable as $$\entropicVarSum=\sum_{\cellindR = 1}^{\MEElements}\sum_{\SGsumIndex=0}^{\SGtruncorder} \entropicVar_{\SGsumIndex,\cellindR} \xiBasisPoly{\SGsumIndex,\cellindR},$$ we again have $\solution = \InvGradEntropy\!(\entropicVarSum)$. Plugging this ansatz into $\lagrange$ and differentiating with respect to $\entropicVar_{\SGeqIndex,\cellindRvar}$ for $\SGeqIndex=0,\ldots,\SGtruncorder$ and $\cellindRvar=1,\ldots,\MEElements$ gives
\begin{align}\label{eq:MEIPMLagrange}
\mathrm{d}_{\entropicVar_{\SGeqIndex,\cellindRvar}} \lagrange(\solution,\entropicVar) &= \MEintcellR{\cellindRvar}{ \DiffEntropy(\solution(\entropicVarSum))\DiffSol(\entropicVarSum)\xiBasisPoly{\SGeqIndex,\cellindRvar}}+ \solution_{\SGeqIndex,\cellindRvar} - \MEintcellR{\cellindRvar}{ \solution(\entropicVarSum) \xiBasisPoly{\SGeqIndex,\cellindRvar} }- \sum_{\cellindR = 1}^{\MEElements}\sum_{\SGsumIndex=0}^{\SGtruncorder} \entropicVar_{\SGsumIndex,\cellindR} \MEintcellR{\cellindRvar}{\DiffSol(\entropicVarSum) \xiBasisPoly{\SGsumIndex,\cellindR} \xiBasisPoly{\SGsumIndexvar,\cellindRvar}} \nonumber \\
&= \MEintcellR{\cellindRvar}{ \entropicVarSum \DiffSol(\entropicVarSum)\xiBasisPoly{\SGeqIndex,\cellindRvar}} +  \solution_{\SGeqIndex,\cellindRvar} -\MEintcellR{\cellindRvar}{ \solution(\entropicVarSum) \xiBasisPoly{\SGeqIndex,\cellindRvar}}-  \MEintcellR{\cellindRvar}{ \DiffSol(\entropicVarSum) \entropicVarSum \xiBasisPoly{\SGeqIndex,\cellindRvar} } \nonumber\\
&= \solution_{\SGeqIndex,\cellindRvar} -\MEintcellR{\cellindRvar}{ \solution(\entropicVarSum) \xiBasisPoly{\SGeqIndex,\cellindRvar}} \stackrel{!}{=} \bm{0}.
\end{align}
In order to minimize the Lagrangian we need to find the root of this equation.
Note that since $\xiBasisPoly{\SGsumIndex,\cellindR}$ has support $\cellR{\cellindR}$, the integral term becomes
\begin{align*}
\MEintcellR{\cellindRvar}{\solution(\entropicVarSum) \xiBasisPoly{\SGsumIndexvar,\cellindRvar}} &= \MEintcellR{\cellindRvar}{ \solution\left(\sum_{\cellindR = 1}^{\MEElements}\sum_{\SGsumIndex=0}^{\SGtruncorder}\entropicVar_{\SGsumIndex,\cellindR}\,\xiBasisPoly{\SGsumIndex,\cellindR}\right) \xiBasisPoly{\SGsumIndexvar,\cellindRvar}}\\
&= \MEintcellR{\cellindRvar}{\solution\left(\sum_{\SGsumIndex=1}^{\SGtruncorder}\entropicVar_{\SGsumIndex,\cellindRvar}\,\xiBasisPoly{\SGsumIndex,\cellindRvar}\right) \xiBasisPoly{\SGsumIndexvar,\cellindRvar}}.
\end{align*}
Hence, we can solve \eqref{eq:MEIPMLagrange} through
\begin{align}\label{eq:ulambda}
 \solution_{\SGsumIndexvar,\cellindRvar} - \MEintcellR{\cellindRvar}{ \InvGradEntropy\left(\sum_{\SGsumIndex=0}^{\SGtruncorder}\entropicVar_{\SGsumIndex,\cellindRvar}\,\xiBasisPoly{\SGsumIndex,\cellindRvar} \right) \xiBasisPoly{\SGsumIndexvar,\cellindRvar}} = \bm{0}, 
\end{align}
for  $\SGeqIndex=0,\ldots,\SGtruncorder$ and $\cellindRvar=1,\ldots,\MEElements$.
In order to find the dual variables $\entropicVar$ as the roots of \eqref{eq:ulambda} we employ Newton's method. 
Moreover, the ME-IPM system reads
\begin{equation}\label{MEIPMsystem}
    \dt
\MEintcellR{\cellindR}{ \solution \,\xiBasisPoly{\SGsumIndexvar,\cellindR} }
+ \nabla_\x\cdot
\MEintcellR{\cellindR}{\fluxvec \big(\entropicVarSum,\uncertainty\big)\,\xiBasisPoly{\SGsumIndexvar,\cellindR}}
= \bm{0},
\end{equation}
with $\SGsumIndexvar=0,\ldots,\SGtruncorder$ and $\cellindR=1,\ldots,\MEElements$. Let us first state that using the \ME{} approach for IPM does not destroy hyperbolicity of the moment system
\begin{theorem}
The \ME{} IPM system \eqref{MEIPMsystem} is hyperbolic.
\end{theorem}
\begin{proof}
The proof is a trivial extension of the hyperbolicity proof for the classical IPM method, which can be found in \cite[Chapter~4.2.3]{poette2019contribution}. As a starting point, we collect the dual variables into the vector $\entropicVar_{\cellindR} = (\entropicVar_{\cellindR,0}^T,\cdots,\entropicVar_{\cellindR,\comment{\SGtruncorder}}^T)\in\mathbb{R}^{d\cdot (\comment{\SGtruncorder}+1)}$. Then the system \eqref{MEIPMsystem} is brought into quasi-conservative form by applying the chain rule, leading to
\begin{align*}
    \bm{H}_{\cellindR}(\entropicVar_{\cellindR})\partial_t \entropicVar_{\cellindR} + \bm{B}_{\cellindR}(\entropicVar_{\cellindR})\partial_x \entropicVar_{\cellindR} = \bm{0}
\end{align*}
for every \ME{} $\cellindR=1,\ldots,\MEElements$. Following \cite[Chapter~4.2.3]{poette2019contribution}, the temporal Jacobians $\bm{H}_{\cellindR}\in\mathbb{R}^{d\cdot (\comment{\SGtruncorder}+1)\times d\cdot (\comment{\SGtruncorder}+1)}$ are symmetric positive definite and the spatial Jacobians $\bm{B}_{\cellindR}$ are symmetric. Hence, the \ME{} IPM equations can be written as a Friedrichs system, which according to \cite{friedrichs1954symmetric} implies hyperbolicity.
\end{proof}
Furthermore, note that the IPM ansatz (at least for the cases considered in this work) only allows the solution to take admissible values. Hence the solution always lies in the admissible set. Besides hyperbolicity as well as the potential to mitigate oscillations, the \ME{} IPM method has a remarkable computational advantage:
\begin{remark}
The IPM optimization problems decouple and one can solve the $\MEElements$ problems per spatial cell individually. Commonly, the multi-element approach allows using a smaller total degree of polynomials. Therefore, instead of solving one expensive optimization problem, we now solve $\MEElements$ cheaper optimization problems in each cell, which can be distributed to different processors.
\end{remark}

We apply the numerical discretization from \secref{sec:numerics} yielding the following algorithm of the \ME{} intrusive polynomial moment method, whereas we now replace the solution operator \eqref{eq:Lhfilter} by 
\begin{align}
    L_h^{(\timeind)}&\big(\solutioncoeff{\SGsumIndex}{\cellind,\cellindy,\cellindR}{\timeind},\entrcm{\timeind}{\cellind,\cellindy,\cellindR}\big) := \solutioncoeff{\SGsumIndex}{\cellind,\cellindy,\cellindR}{\timeind}\nonumber\\[.2cm]
		&- \frac{\dtstepsize}{\Delta\x} \! \int_{\cellR{\cellindR}}\!\! \Big(\hat{\flux}_1\big(\InvGradEntropy(\entrcm{\timeind}{\cellind,\cellindy,\cellindR}),\InvGradEntropy(\entrcm{\timeind}{\cellind+1,\cellindy,\cellindR}),\uncertainty\big)\!-\!\hat{\flux}_1\big(\InvGradEntropy(\entrcm{\timeind}{\cellind-1,\cellindy,\cellindR}),\InvGradEntropy(\entrcm{\timeind}{\cellind,\cellindy,\cellindR}),\uncertainty\big)\!\Big)
		\xiBasisPoly{\SGsumIndex,\cellindR}\locxiPDF{\cellindR} \mathrm{d}\uncertainty \nonumber\\[.2cm]
		&- \frac{\dtstepsize}{\Delta\y}\! \int_{\cellR{\cellindR}}\!\! \Big(\hat{\flux}_2\big(\InvGradEntropy(\entrcm{\timeind}{\cellind,\cellindy,\cellindR}),\InvGradEntropy(\entrcm{\timeind}{\cellind,\cellindy+1,\cellindR}),\uncertainty\big)\!-\!\hat{\flux}_2\big(\InvGradEntropy(\entrcm{\timeind}{\cellind,\cellindy-1,\cellindR}),\InvGradEntropy(\entrcm{\timeind}{\cellind,\cellindy,\cellindR}),\uncertainty\big)\!\Big)
		\xiBasisPoly{\SGsumIndex,\cellindR}\locxiPDF{\cellindR} \mathrm{d}\uncertainty \label{eq:LhIPM},
\end{align}
for $\SGsumIndex=0,\ldots,\SGtruncorder$,  $(\cellindx,\cellindy)\in\ncells$ and $\cellindR=1,\ldots,\MEElements$, defining one forward Euler time step of the FVM for the ME-IPM system \eqref{MEIPMsystem}. Here, $\entrcm{\timeind}{\cellind,\cellindy,\cellindR}$ describes the cell mean of the entropic variable in cell $\cell{\cellindx,\cellindy}$ and \ME{} $\cellR{\cellindR}$ at $\timevar_\timeind$.

\begin{algorithm}[H] 
	\caption{Multi-Element Intrusive Polynomial Moment Method (ME-IPM)}
	\label{algo:MEIPMM}
	\begin{algorithmic}[1]
		\State $\entrcm{0}{}\leftarrow \Big(\sum_{\SGsumIndex=0}^\SGtruncorder\int_{\cell{\cellind,\cellindy}}\!\int_{\cellR{\cellindR}}\entropy'\big(\solution^{(0)}\big)\,\xiBasisPoly{\SGsumIndex,\cellindR}\locxiPDF{\cellindR} \mathrm{d}\uncertainty\mathrm{d}(\x,\y)\,\xiBasisPoly{\SGsumIndex,\cellindR}\Big)_{(\cellind,\cellindy)\in\ncells,\cellindR=1:\MEElements}$\hfill \textit{\# initial state}\smallskip
		\For{$\timeind=0,\ldots,N_\tend$} \hfill \textit{\# time loop}\smallskip
		\State $\overline{\solution}^{(\timeind)} \leftarrow \Big(\int_{\cell{\cellind,\cellindy}}\!\int_{\cellR{\cellindR}} \InvGradEntropy\!\big( \entrcm{\timeind}{\cellind,\cellindy,\cellindR} \big) \xiBasisPoly{\SGsumIndex} \mathrm{d}\uncertainty\mathrm{d}(\x,\y)\Big)_{(\cellind,\cellindy)\in\ncells,\cellindR=1:\MEElements,\SGsumIndex:\SGtruncorder}$ \hfill \textit{\# variable map}\smallskip
		\State  $\overline{\solution}^{(\timeind+1)}\leftarrow
		L_h^{(\timeind)}(\overline{\solution}^{(\timeind)},\entrcm{\timeind}{})$ \hfill \textit{\# FVM \eqref{eq:LhIPM}}\smallskip
		\State $\entrcm{\timeind+1}{} \leftarrow \text{arg}\,\text{min}_{\entropicVarSum}\lagrange(\overline{\solution}^{(\timeind)},\entrcm{\timeind}{})$
		\hfill \textit{\# minimize Lagrangian using \eqref{eq:MEIPMLagrange}}\medskip
		\EndFor
	\end{algorithmic}
\end{algorithm}

\begin{remark}\label{rem:acceleration}
In addition to the described strategies to reduce computational costs, the steady-state test case in Section~\ref{sec:NACA} makes use of two acceleration techniques taken from \cite{kusch2020intrusive}. 
\begin{itemize}
    \item Note that the minimization of the Lagrangian is usually performed using Newton iterations. The iteration is stopped when the gradient of the Lagrangian is sufficiently close to zero. When solving steady-state problems, the FVM steps can be interpreted as a fixed point iteration, i.e. the time loop is iterated until a steady-state is reached. The idea of one-shot IPM (osIPM) is to perform only a single Newton iteration to minimize the Lagrangian in every iteration of the time loop. In this case, we iterate both, the moments as well as the dual variables to their steady-state solution simultaneously. It can be shown that this iteration converges locally \cite{kusch2020intrusive}. The idea follows the concept of the one-shot approach used in PDE constraint optimization \cite{Hazra2005}. 
    \item We make use of adaptivity in stochastic space, which is one core advantage of intrusive methods \cite{tryoen2012adaptive,buerger2014hybrid,meyer2020posteriori,kusch2020intrusive}. \comment{To reduce computational costs, both the gPC degree $\SGtruncorder$ and the number of quadrature points to compute the integrals in \eqref{LaxFriedrichs} are adapted in time and space according to the smoothness of the solution. To allow for an efficient evaluation of numerical fluxes in \eqref{LaxFriedrichs} when neighboring cells have different refinement levels, Clenshaw-Curtis nodes are chosen. In this case, the nestedness of these nodes allows for a computationally efficient evaluation, since spatial stencils can be evaluated on the same nodes. For a detailed presentation of the adaptivity strategy used in this work, see \cite[Section~5]{kusch2020intrusive}. Note that we do not use adaptivity in the number of multi-elements or the spatial domain.}
    
    \item To ensure that a large number of time iterations is performed on a low refinement level (i.e. on a small number of moments), the maximal truncation order is only allowed to increase once the approximation is sufficiently close to the steady-state solution. Thereby, a large amount of time iterations use the lowest refinement level (i.e. a small number of moments), which can be understood as a preconditioning step. Keeping a low refinement level for a large amount of time steps is called refinement retardation.
\end{itemize}
\end{remark}

\begin{remark}
A common issue of minimal entropy methods such as IPM is the loss of realizability due to errors when numerically solving the IPM optimization problem. This means that moment vectors generated by the numerical method can leave the hyperbolicity set, i.e. they do not belong to an admissible solution. The variable map ensures hyperbolicity as discussed in \cite{Kusch2017}.
\end{remark}

\comment{
\begin{remark}
A parallelization of the proposed methods is performed with MPI. To reduce communication overhead, the computational resources focus on the parallelization of quadrature rules for numerical fluxes as well as IPM optimization problems. For more details on the chosen parallelization strategy, see \cite[Section~6]{kusch2020intrusive}.
\end{remark}}
\section{Numerical Results}\label{sec:five}
In the following section we apply the oscillation mitigation intrusive UQ schemes from \secref{sec:three} and \secref{sec:four} to the one- and two-dimensional Euler equations. In particular, we evaluate and compare the stochastic Galerkin and intrusive polynomial moment method with and without the \ME{} ansatz in terms of their relative error and the computational costs. 

The code that we use to obtain the proceeding numerical results is openly available at \cite{uqcreator} to allow reproducibility. To guarantee admissibility for both, the IPM and hSG solutions, all numerical calculations use the HLL numerical flux \cite{harten1983upstream} which has the property of being positivity--preserving \cite{munz1994godunov}.

\subsection{Sod's shock tube}\label{sec:Sod}
To demonstrate the behavior of the different proposed methods, we first investigate the uncertain Sod's shock tube test case as proposed in \cite{Poette2009}. \comment{The uncertain Sod's shock tube from \cite{Poette2009} extends the original test case \cite{sod1978survey} to an uncertain shock position.} Here, the random system of conservation laws are the one-dimensional Euler equations \eqref{eq:euler}, where the heat capacity ratio $\gamma$ is chosen to be $1.4$. 
Initially, the gas is in a shock state with random position $x_{\text{interface}}(\xi) = x_0+\sigma \xi$, where $\xi$ is uniformly distributed in the interval $[-1,1]$. Hence, we have
\begin{eqnarray*}
\rho_{\text{IC}} &=& \begin{cases} \rho_L &\mbox{if } x < x_{\text{interface}}(\xi) \\
\rho_R & \mbox{else } \end{cases}\;, \\
(\rho v)_{\text{IC}} &=& 0\;, \\
(\rho e)_{\text{IC}} &=& \begin{cases} \rho_L e_L &\mbox{if } x < x_{\text{interface}}(\xi) \\
\rho_R e_R & \mbox{else } \end{cases}\;.
\end{eqnarray*}
The IPM method chooses the entropy
\begin{equation*}
\entropy(\rho,\rho v,\rho e) = -\rho \ln\left(\rho^{-\gamma}\left(\rho e-\frac{(\rho v)^2}{2\rho}\right)\right)\;.
\end{equation*}
Further parameter values are
\begin{table}[h!]
\centering
    \begin{tabular}{ | l | p{5.5cm} |}
    \hline
    $\Domain=[\x_L,\x_R]=[0,1]$ & range of spatial domain \\
    $\tend=0.14$ & end time \\
    $x_0 = 0.5, \sigma = 0.05$ & interface position parameters\\
    $\rho_L = 1.0,e_L = 2.5, \rho_R = 0.125,e_R = 0.25$ & initial states\\
    \hline
    \end{tabular}
    \caption{Parameter values for Sod's shock tube.}
    \label{tab:paramsSod}
\end{table}

Moreover, we divide the spatial domain into $\ncells=2000$ cells.
In a first simulation, we compare the expected gas density as well as its variance when employing stochastic Galerkin with and without the multi-element ansatz. To guarantee an admissible solution, both methods require the hyperbolicity-preserving limiter discussed in \secref{sec:hypPresLimiter}. The classical hSG scheme uses polynomials up to a degree of $14$, which we denote by hSG$_{14}$. For multi-element hSG we decompose the stochastic domain into three multi-elements, each of them with polynomials up to degree $4$. We denote this method by ME-hSG$_{3,4}$. Furthermore, a Gauss-Legendre quadrature is chosen. Here, the chosen number of points is $30$ for the classical hSG method and $10$ points per multi-element when using ME-hSG.

\begin{figure}[h!]
\centering
	\begin{subfigure}{0.49\linewidth}
		\centering
				\includegraphics[width=\linewidth]{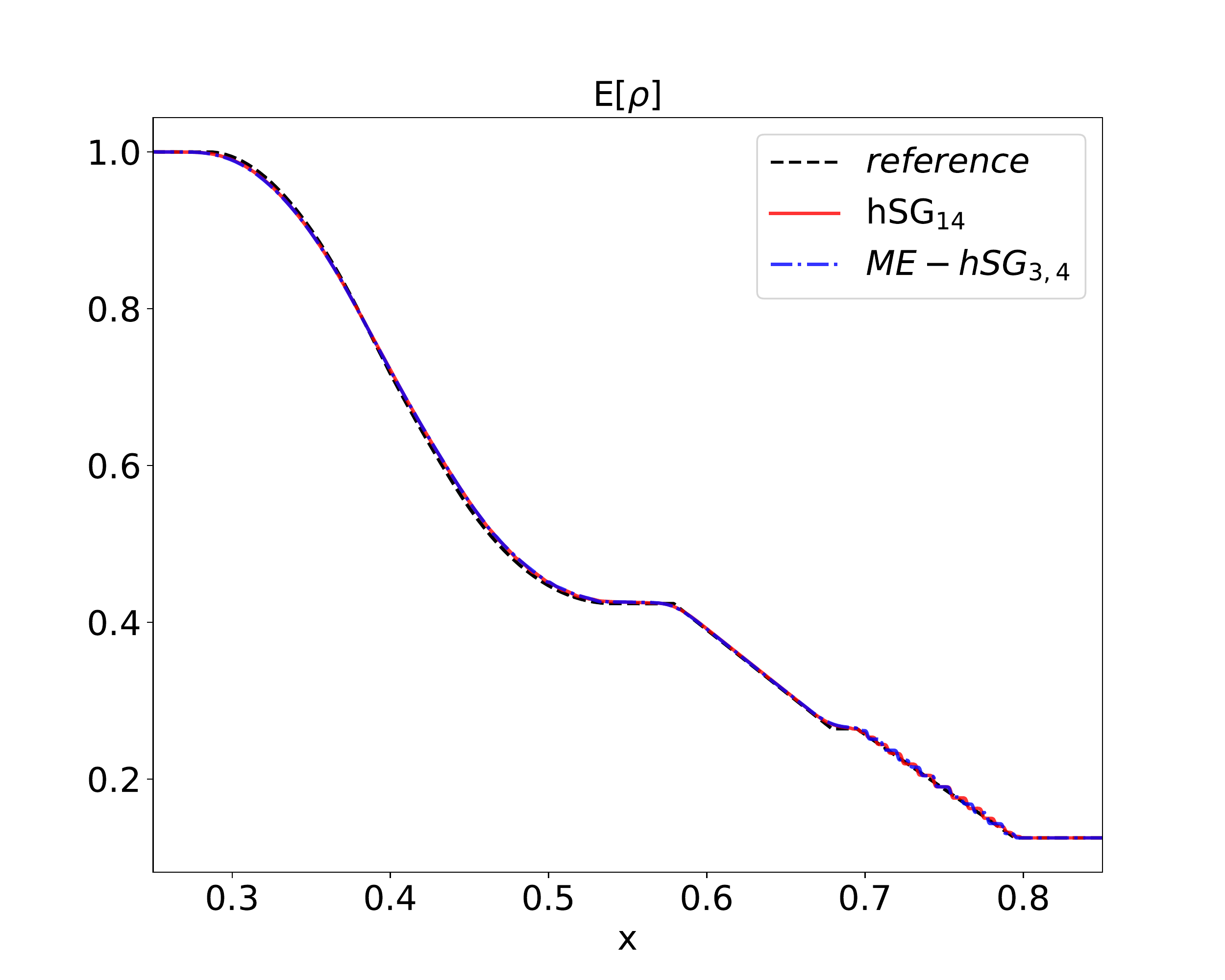}
		\label{fig:referenceSolutionsub1}
	\end{subfigure}
	\hfill
	\begin{subfigure}{0.49\linewidth}
		\centering
				\includegraphics[width=\linewidth]{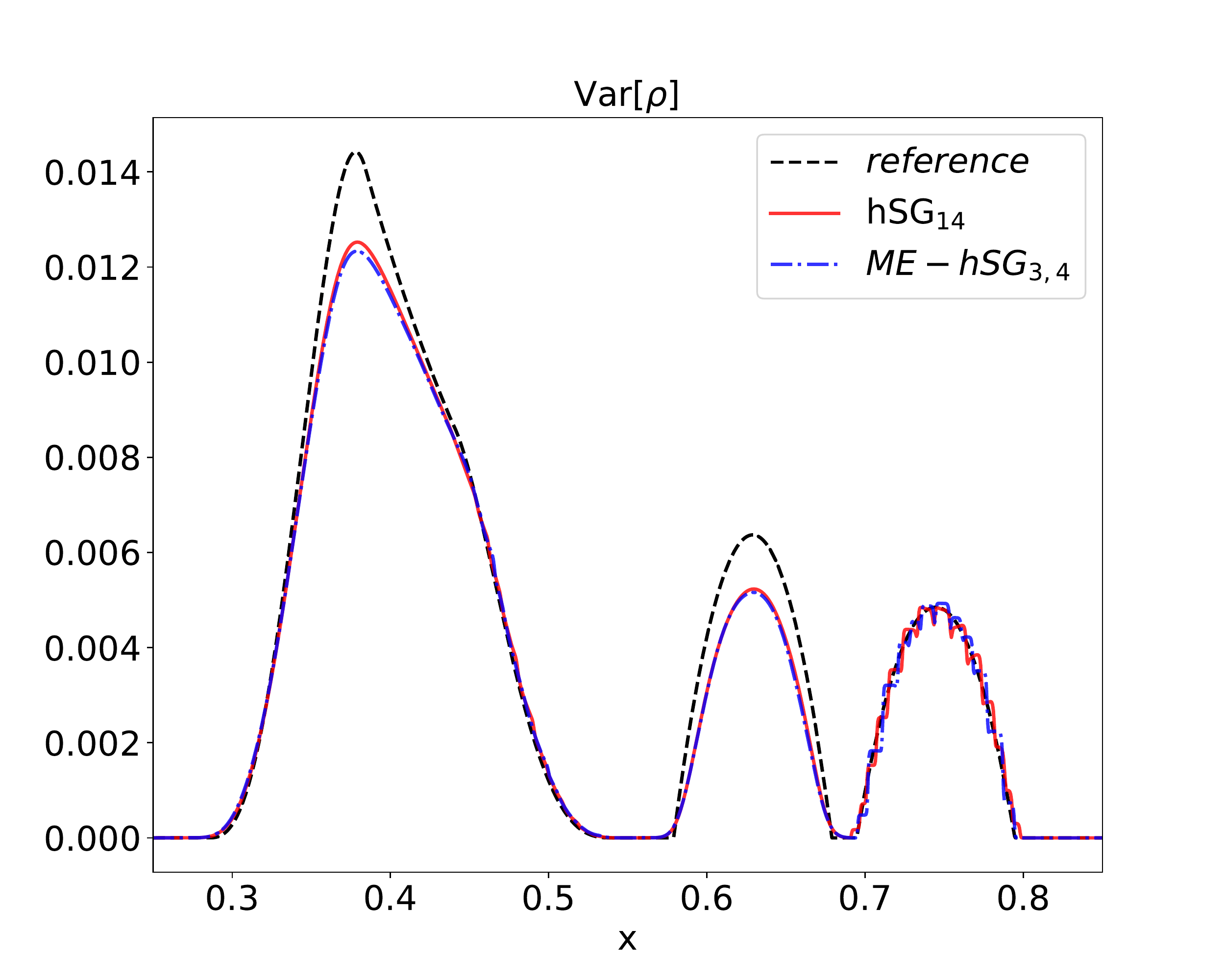}
		\label{fig:referenceSolutionsub2}
	\end{subfigure}
	\caption{Comparison of hSG using $15$ moments and ME-hSG for $3$ multi-elements with $5$ moments. The number of unknowns per spacial cell is identical, which yields the same runtime and similar results.}
	\label{fig:ComparisonSGMESG}
\end{figure}

The results are shown in Figure~\ref{fig:ComparisonSGMESG}, \comment{where we include the analytic reference solution to Sod's shock tube test case \cite{sod1978survey}}. From left to right, we can see the expected value and variance of the three main characteristics of this test case, namely a rarefaction wave, a contact discontinuity and a shock. Since the number of unknowns per spatial cell is identical for both hSG and ME-hSG, the computational costs are the same leading to similar runtimes (24.675 seconds for hSG and 23.761 seconds for ME-hSG). Expected value and variance of both methods look similar. In particular, the solution emits a step-like profile of the expected value and an oscillatory variance at the shock. In this case, the hSG and ME-hSG results almost coincide.

When applying the IPM method with and without multi-elements we observe a clear improvement of the overall runtime. We denote the IPM method using polynomials up to degree $14$ by IPM$_{14}$ and the multi-element IPM method with three multi-elements and gPC degree four by ME-IPM$_{3,4}$, where we use the same quadrature as for hSG and ME-hSG. The numerical results are depicted in Figure~\ref{fig:ComparisonIPMMEIPM}.

\begin{figure}[h!]
\centering
	\begin{subfigure}{0.49\linewidth}
		\centering
				\includegraphics[width=\linewidth]{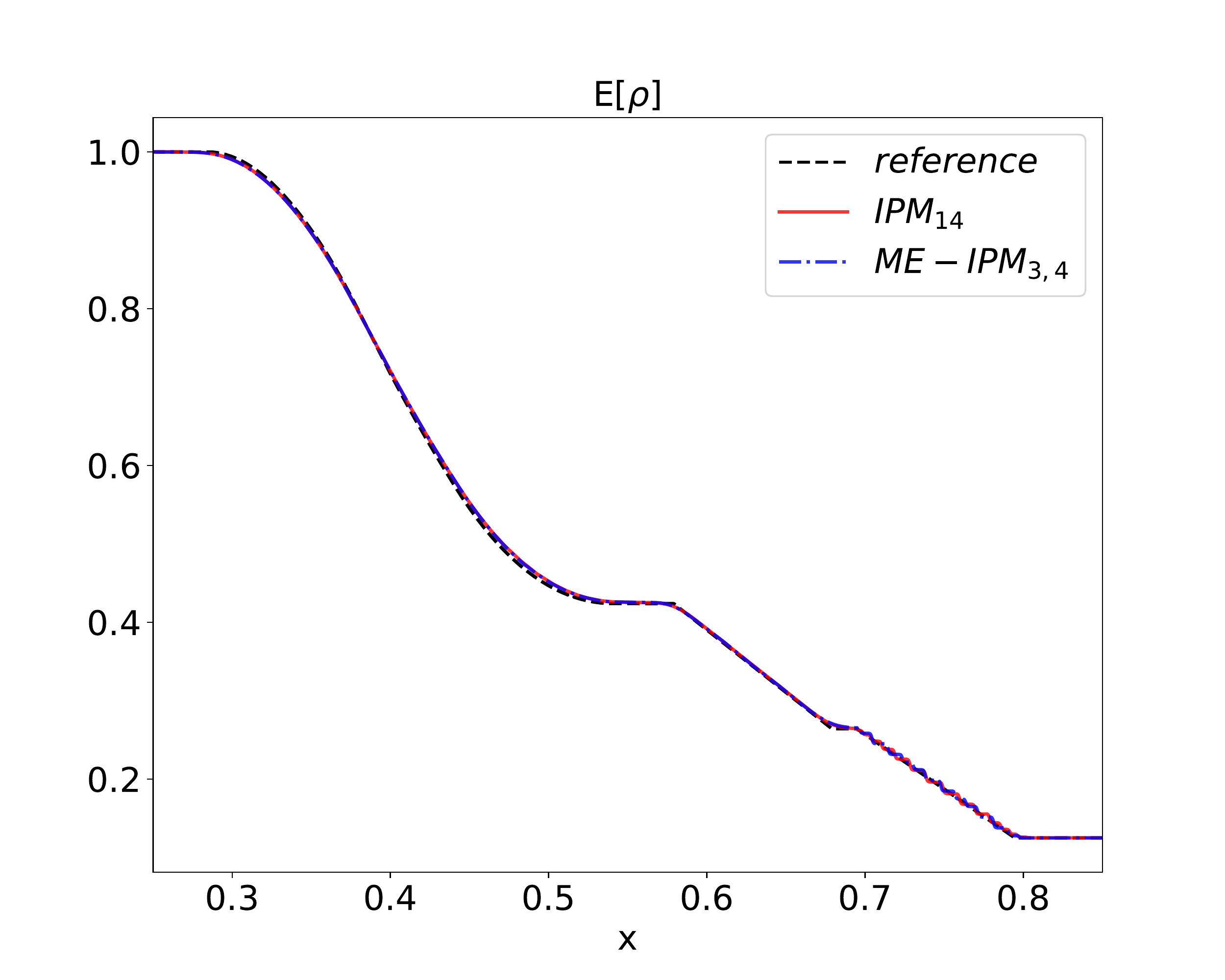}
		\label{fig:referenceSolutionsub1}
	\end{subfigure}
	\hfill
	\begin{subfigure}{0.49\linewidth}
		\centering
				\includegraphics[width=\linewidth]{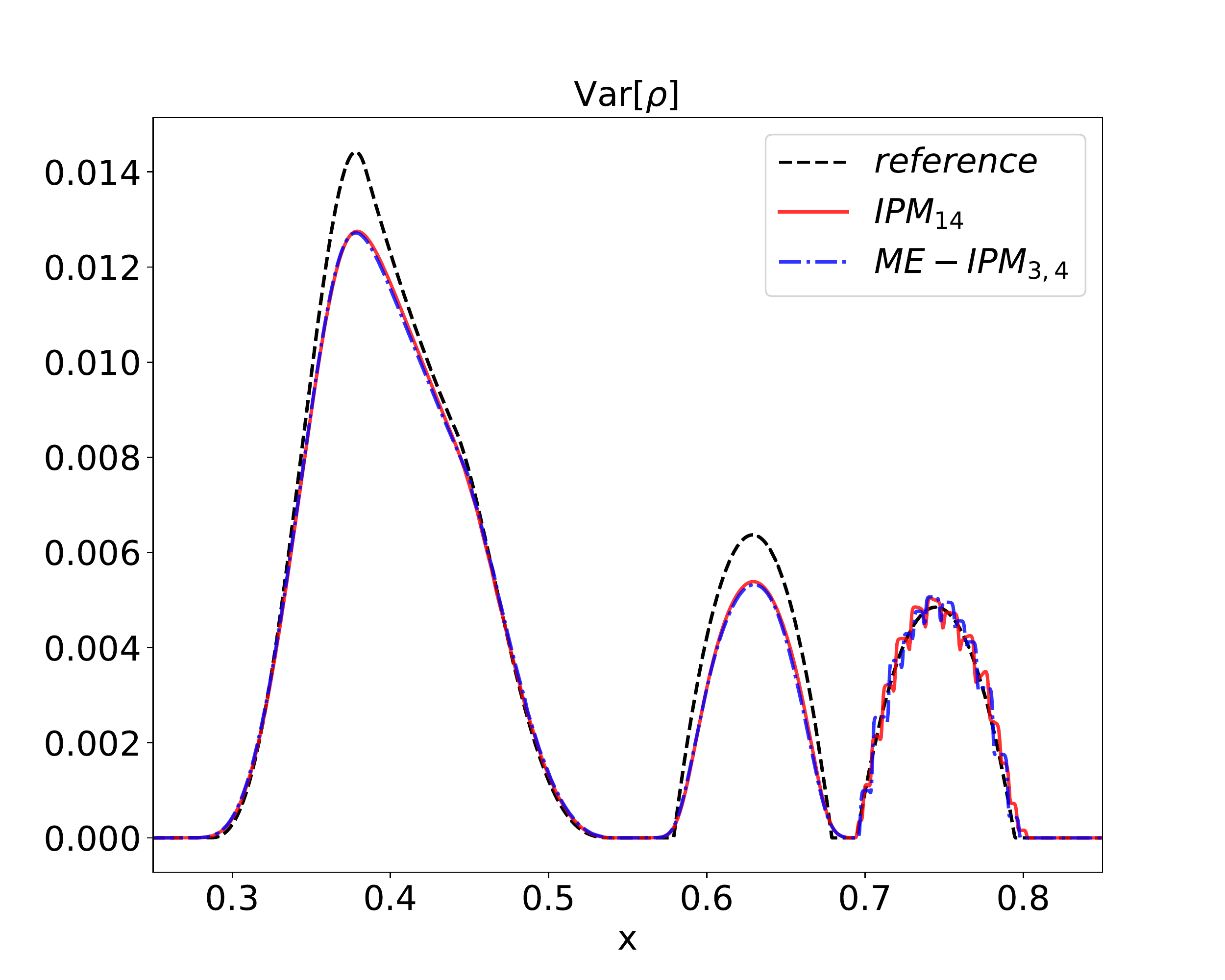}
		\label{fig:referenceSolutionsub2}
	\end{subfigure}
	\caption{Comparison of IPM using $15$ moments and ME-IPM for $3$ multi-elements with $5$ moments. Even though the number of unknowns per spacial cell is identical, the runtime of the multi-element IPM method is reduced by a factor of $7.72$ while yielding a similar result.}
	\label{fig:ComparisonIPMMEIPM}
\end{figure}

First, we notice that the multi-element ansatz heavily reduces the computational costs, since every cell requires solving three decoupled optimization problems with five unknowns respectively. In contrast, the classical IPM method with the same number of unknowns requires solving one optimization problem with $15$ unknowns. This reduction in numerical costs is reflected by the runtime which decreases by a factor of $7.72$ (from $4364$ to $564$ seconds). Note that the IPM results show a slightly improved approximation compared to the hSG results in Figure~\ref{fig:ComparisonSGMESG}. This however does not justify the increased computational costs of both IPM methods, which is why we can conclude that in this test case the hSG method performs better than IPM.

Lastly, we investigate the multi-element ansatz for filtered hSG. The results are depicted in Figure~\ref{fig:ComparisonMESGvsFMESG}, we additionally show a zoomed view of the shock region in Figure~\ref{fig:ComparisonMESGvsFMESGZoom}.

\begin{figure}[h!]
\centering
	\begin{subfigure}{0.49\linewidth}
		\centering
				\includegraphics[width=\linewidth]{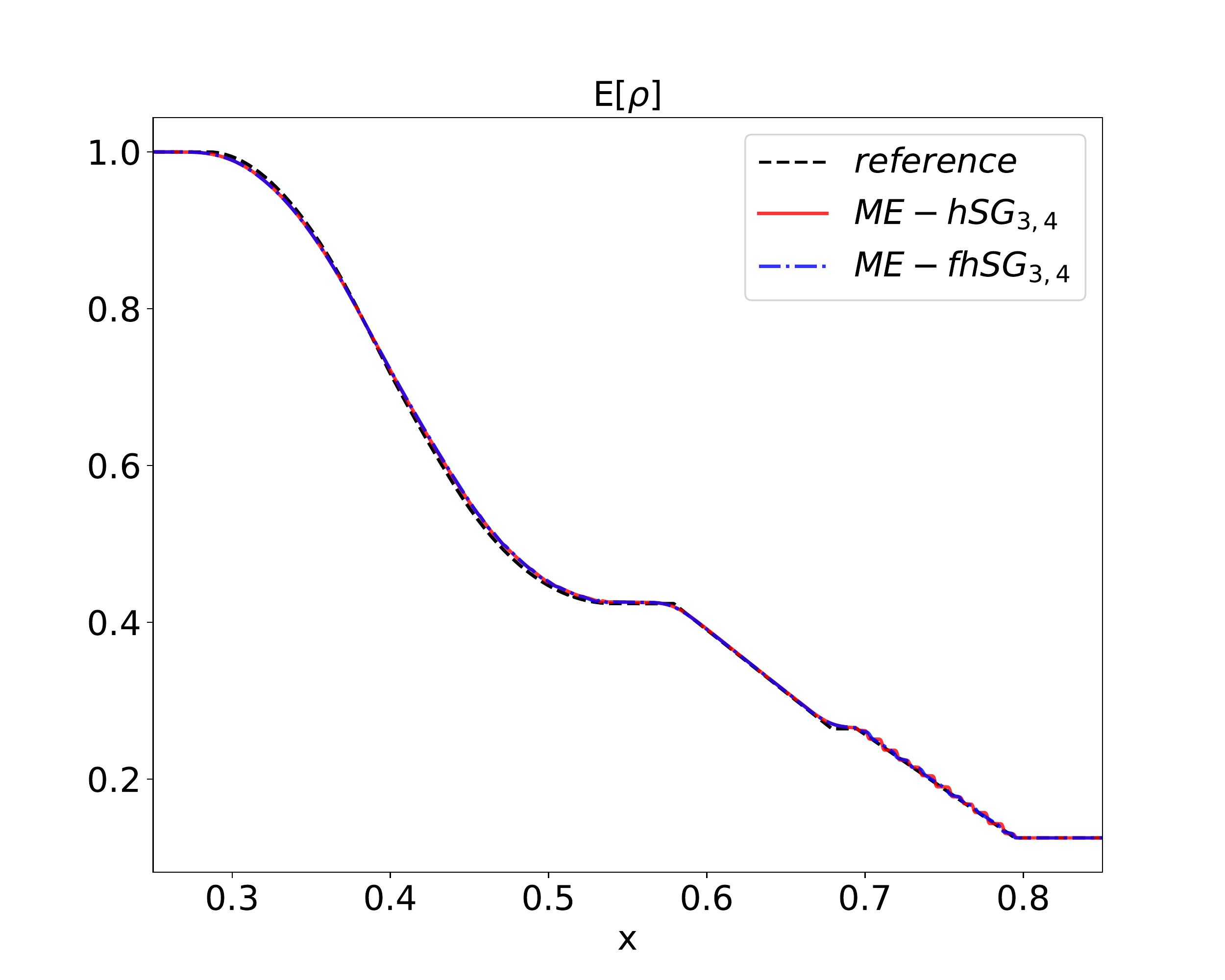}
		\label{fig:referenceSolutionsub1}
	\end{subfigure}
	\hfill
	\begin{subfigure}{0.49\linewidth}
		\centering
				\includegraphics[width=\linewidth]{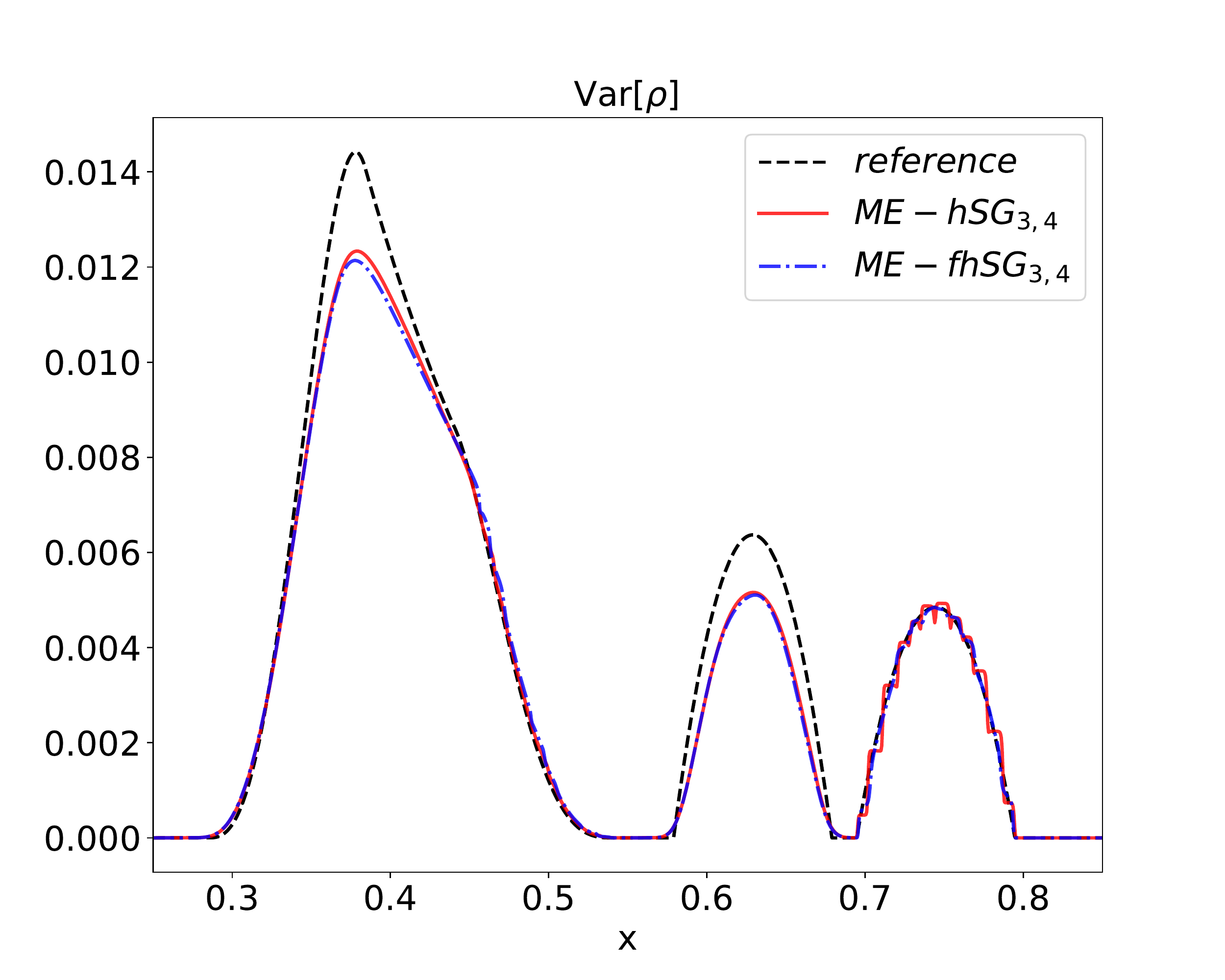}
		\label{fig:referenceSolutionsub2}
	\end{subfigure}
	\caption{Comparison of ME-hSG and ME-fhSG using $3$ multi-elements with $5$ moments. While the filter slightly dampens the variance at the contact discontinuity and the rarefaction wave, the shock approximation shows good agreement with the reference solution. }
	\label{fig:ComparisonMESGvsFMESG}
\end{figure}
\begin{figure}[h!]
\centering
	\begin{subfigure}{0.49\linewidth}
		\centering
				\includegraphics[width=\linewidth]{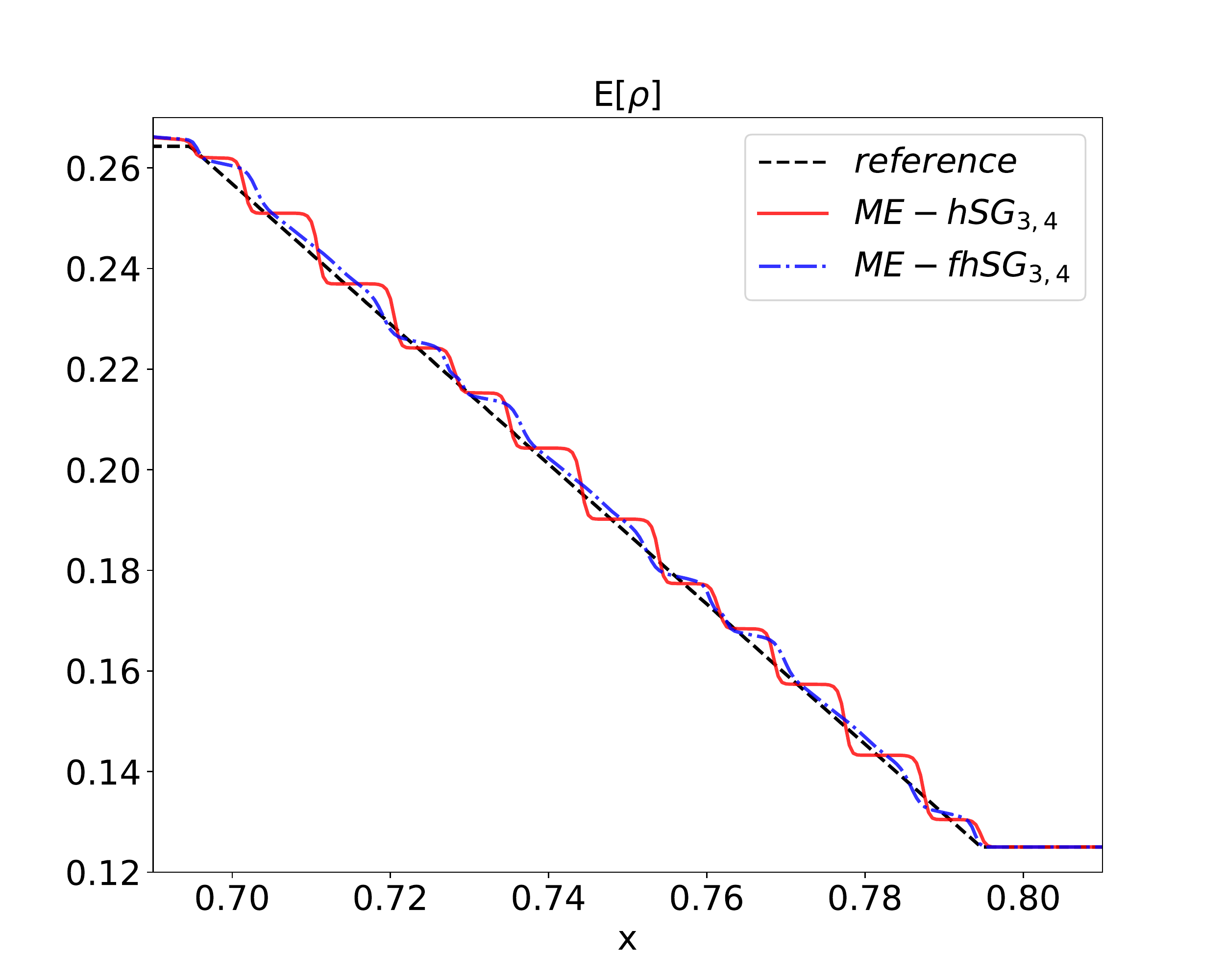}
		\label{fig:referenceSolutionsub1}
	\end{subfigure}
	\hfill
	\begin{subfigure}{0.49\linewidth}
		\centering
				\includegraphics[width=\linewidth]{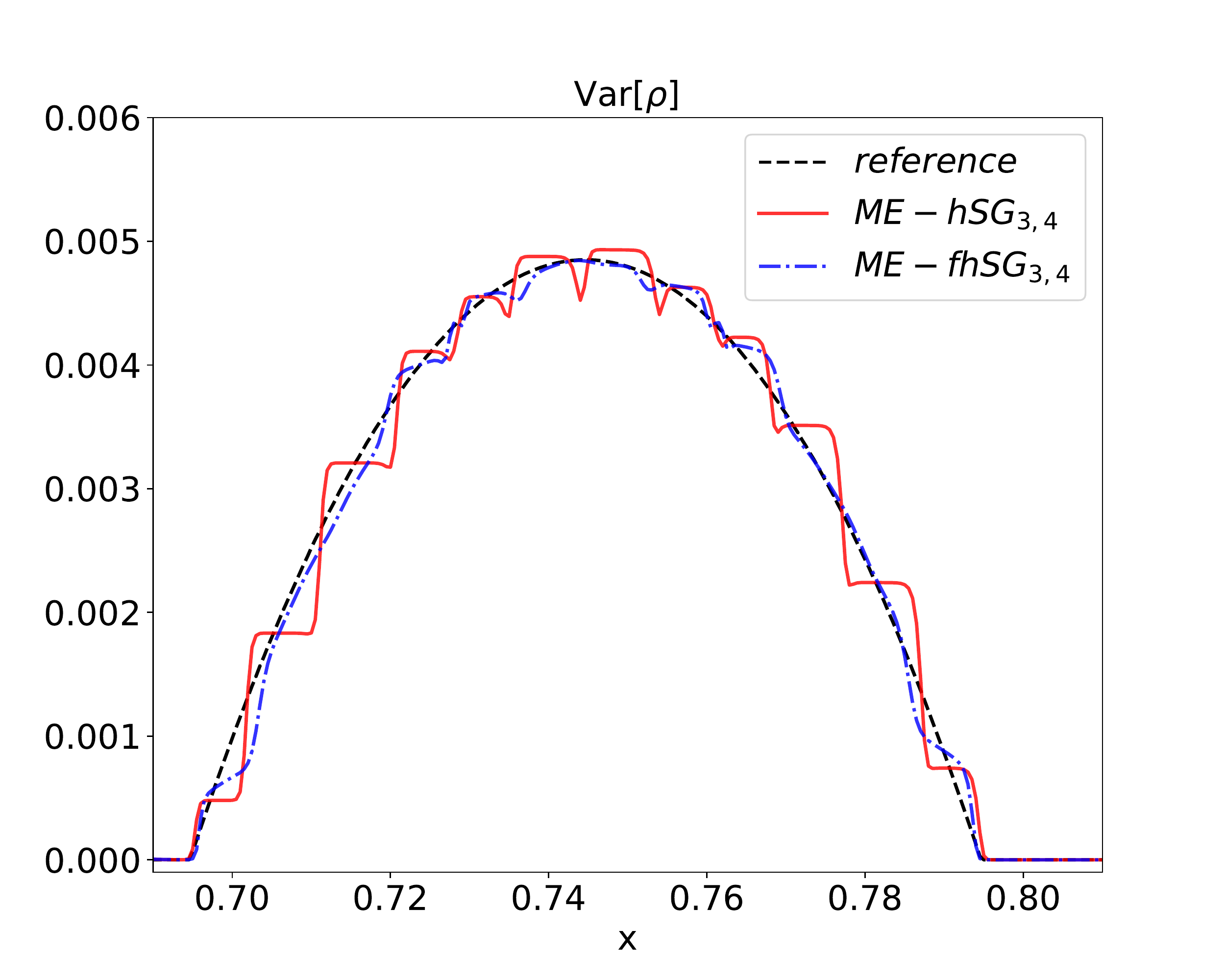}
		\label{fig:referenceSolutionsub2}
	\end{subfigure}
	\caption{Results for ME-hSG and ME-fhSG using $3$ multi-elements with $5$ moments. Zoom on shock. }
	\label{fig:ComparisonMESGvsFMESGZoom}
\end{figure}

The chosen filter is the exponential filter discussed in \secref{sec:fSG}. We pick the filter strength $\lambda=2$ and the order $\alpha=10$. These values have been chosen by a parameter study in which we found that an increased order helps to distinguish between the rarefaction wave, which does not require filtering and the shock position where the filter needs to be active. At a high filter order, the filtering does not heavily effect the rarefaction wave while focusing on the shock. As expected, choosing the filter strength too small will not dampen oscillations efficiently while picking the filter strength too large filters out the uncertainty. At the chosen filter parameters, we observe that the filter significantly improves the expected value and the variance approximation at the shock. Unfortunately, the filter slightly dampens the variance at the rarefaction wave. Overall, the filter yields a satisfactory solution approximation without increasing the computational time.

\subsection{NACA0012 airfoil}\label{sec:NACA}
In the following we consider the uncertain NACA test case from \cite{kusch2020intrusive}. The test case investigates the effects of an uncertain angle of attack $\phi\sim U(0.75,1.75)$ for a NACA0012 airfoil with a length of one meter. The governing equations are the stochastic Euler equations in two dimensions, which read
\begin{align}\label{eq:eulerxy}
\partial_t
\begin{pmatrix}
\rho \\ \rho v_1 \\ \rho v_2 \\ \rho e
\end{pmatrix}
+\partial_{\x}
\begin{pmatrix}
\rho v_1 \\ \rho v_1^2 +p \\ \rho v_1 v_2 \\  v_1 (\rho e+p)
\end{pmatrix}
+\partial_{\y}
\begin{pmatrix}
\rho v_2 \\ \rho v_1 v_2 \\ \rho v_2^2+p \\ v_2 (\rho e+p)
\end{pmatrix}
=\bm{0}.
\end{align}
A closure for the pressure $p$ in two dimensions is given by
\begin{align*}
p = (\gamma-1)\rho\left(e-\frac12(v_1^2+v_2^2)\right).
\end{align*}
The hyperbolicity set from \defref{def:hypset} is therefore given by 
$$	\realizableSet = \left\{ \solution = \begin{pmatrix}
\rho \\ \rho v_1 \\ \rho v_2 \\ \rho e
\end{pmatrix}\,\Bigg|~\density>0,~\pressure = (\gamma-1)\rho\left(e-\frac12(v_1^2+v_2^2)\right) >0\right\}.$$
We again choose the heat capacity ratio $\gamma$ as $1.4$. At the airfoil surface, the Euler slip condition $\bm v^T\bm n = 0$ is used, where $\bm n$ denotes the surface normal. At a distance of $40$ meters from the airfoil, one assumes Dirichlet boundary conditions, i.e., the airfoil does not influence the flow at this boundary. Here, the constant flow field is given by Mach number $Ma = 0.8$, pressure $p = 101\;325$ Pa and a temperature of $273.15$ K. With the specific gas constant $R=287.87\frac{J}{kg\cdot K}$ and a given angle of attack $\phi$ the conserved variables at the far field boundary can be determined uniquely, since
\begin{align*}
\rho &= \frac{p}{R\cdot T}, \\ 
\rho v_1 &= \rho\cdot\text{Ma}\cdot \sqrt{\gamma R T} 
\cos\phi ,\\ 
\rho v_2 &= \rho\cdot\text{Ma}\cdot \sqrt{\gamma R T} 
\sin\phi,\\ 
\rho e =& \frac{p}{\gamma - 1} + \frac{1}{2}\rho (v_1^2 + v_2^2).
\end{align*}
The initial condition is chosen to be equal to the far field boundary values. In this case, the slip condition at the airfoil will correct the flow solution until a steady-state flow is reached after a certain time. The iteration in time is performed until the time residual fulfills the stopping criterion
\begin{align}\label{eq:residualSteady}
\sum_{(\cellindx,\cellindy)\in\ncells}\sum_{\cellindR=1}^\MEElements \Delta \x \Delta\y \Vert \bm{\overline{u}}_{\cellindx,\cellindy,\cellindR}^{(\timeind)} - \bm{\overline{u}}_{\cellindx,\cellindy,\cellindR}^{(\timeind-1)} \Vert \leq \varepsilon.
\end{align}
In all steady-state calculations, we choose $\varepsilon=5\cdot10^{-6}$.

Unlike the Shock tube, this test case does not have an analytical solution. Therefore, we compute a reference solution by stochastic collocation with $100$ Gauss-Legendre quadrature points. The resulting expected value and variance of the density as well as the implemented mesh are depicted in Figure~\ref{fig:referenceSolution}. The mesh is composed of 22361 triangular elements and is used for all computations of the NACA test case.

\begin{remark}
Within the numerical implementation of this test case, we divide the two-dimensional spatial domain into a triangular mesh. The theory of \secref{sec:numerics} for uniform grids can easily be applied to such a discretization.
\end{remark}

\begin{figure}[h!]
\centering
	\begin{subfigure}{0.329\linewidth}
		\centering
				\includegraphics[width=\linewidth]{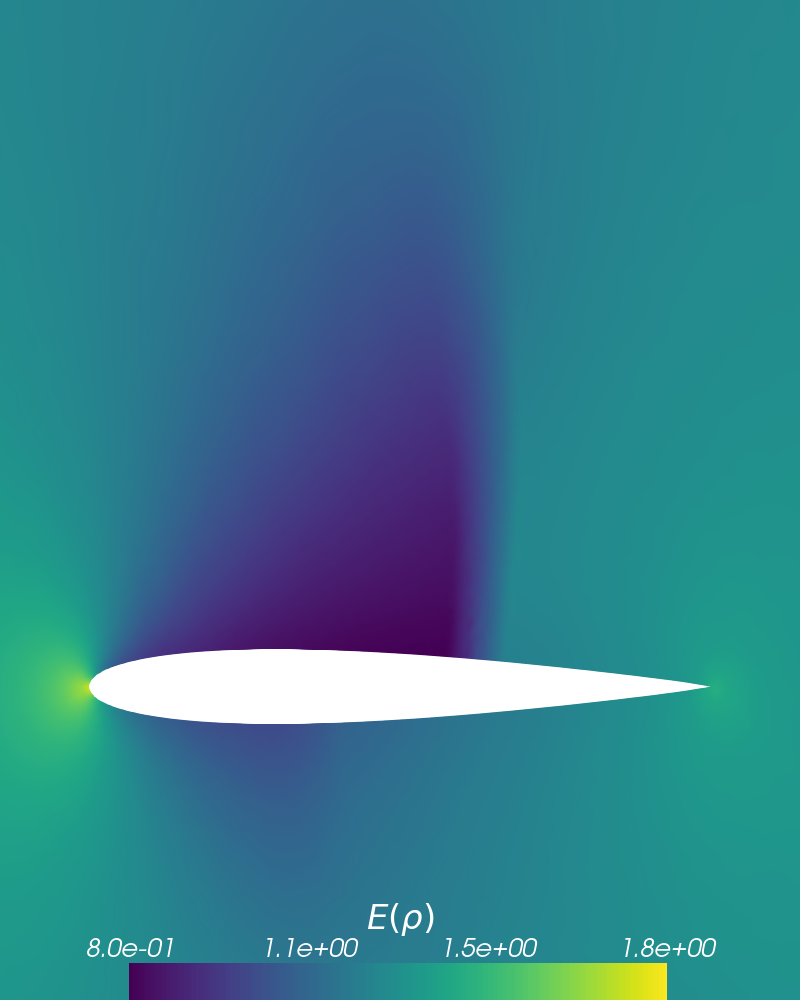}
		\label{fig:referenceSolutionsub1}
	\end{subfigure}
	\hfill
	\begin{subfigure}{0.329\linewidth}
		\centering
				\includegraphics[width=\linewidth]{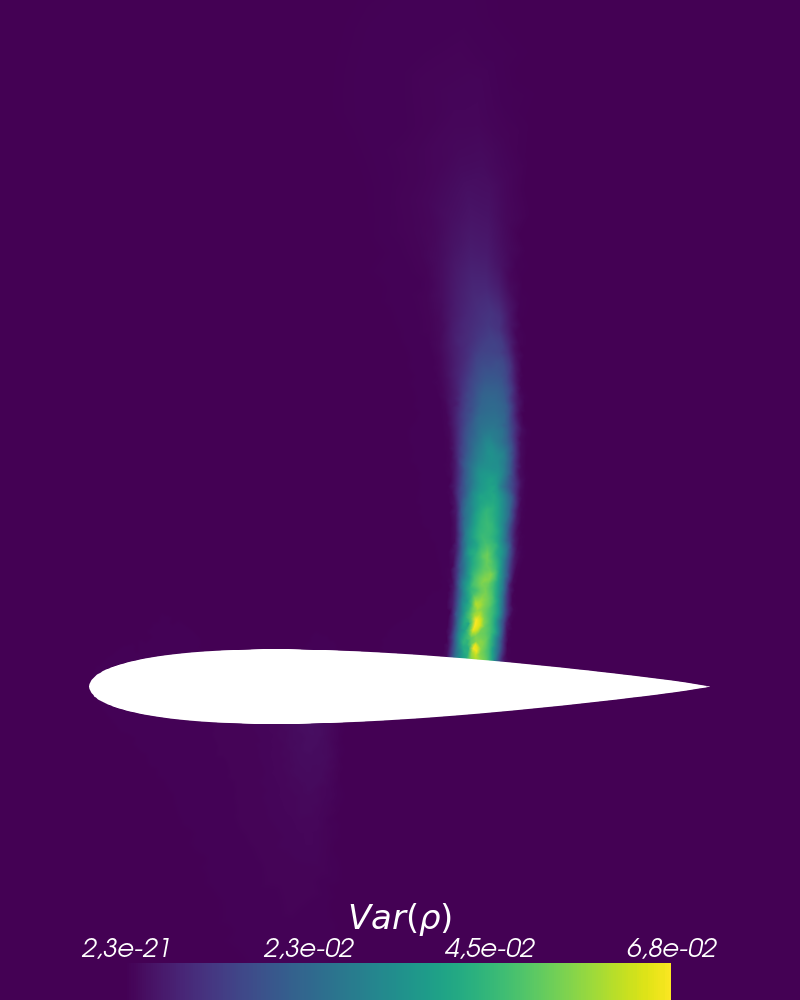}
		\label{fig:referenceSolutionsub2}
	\end{subfigure}
	\hfill
	\begin{subfigure}{0.329\linewidth}
		\centering
				\includegraphics[width=\linewidth]{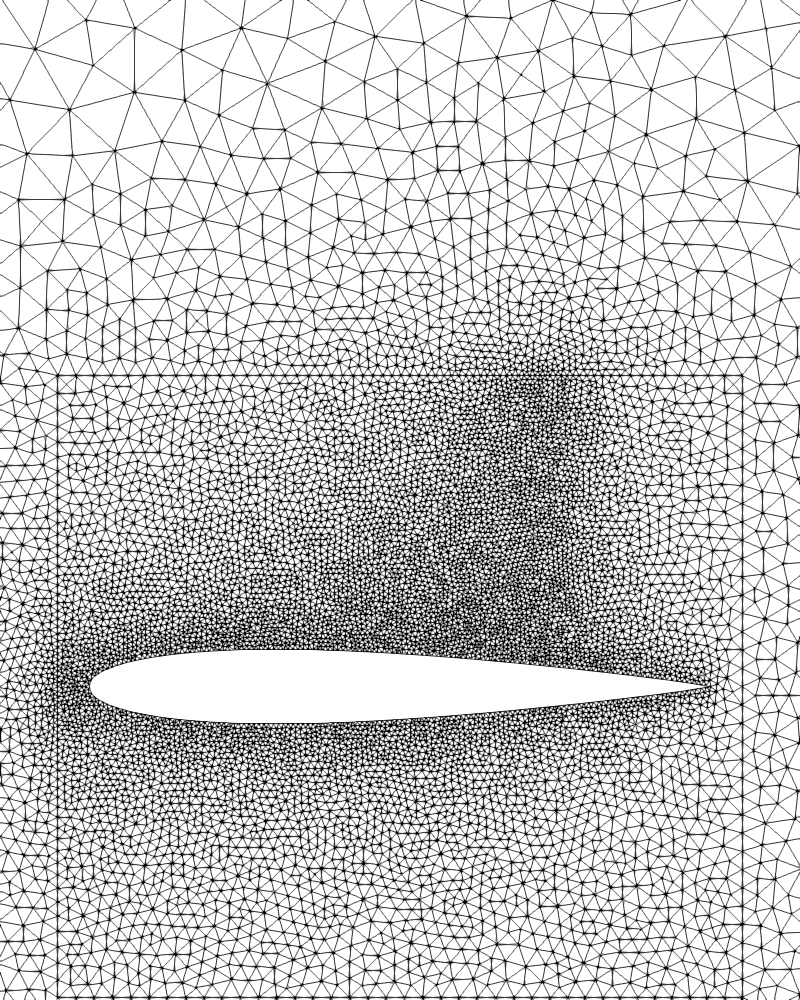}
		\label{fig:referenceSolutionsMesh}
	\end{subfigure}
	\caption{Left: Reference solution E$[\rho]$. Center: Var$[\rho]$. Right: Mesh around the airfoil. The mesh is a circle around the airfoil, consisting of 22361 triangular cells.}
	\label{fig:referenceSolution}
\end{figure}

In the following, we aim at approximating this solution with our oscillation mitigation intrusive UQ methods. To quantify the efficiency of the different approaches, we investigate the relative error of the density's expected value and variance over time. We define the L$_2$-error of a discrete quantity $\bm e_{\Delta}\in\mathbb{R}^{N_\x\times N_\y}$
\begin{align*}
\Vert \bm e_{\Delta} \Vert_{\Delta} := \sqrt{\sum_{(\cellindx,\cellindy)\in\ncells} \Delta \x \Delta\y\, \bm e_{\cellindx,\cellindy}^2}\;.
\end{align*}
Then, if we denote the solution obtained with the numerical method by $\solution_{\Delta}$ and the reference solution by $\widehat{\solution}_{\Delta}$, the relative error of the expected value and variance is
\begin{align}\label{eq:relErrors}
\frac{\Vert \text{E}[\widehat{\solution}_{\Delta}] - \text{E}[\solution_{\Delta}] \Vert_{\Delta}}{\Vert \text{E}[\widehat{\solution}_{\Delta}] \Vert_{\Delta}} \qquad \text{ and }\qquad \frac{\Vert \text{Var}[\widehat{\solution}_{\Delta}] - \text{Var}[\solution_{\Delta}] \Vert_{\Delta}}{\Vert \text{Var}[\widehat{\solution}_{\Delta}] \Vert_{\Delta}}.
\end{align}
Following \cite{kusch2020intrusive}, we record this error inside a box of one meter height and 1.1 meters length around the airfoil. This strategy allows us to exclude fluctuations in the far field solution. The resulting efficiency curves are depicted in Figure~\ref{fig:ConvergenceSGMESGIPMMEIPM}. Here, we plot the relative errors according to \eqref{eq:relErrors} over the computational time. All methods are converged to a time residual of $\varepsilon=5\cdot10^{-6}$, cf. \eqref{eq:residualSteady}. For most methods, a satisfactory solution is however reached at an earlier stage when the efficiency curve in Figure~\ref{fig:ConvergenceSGMESGIPMMEIPM} is saturated. In this case, the error of the steady-state approximation reaches the same error level as the approximation error of the uncertainty.

\begin{figure}[h!]
\centering
	\begin{subfigure}{0.49\linewidth}
		\centering
				\includegraphics[width=\linewidth]{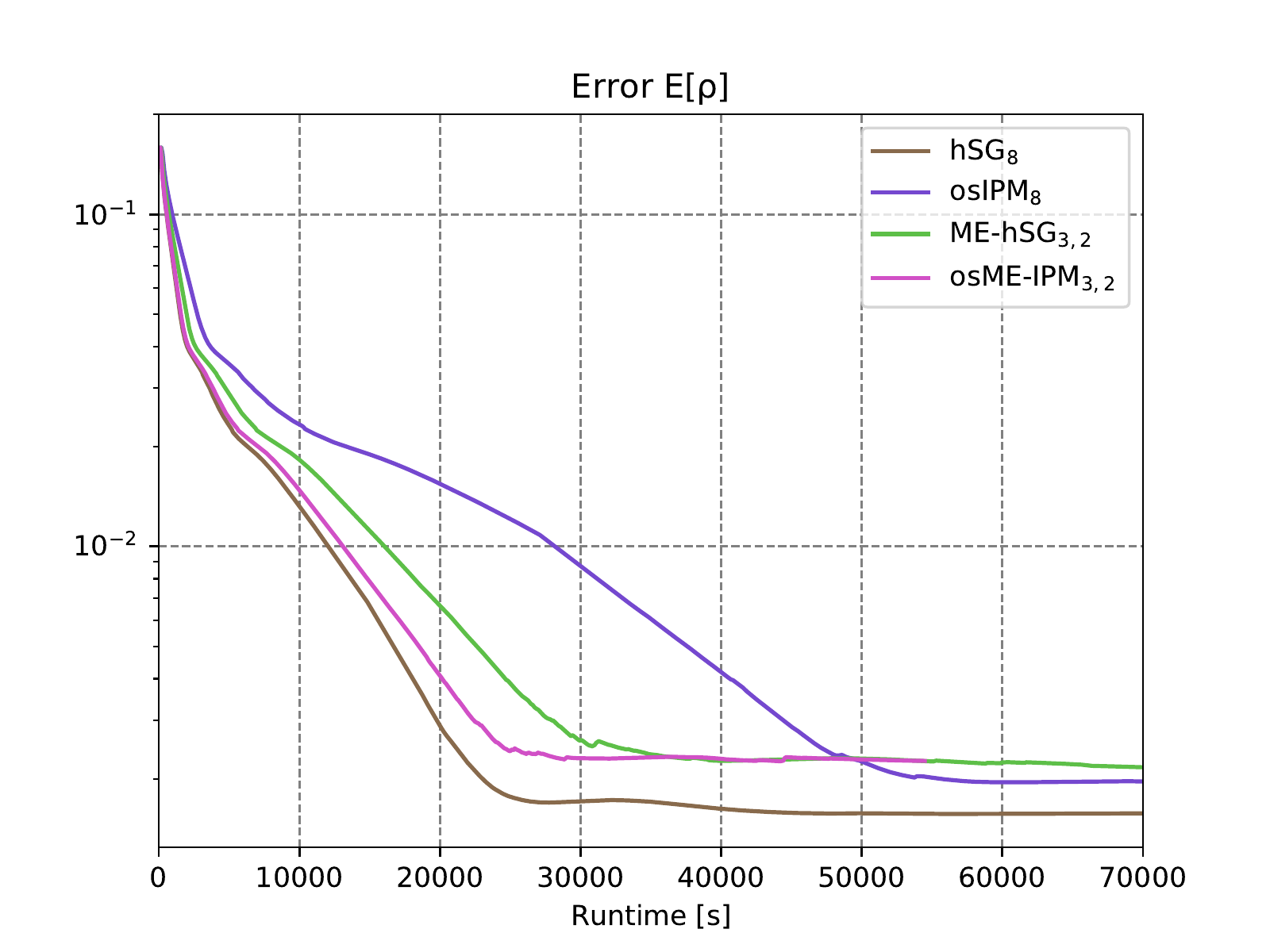}
		\label{fig:referenceSolutionsub1}
	\end{subfigure}
	\hfill
	\begin{subfigure}{0.49\linewidth}
		\centering
				\includegraphics[width=\linewidth]{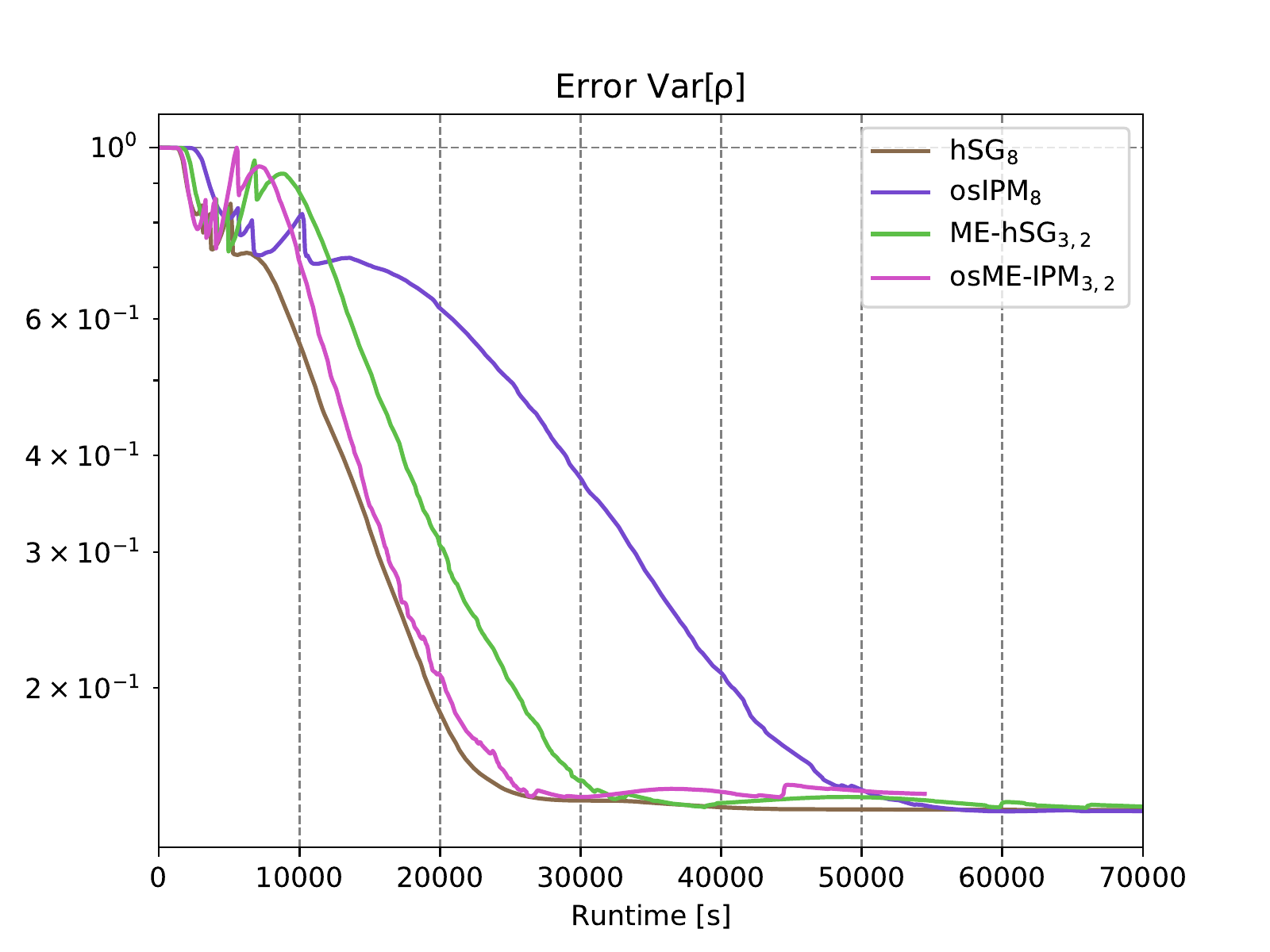}
		\label{fig:referenceSolutionsub2}
	\end{subfigure}
	\caption{Relative errors \eqref{eq:relErrors} computed over runtime. Computations are carried out using 5 MPI threads. As for the one-dimensional shock tube, the multi-element ansatz significantly reduces the overall runtime when using the IPM method. Multi-element IPM beats multi-element hSG, however the fastest method yielding the smallest error is hSG.}
	\label{fig:ConvergenceSGMESGIPMMEIPM}
\end{figure}

All simulations have been computed in parallel on 5 processors. For details on the parallelization of the code, see \cite{kusch2020intrusive}. The stochastic Galerkin calculations again require the hyperbolicity limiter presented Section~\ref{sec:hypPresLimiter}. First, we compare the results obtained with IPM. Here, we make use of one-shot IPM (osIPM) to reduce numerical costs. More information on this method can be found in \cite{kusch2020intrusive}. The number of unknowns per spatial cell for multi-element IPM and IPM is chosen to be $9$: While ME-IPM and ME-SG employ $3$ multi-elements with gPC degree two, the classical SG and IPM methods have gPC degree $8$. We use a Clenshaw-Curtis quadrature of level $4$ ($17$ points) when not using the multi-element ansatz. The multi-element methods use a level $2$ ($5$ points) Clenshaw-Curtis rule per element.
As expected, the multi-element ansatz again heavily decreases the numerical costs, while yielding a similar error level. Here, the error obtained with the standard IPM method is slightly better compared to the multi-element error. Note that ME-IPM beats ME-hSG in terms of efficiency. In this test case the most efficient method is the hyperbolicity-preserving stochastic Galerkin approach.

In order to further decrease the error obtained with the multi-element methods, we make use of adaptivity and refinement retardation, cf. Remark~\ref{rem:acceleration}. 
Here, we fix the choice of $3$ elements, however let the number of polynomials in every element depend on the solution's smoothness. Furthermore, all quadratures are approximated by a Clenshaw-Curtis quadrature rule, which adapts to the moment order. The different gPC degrees and the corresponding quadrature points for each refinement level are given by table \tabref{tab:levels}. We choose the refinement barriers as $\delta_{-} = 2\cdot 10^{-4}$ and $\delta_{+} = 2\cdot 10^{-5}$. For more information on \tabref{tab:levels} and refinement barriers see \cite{kusch2020intrusive}.

\begin{table}[h!]
\centering
    \begin{tabular}{ | l || c | c | c | c | c |}
    \hline
    refinement level & 0 &  1 & 2 & 3 & 4 \\
    degree & 2 &  3 & 4 & 5 & 6 \\
    quadrature points & 5 &  9 & 9 & 9 & 17 \\
    \hline
    \end{tabular}
    \caption{Refinement levels with their moment orders and quadrature points for adaptive refinements in ME-hSG and ME-IPM.}
    \label{tab:levels}
\end{table}

We converge the solution on refinement level zero to a residual of $\varepsilon=4\cdot 10^{-5}$ and then let the refinement level allow to increase up to level $4$ until $\varepsilon=5\cdot10^{-6}$. The resulting efficiency is depicted in Figure~\ref{fig:ConvergenceAdaptivity}. In this case, the ME stochastic Galerkin method is more efficient than ME-IPM. This is most likely caused by the increased number of moments per multi-element, which increases the costs of solving the IPM optimization problem. Thus, the ME-IPM method is most efficient on low polynomial degrees. A refinement strategy, which overcomes this issue would be to increase or decrease the number of multi-elements instead of the number of polynomials. We however leave this idea to future work.

\begin{figure}[h!]
\centering
	\begin{subfigure}{0.49\linewidth}
		\centering
				\includegraphics[width=\linewidth]{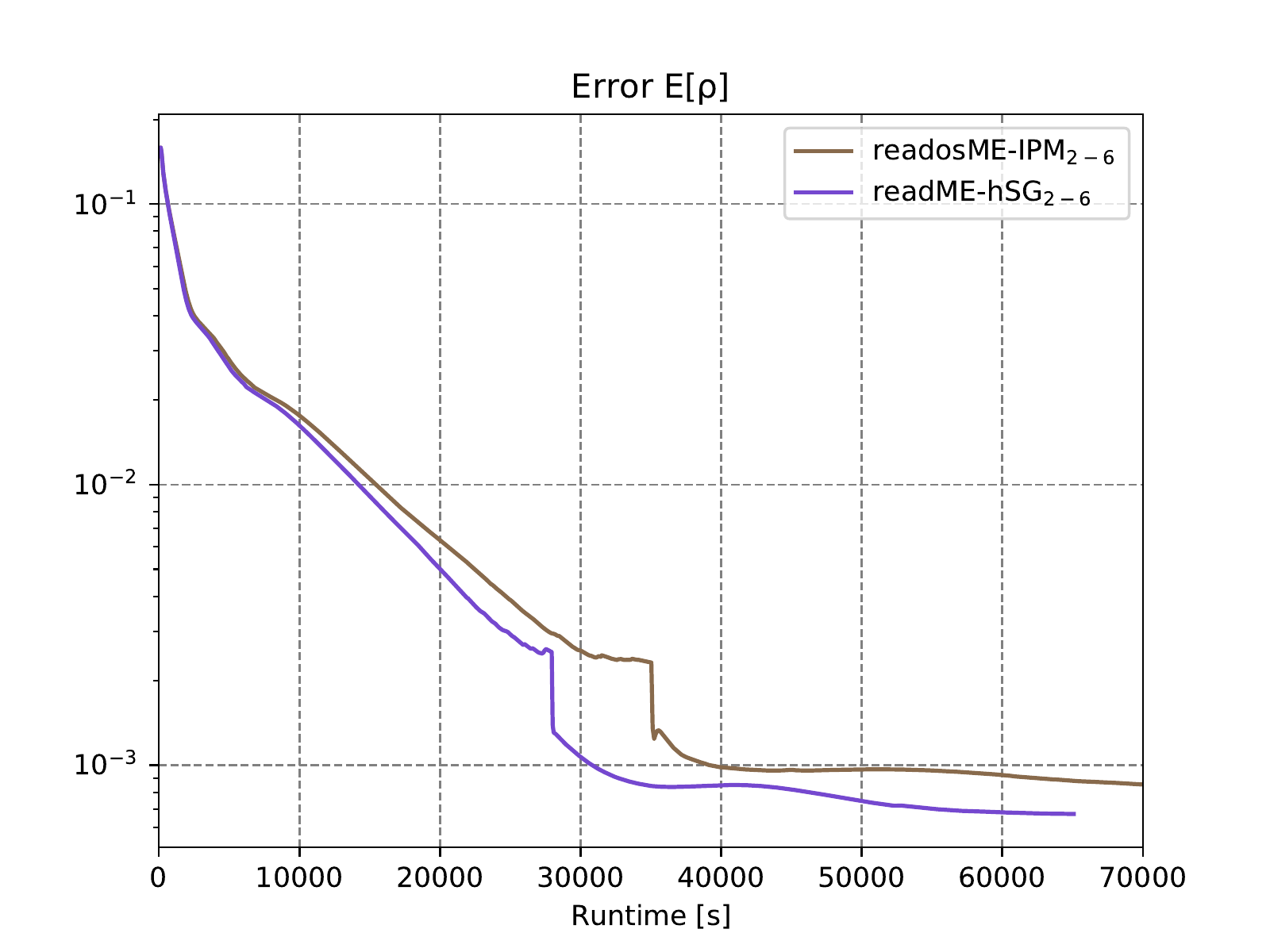}
	\end{subfigure}
	\hfill
	\begin{subfigure}{0.49\linewidth}
		\centering
				\includegraphics[width=\linewidth]{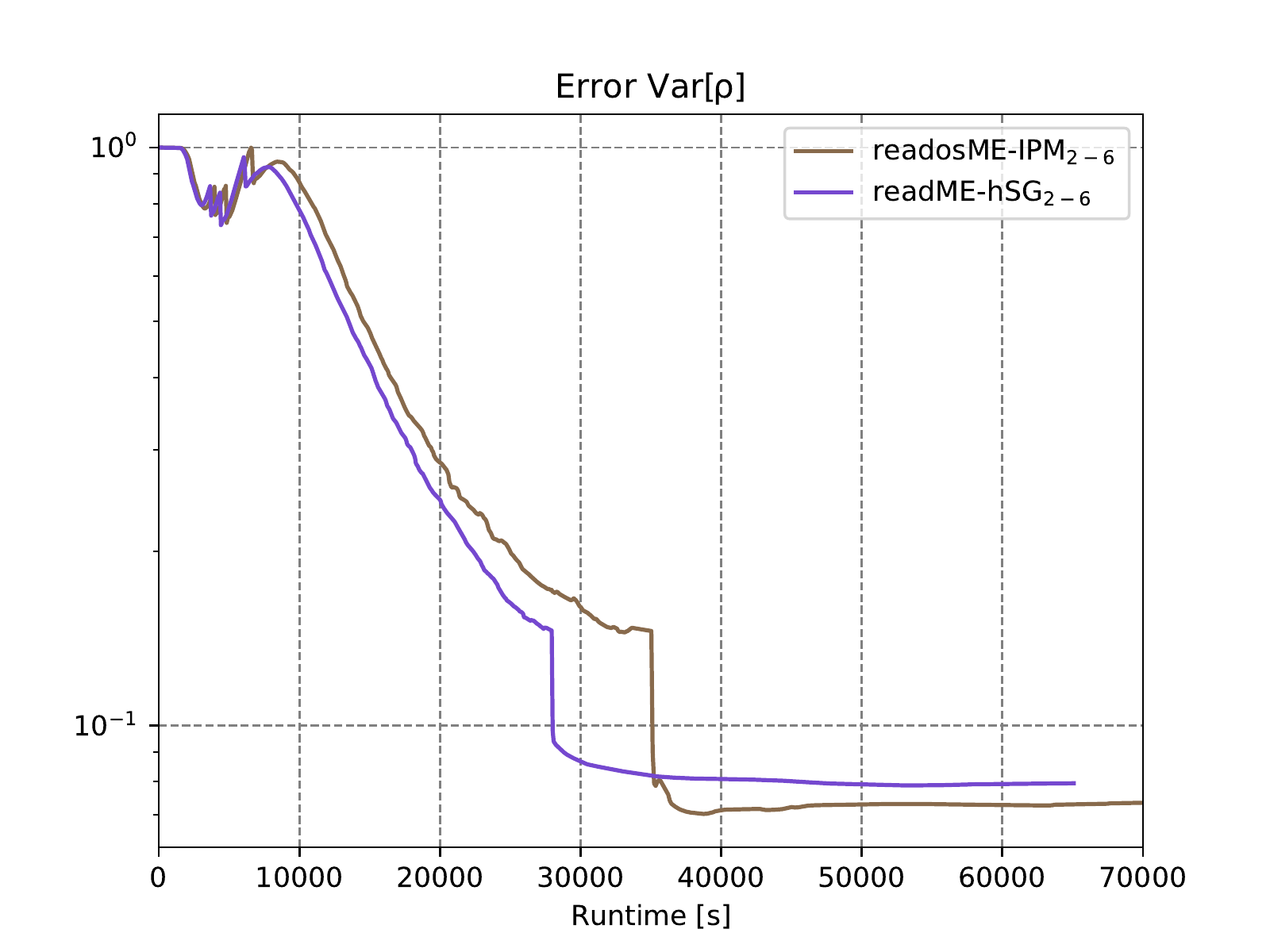}
	\end{subfigure}
	\caption{Relative errors \eqref{eq:relErrors} computed over runtime. Computations are carried out using 5 MPI threads. We use refinement retardation (re) and adaptivity (ad) to accelerate the computation while achieving an improved error level. }
	\label{fig:ConvergenceAdaptivity}
\end{figure}

Figure~\ref{fig:MEhSGNACA} shows the approximated expected value and variance by refinement retardation adaptive ME-hSG (readME-hSG) as well as the corresponding refinement level at the steady-state solution. The results of readosME-IPM look similar and are therefore left out at this point. It can be seen that the approximation does not show oscillatory behaviour, which stems from the sufficiently high number of unknowns chosen by the adaptive scheme on top of the airfoil. The approximation agrees well with the reference solution in Figure~\ref{fig:referenceSolution}. 

\begin{figure}[h!]
\centering
	\begin{subfigure}{0.329\linewidth}
		\centering
				\includegraphics[width=\linewidth]{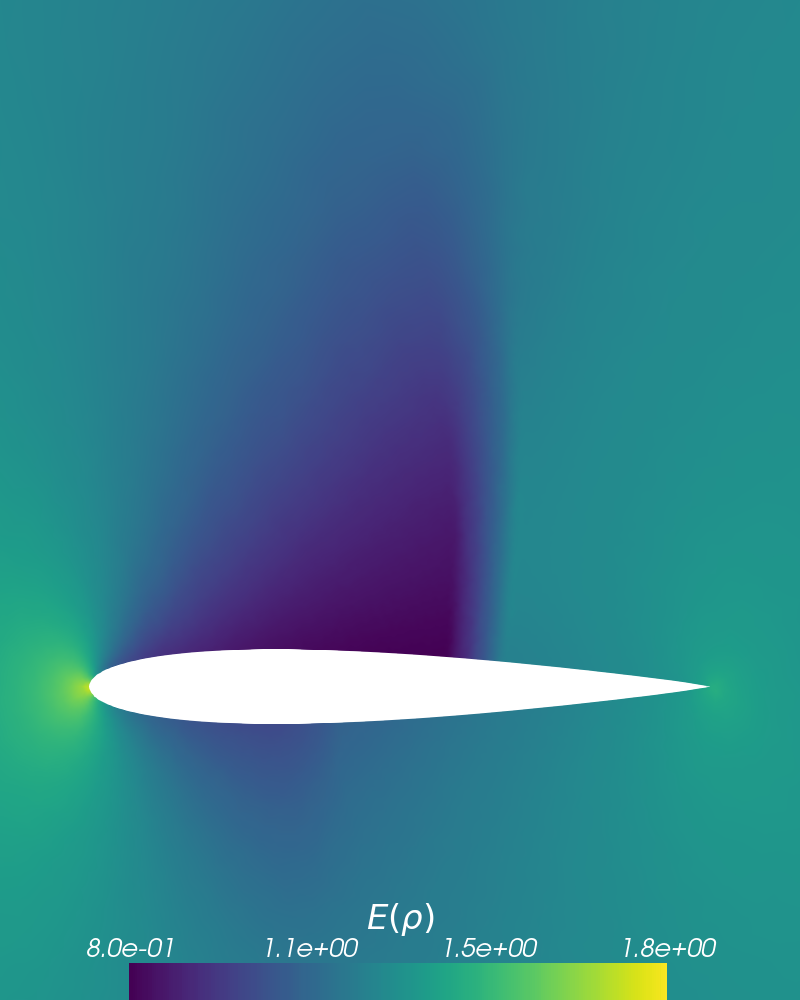}
	\end{subfigure}
	\hfill
	\begin{subfigure}{0.329\linewidth}
		\centering
				\includegraphics[width=\linewidth]{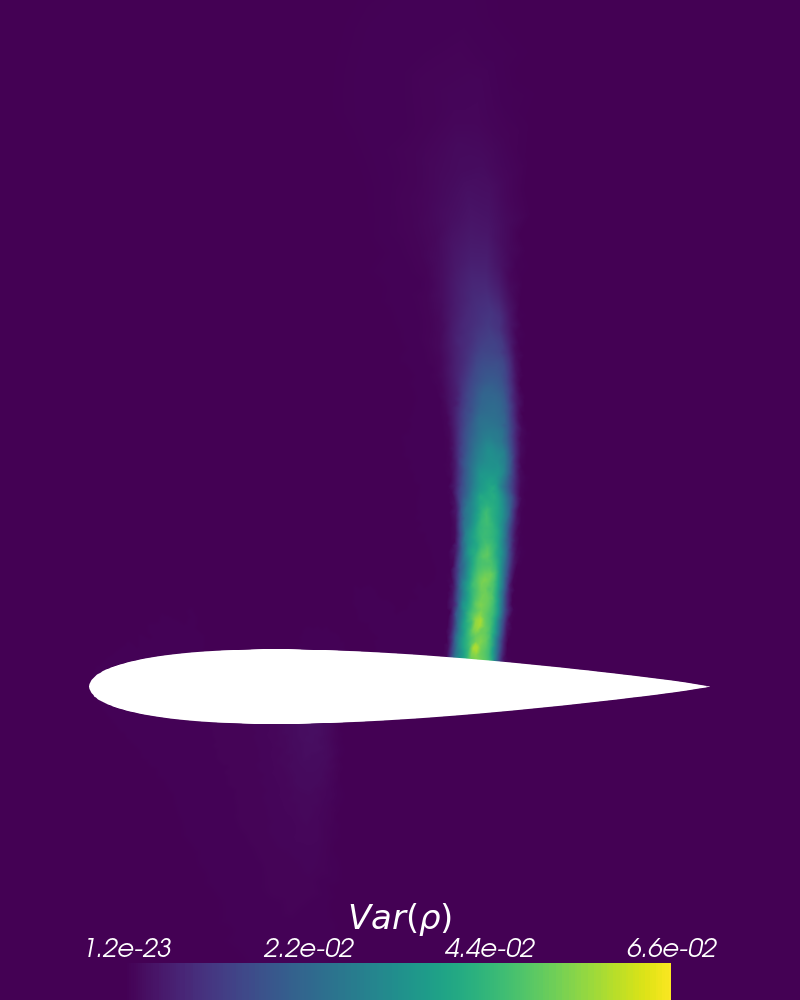}
	\end{subfigure}
	\hfill
	\begin{subfigure}{0.329\linewidth}
		\centering
				\includegraphics[width=\linewidth]{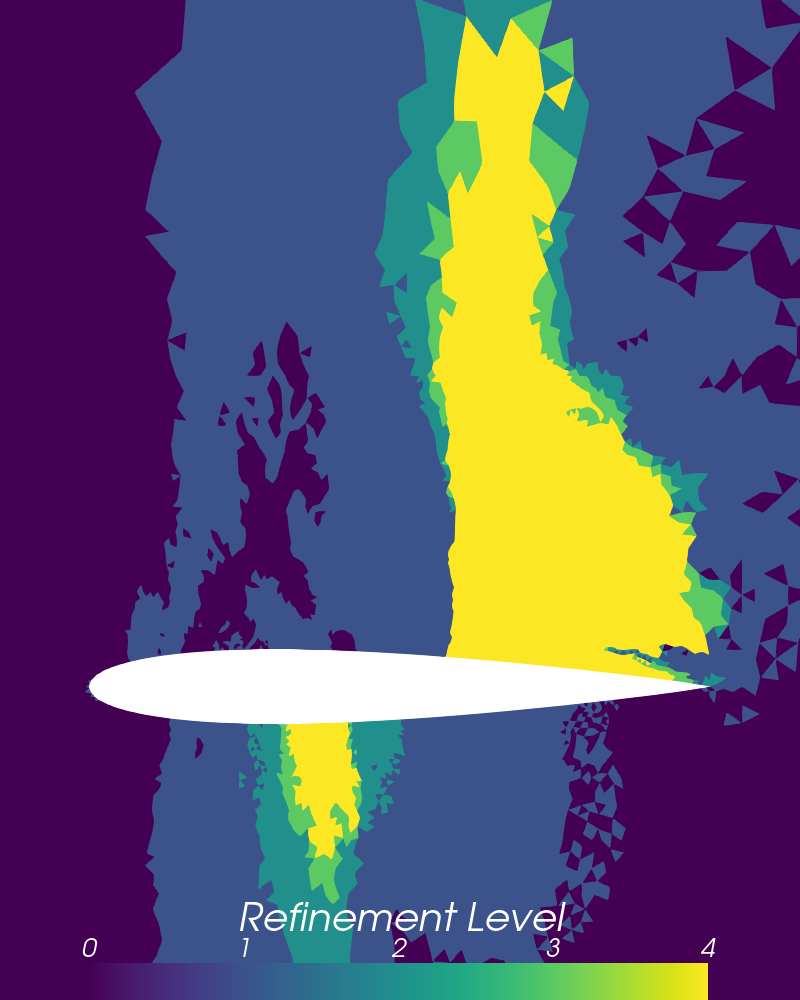}
	\end{subfigure}
	\caption{Expected density (left) and corresponding variance (right) computed with ME-hSG using refinement retardation and adaptivity. Right: Refinement level at final solution.}
	\label{fig:MEhSGNACA}
\end{figure}

We leave the use of filters for steady-state problems to future work. Here, the main issue is that the filter is applied in every pseudo-time step, which commonly leads to a heavy modification of the solution.

\subsection{Nozzle with uncertain shock}
Lastly, we investigate a nozzle with an uncertain shock. Again, the test case relies on the two-dimensional Euler equations \eqref{eq:eulerxy} inside a nozzle geometry (see Figure~\ref{fig:RefNozzle}). The nozzle is composed of a chamber on the left, a throat in the center and the main nozzle part on the right. The shock position is uniformly distributed within the throat. The shock states are identical to the one-dimensional shock tube experiment from Section~\ref{sec:Sod}, namely
\begin{eqnarray*}
\rho_{\text{IC}} &=& \begin{cases} \rho_L &\mbox{if } x < x_{\text{interface}}(\xi) \\
\rho_R & \mbox{else } \end{cases}\;, \\
(\rho v_1)_{\text{IC}} &=& 0\;, \\
(\rho v_2)_{\text{IC}} &=& 0\;, \\
(\rho e)_{\text{IC}} &=& \begin{cases} \rho_L e_L &\mbox{if } x < x_{\text{interface}}(\xi) \\
\rho_R e_R & \mbox{else } \end{cases}\;,
\end{eqnarray*}
with $x_{\text{interface}}(\xi) = x_0+\sigma \xi$, $\uncertainty\sim\mathcal{U}(-1,1)$. We choose $x_0$ to be the x-coordinate of the center of the throat. The throat ranges from $x_0-0.5$ to $x_0+0.5$, i.e., we set $\sigma = 0.5$ to ensure a uniformly distributed shock position in the entire throat section.  
We calculate the solution at an end time of $t = 4$. A reference solution, which has been computed by stochastic collocation with a 100 point Gauss-Legendre quadrature, is shown in Figure~\ref{fig:RefNozzle}. All computations are carried out on a mesh with $\ncells=225758$ triangular cells. Note that similar to Sod's shock tube we observe structures stemming from the rarefaction wave (left), the contact discontinuity (center) and the shock (right).

\begin{figure}[h!]
\centering
	\begin{subfigure}{1.0\linewidth}
		\centering
		\includegraphics[scale=0.32]{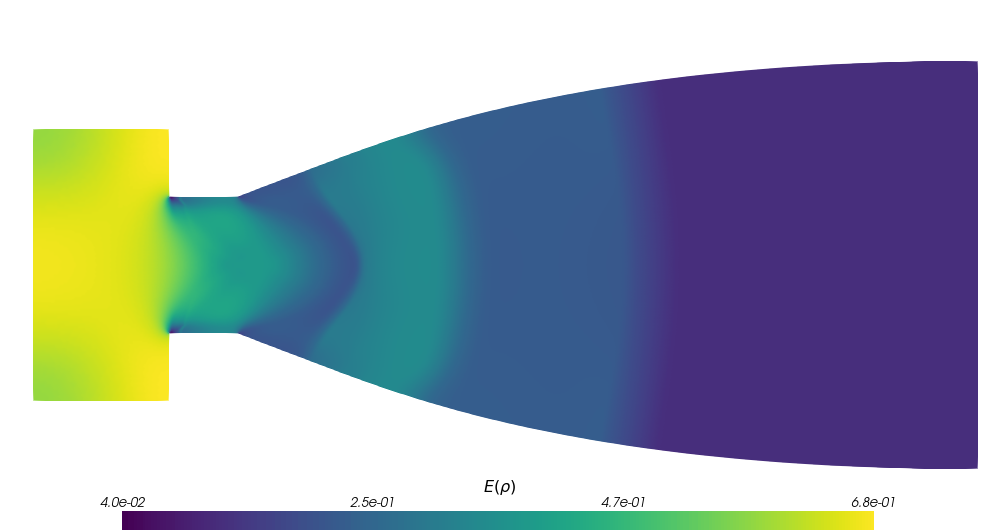}
		
		\label{fig:sub1}
	\end{subfigure}
	\begin{subfigure}{1.0\linewidth}
		\centering
		\includegraphics[scale=0.32]{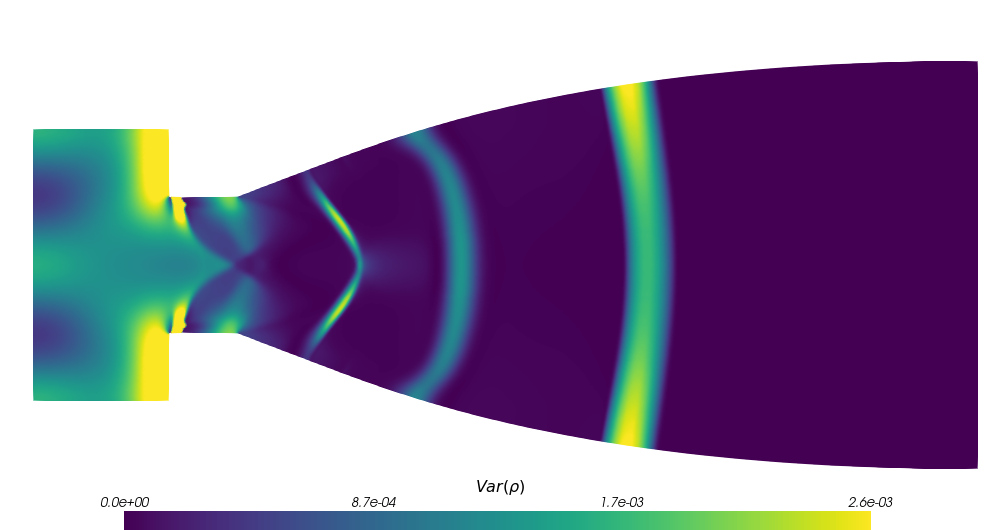}
		
		\label{fig:sub2}
	\end{subfigure}
	\caption{Reference solution for density. Top: Expected value. Bottom: Variance.}
	\label{fig:RefNozzle}
\end{figure}

We first compare the different methods without multi-elements. Here, we use gPC polynomials up to degree 5 as well as a 20 point Gauss-Legendre quadrature rule. We run simulations with hSG, fhSG and IPM. The filtered hSG method is implemented for the exponential filter with order $\alpha = 7$ and strength $\lambda=0.1$. Note that in this case, the filter parameters have again been picked by a brief parameter study.

We find the resulting expected density in Figure~\ref{fig:ExpNozzleSGIPM} and the corresponding variance in Figure~\ref{fig:VarNozzleSGIPM}. Comparing these results with the reference solution in Figure~\ref{fig:RefNozzle}, we observe that hSG as well as IPM yield a step-like expected density and an oscillatory variance. However, combined with a filter, we achieve a remarkably good approximation: While the variance is not dampened significantly at the contact discontinuity and the rarefaction wave, the variance at the shock does not oscillate and agrees well with the reference solution. Furthermore, the expected value of the filtered solution has a linear connection at the shock, which coincides with the expected value of the reference solution. 
All methods reveal errors for the variance in the throat region.

\begin{figure}[h!]
\centering
	\begin{subfigure}{1.0\linewidth}
		\centering
		\includegraphics[scale=0.32]{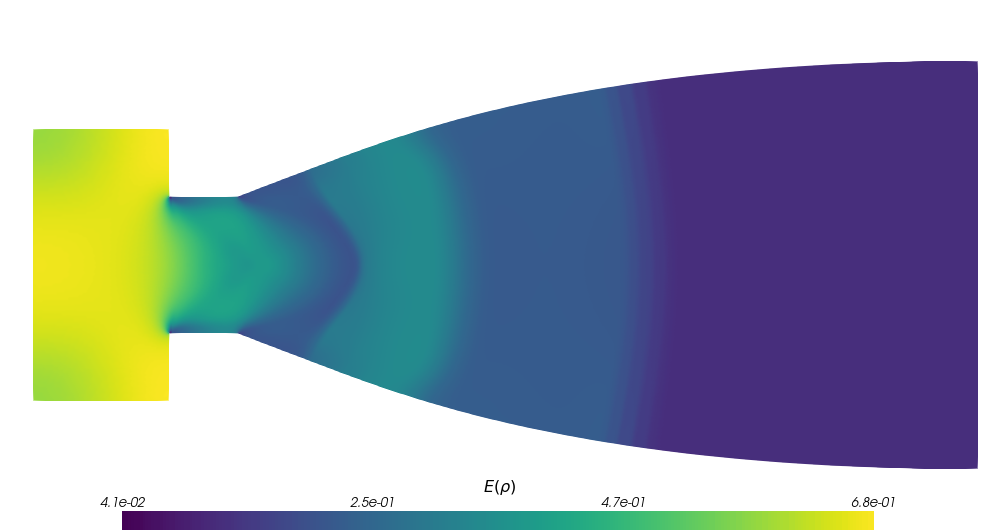}
		\label{fig:sub1}
	\end{subfigure}
	\begin{subfigure}{1.0\linewidth}
		\centering
		\includegraphics[scale=0.32]{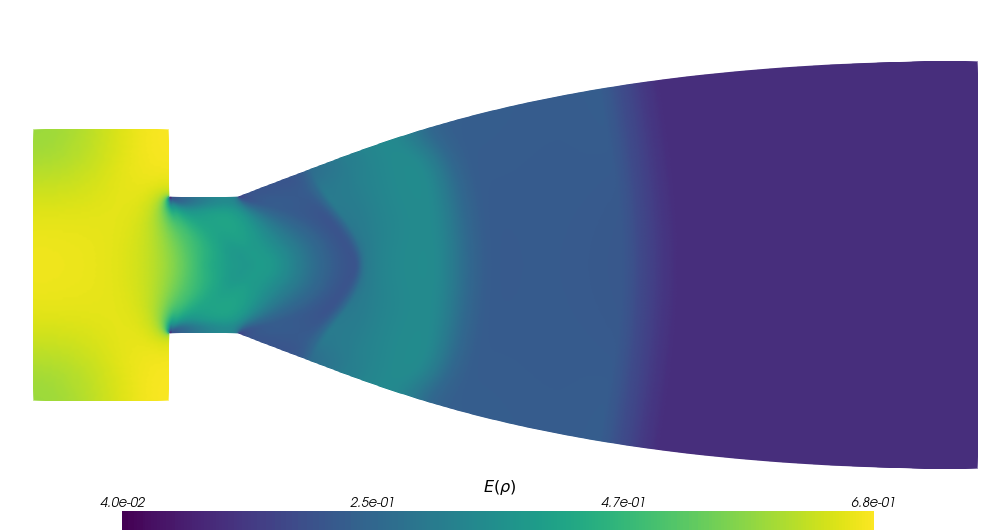}
		\label{fig:sub1}
	\end{subfigure}
	\begin{subfigure}{1.0\linewidth}
		\centering
		\includegraphics[scale=0.32]{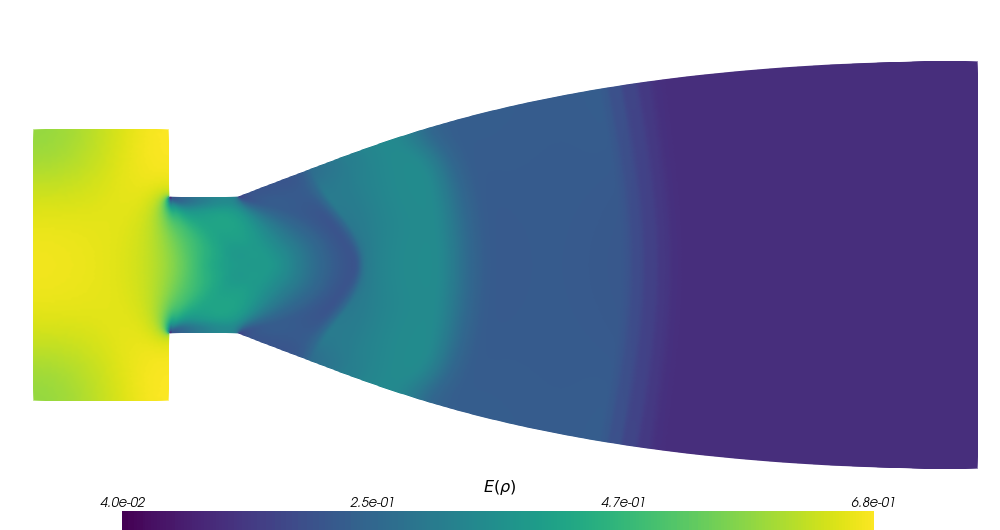}
		\label{fig:sub2}
	\end{subfigure}
	\caption{Expected density with different methods. From top to bottom: hSG, fhSG, IPM.}
	\label{fig:ExpNozzleSGIPM}
\end{figure}
\begin{figure}[h!]
\centering
	\begin{subfigure}{1.0\linewidth}
		\centering
		\includegraphics[scale=0.32]{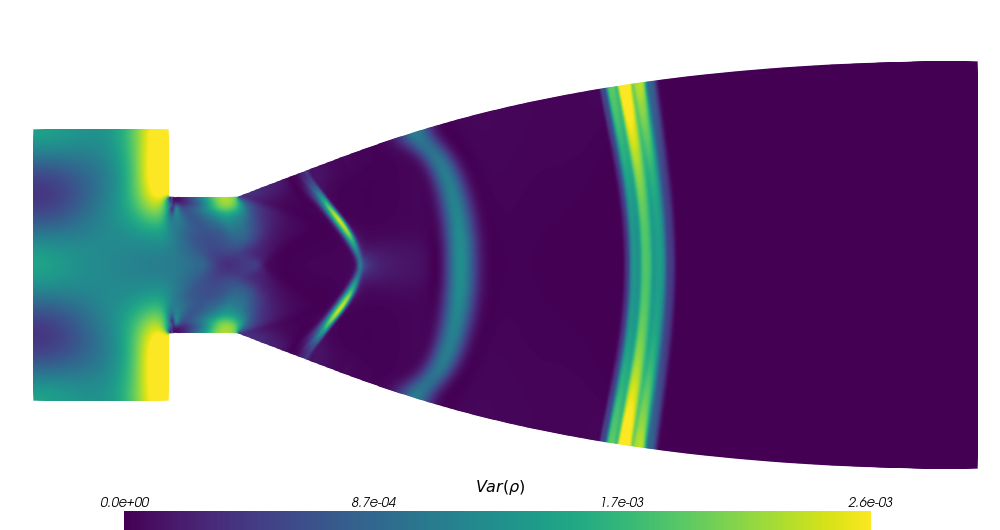}
		\label{fig:sub1}
	\end{subfigure}
	\begin{subfigure}{1.0\linewidth}
		\centering
		\includegraphics[scale=0.32]{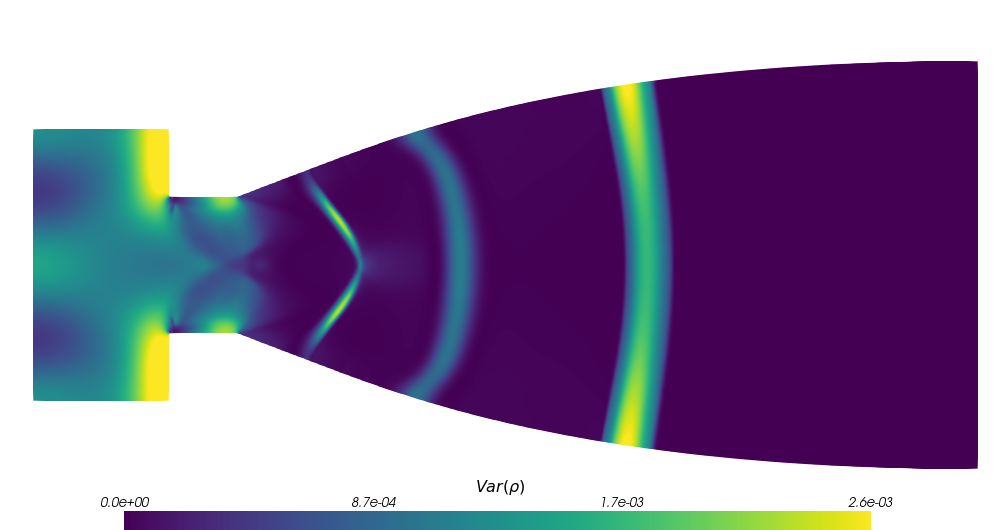}
		\label{fig:sub1}
	\end{subfigure}
	\begin{subfigure}{1.0\linewidth}
		\centering
		\includegraphics[scale=0.32]{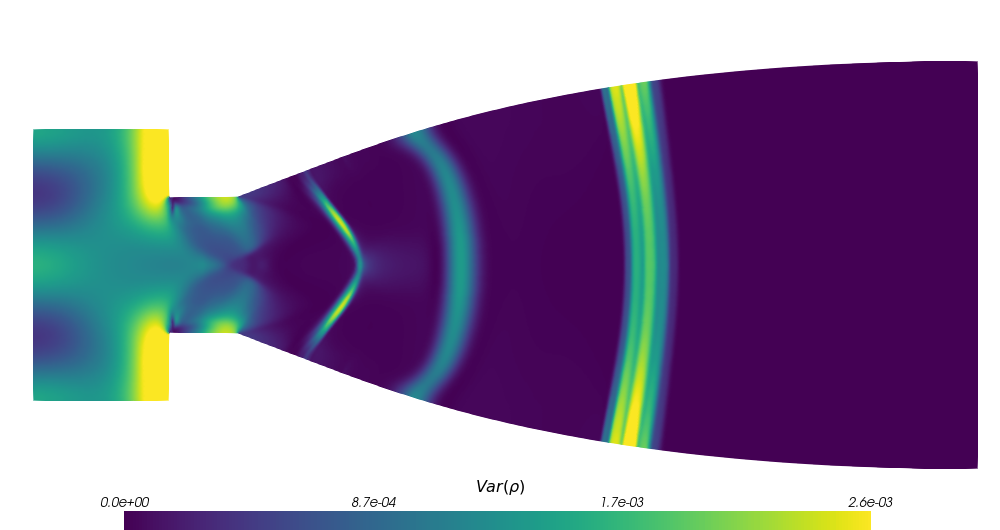}
		\label{fig:sub2}
	\end{subfigure}
	\caption{Variance of the density with different methods. From top to bottom: hSG, fhSG, IPM.}
	\label{fig:VarNozzleSGIPM}
\end{figure}

Let us now turn to the multi-element methods. Here, we employ two multi-elements with polynomials up to degree two, i.e., the number of unknowns equals the previous test case. In each multi-element we have 10 Gauss-Legendre quadrature points. 

The resulting expected values and variances of these methods are depicted in Figures~\ref{fig:ExpNozzleMESGMEIPM} and \ref{fig:VarNozzleMESGMEIPM}. We observe that ME-hSG and ME-IPM lead to good solution approximations compared to their classical counterparts. Unfortunately, we still observe a step-like profile of the expected value as well as oscillations in the variance. However, these spurious artifacts are mitigated by the multi-element ansatz.

As before, we apply the exponential filter with order $\alpha = 7$ and strength $\lambda=0.1$. When using multi-elements, the filter is able to smear out oscillations, leading to an improved approximation of the expected value and the variance. However, the effectiveness of the filter appears to be increased when applying it to the standard hSG method. Furthermore, note that all multi-element methods yield a better resolution of the variance inside the throat region.

\begin{figure}[h!]
\centering
	\begin{subfigure}{1.0\linewidth}
		\centering
		\includegraphics[scale=0.32]{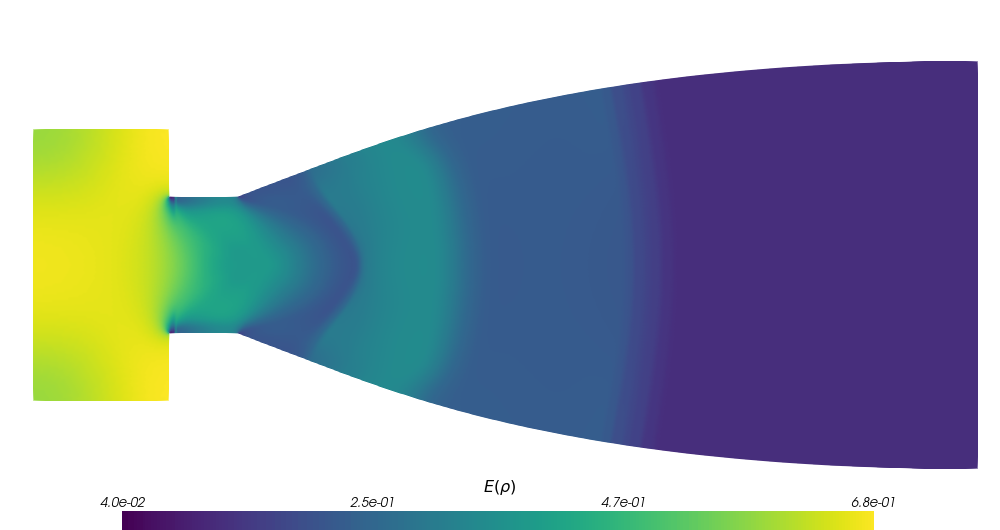}
		\label{fig:sub1}
	\end{subfigure}
	\begin{subfigure}{1.0\linewidth}
		\centering
		\includegraphics[scale=0.32]{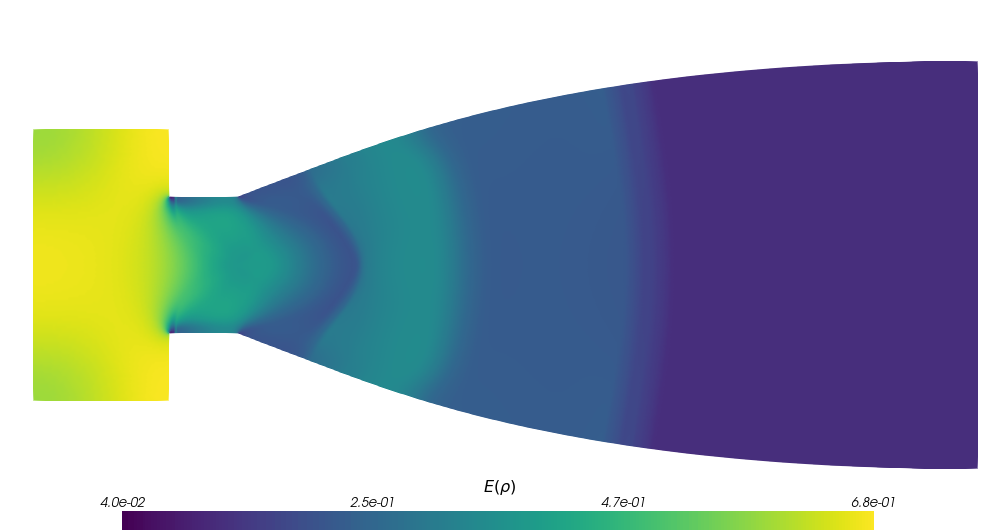}
		\label{fig:sub1}
	\end{subfigure}
	\begin{subfigure}{1.0\linewidth}
		\centering
		\includegraphics[scale=0.32]{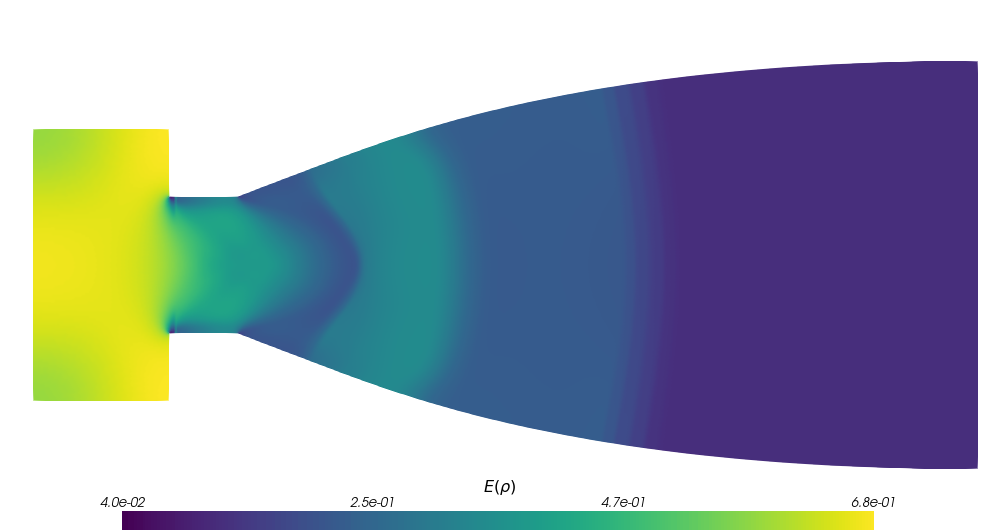}
		\label{fig:sub2}
	\end{subfigure}
	\caption{Expected density with different methods. From top to bottom: ME-hSG, ME-fhSG, ME-IPM.}
	\label{fig:ExpNozzleMESGMEIPM}
\end{figure}

\begin{figure}[h!]
\centering
	\begin{subfigure}{1.0\linewidth}
		\centering
		\includegraphics[scale=0.32]{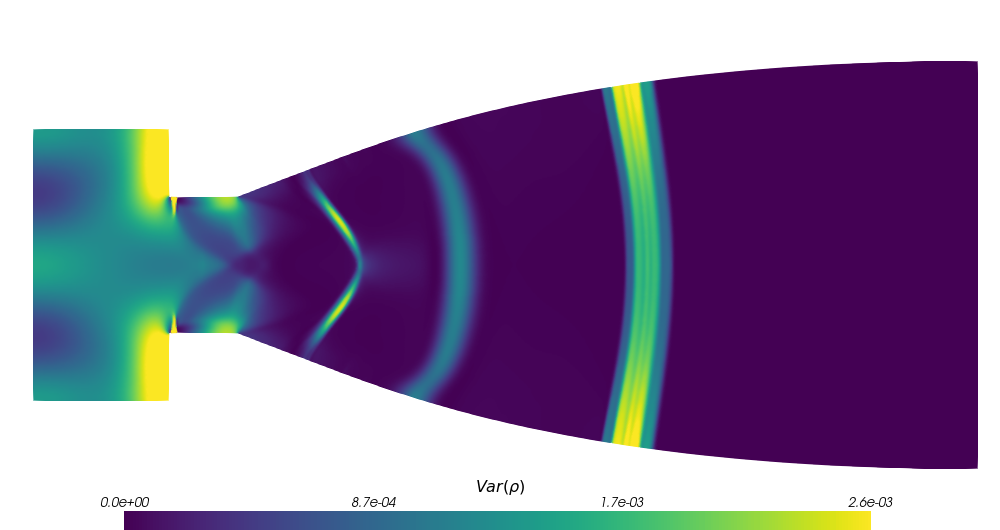}
		\label{fig:sub1}
	\end{subfigure}
	\begin{subfigure}{1.0\linewidth}
		\centering
		\includegraphics[scale=0.32]{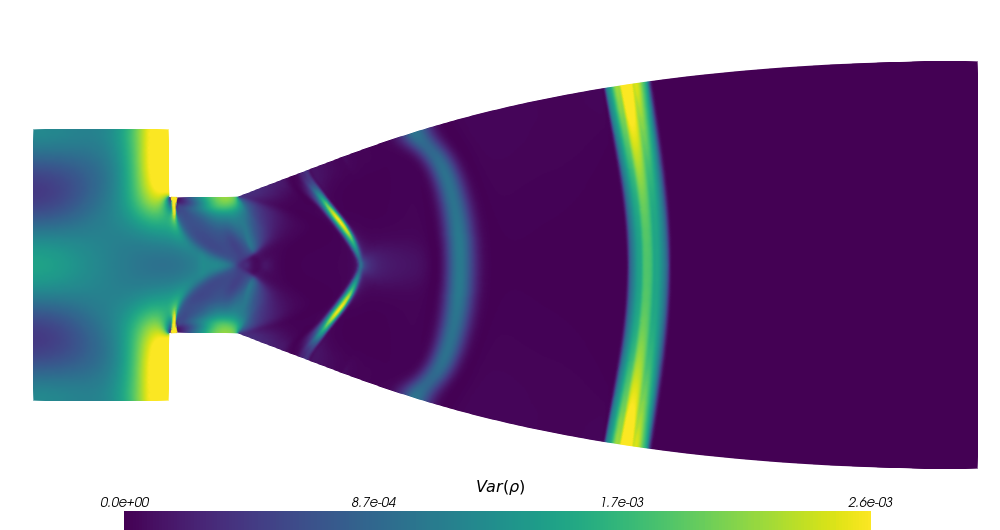}
		\label{fig:sub1}
	\end{subfigure}
	\begin{subfigure}{1.0\linewidth}
		\centering
		\includegraphics[scale=0.32]{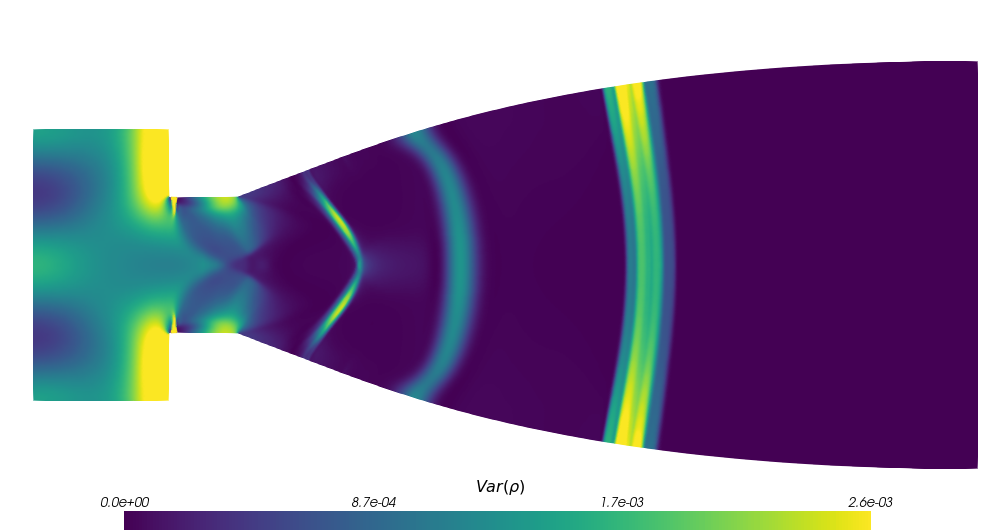}
		\label{fig:sub2}
	\end{subfigure}
	\caption{Variance of the density with different methods. From top to bottom: ME-hSG, ME-fhSG, ME-IPM.}
	\label{fig:VarNozzleMESGMEIPM}
\end{figure}
\clearpage
\section{Conclusions and Outlook}
In this work, we combined different methods to mitigate oscillations and spurious artifacts that arise from intrusive UQ methods while maintaining hyperbolicity of the moment system. To mitigate oscillations, we made use of filtering as well as a multi-element ansatz. Furthermore, hyperbolicity is ensured by using a hyperbolicity limiter as well as the IPM method.
First, we combined filtering with the hyperbolicity-preserving limiter and then extended the resulting method with the multi-element ansatz. Second, the multi-element ansatz has been applied to the IPM method. Both strategies dampen oscillations which leads to an improved approximation of expected values and variances. In addition to that, the multi-element approach in combination with IPM allows heavily reducing the required runtime to achieve a certain error level. When comparing the different methods, we observe that the optimal method choice is problem dependent. In the case of Sod's shock tube, filtering in combination with the multi-element ansatz in hSG yields the best approximation. For the nozzle test case, applying the filter to the original hSG method appears to be the method of choice as it closely captures the reference solution. 
A solely application of the (necessary) hyperbolicity limiter or the IPM closure to this test case resulted in oscillatory solution approximations. However, either the use of the multi-element ansatz or the filter significantly improves the resolution quality.
The presented NACA test case underlines the efficiency of the multi-element approach in combination with IPM. 

The different methods are so far applied for first-order numerical schemes and it would be interesting to be combined with WENO reconstructions in order to obtain high-order approximations and to further reduce the impact of Gibbs phenomenon.  Moreover, we considered numerical test cases with only one-dimensional uncertainties, however, intrusive gPC approaches are mostly used in low-dimensional random spaces and non-intrusive schemes such as stochastic Collocation are preferred otherwise.
In addition to that, the filters that we presented within this article can be used together with IPM in order to filter the coefficients of the entropic variable. This might yield a further reduction of oscillations within IPM. Since it is not clear how to achieve hyperbolicity of filtered moments, we leave this task to future work. Furthermore, the applicability of filters for steady problems should be investigated to allow the use of filters in steady state applications.

\section*{Acknowledgments}
Funding by the Deutsche Forschungsgemeinschaft (DFG) within the RTG GrK 1932 ``Stochastic Models for Innovations in the Engineering Science'' is gratefully acknowledged. Jonas Kusch has been supported by the DFG under grant FR 2841/6-1.

\bibliographystyle{siam}
\bibliography{library,bibliography}

\begin{thebibliography}{10}

\bibitem{Abgrall2007}
{\sc R.~Abgrall}, {\em {A simple, flexible and generic deterministic approach
  to uncertainty quantifications in non-linear problems: application to fluid
  flow problems}}, tech. rep., 2008.

\bibitem{Abgrall2013}
{\sc R.~Abgrall and P.~M. Congedo}, {\em {A semi-intrusive deterministic
  approach to uncertainty quantification in non-linear fluid flow problems}},
  Journal of Computational Physics, 235 (2013), pp.~828--845.

\bibitem{abgrall2017uncertainty}
{\sc R.~Abgrall and S.~Mishra}, {\em {Uncertainty quantification for hyperbolic
  systems of conservation laws}}, in Handbook of Numerical Analysis, vol.~18,
  Elsevier, 2017, pp.~507--544.

\bibitem{Alpert1993}
{\sc B.~K. Alpert}, {\em {A Class of Bases in L2 for the Sparse Representation
  of Integral Operators}}, SIAM Journal on Mathematical Analysis, 24 (1993),
  pp.~246--262.

\bibitem{Barth2012}
{\sc T.~Barth}, {\em {On the propagation of statistical model parameter
  uncertainty in CFD calculations}}, Theoretical and Computational Fluid
  Dynamics, 26 (2012), pp.~435--457.

\bibitem{barth2013non}
\leavevmode\vrule height 2pt depth -1.6pt width 23pt, {\em {Non-intrusive
  uncertainty propagation with error bounds for conservation laws containing
  discontinuities}}, in Uncertainty quantification in computational fluid
  dynamics by H. Bijl, D. Lucor, S. Mishra, C. Schwab, Springer, 2013,
  pp.~1--57.

\bibitem{boyd1996erfc}
{\sc J.~P. Boyd}, {\em {The erfc-log filter and the asymptotics of the Euler
  and Vandeven sequence accelerations}}, in Proceedings of the Third
  International Conference on Spectral and High Order Methods, Houston Math. J,
  1996, pp.~267--276.

\bibitem{boyd2001chebyshev}
{\sc J.~P. Boyd}, {\em {Chebyshev and Fourier spectral methods}}, Courier
  Corporation, 2001.

\bibitem{buerger2014hybrid}
{\sc R.~Buerger, I.~Kroeker, and C.~Rohde}, {\em A hybrid stochastic {G}alerkin
  method for uncertainty quantification applied to a conservation law modelling
  a clarifier-thickener unit}, 2014.

\bibitem{Chen2005}
{\sc Q.~Y. Chen, D.~Gottlieb, and J.~S. Hesthaven}, {\em {Uncertainty analysis
  for the steady-state flows in a dual throat nozzle}}, Journal of
  Computational Physics, 204 (2005), pp.~378--398.

\bibitem{debusschere2004numerical}
{\sc B.~J. Debusschere, H.~N. Najm, P.~P. P{\'e}bay, O.~M. Knio, R.~G. Ghanem,
  and O.~P. Le~Ma{\i}tre}, {\em Numerical challenges in the use of polynomial
  chaos representations for stochastic processes}, SIAM journal on scientific
  computing, 26 (2004), pp.~698--719.

\bibitem{deshpande1986kinetic}
{\sc S.~Deshpande}, {\em Kinetic theory based new upwind methods for inviscid
  compressible flows}, in 24th Aerospace Sciences Meeting, 1986, p.~275.

\bibitem{despres2013robust}
{\sc B.~Despr{\'e}s, G.~Po{\"e}tte, and D.~Lucor}, {\em Robust Uncertainty
  Propagation in Systems of Conservation Laws with the Entropy Closure Method},
  Springer International Publishing, 2013, pp.~105--149.

\bibitem{Meyer2019a}
{\sc J.~D{\"{u}}rrw{\"{a}}chter, T.~Kuhn, F.~Meyer, L.~Schlachter, and
  F.~Schneider}, {\em {A hyper-bolicity-preserving discontinuous stochastic
  Galerkin scheme for uncertain hyperbolic systems of equations}}, Journal of
  Computational and Applied Mathematics, 370 (2020), p.~112602.

\bibitem{friedrichs1954symmetric}
{\sc K.~O. Friedrichs}, {\em Symmetric hyperbolic linear differential
  equations}, Communications on pure and applied Mathematics, 7 (1954),
  pp.~345--392.

\bibitem{gerster2020entropies}
{\sc S.~Gerster and M.~Herty}, {\em {Entropies and symmetrization of hyperbolic
  stochastic Galerkin formulations}}, Comm. Computat. Phys., to appear,
  (2020).

\bibitem{Giles2008}
{\sc M.~B. Giles}, {\em {Multilevel Monte Carlo path simulation}}, Operations
  Research, 56 (2008), pp.~607--617.

\bibitem{Gottlieb2008}
{\sc D.~Gottlieb and D.~Xiu}, {\em {Galerkin method for wave equations with
  uncertain coefficients}}, Communications in Computational Physics, 3 (2008),
  pp.~505--518.

\bibitem{harten1983upstream}
{\sc A.~Harten, P.~D. Lax, and B.~v. Leer}, {\em On upstream differencing and
  godunov-type schemes for hyperbolic conservation laws}, SIAM review, 25
  (1983), pp.~35--61.

\bibitem{hauck2011high}
{\sc C.~D. Hauck}, {\em {High-order entropy-based closures for linear transport
  in slab geometry}}, Communications in Mathematical Sciences, 9 (2011),
  pp.~187--205.

\bibitem{Hazra2005}
{\sc S.~B. Hazra, V.~Schulz, J.~Brezillon, and N.~R. Gauger}, {\em {Aerodynamic
  shape optimization using simultaneous pseudo-timestepping}}, Journal of
  Computational Physics, 204 (2005), pp.~46--64.

\bibitem{Heinrich2001}
{\sc S.~Heinrich}, {\em {Multilevel Monte Carlo methods}}, Large-scale
  scientific computing, Third international conference LSSC 2001, Sozopol,
  Bulgaria, 2170 (2001), pp.~58--67.

\bibitem{hesthaven2007spectral}
{\sc J.~S. Hesthaven, S.~Gottlieb, and D.~Gottlieb}, {\em {Spectral methods for
  time-dependent problems}}, vol.~21, Cambridge University Press, 2007.

\bibitem{hoskins1980representation}
{\sc B.~J. Hoskins}, {\em Representation of the earth topography using
  spherical harmonies}, Monthly Weather Review, 108 (1980), pp.~111--115.

\bibitem{Jin2017c}
{\sc S.~Jin and L.~Pareschi}, eds., {\em {Uncertainty Quantification for
  Hyperbolic and Kinetic Equations}}, vol.~14 of SEMA SIMAI Springer Series,
  Springer International Publishing, Cham, 2017.

\bibitem{Kolb2018}
{\sc O.~Kolb}, {\em {A Third Order Hierarchical Basis WENO Interpolation for
  Sparse Grids with Application to Conservation Laws with Uncertain Data}},
  Journal of Scientific Computing, 74 (2018), pp.~1480--1503.

\bibitem{Kusch2017}
{\sc J.~Kusch, G.~W. Alldredge, and M.~Frank}, {\em
  {Maximum-principle-satisfying second-order Intrusive Polynomial Moment
  scheme}}, The SMAI journal of computational mathematics, 5 (2019),
  pp.~23--51.

\bibitem{Kusch2018a}
{\sc J.~Kusch and M.~Frank}, {\em {Intrusive methods in uncertainty
  quantification and their connection to kinetic theory}}, International
  Journal of Advances in Engineering Sciences and Applied Mathematics, 10
  (2018), pp.~54--69.

\bibitem{Kusch2019}
\leavevmode\vrule height 2pt depth -1.6pt width 23pt, {\em {An adaptive
  quadrature-based moment closure}}, International Journal of Advances in
  Engineering Sciences and Applied Mathematics, 11 (2019), pp.~174--186.

\bibitem{Kusch2018}
{\sc J.~Kusch, R.~G. McClarren, and M.~Frank}, {\em {Filtered Stochastic
  Galerkin Methods For Hyperbolic Equations}}, Journal of Computational
  Physics, 403 (2020).

\bibitem{uqcreator}
{\sc J.~Kusch, L.~Schlachter, and J.~Wolters}, {\em UQCreator testcases for
  "Oscillation Mitigation of Hyperbolicity-Preserving Intrusive Uncertainty
  Quantification Methods for Systems of Conservation Laws"}, 2020.
\newblock
  \url{https://git.scc.kit.edu/uqcreator/publication-oscillation-mitigation-of-intrusive-uncertainty-quantification-methods-for-hyperbolic-systems}.

\bibitem{kusch2020intrusive}
{\sc J.~Kusch, J.~Wolters, and M.~Frank}, {\em Intrusive acceleration
  strategies for uncertainty quantification for hyperbolic systems of
  conservation laws}, Journal of Computational Physics,  (2020), p.~109698.

\bibitem{levermore1996moment}
{\sc C.~D. Levermore}, {\em Moment closure hierarchies for kinetic theories},
  Journal of statistical Physics, 83 (1996), pp.~1021--1065.

\bibitem{lye2016multilevel}
{\sc K.~O. Lye}, {\em {Multilevel Monte-Carlo for measure valued solutions}},
  (2016).

\bibitem{Ma2009}
{\sc X.~Ma and N.~Zabaras}, {\em {An adaptive hierarchical sparse grid
  collocation algorithm for the solution of stochastic differential
  equations}}, Journal of Computational Physics, 228 (2009), pp.~3084--3113.

\bibitem{McClarren2010}
{\sc R.~G. McClarren and C.~D. Hauck}, {\em {Robust and accurate filtered
  spherical harmonics expansions for radiative transfer}}, Journal of
  Computational Physics, 229 (2010), pp.~5597--5614.

\bibitem{meyer2020posteriori}
{\sc F.~Meyer, C.~Rohde, and J.~Giesselmann}, {\em A posteriori error analysis
  for random scalar conservation laws using the stochastic galerkin method},
  IMA Journal of Numerical Analysis, 40 (2020), pp.~1094--1121.

\bibitem{Mishra2016}
{\sc S.~Mishra, N.~H. Risebro, C.~Schwab, and S.~Tokareva}, {\em {Numerical
  solution of scalar conservation laws with random flux functions}}, SIAM-ASA
  Journal on Uncertainty Quantification, 4 (2016), pp.~552--591.

\bibitem{Mishra2014}
{\sc S.~Mishra and C.~Schwab}, {\em {Sparse tensor multi-level Monte Carlo
  finite volume methods for hyperbolic conservation laws with random initial
  data}}, Mathematics of Computation, 81 (2012), pp.~1979--2018.

\bibitem{mishra2012multi}
{\sc S.~Mishra, C.~Schwab, and J.~{\v{S}}ukys}, {\em {Multi-level Monte Carlo
  finite volume methods for nonlinear systems of conservation laws in
  multi-dimensions}}, Journal of Computational Physics, 231 (2012),
  pp.~3365--3388.

\bibitem{mishra2013}
\leavevmode\vrule height 2pt depth -1.6pt width 23pt, {\em {Multi-level Monte
  Carlo Finite Volume Methods for Uncertainty Quantification in Nonlinear
  Systems of Balance Laws}}, Lecture Notes in Computational Science and
  Engineering, 92 (2013).

\bibitem{munz1994godunov}
{\sc C.~Munz}, {\em On godunov-type schemes for lagrangian gas dynamics}, SIAM
  Journal on Numerical Analysis, 31 (1994), pp.~17--42.

\bibitem{Nobile2008}
{\sc F.~Nobile, R.~Tempone, and C.~G. Webster}, {\em {A Sparse Grid Stochastic
  Collocation Method for Partial Differential Equations with Random Input
  Data}}, SIAM Journal on Numerical Analysis, 46 (2008), pp.~2411--2442.

\bibitem{perthame1990boltzmann}
{\sc B.~Perthame}, {\em Boltzmann type schemes for gas dynamics and the entropy
  property}, SIAM Journal on Numerical Analysis, 27 (1990), pp.~1405--1421.

\bibitem{perthame1992second}
{\sc B.~Perthame}, {\em Second-order boltzmann schemes for compressible euler
  equations in one and two space dimensions}, SIAM Journal on Numerical
  Analysis, 29 (1992), pp.~1--19.

\bibitem{pettersson2014stochastic}
{\sc P.~Pettersson, G.~Iaccarino, and J.~Nordstr{\"o}m}, {\em A stochastic
  {Galerkin} method for the {Euler} equations with {Roe} variable
  transformation}, Journal of Computational Physics, 257 (2014), pp.~481--500.

\bibitem{poette2019contribution}
{\sc G.~Po{\"e}tte}, {\em Contribution to the mathematical and numerical
  analysis of uncertain systems of conservation laws and of the linear and
  nonlinear Boltzmann equation}, PhD thesis, 2019.

\bibitem{Poette2009}
{\sc G.~Po{\"{e}}tte, B.~Despr{\'{e}}s, and D.~Lucor}, {\em {Uncertainty
  quantification for systems of conservation laws}}, Journal of Computational
  Physics, 228 (2009), pp.~2443--2467.

\bibitem{poette2011treatment}
{\sc G.~Po{\"e}tte, B.~Despr{\'e}s, and D.~Lucor}, {\em Treatment of uncertain
  material interfaces in compressible flows}, Computer Methods in Applied
  Mechanics and Engineering, 200 (2011), pp.~284--308.

\bibitem{Radice2013}
{\sc D.~Radice, E.~Abdikamalov, L.~Rezzolla, and C.~D. Ott}, {\em {A new
  spherical harmonics scheme for multi-dimensional radiation transport I.
  Static matter configurations}}, Journal of Computational Physics, 242 (2013),
  pp.~648--669.

\bibitem{radice2013new}
{\sc D.~Radice, E.~Abdikamalov, L.~Rezzolla, and C.~D. Ott}, {\em A new
  spherical harmonics scheme for multi-dimensional radiation transport i.
  static matter configurations}, Journal of Computational Physics, 242 (2013),
  pp.~648--669.

\bibitem{Schillings2017}
{\sc C.~Schillings and A.~M. Stuart}, {\em {Analysis of the ensemble Kalman
  filter for inverse problems}}, SIAM Journal on Numerical Analysis, 55 (2017),
  pp.~1264--1290.

\bibitem{Schlachter2018}
{\sc L.~Schlachter and F.~Schneider}, {\em {A hyperbolicity-preserving
  stochastic Galerkin approximation for uncertain hyperbolic systems of
  equations}}, Journal of Computational Physics, 375 (2018), pp.~80--98.

\bibitem{Schlachter2019}
{\sc L.~Schlachter, F.~Schneider, and O.~Kolb}, {\em {Weighted Essentially
  Non-Oscillatory stochastic Galerkin approximation for hyperbolic conservation
  laws}}, Journal of Computational Physics, 419 (2020), p.~109663.

\bibitem{sod1978survey}
{\sc G.~A. Sod}, {\em A survey of several finite difference methods for systems
  of nonlinear hyperbolic conservation laws}, Journal of computational physics,
  27 (1978), pp.~1--31.

\bibitem{Tryoen2010}
{\sc J.~Tryoen and A.~Ern}, {\em {Adaptive Anisotropic Spectral Stochastic
  Methods}},  (2010).

\bibitem{tryoen2012adaptive}
{\sc J.~Tryoen, O.~L. Ma{\i}tre, and A.~Ern}, {\em Adaptive anisotropic
  spectral stochastic methods for uncertain scalar conservation laws}, SIAM
  Journal on Scientific Computing, 34 (2012), pp.~A2459--A2481.

\bibitem{Bos2018}
{\sc L.~M.~M. van~den Bos, B.~Sanderse, W.~A. A.~M. Bierbooms, and G.~J.~W. van
  Bussel}, {\em {Bayesian model calibration with interpolating polynomials
  based on adaptively weighted Leja nodes}}, Communications in Computational
  Physics, 27 (2018), pp.~33--69.

\bibitem{Wan2005}
{\sc X.~Wan and G.~E. Karniadakis}, {\em {An adaptive multi-element generalized
  polynomial chaos method for stochastic differential equations}}, Journal of
  Computational Physics, 209 (2005), pp.~617--642.

\bibitem{Wan2006a}
\leavevmode\vrule height 2pt depth -1.6pt width 23pt, {\em {Long-term behavior
  of polynomial chaos in stochastic flow simulations}}, Computer Methods in
  Applied Mechanics and Engineering, 195 (2006), pp.~5582--5596.

\bibitem{Wan2006}
\leavevmode\vrule height 2pt depth -1.6pt width 23pt, {\em {Multi-Element
  Generalized Polynomial Chaos for Arbitrary Probability Measures}}, SIAM J.
  Sci. Comput., 28 (2006), pp.~901--928.

\bibitem{Wan2009}
\leavevmode\vrule height 2pt depth -1.6pt width 23pt, {\em {Error control in
  multi-element generalized polynomial chaos method for elliptic problems with
  random coefficients}}, Communications in Computational Physics, 5 (2009),
  pp.~793--820.

\bibitem{Wiener1938}
{\sc N.~Wiener}, {\em {The homogeneous chaos.}}, Amer. J. Math, 60 (1938),
  pp.~897--936.

\bibitem{Witteveen2009}
{\sc J.~A. Witteveen, A.~Loeven, and H.~Bijl}, {\em {An adaptive Stochastic
  Finite Elements approach based on Newton-Cotes quadrature in simplex
  elements}}, Computers and Fluids, 38 (2009), pp.~1270--1288.

\bibitem{Wu2017}
{\sc K.~Wu, H.~Tang, and D.~Xiu}, {\em {A stochastic Galerkin method for
  first-order quasilinear hyperbolic systems with uncertainty}}, Journal of
  Computational Physics, 345 (2017), pp.~224--244.

\bibitem{Xiu2005}
{\sc D.~Xiu and J.~S. Hesthaven}, {\em {High-order collocation methods for
  differential equations with random inputs}}, 27 (2005), pp.~1118--1139.

\bibitem{Xiu2003}
{\sc D.~Xiu and G.~E. Karniadakis}, {\em {The Wiener-Askey polynomial chaos for
  stochastic differential equations}}, SIAM Journal on Scientific Computing, 24
  (2003), pp.~619--644.

\bibitem{Yan2019}
{\sc L.~Yan and T.~Zhou}, {\em {Adaptive multi-fidelity polynomial chaos
  approach to Bayesian inference in inverse problems}}, Journal of
  Computational Physics, 381 (2019), pp.~110--128.

\end{thebibliography}
\end{document}